\documentclass[11pt]{amsart}

\pdfoutput=1

\usepackage{amsmath,amsthm, amssymb, amsfonts}
\usepackage[mathscr]{eucal}
\usepackage{graphicx}
\usepackage[usenames,dvipsnames]{color}
\usepackage{subfigure}

\usepackage{pgf}
\usepackage{tikz}
\usetikzlibrary{arrows,automata}
\usepackage[latin1]{inputenc}

\usepackage[colorlinks=true, urlcolor=NavyBlue, linkcolor=NavyBlue, citecolor=NavyBlue, pdftitle={Sutured Legendrian Invariants and Invariants of Open Contact Manifolds}, pdfauthor={John B. Etnyre, David Shea Vela-Vick, Rumen Zarev}, pdfsubject={Transverse Invariants}, pdfkeywords={Legendrian Links, transverse Links, Heegaard Floer homology, 57M27, 57R58}]{hyperref}

\newtheorem{theorem}{Theorem}[section]

\newtheorem{question}{Question}

\newtheorem{othertheorem}[theorem]{Theorem}
\newtheorem{proposition}[theorem]{Proposition}

\newtheorem{lemma}[theorem]{Lemma}

\theoremstyle{definition}
\newtheorem{definition}[theorem]{Definition}

\theoremstyle{definition}
\newtheorem{remark}[theorem]{Remark}

\newcommand{\mr}[1]{\mathrm{#1}}
\newcommand{\mscr}[1]{\mathscr{#1}}

\newcommand{\xib}{\overline{\xi}}

\newcommand{\HFKH}{\widehat{\mathrm{HFK}}}
\newcommand{\CFKH}{\widehat{\mathrm{CFK}}}
\newcommand{\HFKM}{\mathrm{HFK}^-}
\newcommand{\HFKP}{\mathrm{HFK}^+}
\newcommand{\CFKP}{\mathrm{CFK}^+}
\newcommand{\CFKM}{\mathrm{CFK}^-}
\newcommand{\HFKI}{\mathrm{HFK}^\infty}
\newcommand{\CFKI}{\mathrm{CFK}^\infty}
\newcommand{\HFH}{\widehat{\mathrm{HF}}}
\newcommand{\CFH}{\widehat{\mathrm{CF}}}
\newcommand{\CF}{\mathrm{CF}}
\newcommand{\CFA}{\widehat{\mathrm{CFA}}}
\newcommand{\CFD}{\widehat{\mathrm{CFD}}}
\newcommand{\CFDM}{\mathrm{CFD}^-}
\newcommand{\CFDP}{\mathrm{CFD}^+}
\newcommand{\BSD}{\widehat{\mathrm{BSD}}}
\newcommand{\BSA}{\widehat{\mathrm{BSA}}}
\newcommand{\BSDA}{\widehat{\mathrm{BSDA}}}
\newcommand{\spinc}{\mathrm{Spin}^{\mathbb{C}}}

\newcommand{\id}{\mathrm{Id}}

\newcommand{\PM}{\partial^-}
\newcommand{\PH}{\widehat{\partial}}

\newcommand{\STH}{\widehat{\mathscr{T}}}
\newcommand{\ST}{\mathscr{T}}
\newcommand{\SLH}{\widehat{\mathscr{L}}}
\newcommand{\SL}{\mathscr{L}}
\newcommand{\EH}{\mathrm{EH}}
\newcommand{\EHL}{\underrightarrow{\mathrm{EH}}}
\newcommand{\EHIL}{\underleftarrow{\mathrm{EH}}}

\newcommand{\SFH}{\mathrm{SFH}}
\newcommand{\SFC}{\mathrm{SFC}}
\newcommand{\SFHL}{\underrightarrow{\mathrm{SFH}}}
\newcommand{\SFHIL}{\underleftarrow{\mathrm{SFH}}}

\newcommand{\tb}{\mathrm{tb}}
\newcommand{\tw}{\mathrm{tw}}
\newcommand{\rot}{\mathrm{r}}

\def\x{\mathbf{x}}
\def\y{\mathbf{y}}

\def\brho{\mathbf{\rho}}

\def\s{\mathfrak{s}}

\newcommand{\A}{\mathcal{A}}
\newcommand{\B}{\mathcal{B}}
\newcommand{\C}{\mathcal{C}}
\newcommand{\HD}{\mathcal{H}}
\newcommand{\THD}{\widetilde{\mathcal{H}}}
\newcommand{\W}{\mathcal{W}}
\newcommand{\D}{\mathcal{D}}
\newcommand{\KF}{K_{\mathrm{fill}}}
\newcommand{\Y}{\mathcal{Y}}
\newcommand{\SF}{\mathcal{F}}
\newcommand{\SZ}{\mathcal{Z}}
\newcommand{\SI}{\mathcal{I}}
\newcommand{\SG}{\mathcal{G}}
\newcommand{\T}{\mathcal{T}}
\newcommand{\TT}{\mathbb{T}}
\newcommand{\TF}{\mathcal{T}_{\mathrm{fill}}}
\newcommand{\SH}{\mathcal{H}}
\newcommand{\F}{\mathbb{F}}
\newcommand{\Q}{\mathbb{Q}}
\newcommand{\R}{\mathbb{R}}
\newcommand{\Z}{\mathbb{Z}}
\newcommand{\I}{\mathbb{I}}

\newcommand{\balpha}{\boldsymbol{\alpha}}
\newcommand{\bbeta}{\boldsymbol{\beta}}

\newcommand{\ba}{\boldsymbol{a}}
\newcommand{\BZ}{\boldsymbol{Z}}

\begin{document}

\thispagestyle{empty}

\title[Legendrian and transverse invariants]{Sutured Floer homology and invariants of Legendrian and transverse knots}
\author{John B. Etnyre}
\address{School of Mathematics \\ Georgia Institute of Technology}
\email{etnyre@math.gatech.edu}
\urladdr{\href{http://www.math.gatech.edu/~etnyre}{http://www.math.gatech.edu/\~{}etnyre}}
\author{David Shea Vela-Vick}
\address{Department of Mathematics \\ Louisiana State University}
\email{shea@math.lsu.edu}
\urladdr{\href{https://www.math.lsu.edu/~shea}{https://www.math.lsu.edu/\~{}shea}}
\author{Rumen Zarev}
\address{Department of Mathematics \\ University of California, Berkeley}
\email{rzarev@math.berkeley.edu}
\urladdr{\href{http://math.berkeley.edu/~rzarev}{http://math.berkeley.edu/\~{}rzarev}}

\keywords{Legendrian knots, Transverse knots, Heegaard Floer homology}
\subjclass[2010]{57M27; 57R58}
\maketitle


\begin{abstract}
	Using contact-geometric techniques and sutured Floer homology, we present an alternate formulation of the minus and plus version of knot Floer homology.  We further show how natural constructions in the realm of contact geometry give rise to much of the formal structure relating the various versions of Heegaard Floer homology.  In addition, to a Legendrian or transverse knot $K \subset (Y,\xi)$, we associate distinguished classes $\EHL(K) \in \HFKM(-Y,K)$ and $\EHIL(K) \in \HFKP(-Y,K)$, which are each invariant under Legendrian or transverse isotopies of $K$.  The distinguished class $\EHL$ is shown to agree with the Legendrian/transverse invariant defined by Lisca, Ozsv\'ath, Stipsicz, and Szab\'o despite a strikingly dissimilar definition. While our definitions and constructions only involve sutured Floer homology and contact geometry, the identification of our invariants with known invariants uses bordered sutured Floer homology to make explicit computations of maps between sutured Floer homology groups.
\end{abstract}

\tableofcontents

\section{Introduction} 
\label{sec:intro}

In \cite{Ju2} Juh\'asz defined the sutured Heegaard Floer homology $\SFH(Y, \Gamma)$ of a balanced sutured manifold $(Y,\Gamma)$ and immediately observed that if $Y(K) = \overline{Y \backslash \nu(K)}$ was the complement of an open tubular neighborhood of a knot $K$ in the manifold $Y$ and $\Gamma_\mu$ was the union of two meridional curves, then $\SFH(Y(K), \Gamma_\mu)$ was isomorphic to the knot Floer homology of $K$, $\HFKH(Y,K)$. A primary aim in this paper is to show how to recover more of the knot Floer homology package from the sutured theory. More specifically, we will show that given a knot $K \subset Y$ we can define its Heegaard Floer-theoretic invariants purely in terms of sutured Floer homology, contact geometry, and certain direct and inverse limits. These invariants share many properties of the knot Floer homology package and in the second part of the paper, using border sutured homology, we show how to identify these limit invariants with the plus and minus knot Floer homologies. 

The original motivation for the present study is found in the work of Stipsicz and V\'ertesi who first established a connection between the Legendrian knot invariant defined by Honda, Kazez, and Mati\'c \cite{HKM2} and the Legendrian/transverse invariant defined by Lisca, Ozsv\'ath, Stipsicz and Szab\'o \cite{LOSS}, hereafter referred to as the LOSS invariant.  Their work naturally gives rise to an alternate, and more geometric characterization of the LOSS hat invariant.  We show here that the correspondence first established by Stipsicz and V\'ertesi fits into a much broader picture encompassing the more general LOSS minus invariant.

Accomplishing the broad goals described in the paragraphs above requires precise computations of the Honda-Kazez-Mati\'c gluing maps for sutured Floer homology in a multitude of nontrivial situations.  To date, only elementary computations, typically relying on formal properties of the HKM gluing maps have been performed.  Such precision is achieved through tools and techniques originating in bordered Floer homology \cite{LOT1} and, specifically, bordered sutured Floer homology \cite{Za2,Za3}, as developed by the third author.

We note that contact geometry plays a key role in our results, adding to a steady stream of evidence that there exist deep connections linking contact geometry and Heegaard Floer theory.  On one hand, Heegaard Floer invariants have proven powerful tools for studying contact-geometric phenomena.  They were instrumental in Lisca and Stipsicz's classification of Seifert-fibered spaces admitting tight contact structures, and have featured prominently in the study of transversally non-simple knot and link types.  In the other direction, contact structures have conspicuously appeared in solutions to several problems in Heegaard Floer theory.  In addition to appearing in Honda, Kazez, and Mati\'c's definition of a Heegaard Floer gluing map, Juh\'asz uses them in an essential way in his construction of cobordism maps for sutured Floer homology. The proof of our results hint at what might be behind this connection: when considering relatively simple manifolds, the rigidity of the algebraic structure in bordered sutured Floer homology, coupled with known properties of contact structures and their induced gluing maps can sometime uniquely determine a given situation. 

In the remainder of the introduction, we provide a more thorough discussion of the geometric and algebraic objects under consideration and statements of the main theorems to be proved in subsequent sections. Here, as in there rest of the paper, we will focus our attention on the direct limit invariants and only sketch the ideas behind the inverse limit invariants since their definition and the proofs of their properties parallel those of the direct limit invariants quite closely.

\subsection{Limit Invariants} 
\label{sub:limit_invariants}

Let $K$ be a knot in a closed 3--manifold $Y$. We denote the knot compliment $\overline{Y(K)\backslash \nu(K)}$. We consider a sequence of pairs of longitudinal sutures $\Gamma_i$ on $\partial Y(K)$ that ``converge'' the union of two meridional curves $\Gamma_\mu$ on $\partial Y(K)$. More precisely the curves that make up $\Gamma_{i+1}$ differ form those that make up $\Gamma_i$ by subtracting a meridian. We can think of $(Y(K),\Gamma_{i})$ as a subset of $(Y(K),\Gamma_{i+1})$ so that $B_i = \overline{(Y(K),\Gamma_{i+1}) \backslash (Y(K),\Gamma_{i})}$ is $T^2\times [0,1]$. On $T^2\times[0,1]$ there are (up to fixing the characteristic foliations on the boundary) two contact structures $\xi_+$ and $\xi_-$ for which the boundary is convex and the dividing curves agree with the sutures. 

In \cite{HKM3}, Honda, Kazez and Mati\'c defined a gluing map for sutured Floer homology.  Loosely speaking, if $(M,\Gamma)$ is a sutured manifold which sits as a submanifold of $(M',\Gamma')$, then a contact structure $\xi$ on the complement $M' - M$ which is compatible with $\Gamma$ and $\Gamma'$ induces a map
 \[
 	\phi_{\xi}: \SFH(-M,-\Gamma) \to \SFH(-M',-\Gamma').
 \]
Thus, using the contact structure $\xi_-$ on on $T^2\times [0,1]$, we have the induced gluing map
\[
	\phi_-: \SFH(-Y(K),-\Gamma_i)\to \SFH(-Y(K),-\Gamma_{i+1}),
\]
for each $i$. 

Taking the directed limit of the above sequence of groups and maps yields our primary object of study, the {\it sutured limit homology} of $K$
\[
	\SFHL(-Y,K) = \varinjlim \SFH(-Y(K_i),-\Gamma_{i}).
\]

Now, considering the contact structure $\xi_+$ on $B_i$, we obtain maps
\[
	\psi_+: \SFH(-Y(K),-\Gamma_i)\to \SFH(-Y(K),-\Gamma_{i+1}).
\]
Using simple facts concerning contact structures on thickened tori we will show that they induce a well-defined map
\[
	\Psi: \SFHL(-Y,K) \to \SFHL(-Y,K).
\]
Thus, the group $\SFHL(-Y,K)$ can be given the structure of an $\F[U]$-module, were $U$ acts by $\Psi$.  

We further show that the $\F[U]$-module $\SFHL(-Y,K)$ is endowed with two natural absolute gradings, which are reminiscent of the usual absolute Alexander and Maslov gradings in knot Floer homology.

In Section~\ref{sec:limit_is_knot}, we prove the following theorem characterizing $\SFHL(-Y,K)$.
\begin{theorem}\label{thm:lim_to_minus}
	Let $K \subset Y$ be a smoothly embedded null-homologous knot. There exists a isomorphism of graded $\F[U]$-modules
	\[
		I_-: \SFHL(-Y,L) \to \HFKM(-Y,L).
	\]
\end{theorem}
\begin{remark}
	Marco Golla \cite{Golla} has obtained results similar to Theorems \ref{thm:lim_to_minus} and \ref{thm:lim_to_minus_leg} for Legendrian knots in the standard contact $3$-sphere via an alternate characterization of the maps induced on sutured Floer homology by bypass attachments. This characterization involves holomorphic triangle counts originally developed by Rasmussen.  Golla is further able to show, in the $S^3$ setting, that the HKM invariant of a Legendrian knot $K$ is determined by the pair of LOSS invariants $\{\SL(K),\SL(-K)\}$ of $K$ and its orientation reverse $-K$. Examples discussed in Section~\ref{sec:examples} show this is not true in general contact manifolds. 
\end{remark}

Each $(Y(K),\Gamma_i)$ can be viewed as a subset of $(Y(K),\Gamma_\mu)$ so that $(Y(K),\Gamma_\mu) \backslash (Y(K),\Gamma_i)$ is $T^2\times [0,1]$. As above, there are two possible tight contact structures $\xi_+$ and $\xi_-$ on $T^2\times[0,1]$ with convex boundary realizing $\Gamma\cup \Gamma_i$ as dividing sets.  Choosing $\xi_-$, the HKM gluing map gives 
\[
 	\phi'_{SV}: \SFH(-Y(K),-\Gamma_i)\to \SFH(-Y(K),-\Gamma_\mu).
\] 
In \cite{Ju}, $\SFH(-Y(K),-\Gamma_\mu)$ was canonically identified with $\HFKH(-Y,K)$. Using this identification we obtain the map
\[
\phi_{SV}:\SFH(-Y(K),-\Gamma_i)\to \HFKH(-Y,K),
\]
we use the subscript $SV$ as these maps were originally defined by Stipsicz and V\'ertesi in \cite{StV}. 

We can again appeal to facts about decompositions of contact structures on thickened tori to show the maps $\phi_{SV}$  induce a map
\[
	\Phi_{SV} : \SFHL(-Y,K) \to \HFKH(-Y,K).
\]
With respect to the isomorphism given in Theorem~\ref{thm:lim_to_minus}, we have the following characterization of $\Phi_{SV}$.
\begin{theorem}\label{thm:SV_map}
	Let $K \subset Y$ be a null-homologous knot type, $I_-: \SFHL(-Y,K) \to \HFKM(-Y,K)$ the isomorphism given by Theorem~\ref{thm:lim_to_minus}, and $p_*: \HFKM(-Y,K) \to \HFKH(-Y,K)$ the map induced on homology by setting the formal variable $U$ equal to zero at the chain level.  The following diagram commutes.
\begin{center}
\begin{tikzpicture}	[->,>=stealth',auto,thick]
	\node (a) at (0,0){$\SFHL(-Y,K)$} ;
	\node (b) at (4,0) {$\HFKM(-Y,K)$} ;
	\node (e) at (2,-2) {$\HFKH(-Y,K)$} ;

	\draw (a) edge node[above] {$I_-$} (b);
	\draw (a) edge node[left] {$\Phi_{SV}$} (e);
	\draw (b) edge node[right] {$p_*$} (e);
\end{tikzpicture}
\end{center}
\end{theorem}

There is an additional natural geometric operation one can perform to the sequence $(Y(K),\Gamma_i)$.  Specifically, one can consider the effect of attaching a meridional contact 2-handle to the boundary of each $(Y(K), \Gamma_i)$.  As a sutured manifold, this space is equal to $Y(1)$ in the language of Juh\'asz \cite{Ju} and its sutured Floer homology can be naturally identified with $\HFH(-Y)$.  By considering the HKM gluing maps associated to this sequence of contact 2-handle attachments, for each $i$, we obtain a sequence of maps $\phi_{\mathrm{2h}}:\SFH(-Y(K),-\Gamma_i)\to \HFH(-Y)$ that induce a map
\[
	\Phi_{2h}: \SFHL(-Y,K) \to \HFH(-Y).
\]
With respect to the isomorphism given in Theorem~\ref{thm:lim_to_minus}, we have the following characterization of $\Phi_{2h}$.
\begin{theorem}\label{thm:2handle}
	Let $K \subset Y$ be a null-homologous knot type, $I_-: \SFHL(-Y,K) \to \HFKM(-Y,K)$ the isomorphism given by Theorem~\ref{thm:lim_to_minus}, and $\pi_*: \HFKM(-Y,K) \to \HFH(-Y)$ the map induced on homology by setting the formal variable $U$ equal to the identity at the chain level.  The following diagram commutes.
\begin{center}
\begin{tikzpicture}	[->,>=stealth',auto,thick]
	\node (a) at (0,0){$\SFHL(-Y,K)$} ;
	\node (b) at (4,0) {$\HFKM(-Y,K)$} ;
	\node (e) at (2,-2) {$\HFH(-Y)$} ;

	\draw (a) edge node[above] {$I_-$} (b);
	\draw (a) edge node[left] {$\Phi_{\mathrm{2h}}$} (e);
	\draw (b) edge node[right] {$\pi_*$} (e);
\end{tikzpicture}
\end{center}
\end{theorem}


\subsection{Legendrian and Transverse Invariants} 
\label{sub:int_lnt_invts}

In 2007, Honda, Kazez, and Mati\'c defined an invariant of Legendrian knots taking values in sutured Floer homology \cite{HKM2}.  Given a Legendrian knot $K \subset (Y,\xi)$, the Honda-Kazez-Mati\'c invariant --- henceforth referred to as the HKM invariant --- is obtained via the following construction.  First, remove an open standard neighborhood of $K$ from $(Y,\xi)$ and denote the resulting space $(Y(K),\xi_K)$.  The HKM invariant is then equal to the contact invariant
\[
	\EH(K) = \EH(Y(K),\xi_K) \in \SFH(-Y(K),-\Gamma_K),
\]
where the set of sutures $\Gamma_K$ is equal to the natural dividing set obtained on torus boundary $\partial Y(K)$.

Later, in 2008, Lisca, Ozsv\'ath, Stipsicz, and Szab\'o defined an alternate invariant of both Legendrian and transverse knots taking values in knot Floer homology.  Given a null-homologous Legendrian knot $K \subset (Y,\xi)$, the Lisca-Ozsv\'ath-Stipsicz-Szab\'o invariant --- henceforth referred to as the LOSS invariant --- is obtained via an open book decomposition adapted to the knot $K$.  Ultimately, their construction yields two invariants
\[
	\SL(K) \in \HFKM(-Y,K) \;\;\;\;\; \text{and} \;\;\;\;\; \SLH(K) \in \HFKH(-Y,K),
\]
which take values in either the minus or hat version of knot Floer homology.

The LOSS invariants possess several features which distinguish them from the HKM invariants.  First, they take values in knot Floer homology and come in two flavors, ``minus'' and ``hat''.  More strikingly, unlike the HKM invariants, the LOSS invariants are unchanged by negative Legendrian stabilization.  This implies that $\SL$ and $\SLH$ define transverse invariants through a process known as Legendrian approximation.

A connection between the HKM and LOSS invariants was discovered by Stipsicz and V\'ertesi in \cite{StV}.  Given a null-homologous Legendrian knot $K \subset (Y,\xi)$, Stipsicz and V\'ertesi identify a natural contact geometric construction which ultimately yields a map
\[
	\phi_{SV}: \SFH(-Y(K),-\Gamma_{K}) \to \HFKH(-Y,K),
\]
for which the image of $\EH(K)$ is $\SLH(K)$. 

The map $\phi_{SV}$ is precisely the one discussed in the previous subsection. More specifically, the Stipsicz-V\'ertesi map is obtained by attaching a contact $T^2 \times I$ layer to the boundary of $(Y(K),\xi_K)$ to obtain a space which we denote $(Y(K),\xib_K)$.  The contact structure on $T^2 \times I$ is chosen in a way which is compatible with negative Legendrian stabilization and which results in a pair of meridional dividing curves along the boundary of the resulting space.  Applying the HKM gluing map gives an identification between $\EH(L)$ and the contact invariant $\EH(Y(K),\xib_K) \in \SFH(-Y(K),-\Gamma_\mu)$.  Since the new dividing set on $\partial Y(K)$ consists precisely of two meridional curves, the sutured Floer homology group $\SFH(-Y(K),-\Gamma_\mu)$ is isomorphic to $\HFKH(-Y,K)$.  An explicit computation using open book decompositions then provides the desired identification between $\EH(Y(K),\xib_K)$ and $\SLH(K)$.

This alternate view of the LOSS hat invariant --- as the contact invariant of a space associated to a given Legendrian or transverse knot --- is quite useful in practice.  It frequently allows one to interpolate between geometric properties of Legendrian and transverse knots and algebraic properties of the LOSS hat invariant.  For instance, this perspective was instrumental in the first and second author's result that Giroux torsion layers are necessarily intersected by the binding of any open book supporting the ambient contact structre \cite{EV}.  The above discussion motivates one to consider a refinement of the Stipsicz-V\'ertesi construction which retains more geometric information associated to a given Legendrian or transverse knot.

If $K \subset (Y,\xi)$ is a Legendrian knot, we denote by $K_i$ the $i^{th}$ negative stabilization of $K$. Let $(Y(K),\xi_i)$ denote the complement of an open standard neighborhood of $K_i$. Note that the boundary of $(Y(K),\xi_i)$ is convex, and, with the appropriate choice of initial longitude, we can identify the dividing set with $\Gamma_i$ from the previous subsection.  Work of Etnyre and Honda \cite{EH} shows that the complement $(Y(K),\xi_{i+1})$ is obtain from $(Y(K),\xi_i)$ by attaching a negatively signed basic slice $(T^2 \times I,\xi_-)$ to the boundary of $(Y(K),\xi_i)$.
Thus, the collection 
\[
	\{\EH(K_i)\in \SFH(-Y(K), -\Gamma_i)\}
\]
of HKM invariants satisfies $\phi_-(\EH(K_i))=\EH(K_{i+1})$ and hence yields an element
\[
	\EHL(K) \in \SFHL(-Y,K),
\]
which defines an invariant of the Legendrian knot $K$.  By construction, the invariant $\EHL(K)$ remains unchanged under negative stabilizations of the Legendrian knot $K$.  Therefore, through the process of Legendrian approximations, we see that $\EHL$ defines an invariant of transverse knots.  In what follows, we shall refer to these as the LIMIT invariants of Legendrian and transverse knots.

With respect to the isomorphism $I_-$ promised by Theorem~\ref{thm:lim_to_minus}, we have the following alternate characterization of the LIMIT invariant $\EHL$.

\begin{theorem}\label{thm:lim_to_minus_leg}
	Let $K \subset (Y,\xi)$ be a null-homologous Legendrian knot.  Under the isomorphism
\[
	I_-: \SFHL(-Y,K) \to \HFKM(-Y,K)
\]
given by Theorem~\ref{thm:lim_to_minus}, the Legendrian invariants $\EHL(K)$ and $\SL(K)$ are identified.
\end{theorem}

Knowing that ``LIMIT$=$LOSS'' allows one to combine the intrinsic advantages of either invariant when attempting to solve a given problem.  In a similar spirit, the second author, in joint work with Baldwin and V\'ertesi \cite{BVV}, showed that the Legendrian and transverse invariants defined in \cite{LOSS} agree with the combinatorial (GRID) invariants of Legendrian and transverse knots defined by Ozsv\'ath, Szab\'o, and Thurston in \cite{OST}.  Thus, one can view Theorem~\ref{thm:lim_to_minus_leg} as the final chapter in a story relating the various Legendrian and transverse invariants defined within the sphere of Heegaard Floer theory.


\subsection{Sutured Inverse Limit Invariants} 
\label{sub:sut_inverse_limit}
The sutured limit invariants are defined by taking a sequence of tori in a knot complement with sutures that limit to meridional sutures through negative longitudinal sutures. If instead one goes beyond the meridional slope and limits through positive longitudinal sutures one can define an inverse limit invariant 
\[
	\SFHIL(-Y,K).
\]
The details of the construction are very similar to those above and presented in Section~\ref{sub:top_inverselimits}. Arguments analogous to the ones we use in proving the theorems above will prove the following relations with Heegaard-Floer knot invariants. 
\begin{theorem}\label{thm:lim_to_plus}
	Let $K \subset Y$ be a smoothly embedded null-homologous knot. There exists a isomorphism of bigraded $\F[U]$-modules
\[
		I_+: \SFHIL(-Y,L) \to \HFKP(-Y,L).
\]
\end{theorem}

\begin{theorem}\label{thm:SV_map2}
	Let $K \subset Y$ be a Legendrian representative of a null-homologous knot type, $I_+: \SFHL(-Y,K) \to \HFKM(-Y,K)$ the isomorphism given by Theorem~\ref{thm:lim_to_plus}, and $\iota_*: \HFKH(-Y,K)\to \HFKP(-Y,K)$ the map induced on homology by the inclusion of complexes.  Then there is a natural geometrically defined map $\Phi_{\mathrm{dSV}}$ so that the following diagram commutes
\begin{center}
\begin{tikzpicture}	[->,>=stealth',auto,thick]
	\node (b) at (0,0){$\SFHIL(-Y,K)$} ;
	\node (c) at (4,0) {$\HFKP(-Y,K)$} ;
	\node (a) at (2,-2) {$\HFKH(-Y,K)$} ;

	\draw (b) edge node[above] {$I_+$} (c);
	\draw (a) edge node[left] {$\Phi_{dSV}$} (b);
	\draw (a) edge node[right] {$\iota_*$} (c);
\end{tikzpicture}.
\end{center}
\end{theorem}

Also, in Section~\ref{sub:top_inverselimits} we define a class $\EHIL(K)$ in $\SFHIL(-Y,K)$ for a Legendrian or transverse knot $K$ in a contact manifold $(Y,\xi)$. While a corresponding invariant in knot Floer homology has not previously been studies we can prove the following result. 
\begin{theorem}\label{thm:lim_to_plus_leg}
	Let $K \subset (Y,\xi)$ be a Legendrian knot.  Under the map
\[
	\Phi_{\mathrm{dSV}}: \HFKH(-Y,K) \to \SFHIL(-Y,K)
\]
given in Theorem~\ref{thm:SV_map2}, the Legendrian invariant $\SLH(K)$ is sent to $\EHIL(K)$.
\end{theorem}


\subsection{Vanishing slopes} 
\label{sub:vanishing_slopes}

The construction of the limit invariants and examples computed in Section~\ref{sec:examples} motivate the definition of an invariant of Legendrian or transverse knots we dub the ``vanishing slope''.

To be more precise, let $K$ be an oriented null-homologous Legendrian knot, and $(Y(K),\xi_K)$ the complement of an open standard neighborhood of $K$.  We define an {\it extension} of $(Y(K),\xi_K)$ to be contact manifold $(Y(K), \xi'_K)$, which is obtained from $(Y(K),\xi_K)$ by attaching a (tight) sequence of basic slices to $\partial Y(K)$.  One similarly defines a {\it positive} or {\it negative} extension to be one in which all of the attached basic slices are positive or negative, respectively. 

Recall that if $K^\pm$ is obtained from $K$ via positive or negative stabilization, then $(Y(K),\xi_{K^\pm})$ is obtained from $(Y(K),\xi_K)$ by attaching a positive or negative basic slice to $\partial Y(K)$ respectively.  In particular, $(Y(K),\xi_{K^\pm})$ is either a positive or negative extension of $(Y(K),\xi_K)$ depending on the sign of the stabilization.  Similarly, the contact 3--manifold $(Y(K),\xib_K)$, obtained via the Stipsicz-V\'ertesi attachment is a negative extension of $(Y(K),\xi_K)$.

Let $(Y(K),\xi'_K)$ be an extension of $(Y(K),\xi_K)$ (positive or negative).  We define the {\it extension slope} of $(Y(K),\xi'_K)$ to be $s(Y(K),\xi'_K) = (-n,r)$, where $n$ is the amount of Giroux ($\pi$-)torsion in $\overline{(Y(K),\xi'_K) \backslash (Y(K),\xi_K)}$ and $r$ is the usual dividing slope of the dividing curves in the boundary of $(Y(K),\xi'_K)$.  Roughly, the extension slope is just the usual dividing slope, enhanced to track the number of times the dividing curves of convex tori contained within the extension rotate beyond the meridional slope as they approach the boundary of $(Y(K),\xi'_K)$.  From this interpretation, we see that extension slopes come equipped with a natural lexicographical ordering.

\begin{definition}\label{def:vanishing}
	Let $K \subset (Y,\xi)$ be a null-homologous Legendrian knot with a given Seifert framing and non-vanishing HKM invariant.  We define the {\it vanishing slope} $\mr{Van}(K)$ to be
\[
	\sup\{s(\Gamma_{\xi'_K}) \; | \; (Y(K),\xi'_K) \; \text{extends} \; (Y(K),\xi_K),\; \EH(Y(K),\xi'_K) = 0\},
\]
where all extensions must be by tight contact structures.
We similarly define the {\it positive} and {\it negative vanishing slopes} $\mr{Van}^\pm(K)$ to be
\[
	\sup\{s(\Gamma_{\xi'_K}) \; | \; (Y(K),\xi'_K) \; \text{positively extends} \; (Y(K),\xi_K), \; \EH(Y(K),\xi'_K) = 0\}
\]
and
\[
	\sup\{s(\Gamma_{\xi'_K}) \; | \; (Y(K),\xi'_K) \; \text{negatively extends} \; (Y(K),\xi_K), \; \EH(Y(K),\xi'_K) = 0\}
\]
respectively.
\end{definition}

The above definitions can be extended to the transverse category as well via the process of Legendrian approximation.  In this case, however, one must restrict the set(s) of allowable extensions to sequences of basic slice attachments, the first of which is the Stipsicz-V\'ertesi attachment.  Otherwise, the corresponding definitions are identical.

We immediately obtain the following observation concerning the relationship of the negative vanishing slope to other invariants considered in this paper.

\begin{proposition}\label{prop:negvan_invts}
	Let $K \subset (Y,\xi)$ be a null-homologous Legendrian knot with given Seifert framing.
	\begin{enumerate}
		\item If $\EH(K) \neq 0$, then $\mr{Van}^-(K) \leq (0,tb(K))$,
		\item If any of the invariants $\SL(K) = \EHL(K)$, $\SLH(K)$, or $\EHIL(K)$ are non-vanishing, then $\mr{Var}^-(K) \leq (0,-\infty)$.
	\end{enumerate}
\qed
\end{proposition}
See Section~\ref{sec:examples} for some explicit computations of the vanishing slope. 


\subsection{Noncompact 3-manifolds} 
\label{sub:noncompact}

The work presented here is part of a broader program to develop Heegaard Floer theoretic invariants for noncompact 3-manifolds with cylindrical ends and a generalized ``suture'' on the boundary at infinity.  These invariants are built in a fashion similar to $\SFHL(-Y,K)$ above, by taking directed limits over collections of maps induced by natural contact-geometric constructions.  In the special case of manifolds with $T^2 \times [0,\infty)$-ends, a generalized suture is equivalent to a choice of ``slope at infinity'', as defined in \cite{Tr}.

When these techniques are applied to a null-homologous knot complement $Y-K$, and the slope at infinity is chosen to be meridional to $K$, the resulting group $\SFH(-Y,K)$ is isomorphic to the minus variant of knot Floer homology $\HFKM(-Y,K)$.  

Such generalizations are the subject of future papers.


\subsection{Supplementary Results and Questions} 
\label{sub:aux_quest}

To prove the theorems discussed above, we must establish a number of supplementary results which may be of independent interest.  Most notably, we discuss a general framework which one can apply to effectively and explicitly compute the HKM gluing maps.  Additionally, as a corollary of our discussion regarding maps induced by bypass attachments, we obtain an independent proof of Honda's bypass exact triangle \cite{Ho}, see Section~\ref{sec:bypass_attachment_maps}.

In Section~\ref{sec:examples}, we provide an example an example of a Legendrian knot $K_1$ for which $\EH(K_1) \neq 0$ despite the fact that $\EHL(K_1) = \SL(K_1) = 0$.  We further demonstrate the existence of a Legendrian knot $K_2$ for which $\SLH(K_2) \neq 0$ while $\EHIL(K) = 0$.  This suggests attention be paid to the following question.

\begin{question}
	What is the difference in information content between the various Legendrian and transverse invariants defined in the context of Heegaard Floer theory?
\end{question}

Marco Golla \cite{Golla} has a beautiful answer to this question for Legendrian knots in the standard contact 3--sphere.  Specifically, he shows that, in terms of information content, the HKM invariant of a Legendrian knot $K$ is equivalent to the pair of LOSS invariants $\{\SL(K),\SL(-K)\}$.  That is, $\EH(K)$ determines the pair $\{\SL(K),\SL(-K)\}$ and vice-versa. As mentioned above our examples in Section~\ref{sec:examples} indicate this is not true in arbitrary contact manifolds. 

In a different direction, Lisca and Stipsicz \cite{LSt} recently showed how to construct a new invariant of Legendrian and transverse knots using contact surgery techniques.  Although their construction is substantially different from that presented in this paper, it is similar in the sense that their invariants take values in an (inverse) limit on Heegaard Floer homology groups. Thus, we ask the following question.

\begin{question}
	What, if any, is the relationship between the inverse limit invariants defined by Lisca and Stipsicz and the directed and inverse  limit invariants defined here?
\end{question}


\subsection*{Organization} 
\label{sub:organization}
 Part~\ref{part1} of the paper gives the definition and properties of the limit sutured homologies and discusses their properties. Specifically, Section~\ref{sec:back} provides background on contact geometry and knot and sutured Floer homology.  In Section~\ref{sec:limits}, we provide a rigorous definition of the sutured limit homologies and the associated Legendrian/transverse invariant.  
Part~\ref{part2} uses bordered sutured Floer homology to identify the invariants form Part~\ref{part1} with their corresponding knot Floer homologies. We begin that part with a review of bordered sutured Floer homology in Section~\ref{backgroud2} and discuss the algebras associated to parameterized sutured surfaces used in our proofs in Section~\ref{sec:algebras}. The following two sections identify our limit invariants with the corresponding knot Floer homologies and the limit Legendrian invariants with the LOSS invariants, respectively. Then in Sections~\ref{sec:SV_limit} and~\ref{sec:proof_lim_hfh} we prove various maps between the limit invariants and knot Floer homology can be identified with corresponding maps purely in knot Floer homology.  In Section~\ref{sec:plus}, we sketch proofs of the various results concerning sutured inverse limit homology.  Having completed our identification of limit invariants with knot Floer homology, in Section~\ref{sec:gradings} we show how to identify natural gradings on sutured limit homology with the classical absolute Alexander and ($\Z/2$) Maslov gradings.  Finally, in Section~\ref{sec:examples}, examples are presented of Legendrain knots exhibiting interesting behavior from the perspective of the Legendrian and transverse invariants defined herein. 


\subsection*{Acknowledgements} 
\label{sub:acknowledgements}

We would like to thank the Mathematical Sciences Research Institute for their hospitality hosting the authors in the spring of 2010.  A significant portion of this work was carried out during the semester on ``Symplectic and Contact Geometry and Topology'' and on ``Homology Theories of Knots and Links''.  We would also like to thank the Banff International Research Station for allowing us to participate in the workshop ``Interactions Between Contact Symplectic Topology and Gauge Theory in Dimensions 3 and 4''. The third author was unable to participate in the completion of this paper, so any issues with the exposition are solely the responsibility of first two authors. The first author gratefully acknowledges the support of NSF grants DMS-0804820 and DMS-1309073.  The second author gratefully acknowledges the support of NSF grant DMS-1249708. The third  author gratefully acknowledges the support of a Simons Postdoctoral Fellowship.



\part{The sutured limit homology package} \label{part1}

In this part of the paper, using only sutured Floer homology and contact geometry, we define the sutured limit $\SFHL(Y,K)$ and sutured inverse limit $\SFHIL(Y,K)$ homologies of a null-homologous knot in a 3--manifold. Together with the sutured Floer homology $\SFH(Y(K),\Gamma_\mu)$ of the knot complement with meridional sutures, these groups are shown to share many of the properties of the knot Floer homology packaged $\mathrm{HFK}^\pm(Y,K)$ and $\HFKH(Y,K)$. We also show that given a Legendrian knot $K$ in a contact manifold $(Y,\xi)$ that there is an invariant $\EHL(K)\in \SFHL(-Y,K)$ that shares many properties of the LOSS invariant $\SL(L) \in \HFKM(-Y,L)$.
\begin{remark}
It will be clear from our discussion that any homology theory for sutured manifolds that poses an appropriate ``gluing" theorem and contact invariant will lead to limit invariants for knots. 
\end{remark}

\section{Background} 
\label{sec:back}

In this section, we review the basic definitions and results used in the first part of the paper to define the sutured limit and inverse limit homologies.  We begin by reviewing  standard notions in contact geometry, convex surfaces and Legendrian and transverse knot theory. In the following subsections we recall basic definitions and results from knot Floer homology, sutured Floer homology and invariants of Legendrian and transverse knots.

\subsection{Contact Geometry} 
\label{sub:contact_geometry}

Recall that a {\it contact structure} on an oriented $3$-manifold $Y$ is a $2$-plane field $\xi$ satisfying an appropriate nonintegrability condition.  In what follows, we assume that our contact structures are always cooriented by a global $1$-form $\alpha$, called a {\it contact form}.  In this case, the nonintegrability condition is equivalent to the statement that $\alpha \wedge d\alpha$ is a volume form defining the given orientation on $Y$. We refer the reader to \cite{Et2, EH, Honda00a} for details concerning contact structures, Legendrian and transverse knots, and convex surfaces, but recall below the basic facts we will need.

\subsubsection{Convex Surfaces and Bypass Attachments} 
\label{sub:conv_surf}

Recall that a surface $\Sigma$ in a contact manifold $(Y,\xi)$ has an induced singular foliation $T\Sigma\cap \xi$ called the characteristic foliation $\Sigma_\xi$ and the characteristic foliation determines $\xi$ in a neighborhood of $\Sigma$. The surface $\Sigma$ is said to be {\it convex} if there exists a vector field $v$ on $Y$ which is transverse to $\Sigma$, and whose flow preserves the contact structure $\xi$.  Given such a surface and vector field the {\it dividing set} $\Gamma \subset \Sigma$ is the collection of points $\{p \in \Sigma:\, v_p \in \xi_p\}$.

Dividing sets are so-called because they divide a convex surface $\Sigma$ into a union of two (possibly disconnected) regions.  Orienting $\Sigma$ so that the vector field is positively transverse to $\Sigma$, the regions are called positive or negative according to whether the transverse vector field $v$ along $\Sigma$ intersects the contact planes positively or negatively.

Convex surfaces have proven tremendously useful in the study of contact structures on $3$-manifolds for the following key reasons.
\begin{enumerate}
	\item If $\Sigma$ is closed or compact with Legendrian boundary (and the twisting of $\xi$ along $\partial \Sigma$ is non-positive), then after possibly applying a $C^0$-isotopy in a neighborhood of the boundary, $\Sigma$ is $C^\infty$-close to a convex surface.
	\item Giroux flexibility: Given a convex surface $\Sigma$ with dividing set $\Gamma$, if $\mathcal{F}$ is a singular foliation on $\Sigma$ that is divided by $\Gamma$ (see \cite{Honda00a} for the precise definition of ``divided by" but in practice it means that $\mathcal{F}$ is the characteristic foliation on $\Sigma$ in some contact structure and $\Gamma$ is isotopic to a dividing set for the foliation), then we may $C^0$-isotope $\Sigma$ so that its characteristic foliation is $\mathcal{F}$. 
	\item Since the characteristic foliation of a surface determines the contact structure in a neighborhood of $\Sigma$, the contact structure on $\xi$ near $\Sigma$ is {\em almost} determined by a the dividing set $\Gamma$. 
\end{enumerate}

An important example of the use of Giroux flexibility is for convex tori. Suppose $T$ is a convex torus in $(Y,\xi)$ with dividing set $\Gamma$ consisting of two parallel curves that split $T$ into two annuli $T_+$ and $T_-$. According to Giroux flexibility we can $C^0$-isotope $T$ so that its characteristic foliation consists of a two circles worth of singularities, one the core of $T_+$ and the other the core of $T_-$. These are called {\em Legendrian divides}. The rest of the foliation is non-singular and give a ruling of $T$ by curves of any pre-selected slope other than the slope of the dividing curves. These non-singular leaves are called the {\em ruling curves}. A torus with such a characteristic foliation will be called a {\em standard convex torus}.

Let $\alpha$ be an arc contained in a convex surface $\Sigma$ and suppose that $\alpha$ intersects the dividing set of $\Sigma$ in three points $p_1$, $p_2$ and $p_3$, where $p_1$ and $p_3$ are the endpoints of $\alpha$.  A {\it bypass} along $\alpha$ is a convex disk $D$ with Legendrian boundary such that
\begin{enumerate}
	\item $D \cap \Sigma = \alpha$,
	\item $\mathrm{tb}(D) = -1$,
	\item $\partial D = \alpha \cup \beta$,
	\item $\alpha \cap \beta = \{ p_1,p_3 \}$ are corners of $D$ and elliptic singularities of $D_\xi$.
\end{enumerate}

When a bypass is attached to a convex surface $\Sigma$, the dividing set on $\Sigma$ changes in the following predictable way.

\begin{othertheorem}[Honda 2000, \cite{Honda00a}]\label{thm:hon_bypass}
Suppose that $\Sigma$ is an oriented convex surface in $(Y,\xi)$. The surface $\Sigma$ locally splits $Y$ into two pieces. Suppose that $D$ is a bypass along $\alpha$ in $\Sigma$ lying on the positive side of $\Sigma$. If $\Sigma\times [0,1]$ is a small one-sided neighborhood of $\Sigma\cup D$ so that $\Sigma=\Sigma\times\{0\}$, then the dividing curves on $\Sigma\times \{1\}$ are the same as the dividing curves on $\Sigma$ except in a neighborhood of $\alpha$ where the change according to Figure~\ref{fig:bp_attach}. The change in the dividing curves if $\Sigma$ is pushed across a bypass on the negative side of $\Sigma$ is also shown in the figure. 
\end{othertheorem}
\begin{figure}[htbp]
	\centering
	\begin{picture}(215,208)
		\put(0,0){\includegraphics{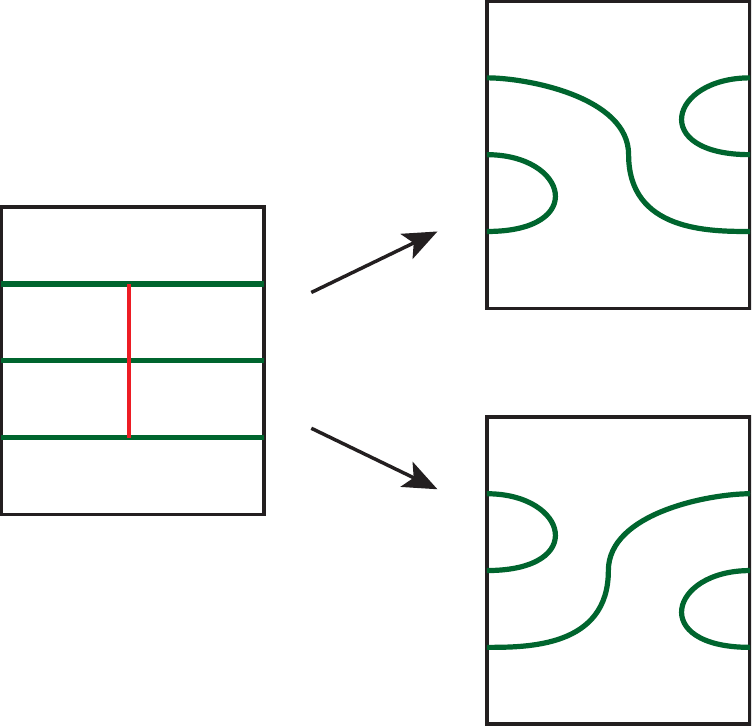}}
		\put(40,112){$\alpha$}
		\put(95,140){\large $P$}
		\put(95,62){\large $N$}
	\end{picture}
	\caption{Effect of a bypass attachment along $\alpha$ from the positive side of the surface (top right) and negative side of the surface (bottom right).}
	\label{fig:bp_attach}
\end{figure}
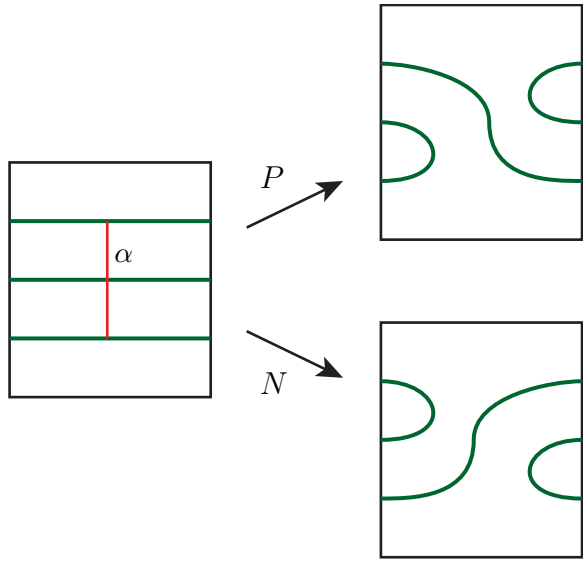


\subsubsection{Legendrian and Transverse Knots} 
\label{sub:leg_trans}

When studying $3$-dimensional contact manifolds $(Y,\xi)$, it is profitable to focus attention on $1$-dimensional subspaces which either lie within or transversely intersect the contact planes.  If a knot $K \subset Y$ satisfied $T_pK \subset \xi_p$ for all $p \in K$, we say that $K$ is {\it Legendrian}.  Similarly, if $K \subset Y$ satisfies $T_pK \pitchfork \xi_p$ for all $p \in K$, we say that $K$ is {\it transverse}.  Since our contact structures are always oriented, we further require that each of the intersections between a transverse knot $K$ and the contact structure $\xi$ be positive.  Legendrian or transverse knots are said to be isotopic if they are isotopic through Legendrian or transverse knots respectively.

Recall that a Legendrian knot always has a framing coming from the contact structure called the {\em contact faming}. If $L$ has a preferred framing $\mathcal{F}$ then we can associate an integer, $\tw(L,\mathcal{F})$, to the contact framing. If $L$ is null-homologous and its preferred framing is the Seifert framing the we call the twisting $\tw(L,\mathcal{F})$ the {\it Thurston-Bennequin invariant} and denote it $\tb(L)$. In addition, when $L$ is null-homologous and oriented we can define the rotation number $\rot(L)$ to be minus the Euler number of $\xi$ restricted to a Seifert surface, relative to an oriented vector field in $\xi$ along $L$. (This number is only well defined module $n$, where $n$ is the generator of the image of the Euler class of $\xi$ in $\Z$.)

It is well known, see \cite{EH}, that any two Legendrian knots have contactomorphic neighborhoods. Thus studying a model situation one can see that given a Legendrian knot $L$ there is a neighborhood of $L$ with convex boundary having two dividing curves of slope $\tb(L)$. If the boundary of this neighborhood is in standard form with any ruling slope then we say this is a {\em standard neighborhood of $L$}. We also recall that given any solid torus $N$ in a contact manifold $(Y,\xi)$ with convex boundary having two dividing curves of slope $n$ and standard form on the boundary and for which $\xi|_N$ is tight, is the standard neighborhood of a unique Legendrian knot $L$ in $N\subset M$. Thus studying Legendrian knots in a given knot type in $(Y,\xi)$ is equivalent to studying such solid tori that represent the given knot type. 

Given an oriented Legendrian knot $K$, one can produce new Legendrian knots $S_+(K)$ and $S_-(K)$ in the same knot type by applying operations called positive, respectively negative, stabilization.  These operations, performed in a standard neighborhood of a point on $L$ are depicted in Figure~\ref{fig:leg_stab}. We will discuss the relation between stabilization and standard neighborhoods of Legendrian knots in the next subsection. 
\begin{figure}[htbp]
	\centering
	\begin{picture}(160,85)
		\put(0,0){\includegraphics[scale=1]{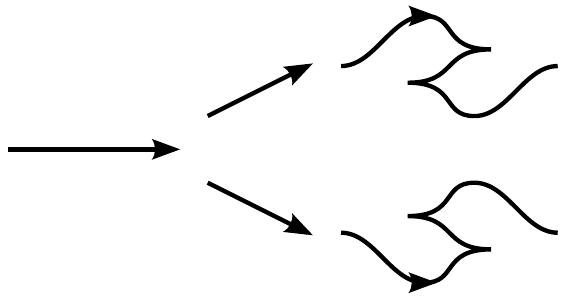}}
		\put(63,15){$S_-$}
		\put(63,65){$S_+$}
	\end{picture}
	\caption{Positive and negative Legendrian Stabilizations}
	\label{fig:leg_stab}
\end{figure}

Given a Legendrian knot $K$, one can produce a canonical transverse knot nearby to $K$, called the transverse pushoff of $K$.  If $T$ is a transverse knot, we say that $K_T$ is a Legendrian approximation of $T$ if the transverse pushoff of $K_T$ is $T$.  For a given transverse knot, there are typically infinitely many distinct Legendrian approximations of $T$.  However, each of these infinitely many distinct Legendrian approximations are related to one another by sequences of negative stabilizations.  Thus, these two constructions are inverses to one another, up to the ambiguity involved in choosing a Legendrian approximation of a given transverse knot (see \cite{EFM,EH}).


\subsubsection{Contact structures on thickened tori} 
\label{sub:tt}

Before discussing contact structures on $T^2\times[0,1]$ we first discuss curves on $T^2$.
Choosing a product structure on $T^2$ we may identify (unoriented) essential curves on $T^2$ with the rational numbers union infinity so that $S^1\times \{pt\}$ is the $\infty$-curve and $\{pt\}\times S^1$ is the 0-curve. It will be useful to compactify $\R$ to $S^1$ and think of the added point as being both $\infty$ and $-\infty$. Having done this the essential curves on $T^2$ are represented by the rational points union infinity on $S^1$. Recall that two curves form an integral basis for $H^1(T^2;\Z)$ if and only if they can be isotoped to intersect exactly once. In terms of the rational numbers $p_0/q_0$ and $p_1/q_1$ associated to the curves, they will form an integral basis if and only if $pq'-q'p=\pm 1$. 

We can encode these ideas in the Farey tessellation, see Figure~\ref{fig:farey}. 
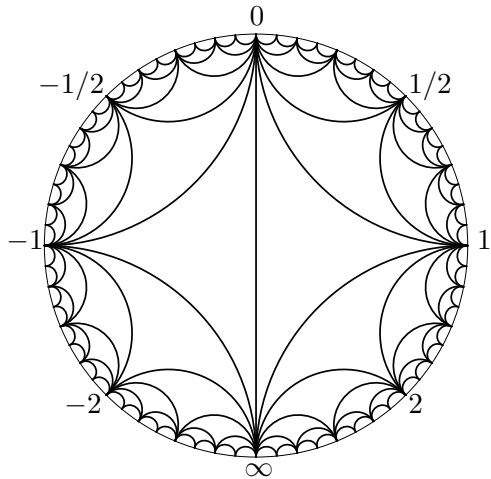
\begin{figure}[htbp]
	\centering
	\begin{picture}(160,165)
		\put(0,0){\includegraphics[scale=1]{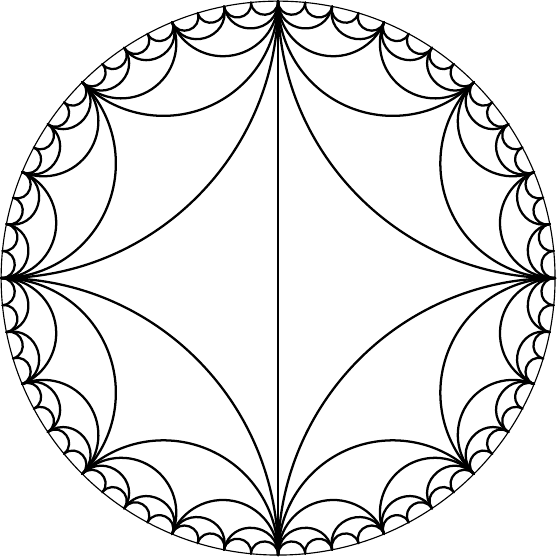}}
		\put(76,-7){$\infty$}
		\put(78,164){$0$}
		\put(164,79){$1$}
		\put(-14,79){$-1$}
		\put(-2,138){$-1/2$}
		\put(138,138){$1/2$}
		\put(138,17){$2$}
		\put(8,17){$-2$}
	\end{picture}
	\caption{The Farey tesselation oriented for use with our convention of slopes.}
	\label{fig:farey}
\end{figure}
Let $D$ be the unit disk in the complex plane. Label the complex number $i$ by 0 and $-i$ by $\pm\infty$ and connect them with a geodesic in $D$ (where we give $D$ the standard hyperbolic metric). Label $1$ by $1$ and connect it to the points labeled $0$ and $\pm\infty$ by geodesics. We will now inductively label the points on $\partial D$ with positive real part. Given an interval on $\partial D$ with positive real part and end points two adjacent points that have been labeled by $p/q$ and $p'/q'$, label its midpoint by $(p+p')/(q+q')$ and connect it to the end points of the interval by geodesics. (Here we think of $0$ as $0/1$ and $\infty$ as $1/0$.) We can similarly label points on $\partial D$ with negative real part (except here we must think of $0$ as $0/1$ and $\infty$ as $-1/0$). This procedure will assign all the rational numbers to points on $\partial D$ and they will appear in order, that is if $a>b$ then $a$ will be in the region that is clockwise of $b$ and counterclockwise of $\infty$. Moreover, the edges will not intersect and two points will be connected by an edge if and only if they correspond to curves that form an integral basis for $H_1(T^2;\Z)$. 

Turning to contact structures, let $\Gamma_i$ be two parallel curves on $T^2$ with slope $s_i$, $i=0,1$. Given a contact structure $\xi$ on $T^2\times[0,1]$ with convex boundary having dividing curves $\Gamma_i$ on $T^2\times\{i\}$, $i=0,1$, we say $\xi$ is {\em minimally twisting} if any other convex torus $T$ in $T^2\times [0,1]$ that is isotopic to the boundary has dividing slope clockwise of $s_0$ and counterclockwise of $s_1$. (Note that a minimally twisting contact structure is necessarily tight.) Recall that the classification of contact structures on thickened tori implies that given any slope that lies clockwise of $s_0$ and counterclockwise of $s_1$ is the dividing slope for some convex torus, thus the minimally twisting condition says that the only convex tori are the ones that must be there. 

A {\em basic slice} is a tight, minimally twisting tight contact structure on $T^2 \times [0,1]$, for which each boundary component is convex with $\Gamma_i$ being the dividing set on $T^2\times\{i\}$, $\Gamma_i$ consisting of two curves each of slope $p_i/q_i$ and $p_0q_1-p_1q_0=\pm1$.

According to \cite{Gi2,Honda00a} there are precisely two basic slice for any given dividing curves (once the characteristic foliations on the boundary are arranged to be the same), called positive and negative.  We denote them by $\xi_{p_0/q_0, p_1/q_1}^\pm$. They are distinguished by their relative Euler class, but are the same up to contactomorphism.  Moreover there is a diffeomorphism taking any basic slice to another. The following theorem relates basic slices to bypass attachments.

\begin{othertheorem}[Honda 2000, \cite{Honda00a}]\label{thm:bp_bslice}
	Let $(T^2 \times [0,1],\xi_{-1,0}^+)$ and $(T^2 \times [0,1],\xi_{-1,0}^-)$ be positive and negative basic slices respectively with dividing slopes $-1$ and $0$.  The contact structures $\xi_{-1,0}^+$ and $\xi_{-1,0}^-$ are obtained from an invariant neighborhood of $T^2\times \{1\}$ by attaching a bypasses layer along the curves $\gamma_+$ and $\gamma_-$ shown in Figure~\ref{fig:slice_bp}, respectively.
\end{othertheorem}
\begin{figure}[htbp]
	\centering
	\begin{picture}(216,100)
		\put(0,-2){\includegraphics[scale=1]{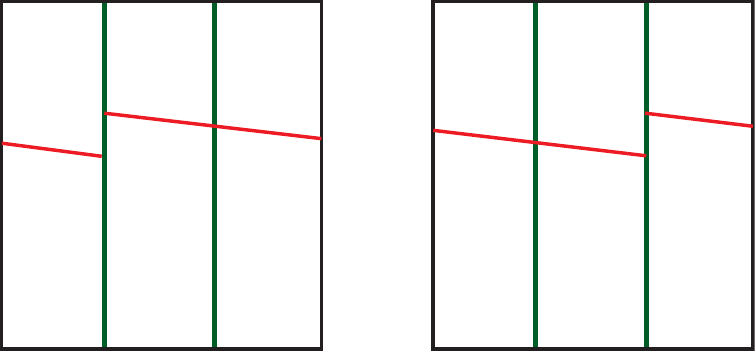}}
		\put(43,72){$\gamma_+$}
		\put(165,65){$\gamma_-$}
		\put(42,28){$+$}
		\put(73,28){$-$}
		\put(165,28){$+$}
		\put(197,28){$-$}
	\end{picture}
	\caption{The bypass attachments for the positive and negative basic slice.}
	\label{fig:slice_bp}
\end{figure}
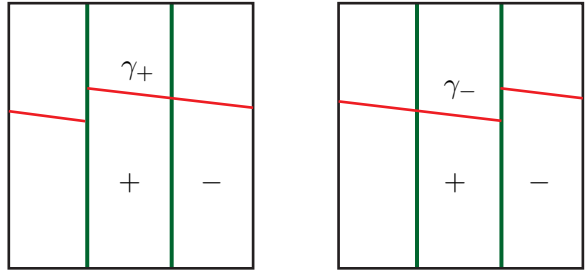

We now recall part of the classification of minimally twisting contact structures on $T^2\times [0,1]$ that we will need below (for details see \cite{Honda00a}).
\begin{enumerate}
\item Given a  minimally twisting contact structure on $T^2\times[0,1]$ with standard convex boundary having dividing slope $s_i$ on $T^2\times {i}$, $i=0,1$, there corresponds a minimal path in the Farey tessellation that goes from $s_0$ clockwise to $s_1$ and signs on each edge in the path.
\item Given the contact structure above any slope $s$ in the interval clockwise of $s_0$ and counterclockwise of $s_1$ can be realized as the dividing slope on some convex torus.
\item\label{tTitem3} Given a minimal path in the Farey tessellation between two numbers $s_0$ and $s_1$ and any assignment of signs to the edges in the path, there is a unique minimally twisting contact structure realizing that path. (Different assignments of signs can correspond to the same contact structure, see \cite{Honda00a}.)
\item Given a non-minimal path in the Farey tessellation between two numbers $s_0$ and $s_1$ and an assignment of signs to the edges it will correspond to a tight (and minimally twisting) contact structure if and only if it can be shortened to a minimal path, otherwise it corresponds to an overtwisted contact structure. A path can be shortened if there are two edges in the path which can be replaced by a third edge and the edges have the same sign, then in the shortening the third edge is assigned the sign of the two edges it replaces. 
\end{enumerate}
We briefly note that each edge in the Farey tessellation corresponds to a basic slice. So the above results basically say that a contact structure on $T^2\times[0,1]$ can be factored into basic slices and when you ``glue" two basic slices together you get a minimally twisting contact structure unless the basic slices have different signs and corresponds to a path that can be shortened.

We now establish some important notation used in the following section to define our limit invariants. Using the product structure above on $T^2$ we denote the basic slice $(T^2\times[0,1],\xi_{-i,-i+1}^\pm)$ by $B_i^\pm$ and $(T^2\times[0,1],\xi_{-\infty,-i}^\pm)$ by $A_i^\pm$. Let  $(T^2\times[0,1],\xi_{i, \infty,}^\pm)$ be denoted by $\widetilde{A}^\pm_i$; and finally, for $i>j$, let $C^\pm_{i,j}$ denote the contact manifold $T^2\times [0,1]$ corresponding to the minimal path in the Farey tessellation from $-i$ to $-j$ with all signs being $\pm$. We note that according to the classification results discussed above we have the following facts.
\begin{proposition}\label{prop:relations}
We have the following relations between contact structures on $T^2\times [0,1]$.
\begin{enumerate}
\item The contact manifold $A^\pm_i\cup C^\pm_{i,0}$ (with the two boundary components having the same slope glued together) is contactomorphic to $A^\pm_0$.
\item For $i>k>j$, the contact manifold $C^\pm_{i,k}\cup C^\pm_{k,j}$ (with the two boundary components having the same slope glued together) is contactomorphic to $C^\pm_{i,j}$.
\item The contact manifold $B_i^\pm\cup B_{i+1}^\mp$ (with the two boundary components having the same slope glued together) is contactomorphic to the contact structure $B_i^\mp\cup B_{i+1}^\pm$. (This does not directly follow from the classification results above but is essentially the ambiguity mentioned in Item~\eqref{tTitem3} above, see \cite{Honda00a}.)
\item The contact manifold $A^\pm_i\cup C^\mp_{i,0}$ (with the two boundary components having the same slope glued together) is overtwisted.
\item The contact manifold $A^\pm_i\cup B_i^\pm$  (with the two boundary components having the same slope glued together) is contactomorphic to $A^\pm_{i-1}$. 
\item The contact manifold $A^\pm_i\cup B_i^\mp$  (with the two boundary components having the same slope glued together) is tight and minimally twisting. 
\end{enumerate}
\end{proposition}

\begin{remark}\label{rem:ends}
Turning the first observation around we notice that there is a sequence of tori $T_i,$ $i=0,1,\ldots,$ in $A^-_0$ such that $T_i$ is a standard convex torus (isotopic to the boundary of $A^-_0$) with dividing slope $-i$ such that $T_i$ cuts $A^-_0$ into two pieces, namely $A^-_i$ and $C^-_{i,0}$. Moreover for $j>i$ the torus $T_j$ is contained in the $A^-_i$ component of the complement of $T_i$.

We have analogous results for $\widetilde{A}^-_0$. Specifically in $\widetilde{A}^-_0$ there is a sequence of tori $\widetilde{T}_i$,  $i=0,1,\ldots,$ such that $\widetilde{T}_i$ is a standard convex torus (isotopic to the boundary of $\widetilde{A}^-_0$) with dividing slope $i$ such that $\widetilde{T}_i$ cuts $\widetilde{A}^-_0$ into two pieces, namely $\widetilde{A}^-_i$ and $C^-_{0,-i}$. Moreover for $j>i$  the torus $\widetilde{T}_j$ is contained in the $\widetilde{A}^-_i$ component of the complement of $\widetilde{T}_i$.
\end{remark}

From the perspective of Legendrian and transverse knot theory we have the following result. 

\begin{othertheorem}[Etnyre-Honda 2001, \cite{EH}]\label{thm:Leg_stab_bs}
	Let $L \subset (Y,\xi)$ be a Legendrian knot and identify the boundary of its complement with $T^2$ so that the meridional curve has slope $\infty$ and the longitude given by the contact framing has slope $0$. Now let $N$ be a standard neighborhood of $L$, $L_+$ and $L_-$ its positive and negative stabilizations inside $N$ and $N_\pm$ standard neighborhoods of $L_\pm$ inside $N$.  The contact manifold $\overline{N-N_\pm}$ is contactomorphic to $B_1^\pm$ (and $C_{0,1}^\pm$). In particular the (closure of the) complement of the standard neighborhoods of $L_+$ and $L_-$ are obtained from the (closure of the) complement of the standard neighborhood of $L$ by a positive and negative basic slice attachments respectively.
\end{othertheorem}


\subsubsection{Open Book Decompositions} 
\label{sub:open_book_decompositions}

In recent years, the primary tool used to study contact structures on $3$-manifolds has been Giroux's correspondence \cite{Gi}.  An {\it open book decomposition} of a $3$-manifold $Y$ is a pair $(B,\pi)$ consisting of an oriented, fibered link $B \subset Y$, together with a fibration of the complement $\pi: (Y-B) \to S^1$ by surfaces whose oriented boundary is $B$.  An open book $(B,\pi)$ is said to be compatible with a contact structure $\xi$ if $B$ is positively transverse to the contact planes and there exists a contact form $\alpha$ for $\xi$ so that $d\alpha$ restricts to an area form on the fibers $S_\theta = \pi^{-1}(\theta)$.

In was shown by Thurston and Winkelnkemper in \cite{TW} that, given an open book, $(B,\pi)$, one can always produce a compatible contact structure.  Giroux showed in \cite{Gi} that two contact structures which are compatible with the same open book are, in fact, isotopic.  He further showed that two open books which compatible with the same contact structure are related by a sequence of ``positive stabilizations'', that is plumbing with positive Hopf bands.  In other words, Giroux proved the following result. 

\begin{othertheorem}[Giroux 2002, \cite{Gi}]\label{thm:giroux}
	There exists a one-to-one correspondence between the set of isotopy classes of contact structures supported by a $3$-manifold $Y$ and the set of open book decompositions of $Y$ up to positive stabilization.
\end{othertheorem}

One can alternatively specify an open book decomposition $(B,\pi)$ of a $3$-manifold $Y$ by specifying a pair $(S,\phi)$ consisting of a fiber surface $S$ and a monodromy map $\phi:S \to S$ corresponding to the fibration $\pi: (Y-B) \to S^1$ (note that $\phi|_{\partial S} = Id$).  The data $(S,\phi)$ is called an {\it abstract open book}, and determines an open book $(B,\pi)$ on the $3$-manifold obtained via the appropriate mapping-cylinder construction, but only up to diffeomorphism.



\subsection{Knot Floer homology} 
\label{sub:knot_floer_homology}

The Heegaard Floer package possesses a specialization to knots and links known commonly as {\it knot Floer homology}.  This specialization was defined independently by Ozsv\'ath and Szab\'o \cite{OS3} and by Rasmussen \cite{Ra}.  In what follows, we review some basic definitions.  The interested reader is encouraged to read the original papers \cite{OS3,Ra} for a more complete and elementary treatment.  We work with coefficients in $\F = \Z/2$ throughout the remainder of the paper.

If $K \subset Y$ is a knot, a doubly-pointed Heegaard diagram for $K$ consists of an ordered tuple $\SH = (\Sigma,\balpha,\bbeta,z,w)$.  We require that the Heegaard diagram $(\Sigma,\balpha,\bbeta)$ specifies the 3-manifold $Y$ and that the knot $K$ is obtained as follows.  Choose oriented, embedded arcs $\gamma_\alpha$ in $\Sigma - \balpha$ and $\gamma_\beta$ in $\Sigma - \bbeta$ connecting the basepoint $z$ to $w$ and $w$ to $z$ respectively. Now, form pushoffs $\gamma_\alpha$ and $\gamma_\beta$ by pushing the interior of these arcs into the $\balpha$ and $\bbeta$ handlebodies respectively. The knot is then the union $K = \gamma_\alpha \cup \gamma_\beta$ of the two curves.

To such a doubly-pointed diagram $\SH$, Ozsv\'ath and Szab\'o associate a chain complex $\CFKI(\SH)$, which is freely generated as an $\F[U,U^{-1}]$-module by the intersections of the tori $\TT_\alpha = \alpha_1 \times \dots \times \alpha_g$ and $\TT_\beta = \beta_1 \times \dots \times \beta_g$ inside the symmetric product $\mathrm{Sym}^g(\Sigma)$.  Given a pair of intersections $\x,\y \in \TT_\alpha \cap \TT_\beta$, a Whitney disk connecting them $\phi \in \pi_2(\x,\y)$ and a generic path of almost complex structures on $\mr{Sym}^g(\Sigma)$, we denote the moduli space of pseudo-holomorphic representatives of $\phi$ by $\mathcal{M}(\phi)$.  It has expected dimension given by the Maslov index $\mu(\phi)$ and possesses a natural $\R$-action given by translation.  We denote the quotient of $\mathcal{M}(\phi)$ under the $\R$-action by $\widehat{\mathcal{M}}(\phi)$.  If $p \in \Sigma - \balpha - \bbeta$, then we denote by $n_p(\phi)$ the intersection number of $\phi$ with the subvariety $V_p = \{p\} \times \mr{Sym}^{g-1}(\Sigma)$.

We define the differential
\[
	\partial^\infty: \CFKI(\SH) \to \CFKI(\SH)
\]
on generators via
\[
	\partial^\infty(\x) = \sum_{\y \in \TT_\alpha \cap \TT_\beta} \sum_{\substack{\phi \in \pi_2(\x,\y)\\ \mu(\phi)=1}} \# \widehat{\mathcal{M}}(\phi) \cdot U^{n_w(\phi)}\y.
\]

For a knot $K$ in a 3-manifold $Y$ with $b_1 = 0$, the complex $(\CFKI(\SH),\partial^\infty)$ comes equipped with two natural gradings. The {\it Maslov} (homological) grading, which is an absolute $\Q$-grading, is specified up to an overall shift by the formula
\[
	M(\x) - M(\y) = \mu(\phi) - 2n_w(\phi),
\]
for $\x,\y \in \TT_\alpha \cap \TT_\beta$ and any $\phi \in \pi_2(\x,\y)$, and the requirement that multiplication by $U$ drop the Maslov grading by two.  The {\it Alexander} grading is again an absolute $\Q$-grading, specified up to an overall shift by the formula
\[
	A(\x) - A(\y) = n_z(\phi) - n_w(\phi),
\]
and the requirement that multiplication by $U$ drop the Alexander grading by one.

From these formulae, we see that the differential $\partial^\infty$ decreases the Maslov grading by one and is $\Z$-filtered with respect to the Alexander grading; $A(\partial^\infty(\x)) \leq A(\x)$ for any $\x \in \TT_\alpha \cap \TT_\beta$.  There is an additional $\Z$-filtration on $(\CFKI(\SH),\partial^\infty)$ obtained by recording the $U$-exponent multiplying a given generator $\x \in \TT_\alpha \cap \TT_\beta$.

By positivity of intersection, $n_w(\phi)$ is always non-negative, so the $\Z[U]$-module $\CFKM(\SH)\subset \CFKI(\SH)$ freely generated by the intersections of the tori $\TT_\alpha$ and $\TT_\beta$ inside the symmetric product $\mathrm{Sym}^g(\Sigma)$ is a sub-complex of $\CFKI(\SH)$. We denote $\partial^\infty$ restricted to $\CFKM(\SH)$ by $\partial^-$. We additionally denote by $(\CFKP(\SH),\partial^+)$ the quotient complex.

\begin{theorem}[Ozsv\'ath-Szabo \cite{OS3}, Rasmussen \cite{Ra}]\label{thm:OSR_HFKM}
	Let $K$ be a null-homologous knot in a 3-manifold $Y$ with $b_1 = 0$, and $\SH$ a doubly-pointed Heegaard diagram for the pair $(Y,K)$.  Then the $\Q$-graded, $\Z \oplus \Z$-filtered chain homotopy type of the complexes $(\CFKI(\SH),\partial^\infty)$, $(\CFKM(\SH),\partial^-)$ and $(\CFKP(\SH),\partial^+)$ are invariants of $(Y,K)$.
\end{theorem}

The homologies of the associated graded object with respect to the Alexander filtration give various types of knot Floer homology.  It is customary to write them as 
\[
	\HFKI(Y,K), \, \HFKM(Y,K), \text{ and } \HFKP(Y,K).
\]

Setting the formal variable $U$ equal to zero in $(\CFKM(\SH),\partial^-)$, we obtain the $\Q$-graded, $\Z$-filtered complex $(\CFKH(\SH),\PH)$.  Taking the homology of the associated graded object with respect to this filtration yields the hat version of knot Floer homology
\[
	\HFKH(Y,K).
\]

The projection $p: (\CFKM(\SH),\PM) \to (\CFKH(\SH),\PH)$ obtained by setting $U=0$ gives rise to a natural map on homology
\[
	p_*: \HFKM(Y,K) \to \HFKH(Y,K).
\]

In a similar spirit, setting the formal variable $U$ equal to the identity gives a projection $\pi: (\CFKM(\SH),\PM) \to (\CFH(Y),\PH)$, inducing a map
\[
	\pi_*: \HFKM(Y,K) \to \HFH(Y),
\]
from the minus version of knot Floer homology to the hat version of the Heegaard Floer theory for the ambient 3-manifold.

We can also identify $(\CFKH(\SH),\PH)$ as the kernel of the $U$ map on $(\CFKP(\SH),\partial^+)$. Thus the inclusion induces a natural map on homology 
\[
	\iota_*: \HFKH(Y,K) \to \HFKP(Y,K).
\]


\subsection{The Lisca-Osv\'ath-Stipsicz-Szab\'o invariant} 
\label{sub:loss_invt}

There is an invariant of Legendrian knots which takes values in knot Floer homology.  Let $L \subset (Y,\xi)$ be a Legendrian knot in the knot type $K$.  In \cite{LOSS}, Lisca, Ozsv\'ath, Stipsicz and Szabo defined invariants
\[
	\SL(L) \in \HFKM(-Y,L)
\]
and
\[
	\SLH(L) \in \HFKH(-Y,L).
\]
Their invariants are constructed in a manor reminiscent of Honda, Kazez and Mati\'c's construction of the usual contact invariant in Heegaard Floer homology.  Since it will be useful in what follows, we recall the construction from \cite{LOSS}.

Given a Legendrian knot $L \subset (Y,\xi)$, we choose an open book decomposition $(B,\pi)$ of $(Y,\xi)$ which contains the knot $L$ sits on a page $S$ of $(B,\pi)$.  We can assume without loss of generality that this page is given by $S = S_{1/2}$, and that $L$ is nontrivial in the homology of $S$.

Now choose a basis $\{a_0,\dots,a_k\}$ for $S$ so that $L$ is intersected only by the arc $a_0$, and does so transversally in a single point.  Next, apply small isotopies to the $a_i$ to obtain a collection of arcs $\{b_0,\dots,b_k\}$.  We require that the endpoints of $b_i$ be obtained from those of those of $a_i$ by shifting along the orientation of $\partial S$, and that $b_i$ intersect $a_i$ in a single transverse point $x_i = a_i \cap b_i$ (see Figure~\ref{fig:loss}).

\begin{figure}[htbp]
	\centering
	\begin{picture}(325,71)
		\put(0,0){\includegraphics[scale=1]{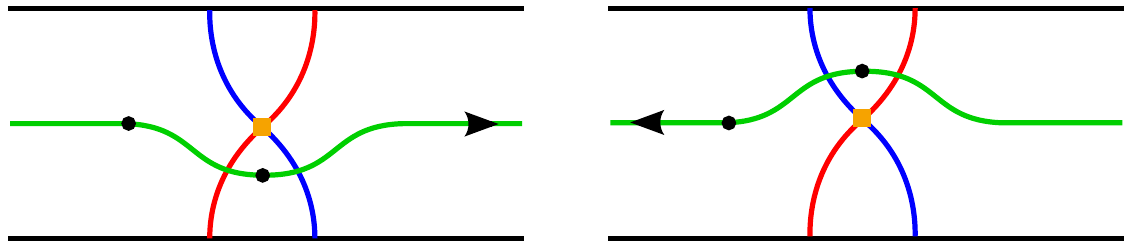}}
		\put(82,32){$x_0$}
		\put(255,35){$x_0$}
		\put(35,40){$z$}
		\put(207,40){$z$}
		\put(72,12){$w$}
		\put(246,55){$w$}
		\put(290,40){$-L$}
		\put(120,40){$L$}
	\end{picture}
	\caption{Construction of the LOSS invariant.}
	\label{fig:loss}
\end{figure}
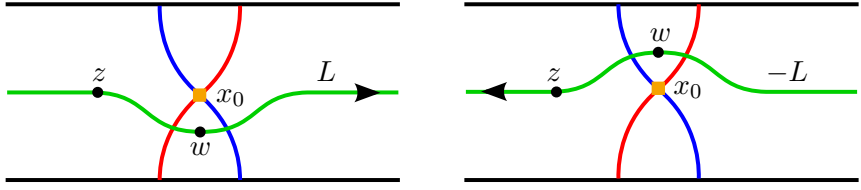

A doubly-pointed Heegaard diagram for the pair $(-Y,L)$ can now be constructed as follows.  The diagram itself is specified by
\[
	(\Sigma,{\bf \beta},{\bf \alpha}) = (S_{1/2} \cup -S_0, (b_i \cup \phi(b_i)), (a_i \cup a_i) ),
\]
where $\phi: S \to S$ is the monodromy map of the fibration $(B,\pi)$ and the arcs $\phi(b_i)$ and the second $a_i$ above sit on the page $-S_{0}$.  The basepoint $z$ is placed on the page $S_{1/2}$, away from the thin strips of isotopy between the $a_i$ and $b_i$.  The second basepoint $w$ is placed inside the thin strip between $a_0$ and $b_0$ as shown in Figure~\ref{fig:loss}.  The two possible choices for the placement of $w$ correspond to the two possible choices of orientation for the Legendrian knot $L$.  

\begin{definition}\label{def:loss}
	Let $L \subset (Y,\xi)$ be a Legendrian knot and let $(\Sigma,{\bf \beta}, {\bf \alpha},z,w)$ be a Heegaard diagram adapted to $L$ constructed as above.  The invariants $\SL(L)$ and $\SLH(L)$ are defined to be
\[
	\SL(L) := [(x_0,\dots,x_k)] \in \HFKM(-Y,L),
\]
and
\[
	\SLH(L) := [(x_0,\dots,x_k)] \in \HFKH(-Y,L)
\]
respectively.
\end{definition}

It was shown in \cite{LOSS} that $\SL(L)$ and $\SLH(L)$ enjoy a number of useful properties, some of which are the following:
\begin{enumerate}
	\item  Under the map $\HFKM(-Y,L) \to \HFKH(-Y,L)$ induced by setting $U = 0$ at the chain level, $\SL(L)$ is sent to $\SLH(L)$.
	\item  Under the map $\HFKM(-Y,L) \to \HFH(-Y)$ induced by setting $U = 1$ at the chain level, $\SL(L)$ is sent to $\EH(Y,\xi)$, the contact invariant of the ambient space.
	\item  If the complement of $L$ is overtwisted (see \cite{LOSS}) or has positive Giroux torsion (see \cite{Ve}), then both $\SL(L)$ and $\SLH(L)$ vanish.
	\item  If $(Y,\xi)$ has non-vanishing contact invariant, then $\SL(L)$ is non-vanishing for every Legendrian $L \subset (Y,\xi)$.
\end{enumerate}

In addition we have the following interesting property. 

\begin{othertheorem}[Lisca-Osv\'ath-Stipsicz-Szab\'o \cite{LOSS}]\label{thm:loss_stab}
	The invariants $\SL$ and $\SLH$ behave as follows under stabilization.  If $L$ is a Legendrian knot and $L_-$ is its negative stabilization, then 
\[
	\SL(L_-) = \SL(L) \;\;\;\; \text{and} \;\;\;\; \SLH(L_-) = \SLH(L).
\]
Similarly, if $L_+$ is the positive stabilization of a Legendrian knot $L$, then 
\[
	\SL(L_+) = U \cdot \SL(L) \;\;\;\; \text{and} \;\;\;\; \SLH(L_+) = 0.
\]
\end{othertheorem}

It immediately follows from Theorem~\ref{thm:loss_stab} that $\SL$ and $\SLH$ define transverse invariants as well.  If $K$ is a transverse knot and $L$ is a Legendrian approximation of $K$, define
\[
	\ST(K) = \SL(L) \;\;\;\; \text{and} \;\;\;\; \STH(K) = \SLH(L).
\]


\subsection{Sutured Floer homology} 
\label{sub:SFH}
Recall that a sutured manifold $(Y,\Gamma)$, with annular sutures, is a manifold $Y$ together with a collection of oriented simple closed curves $\Gamma$ on $\partial Y$ such that each component of $\partial Y$ contains at least one curve in $\Gamma$ and $\partial Y-\Gamma$ consists of two surface $\partial Y_+$ and $\partial Y_-$ so that $\Gamma$ is the oriented boundary of $\partial Y_+$ and $-\Gamma$ is the oriented boundary of $\partial Y_-$. We say that $(Y,\Gamma)$ is balanced if $\partial Y_+$ and $\partial Y_-$ have the same Euler characteristic.

In \cite{Ju} Juh{\'a}sz showed how to associate to a balanced sutured manifold $(Y,\Gamma)$ the sutured Heegaard-Floer homology groups $\SFH(Y,\Gamma)$. We will see a generalization of this in Section~\ref{backgroud2} below, so we will not give the detail of the construction of $\SFH(Y,\Gamma)$ here, but merely recall facts relevant to the definition of our limit sutured homology and its properties. In addition we note that, as in ordinary Heegaard Floer theory, the chain groups are generated by the intersection of tori coming from the curves used in a Heegaard diagram for $(Y,\Gamma)$. The first two results we need relate the sutured Floer theory to previous flavors of Heegaard Floer homology.
\begin{theorem}[Juh\'asz 2006, \cite{Ju}]
\label{thm:SFHisHFH}
Let $Y$ be a closed 3-manifold and denote by $Y(1)$ the sutured manifold obtained from $Y$ by deleting an open ball and placing a single suture on the resulting 2--sphere boundary.  Then, there exists an isomorphism
\[
\SFH(Y(1)) \to \HFH(Y).
\]
\end{theorem}
\begin{theorem}[Juh\'asz 2006, \cite{Ju}]
\label{thm:SFHisHFK}
Let $K$ be a knot in a closed 3-manifold $Y$ and denote by $Y(K)$ the complement of an open tubular neighborhood of $K$ in $Y$. Let $\Gamma_\mu$ be two disjoint meridional sutures on $\partial Y(K)$. Then there is an isomorphism
\[
\SFH(Y(K),\Gamma_\mu)\to \HFKH(Y,K).
\] 
\end{theorem}


\subsection{Relative Spin$^{\mathbb{C}}$ structures and gradings} 
\label{sub:rel_spinc_gradings}

Here, we discuss how to put a grading on the sutured Floer homology groups using relative $\spinc$ structures \cite{Ju2} and, in the case where the sutured manifold comes from a null-homologous knot complement (with meridional sutures), we can see that this grading and Theorem~\ref{thm:SFHisHFK} can be used to recover the Alexander grading on knot Floer homology.

\subsubsection{Relative $\spinc$ structures} 
\label{relspinc}
Given a manifold $Y$ with boundary, choose a non-zero vector field $v_0$ in $TY$ along $\partial Y$. We can define the relative $\spinc$ structures on $Y$ to be  is the set of homology classes of non-zero vector field on $Y$ that restrict to $v_0$ on $\partial Y$. We say two non-zero vector fields are homologous if they are homotopic in the compliment of a 3-ball in the interior of $Y$. Notice that if $v_0'$ is another non-zero vector field along $\partial Y$ that is homotopic to $v_0$ through non-zero vector fields then we can use the homotopy to identify the relative $\spinc$ structures defined by $v_0$ and those defined by $v_0'$ and, if we restrict attention to a contractible set of choices for $v_0$, then these identifications are canonical. 

Consider a sutured manifold $(Y,\Gamma)$. In \cite{Ju2} relative $\spinc$ structures were defined by choosing a vector field $v_0$ that points out of $Y$ along $\partial Y_+$ and into $Y$ along $\partial Y_-$ (and is tangent to $\partial Y$ along $\Gamma$ and pointing into $\partial Y_+$). The set of relative $\spinc$ structures on $(Y,\Gamma)$ defined using such a $v_0$ is denoted by $\spinc(Y,\Gamma)$ and is well defined independent of $v_0$ since the possible choices for $v_0$ form a contractible set.

There is the standard map from the generators of the sutured homology chain groups to $\spinc$ structures
\[
{\s}: \TT_\alpha \cap \TT_\beta \to \spinc(Y,\Gamma)
\]
defined by using the intersections corresponding to points in $\TT_\alpha \cap \TT_\beta$ to pair the critical points of a Morse function corresponding to the chosen Heegaard diagram which used to compute the sutured Floer homology. 
The sutured Floer homology groups can be decomposed by $\spinc$ strucure
\[
\SFH(Y,\Gamma)=\bigoplus_{\s \in \spinc(Y,\Gamma)} \SFH(Y,\Gamma,\s).
\]

Given a vector field $v$ representing an element $\s\in \spinc(Y,\Gamma)$ let $v^\perp$ denote the orthogonal complement of $v$ (using some fixed auxiliary metric). If each component $S$ of $\partial Y$ satisfies $\chi(S_+)=\chi(S_-)$ then $(Y,\Gamma)$ is called strongly balanced. (For manifolds with connected boundary this is of course the same as being balanced.) In this case $v_0^\perp$ is necessarily a trivial plane field, \cite{Ju2}, so there is a non-zero section which we denote $t_0$. We can then define the Euler class  of $\s$ relative to $t_0$, $c_1(\s,t_0)\in H^2(Y,\partial Y;\Z)$ as the obstruction to extending $t_0$ to a non-zero section of $v^\perp$.


\subsubsection{Knot complements and the Alexander grading} 
\label{sub:knot_alex}

We now consider the case of knot complements. Suppose that $K \subset Y$ is a null-homologous knot, and let $F$ be a Seifert surface for $K$.  To the pair $(Y,K)$, we associate the compact sutured manifold $(Y(K),\Gamma_\mu)$ as discussed above; where $Y(K)$ is the complement of an open tubular neighborhood of $K$, and $\Gamma_\mu$ consists of a pair of oppositely-oriented meridional sutures on $\partial Y(K)$.

The set of relative $\spinc$ structures on $(Y(K),\Gamma_\mu)$ is naturally an affine space over $\mr{H}^2(Y(K),\partial Y(K))$.

Choosing an orientation on the knot $K$, we have the natural map
\[
	i^* : \mr{H}^2(Y(K),\partial Y(K); \Z) \to \mr{H}^2(Y; \Z),
\]
induced by Poincar\'e duality and the inclusion of $Y(K)$ into $Y$.  Given a relative $\spinc$ structure on $(Y(K),\Gamma_K)$, Ozsv\'ath and Szab\'o show in Sections 2.2 and 2.4 of \cite{OS6} how to extend this relative $\spinc$ structure to a $\spinc$ structure on $Y$.\footnote{Strictly speaking, Ozsv\'ath and Szab\'o's construction takes a relative $\spinc$ on $Y(K)$, normalized to point outward along the boundary, and produces an absolute $\spinc$ structure on $Y$.  A careful reading of Sections 2.2 and 2.4 of \cite{OS6}, however, indicates that their techniques apply in this more general context.}  Thus, we obtain the natural map
\[
	G_{Y,K}: \spinc(Y(K),\Gamma_\mu) \to \spinc(Y),
\]
which is equivariant with respect to the actions of $\mr{H}^2(Y(K),\partial Y(K))$ and $\mr{H}^2(Y)$ on $\spinc(Y(K),\Gamma_\mu)$ and $\spinc(Y)$, respectively. 

Let $[F,\partial F]$ be a homology class Seifert surface in $\mr{H}^2(Y(K),\partial Y(K); \Z)$ for the null-homologous (oriented) knot $K \subset Y$, and let $\s \in \spinc(Y(K),\Gamma_K)$ be a relative $\spinc$ structure, with non-zero section along the boundary given by the oriented meridional vector field $t_\mu$.  We define the {\it Alexander grading} of $\s$ with respect to $[F,\partial F]$ via the formula
\begin{equation}\label{eqn:alexander}
	A_{[F,\partial F]}(\s) = \frac{1}{2} \langle c_1(\s,t_\mu), [F,\partial F] \rangle
\end{equation}
and notice that the isomorphism in Theorem~\ref{thm:SFHisHFK} preserves Alexander gradings. (In essence the kernel of the map $G_{Y,K}$ is $\Z$ and we can get a map from $ \spinc(Y(K),\Gamma_K)$ to $\Z$ by a choice of Seifert surface for $K$. Moreover, for the choices made here $c_1(\s,t_\mu)$ is an even cohomology class and so we can divide by $2$ to obtain a map onto $\Z$.)

\subsubsection{Convex surfaces and relative $\spinc$ structures}\label{ssec:convexandspin}
We now extend our discussion of relative Spin$^{\mathbb{C}}$ from above so that they are better suited for contact geometry. To a sutured manifold $(Y,\Gamma)$ notice that the set of vector fields $v_0$ that are positively transverse to $\partial Y_+$, negatively transverse to $\partial Y_-$, and in $T\partial Y$ is positively tangent to $\Gamma$, is contractible. So we could use any such $v_0$ to define relative $\spinc$ structures of $(Y,\Gamma)$ instead of the ones used above. Moreover by a homotopy supported near $\Gamma$ that will take one of these vector fields to one of those from Section~\ref{relspinc} and vice versa. Thus when defining relative $\spinc$ structures on $(Y,\Gamma)$ we are free to use either type of vector field along $\partial Y$. 

Notice that the plane field $v_0^\perp$ is transverse to $\Gamma$. More generally, we can homotope $v_0$ so that the plane field $v_0^\perp$ intersected with $T\partial Y$ induces any characteristic foliation for a convex surface divided by $\Gamma$. (Note we are not bringing contact geometry into the picture yet, just indicating the flexibility we have in choosing $v_0$.)

When $(Y,\Gamma)$ is a sutured manifold with torus boundary and $\Gamma$ consists of two parallel curves then we can always choose a $v_0$ so that $v_0^\perp$ induces a standard foliation on the boundary (that is agrees with the standard characteristic foliation on a torus described in Section~\ref{sub:tt}). Given this situation we can take a section $t_0$ of $v_0^\perp$ that is tangent to the ruling curves to define $c_1(\xi,t_0)$. One may easily check, {\em cf.\  }\cite[Lemma 4.6]{Honda00a}, that the class $c_1(\xi,t_0) \in H^2(Y,\partial Y;\Z)$ is independent of the ruling slope on $\partial Y$. 

We discuss a particular case of the relative Euler classes that will be useful in our construction. Recall from Section~\ref{sub:tt} the basic slice $A_i^\pm$ has dividing slope $\infty$ on the back torus and slope $-i$ on the front torus. Once may easily compute, or consult \cite[Section 4.7.1]{Honda00a}, that the relative Euler class $c_1(\xi,t_0)$ is the Poincar\'e dual of $\pm [(i+1)\mu+\lambda]$ (where $t_0$ is a section as above so that the oriented tangent vector to an $\infty$-curve followed by $t_0$ induces the positive orientation on the torus). Similarly the relative Euler class for $\widetilde{A}^\pm_i$ is the Poincar\'e dual of $\mp [(i-1)\mu+\lambda]$.

More generally once can compute that the relative Euler class of the contact structure on $C^\pm_{j,i}$ (see Section~\ref{sub:tt} to recall this notation) is the Poincar\'e dual of $\pm(i-j)\mu$.



\subsection{Contact structures and sutured Floer homology} 
\label{sub:HKM}

Given a balanced sutured manifold $(Y,\Gamma)$ and any contact structure $\xi$ on $Y$ that has convex boundary with dividing set $\Gamma$, Honda, Kazez and Mati\'c \cite{HKM2} defined a class
\[
\EH(\xi)\in \SFH(-Y,-\Gamma_\xi,\s_\xi)
\] 
that is an invariant of $\xi$, where $\s_\xi$ is the relative $\spinc$ structure on $Y$ corresponding to $\xi$. Actually this invariant is only defined up to sign when using $\Z$-coefficients, but in this paper we work exclusively over $\F = \Z/2$, and can thus ignore the sign ambiguity. 

A key component of our constructions will be the following  gluing theorem for sutured Floer homology of Honda, Kazez and Mati\'c. The map in the theorem spiritually amounts to ``tensoring with the contact class''.  Henceforth, we will refer to this map as the HKM gluing map.

\begin{theorem}[Honda-Kazez-Mati\'c \cite{HKM3}]\label{thm:hkm_gluing}
	Let $(Y_1,\Gamma_1)$ and $(Y_2, \Gamma_2)$ be two balanced sutured 3--manifolds, Suppose that $Y_1 \subset Y_2$ and $\xi$ is a contact structure on $Y_2-\mathrm{int } (Y_1)$ with convex boundary divided by $\Gamma_1\cup\Gamma_2$ so that each component of $Y_2-\mathrm{int } (Y_1)$  contains a boundary component of $Y_2$. Then there exists a ``gluing" map
\[
	\phi_{\xi} : \SFH(-Y_1,-\Gamma_{1}) \to \SFH(-Y_2,-\Gamma_{2}).
\] 
\end{theorem}
The map in this theorem is only well defined up to sign when $\Z$-coefficients are used, but we can again ignore this ambiguity since we are working over $\F$.

Furthermore, the map above respects contact invariants.

The invariant $\EH(\xi)$ of a contact structure respects the map in Theorem~\ref{thm:hkm_gluing}.
\begin{theorem}[Honda-Kazez-Mati\'c \cite{HKM3}]\label{thm:gluingandctinvt}
Let $(Y_1,\xi_1)$ and $(Y_2, \xi_2)$ be a compact contact 3-manifolds with convex boundary, and suppose that $(Y_1,\xi_1) \subset (Y_2,\xi_2)$.  If each component of $Y_2\backslash\mathrm{int } (Y_1)$  contain a boundary component of $Y_2$, then the map $\phi_{\xi_2-\xi_1}$ from Theorem~\ref{thm:hkm_gluing} respects contact invariants. That is,
\[
	\phi_{\xi_2-\xi_1}(\EH(Y_1,\xi_1))= \EH(Y_2,\xi_2).
\] 
\end{theorem} 

By associating Honda, Kazez and Mati\'c's contact invariant to the complement of an open standard neighborhood of a Legendrian knot $L$, one obtains an invariant of $L$
\[
	\EH(L) \in \SFH(-Y(L),-\Gamma_L),
\]
which lives in the sutured Floer homology groups of the complement $Y(L)$, with sutures given by the resulting dividing curves on $\partial Y(L)$.

Since $\EH(L)$ is, by definition, the contact invariant of the complement $(Y(L),\xi_L)$, it follows from the theorem above that $\EH(L)$ vanishes if the complement of $L$ possesses a compact submanifold $(N,\xi|_N)$ with $\EH(N,\xi|_N) =0$.  Recall that convex neighborhoods of both overtwisted disks \cite{HKM2} and Giroux torsion layers \cite{GHV} have vanishing contact invariant.  Therefore, the invariant $\EH(L)$ vanishes if either the complement of $L$ is overtwisted or has positive Giroux torsion.


\subsection{Relationships between sutured Legendrian invariants} 
\label{sub:relationships}

It is natural to seek connections and commonalities between the invariant defined by Honda, Kazez and Mati\'c and those defined by Lisca, Ozsv\'ath, Stipsicz and Szab\'o.  The first substantive progress along these lines was accomplished by Stipsicz and V\'ertesi in \cite{StV}.  There, they proved

\begin{othertheorem}[Stipsicz-V\'ertesi \cite{StV}]\label{thm:sv_hkm_loss}
	Let $L \subset (Y,\xi)$ be a Legendrian knot.  Then there exists a map
\[
	\phi_{SV}: \SFH(-Y(L), -\Gamma_L) \to \HFKH(-Y,L)
\]
which sends the invariant $\EH(L)$ to $\SLH(L)$.
\end{othertheorem}

Stipsicz and V\'ertesi use the  of Honda, Kazez and Mati\'c gluing map from Theorem~\ref{thm:hkm_gluing} to construct their map as follows.  First, they attach a basic slice to the boundary of $Y(L)$ so that the dividing set on the resulting manifold consists of two meridional sutures.  A picture of this basic slice attachment is depicted on the left hand side of Figure~\ref{fig:YB}.  Recall from Section~\ref{sub:conv_surf} that there are two possible signs, positive and negative, one can choose for this basic slice.  Stipsicz and Vertesi choose to attach a negative basic slice to ensure that the contact $3$-manifold obtained via their construction does not change if we modify $L$ by negative stabilization.

\begin{definition}\label{def:st_attach}
	Let $L \subset (Y,\xi)$ be a Legendrian knot and $(Y(L),\xi_L)$ the complement of an open standard neighborhood of $L$.  We call the basic slice attachment discussed in the above paragraph a {\it Stipsicz-V\'ertesi} attachment, and denote the resulting contact $3$-manifold $(Y(L),\xib_L)$
\end{definition}

\begin{figure}[htbp]
	\centering
	\includegraphics[scale=1]{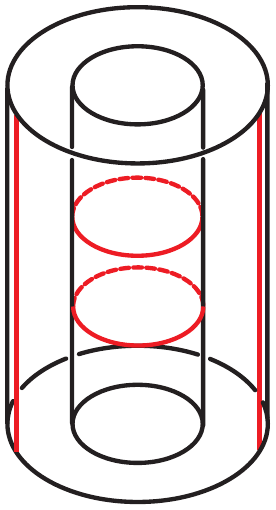}
	\hspace{30px}
	\includegraphics[scale=1]{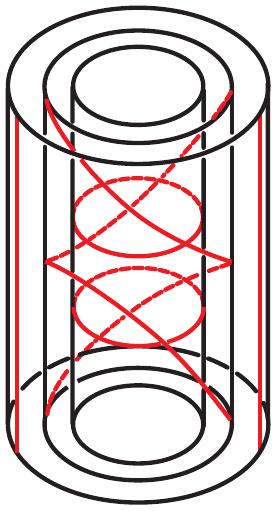}
	\caption{The Stipsicz-V\'ertesi attachment and a factorization.}
	\label{fig:page_and_slice}\label{fig:YB}\label{fig:YB_factor}
\end{figure}

It follows immediately from this definition that the space $(Y(L),\xib_L)$ depends only on the Legendrian $L$ up to negative stabilization.  Recall  from the discussion in Section~\ref{sub:conv_surf} that if $L^-$ is the negative stabilization of $L$, then the complement of $(Y(L^-),\xi_{L^-})$ is obtained form $(Y(L),\xi_L)$ by attaching a negatively signed basic slice $A^-_{n}$ where $n$ is the Thurston-Bennequin invariant of $L$.  By factoring the basic slice attachment yielding $(Y(L),\xib_L)$ as shown on the right hand side of Figure~\ref{fig:YB_factor}, we see it as a composition of two attachments, the first yielding $(Y(L^-),\xi_{L^-})$ and the second $(Y(L^-),\xib_{L^-}) = (Y(L),\xib_L)$.

Since the dividing set on $(Y(L),\xib_L)$ consists of two meridional sutures, it follows from \cite{Ju} that
\[
	\SFH(-Y(L),-\Gamma_{\xib_L}) \cong \HFKH(-Y,L).
\]
By analyzing a Heegaard diagram adapted to both $\EH(L)$ and $\SLH(L)$, Stipsicz and V\'ertesi are able to conclude that $\EH(Y(L),\xib_L) = \SLH(L)$.

From this construction, one obtains a simple proof of Theorem~\ref{thm:loss_stab} for $\SLH$.  If $L^+$ is the positive stabilization of $L$, then $(Y(L^+),\xi_{L^+})$ is obtained from $(Y(L),\xi_{L})$ by attaching a positive basic slice to the boundary of $Y(L)$.  The composition of this positive basic slice with the Stipsicz-V\'ertesi basic slice attachment is then overtwisted, forcing $\EH(Y(L^+),\xib_{L^+})$ to vanish.  Since, as discussed above, negative stabilizations of $L$ factor through the Stipsicz-V\'ertesi attachment, Theorem~\ref{thm:loss_stab} follows.

We also notice that $\EH(Y(L),\xib_L) = \SLH(L)$ sits in the $\spinc$ component of suture Floer homology $\SFH(-Y(L),-\Gamma_\mu, \s_{\xib_L})$, so to see its Alexander grading we need to evaluate $c_1(s_{\xib_L},t_\mu)$ on the Seifert surface $F$ for $L$. Choosing ruling curves on all tori involved that are parallel the $\partial F$ we see that $\langle c_1(\s,t_\mu), [F,\partial F] \rangle$ can be evaluated in two steps. When the Euler class of $\xi_L$ is evaluated on  the component of $F$ contained in $Y(L)$ it is well known to contribute minus the rotation number $-r(L)$, see \cite{EH}.
In Section~\ref{ssec:convexandspin} above we saw that the contact structure on $A^-_n$ will evaluate to $n+1$ on the annulus $F\cap A^-_n$, where $n=tb(L)$ is the Thurston-Bennequin invariant of $L$. Thus, the Alexander grading of $\EH(Y(L),\xib_L) = \SLH(L)$ is
\begin{equation}\label{eqn:LOSS_Alex}
 	\frac{1}{2} \langle c_1(\s,t_\mu), [F,\partial F] \rangle= \frac12 (tb(L)-r(L)+1).
\end{equation}



\section{Limits and Invariants of Knots} 
\label{sec:limits}

In  Subsections~\ref{sub:top_limits} and \ref{sub:positive_stabilization}  we present definitions of the sutured limit invariant, $\SFHL$, and its $\F[U]$-module structure. In Subsection~\ref{sub:alexander} an ``Alexander grading" is given to $\SFHL$. We then discuss some of the properties of this invariant in Subsection~\ref{sub:natural_maps}. In Subsection~\ref{sub:leg_limits} we define the limit invariant $\EHL$ of Legendrian and transverse knots and discuss its properties. 
The definition of the inverse limit invariant  $\SFHIL$ is quite similar to the definition of $\SFHL$. In Subsection~\ref{sub:top_inverselimits} we quickly define the inverse limit invariants $\SFHIL$, discuss it properties and define the corresponding Legendrian and transverse invariant $\EHIL$. 

\subsection{The sutured limit homology groups of a knot} 
\label{sub:top_limits}

Given a knot $K$ in a closed 3--manifold $Y$ denote the complement of an open tubular neighborhood of $K$ by $Y(K)$. A choosing a framing on $K$ is equivalent to choosing a longitude $\lambda$ on $\partial Y(K)$. We now fix a choice of longitude $\lambda$ and let $\Gamma_0$ be a union of two disjoint, oppositely oriented copies of $\lambda$ on $\partial Y(K)$. Then $(Y(K),\Gamma_0)$ is a balanced sutured manifold.

Using notation from the end of Subsection~\ref{sub:tt} we define the {\em meridional completion} of $Y(K)$ to be
\[
	(Y(K),\Gamma_\mu) := [(Y(K),\Gamma_0) \cup A_0^-]/\sim,
\]
where $T^2\times \{1\}$ is identified to $\partial Y(K)$ so that $S^1\times\{pt\}$ is mapped to a meridian of $K$ and the dividing curves on $T^2\times\{1\}$ are mapped to the sutures on $\partial Y(K)$. (Note that this can be done since the dividing curves and sutures are ``longitudinal".) The manifold $(Y(K),\Gamma_\mu)$ is naturally a sutured manifold with sutures $\Gamma_\mu$ coming from the dividing curves on $\partial A_0^-$, that is $\Gamma_\mu$ consists of two meridional curves. As noted in Subsection~\ref{sub:tt} there are convex tori $T_i$ in $A_0^-\subset (Y(K),\Gamma_\mu)$ whose dividing curves are parallel to $\lambda-i\mu$ and such that $T_j$ is closer to the boundary of $(Y(K),\Gamma_\mu)$ than $T_i$ if $j>i$. Thus, we have a sequence of sutured manifolds $(Y(K), \Gamma_i)$ given as the closure of the component of the complement of $T_i$ in $(Y(K),\Gamma_\mu)$ not containing $\partial (Y(K),\Gamma_\mu)$, with sutures coming from the dividing curves of $T_i$\footnote{Strictly speaking, we have a sequence of distinct manifolds $\{Y_i(K)\}$, each contained in the next.  However, since the $Y_i(K)$ are all pairwise diffeomorphic, we drop the subscript to avoid obscuring future discussions.}.

Note that for any $j>i$ we have the inclusion $(Y(K),\Gamma_i) \subset (Y(K),\Gamma_j)$ and that $\overline{(Y(K),\Gamma_j) \backslash (Y(K),\Gamma_i)}$ has a contact structure on it. More specifically 
\[
	\overline{(Y(K),\Gamma_j) \backslash (Y(K),\Gamma_i)}
\] 
is the contact manifold $C^-_{j,i}$. Using the HKM-gluing maps in sutured Floer homology discussed in Theorem~\ref{thm:hkm_gluing} we obtain maps 
\[
\phi_{ij}:\SFH(-Y(K),-\Gamma_i)\to \SFH(-Y(K),-\Gamma_j)
\]
if $i\leq j$.  
\begin{proposition}\label{prop:limitgroup}
	Let $K$ be a knot in $Y$. With the notation above the collection 
\[
	(\{\SFH(-Y(K),-\Gamma_i)\},\{\phi_{ij}\})
\] 
of sutured Floer homology groups and maps together form a directed system. 
\end{proposition}
\begin{proof}
From Proposition~\ref{prop:relations} we know the contact structure on 
\[
\overline{(Y(K),\Gamma_j) \backslash (Y(K),\Gamma_i)}
\]
is the same as the one on 
\[
\overline{(Y(K),\Gamma_j) \backslash (Y(K),\Gamma_k)} \cup  \overline{(Y(K),\Gamma_k) \backslash (Y(K),\Gamma_i)}
\]
for any $j>k>i$. The proposition follows by the well-definedness of the gluing map in sutured Floer homology.  
\end{proof}

This leads us to the following definition.

\begin{definition}\label{def:limitgroup}
	Let $K \subset (Y,\xi)$ be a null-homologous Legendrian knot and consider the associated directed system $(\{\SFH(-Y(K),-\Gamma_{i})\},\{\phi_{ij}\})$ given by Proposition~\ref{prop:limitgroup}.  The \emph{sutured limit homology} of $(-Y,K)$ is defined by taking the directed limit
\[
	\SFHL(-Y,K) = \varinjlim_{\phi_{ij}} \SFH(-Y(K),-\Gamma_{i}).
\]
\end{definition}
Denoting $\phi_{i,i+1}$ by $\phi_-$ for each $i$, and noting that the maps 
\begin{center}\label{fig:alglimit}
\begin{tikzpicture}	[->,>=stealth',auto,thick]
	\node (a) at (0,0){$\SFH(-Y(K),-\Gamma_0)$} ;
	\node (b) at (4.25,0) {$\SFH(-Y(K),-\Gamma_{1})$} ;
	\node (c) at (8.75,0) {$\SFH(-Y(K),-\Gamma_{2})$} ;
	\node (d) at (11.75,0) {$\dots$} ;
	
	\draw (a) edge node[above] {\small $\phi_-$} (b);
	\draw (b) edge node[above] {\small $\phi_-$} (c);
	\draw (c) edge node[above] {\small $\phi_-$} (d);
\end{tikzpicture}
\end{center}
form a cofinal sequence in our directed system, we can compute the sutured limit homology using just the $\phi_-$ maps. 

The only choices made in the definition of the sutured limit homology was that of a framing on $K$. We note that the sutured limit homology is independent of that choice and so is only an invariant of the knot $K$ in $Y$.

\begin{theorem}\label{thm:limit_top}
The sutured limit invariant $\SFHL(-Y,K)$ depends only on the knot type of $K$ in $Y$ and not the choice of framing used in the definition. 
\end{theorem}

\begin{proof}[Proof of Theorem~\ref{thm:limit_top}]\label{pf:limit_top}
Let $\lambda_a$ and $\lambda_b$ be two longitudes for a knot $K$ in $Y$. Let $Y^{ac}(K)$ and $Y^{bc}(K)$ be the meridional completions of $Y(K)$ with respect to the two different longitudes. We note that both these completions are canonically diffeomorphic to $Y(K)$ with meridional sutures. Moreover we can assume that there is some non-positive number $n$ such that  $\lambda_a=\lambda_b+n\mu$. Given this we see that $(Y^b_{n+i}(K),\Gamma^b_{n+i})$ is canonically (up to isotopy) diffeomorphic to $(Y^a_i(K),\Gamma^a_i)$. These diffeomorphism induce isomorphisms of the sutured Floer homology groups $\SFH(-Y^b_{n+i}(K),-\Gamma^b_{n+i})$ and $\SFH(-Y^a_i(K),-\Gamma^a_i)$. These isomorphisms commute with the maps $\phi_{ij}$ and thus induce an isomorphism of the resulting direct limits. 
\end{proof}


\subsection{The $U$-action on the sutured limit homology} 
\label{sub:positive_stabilization}

Recall, using the notation from the previous section, that
\[
\overline{(Y(K),\Gamma_{i+1}) \backslash (Y(K),\Gamma_i)}
\] 
is the basic slice $B_{i+1}^-$. And the contact structure on $B_{i+1}^-$ gave rise to the gluing map
\[
	\phi_-:\SFH(-Y(K),-\Gamma_i) \to \SFH(-Y(K),-\Gamma_{i+1}).
\]
The region $\overline{(Y(K),\Gamma_{i+1}) \backslash (Y(K),\Gamma_i)}$ can also be given the contact structure $B_{i+1}^+$. That contact structure will induce a gluing map
\[
	\psi_+:\SFH(-Y(K),-\Gamma_i)\to \SFH(-Y(K),-\Gamma_{i+1}).
\]

The maps $\phi_-$ and $\psi_+$ together fit into a diagram, shown in Figure~\ref{fig:limit_action} whose commutativity is the content of Proposition~\ref{pro:pos_commute}.
\begin{figure}[htbp]
	\centering
	\begin{tikzpicture}	[->,>=stealth',auto,thick]
		\node (a) at (0,0){$\SFH(-Y(K),-\Gamma_0)$} ;
		\node (b) at (4.75,0) {$\SFH(-Y(K),-\Gamma_{1})$} ;
		\node (c) at (9.5,0) {$\SFH(-Y(K),-\Gamma_{2})$} ;
		\node (d) at (12.5,0) {$\dots$} ;
		\node (e) at (0,-2){$\SFH(-Y(K),-\Gamma_{1})$} ;
		\node (f) at (4.75,-2) {$\SFH(-Y(K),-\Gamma_{2})$} ;
		\node (g) at (9.5,-2) {$\SFH(-Y(K),-\Gamma_{3})$} ;
		\node (h) at (12.5,-2) {$\dots$} ;

		\draw (a) edge node[above] {\small $\phi_-$} (b);
		\draw (b) edge node[above] {\small $\phi_-$} (c);
		\draw (c) edge node[above] {\small $\phi_-$} (d);
		\draw (e) edge node[above] {\small $\phi_-$} (f);
		\draw (f) edge node[above] {\small $\phi_-$} (g);
		\draw (g) edge node[above] {\small $\phi_-$} (h);
		\draw (a) edge node[right] {\small $\psi_+$} (e);
		\draw (b) edge node[right] {\small $\psi_+$} (f);
		\draw (c) edge node[right] {\small $\psi_+$} (g);
	\end{tikzpicture}
	\caption{}
	\label{fig:limit_action}
\end{figure}

\begin{proposition}\label{pro:pos_commute}
	The diagram shown in Figure~\ref{fig:limit_action} is commutative.
\end{proposition}

\begin{proof}
On the thickened torus $\overline{(Y(K),\Gamma_{i+2}) \backslash (Y(K),\Gamma_i)}$ one can consider the contact structures $B_{i+1}^+\cup B_{i+2}^-$ and $B_{i+1}^-\cup B_{i+2}^+$. The former induces the map $\psi_+\circ\phi_-$ and the later induces  the map $\phi_-\circ\psi_+$. Item (3) in Proposition~\ref{prop:relations} says that  $B_{i+1}^+\cup B_{i+2}^-$ and $B_{i+1}^-\cup B_{i+2}^+$ are the same contact structure so the well-definedness of the gluing maps in sutured Floer homology implies that $\psi_+\circ\phi_-=\phi_-\circ\psi_+$.
\end{proof}

It follows from Proposition~\ref{pro:pos_commute} that the collection of maps $\{\psi_+\}$ together induce a well-defined map on sutured limit homology
\[
	\Psi: \SFHL(-Y,K) \to \SFHL(-Y,K).
\]
As an immediate consequence, we obtain the following theorem.

\begin{theorem}\label{thm:U_action}
	Let $K$ be a smoothly embedded null-homologous knot in a 3-manifold $Y$.  The sutured limit homology $\SFHL(-Y,K)$ of the pair $(Y,K)$ can be given the structure of an $\F [U]$-module, where $U$ acts on elements of $\SFHL(-Y,K)$ via the map $\Psi$:
\[
	U \cdot [\x] = \Psi([\x]).
\]\qed
\end{theorem}


\subsection{An Alexander grading} 
\label{sub:alexander}
In this section, we show how to endow the sutured limit homology groups with an absolute Alexander grading which will later be shown to agree with the usual Alexander grading on knot Floer homology. We note that in the previous subsections all definitions could be made whether or not $K$ in $Y$ was null-homologous. To define the Alexander grading it is important that $K$ is null-homologous and that in the definition of $\SFHL(-Y,K)$ we take our initial longitude $\lambda$ to be the one coming from the Seifert surface for $K$. 

Let $K$ be a null-homologous knot in a 3--manifold $Y$ and $F$ a Seifert surface for $K$.  If $\SH = (\Sigma,\balpha,\bbeta,z,w)$ is a sutured Heegaard diagram for the space $(Y(K),\Gamma_i)$, we define the Alexander grading of a generator $\x \in \SG(\SH)$ via the formula
\[
	A_{[F,\partial F]}(\x) = \frac{1}{2} \langle c_1(\s(\x)),t_0), [F,\partial F] \rangle,
\]
where $t_0$ is any non-zero section as discussed in Section~\ref{ssec:convexandspin}.

Recall that the maps $\phi_{ij}:\SFH(-Y(K),-\Gamma_i) \to \SFH(-Y(K),-\Gamma_j)$ used to define the sutured limit invariants are defined via the contact manifold $C^-_{j,i}$.  From the discussion at the end of Section~\ref{ssec:convexandspin}, we see that the map $\phi_{ij}$ is Alexander-homogeneous of degree $(j-i)/2$.  We similarly see that the maps $\psi_+$, which are induced by positive basic slice attachment, are Alexander homogeneous of degree $-1/2$.

To obtain a well-defined Alexander grading on the sutured limit homology groups $\SFHL(-Y,K)$, we introduce shift operators into the directed system.  Specifically, we consider the sequence
\begin{center}
\begin{tikzpicture}	[->,>=stealth',auto,thick]
	\node (a) at (0,0){$\SFH(-Y(K),-\Gamma_0)[-1/2]$} ;
	\node (b) at (4,0) {$\dots$} ;
	\node (c) at (8,0) {$\SFH(-Y(K),-\Gamma_i)[-(i+1)/2]$} ;
	\node (d) at (12,0) {$\dots$} ;
	
	\draw (a) edge node[above] {\small $\phi_-$} (b);
	\draw (b) edge node[above] {\small $\phi_-$} (c);
	\draw (c) edge node[above] {\small $\phi_-$} (d);
\end{tikzpicture}
\end{center}

It follows from the discussion in Section~\ref{ssec:convexandspin} above that each of the maps in the collections $\{\phi_-\}$ and $\{\psi_+\}$ are Alexander-homogeneous of degrees $0$ and $-1$ respectively.  Thus, upon taking the direct limit, we obtain a well-defined Alexander grading on sutured limit homology for which multiplication $U$ decreases grading by a factor of $1$.  The initial grading shift $[-1/2]$ ensures that the Alexander grading we have just defined on sutured limit homology matches the usual one knot Floer homology.


\subsection{Natural Maps} 
\label{sub:natural_maps}

We now turn our attention to natural maps on sutured limit homology induced by the Stipsicz-V\'ertesi basic slice attachment and meridional 2-handle attachment respectively.  Proofs of Theorems \ref{thm:SV_map} and \ref{thm:2handle}, which characterize the maps $\Phi_{SV}$ and $\Phi_{2h}$ in terms of the identification between $\SFHL(-Y,K)$ and $\HFKM(-Y,K)$ will be given in Sections \ref{sec:SV_limit} and \ref{sec:proof_lim_hfh} respectively.

We begin by focussing on the map induced by the Stipsicz-V\'ertesi basic slice attachment --- henceforth referred to as the ``SV attachment''.  Recall that given $(Y(K),\Gamma_i)$ we can attach the basic slice $A_i^-$ to obtain the manifold $(Y(K),\Gamma_\mu)$. As noted in Subsection~\ref{sub:SFH}  we know that $\SFH(-Y(K),-\Gamma_\mu)$ is isomorphic to $\HFKH(-Y,K)$. Thus the gluing map coming from the contact structure on $A^-_i$ induces the Stipsicz-V\'ertesi map
\[
	\phi_{SV}: \SFH(-Y(K),-\Gamma_i) \to \HFKH(-Y,K).
\]
\begin{proposition}
The collection of gluing maps formed by applying the SV attachment to $(Y(K),\Gamma_{i})$, for each $i \geq 0$, together fit into the commutative diagram depicted in Figure~\ref{fig:big_SV} and all maps in the diagram respect the Alexander grading. 
\end{proposition}
\begin{figure}[htbp]
	\centering
	\begin{tikzpicture}	[->,>=stealth',auto,thick]
		\node (a) at (0,0){$\SFH(-Y(K),-\Gamma_0)$} ;
		\node (b) at (4.5,0) {$\SFH(-Y(K),-\Gamma_{1})$} ;
		\node (c) at (9.5,0) {$\SFH(-Y(K),-\Gamma_{2})$} ;
		\node (d) at (12.5,0) {$\dots$} ;
		\node (e) at (0,-2){$\HFKH(-Y,K)$};

		\draw (a) edge node[above] {\small $\phi_-$} (b);
		\draw (b) edge node[above] {\small $\phi_-$} (c);
		\draw (c) edge node[above] {\small $\phi_-$} (d);
		\draw (a) edge node[right] {\small $\phi_{SV}$} (e);
		\draw (b) edge node[right] {\small $\quad\phi_{SV}$} (e);
		\draw (c) edge node[right] {\small $\qquad\phi_{SV}$} (e);
	\end{tikzpicture}
	\caption{ }
	\label{fig:big_SV}
\end{figure}
\begin{proof}
This is again a simple consequence of the classification of contact structures given in Proposition~\ref{prop:relations} and the naturality of the HKM gluing maps.  
\end{proof}

Therefore, the collection $\{\phi_{SV}:\SFH(-Y(K),-\Gamma_i) \to \HFKH(-Y,K)\}$ induces a map on the sutured limit homology.

\begin{proposition}\label{pro:lim_hat}
	Let $K$ be a null-homologous knot in a 3--manifold $Y$.  There exists a well-defined, Alexander grading preserving map $\Phi_{\mr{SV}}: \SFHL(-Y,L) \to \HFKH(-Y,L)$ which is induced by the SV attachment, and whose constituent maps are depicted in Figure~\ref{fig:big_SV}.  \qed
\end{proposition}

There is one additional geometrically meaningful construction one can perform to the space $(Y(K),\Gamma_i)$  --- meridional contact 2--handle attachment.  We obtain the topological manifold $Y^{2h}(K)$ from $Y(K)$ by attaching a topological 2--handle along a meridional curve in $Y(K)$ that intersects $\Gamma_i$ minimally (twice). The boundary of $Y^{2h}(K)$ consists of the annulus $A$ that was part of the boundary of $Y(K)$ and two disks coming form the 2--handle. The sutures $\Gamma^{2h}$ on $Y^{2h}(K)$ consists of $A \cap \Gamma_i$  (that is two arcs) and an arc in each disk coming from the 2--handle that connects the endpoints of $A \cap \Gamma_i$. Notice that $\partial Y^{2h}(K)$ is a sphere and $\Gamma^{2h}$ is a simple closed curve. In other words, $(Y^{2h}(K),\Gamma^{2h}) = Y(1)$.  Thus, as discussed in Subsection~\ref{sub:SFH}, there exists a natural identification
\[
	\SFH(-Y^{2h}(K),-\Gamma^{2h}) \to \HFH(-Y).
\]

There is a unique tight contact structure (up to a choice of compatible characteristic foliation on the boundary) on the 2-handle so that the boundary is convex with corners and the sutures are the induced dividing curves. We use this contact structure to obtain the gluing map
\[
	\phi_{2h}: \SFH(-Y(K),-\Gamma_i) \to \HFH(-Y).
\]
It follows that the collection of gluing maps formed by attaching meridional contact 2-handles to $(Y(K),\Gamma_i)$, for each $i \geq 0$, together fit into the diagram depicted in Figure~\ref{fig:lim_hat_closed}, whose commutativity is the subject of Proposition~\ref{pro:lim_hat_closed_comm}.
\begin{figure}[htbp]
	\centering
	\begin{tikzpicture}	[->,>=stealth',auto,thick]
		\node (a) at (0,0){$\SFH(-Y(K),-\Gamma_0)$} ;
		\node (b) at (4.5,0) {$\SFH(-Y(K),-\Gamma_{1})$} ;
		\node (c) at (9,0) {$\SFH(-Y(K),-\Gamma_{2})$} ;
		\node (d) at (12,0) {$\dots$} ;
		\node (e) at (0,-2){$\HFH(-Y)$};

		\draw (a) edge node[above] {\small $\phi_-$} (b);
		\draw (b) edge node[above] {\small $\phi_-$} (c);
		\draw (c) edge node[above] {\small $\phi_-$} (d);
		\draw (a) edge node[right] {\small $\phi_{2h}$} (e);
		\draw (b) edge node[right] {\small $\quad\phi_{2h}$} (e);
		\draw (c) edge node[right] {\small $\qquad\phi_{2h}$} (e);
	\end{tikzpicture}
	\caption{ }
	\label{fig:lim_hat_closed}
\end{figure}

\begin{proposition}\label{pro:lim_hat_closed_comm}
There exists a well-defined map $\Phi_{\mr{2h}}: \SFHL(-Y,L) \to \HFH(-Y)$ which is induced by meridional contact 2-handle attachment, and whose constituent maps are depicted in Figure~\ref{fig:lim_hat_closed}.  
\end{proposition}

\begin{proof}\label{pf:lim_hat_closed_comm}
Let $M$ be the contact manifold obtained from the vertically invariant contact structure on $T^2\times[0,1]$ with dividing curves of slope $-i$ by attaching a contact 2--handle to $T^2\times\{0\}$. Similarly let $M'$ be the contact manifold obtained from the basic slice $B_{i+1}^-$ by attaching a contact 2--handle to $T^2\times\{0\}$. One may easily check that both of these contact structures are contactomorphism to the complement of an open standard contact ball inside the tight contact structure on the solid torus with convex boundary having dividing slope $i$. Thus, the naturality of the HKM gluing maps yields the claimed result. 
\end{proof}


\subsection{Legendrian and Transverse Invariants: Definition and Properties} 
\label{sub:leg_limits}

We now turn our attention to defining an invariant $\EHL$ of Legendrian and transverse knots which takes values in the sutured limit homology groups $\SFHL(-Y,K)$.  Although its definition is qualitatively different, we will see Section~\ref{sec:equiv_leg} that the invariant $\EHL$ is identified with the Legendrian/transverse invariants defined by Lisca, Ozsv\'ath, Stipsicz and Szab\'o in \cite{LOSS} under the isomorphism given in Theorem~\ref{thm:lim_to_minus}.

\subsubsection{Definition of the Legendrian/transverse invariant} 
\label{sub:definition_leg_invt}

Let $K \subset (Y,\xi)$ be a Legendrian knot.  In Section~\ref{sub:top_limits}, we defined the sutured limit homology group $\SFH(-Y,K)$ by forming the directed limit of the following sequence of groups and maps.
\begin{center}
\begin{tikzpicture}	[->,>=stealth',auto,thick]
	\node (a) at (0,0){$\SFH(-Y(K),-\Gamma_0)$} ;
	\node (b) at (4.5,0) {$\SFH(-Y(K),-\Gamma_{1})$} ;
	\node (c) at (9,0) {$\SFH(-Y(K),-\Gamma_{2})$} ;
	\node (d) at (12.5,0) {$\dots$} ;
	
	\draw (a) edge node[above] {\small $\phi_-$} (b);
	\draw (b) edge node[above] {\small $\phi_-$} (c);
	\draw (c) edge node[above] {\small $\phi_-$} (d);
\end{tikzpicture}
\end{center}
We also showed that the resulting $\F[U]$-module $\SFHL(-Y,K)$ depends only on the topological type of the Legendrian knot $K$.

Notice that if we choose the framing on $K$ used in the definition of $\SFHL$ to be the contact framing, then the sutured manifold $(Y(K),\Gamma_0)$ is precisely the sutured manifold one obtains by removing a standard neighborhood of $K$ from $Y$. Moreover $(Y(K),\Gamma_i)$ is precisely the sutured manifold obtained by removing a standard neighborhood of the $i$-times negatively stabilized $K$, $S_-^i(K)$, from $Y$. Thus there is a natural contact structure $\xi_{K,i}$ on $(Y(K),\Gamma_i)$ coming from the complement of a standard neighborhood of $S_-^i(K)$. Therefore, associated to the Legendrian knot $K$, we have a collection of contact invariants $\{\EH(S^i_-(K))\in \SFH(-Y(K),-\Gamma_i)\}$.

Theorem~\ref{thm:Leg_stab_bs} says that the contact manifold $(Y(K),\xi_{K,i})$ with the basic slice $B_{i+1}^-$ attached to it, is contact isotopic to $(Y(K),\xi_{K,i+1})$. Thus the collection $\{\EH(S^i_-(K))\}$ satisfies $\phi_i(\EH(S^i_-(K))) = \EH(S^{i+1}_-(K))$ for each $i \geq 0$.

\begin{definition}\label{def:leg_lim_invt}
	Let $K \subset (Y,\xi)$ be a Legendrian knot and $S^i_-(K)$ its $i^{th}$ negative stabilization.  We define the {\em LIMIT invariant} of $K$ to be the element $\EHL(K) \in \SFHL(-Y,K)$ given as the residue class of the collection $\{\EH(S^i_-(K))\}$ of HKM invariants associated to the $S^i_-(K)$s inside $\SFHL(-Y,K)$.
\end{definition}

From the discussing at the end of Section~\ref{sub:relationships}, we have that the Alexander grading of $\EHL(K)$ in $\SFHL(-Y,K)$ is $\frac{1}{2} (tb(K)-r(k)+1)$.

From Definition~\ref{def:leg_lim_invt}, we see that the class $\EHL$ defines a Legendrian invariant.  Furthermore, since the invariant $\EHL$ is obtained as a residue class over all possible negative stabilizations of a given Legendrian knot, we have the following.

\begin{theorem}\label{thm:ehl_neg_stab}
	Let $K$ be a null-homologous Legendrian knot and let $K_-$ denote its negative Legendrian stabilization, then $\EHL(K_-) = \EHL(K)$.
\end{theorem}

It follows immediately from Theorem~\ref{thm:ehl_neg_stab} that $\EHL$ gives rise to a transverse invariant through Legendrian approximation.

\begin{definition}\label{def:trans_lim_invt}
	Let $K \subset (Y,\xi)$ be a transverse knot and $L_K$ a Legendrian approximation of $K$.  We define $\EHL(K) = \EHL(L_K)$.
\end{definition}


\subsubsection{Properties of the Legendrian/Transverse Invariant} 
\label{sub:prop_leg_invt}

We now take a moment to discuss some useful and important properties of the Legendrian/transverse invariant $\EHL$ defined above.  These properties should be compared with their analogues for the invariant $\SL$ defined by Lisca, Ozsv\'ath, Stipsicz and Szab\'o in light of the equivalence promised by Theorem~\ref{thm:lim_to_minus_leg}.

Recall that Theorem~\ref{thm:ehl_neg_stab} states that $\EHL$ remains unchanged under negative Legendrian stabilization.  The following theorem describes the corresponding behavior of $\EHL$ under positive Legendrian stabilization.

\begin{theorem}\label{thm:ehl_pos_stab}
	Let $K$ be a Legendrian knot and let $K_+$ denote its positive Legendrian stabilization, then 
	\[\EHL(K_+) = U \cdot \EHL(K).\]
\end{theorem}

\begin{proof}\label{pf:ehl_pos_stab}
Denote the $i^{th}$ negative stabilizations of $K$ and $K_+$ by $K_i$ and $K_{+,i}$, respectively.  Then, for each $i \geq 0$, the contact manifold $(Y(K_{+}),\xi_{K_{+,i}})$ is obtained from $(Y(K),\xi_{K_i})$ by attaching a positively signed basic slice to its boundary.  The gluing maps induced by these basic slice attachments are precisely the $\psi_+$ maps defining $U$-multiplication on $\SFHL(-Y,K)$, and discussed in Section~\ref{sub:positive_stabilization}.

Since the HKM gluing maps respect contact invariants, we have that, for each $i \geq 0$,
\[
	\EH(K_{+,i}) = \psi_+(\EH(K_i)).
\]
Thus,
\[
	\EHL(K_+) = \Psi(\EHL(K)) = U \cdot \EHL(K),
\]
completing the proof of Theorem~\ref{thm:ehl_pos_stab}.
\end{proof}

The next three theorems illustrate some natural relations connecting $\EHL$ to previously defined invariants of Legendrian and transverse knots.  We begin with a theorem concerning the relationship between the LIMIT invariant and the HKM invariant, whose truth follows immediately from the definitions of the sutured limit homology $\SFHL$ and the LIMIT invariant $\EHL$.

\begin{theorem}\label{thm:HKM_to_ehl}
	Let $K \subset (Y,\xi)$ be a null-homologous Legendrian knot and $(Y(K),\xi_K)$ the contact manifold obtained by removing a open standard tubular neighborhood of $K$ from $(Y,\xi)$.  Under the natural map
\[
	\iota: \SFH(-Y(K),-\Gamma_K) \to \SFHL(-Y,K),
\]
induced by inclusion, the invariant $\EH(K)$  is sent to $\EHL(K)$.\qed
\end{theorem}

The next theorem describes the result of applying the Stipsicz-V\'ertesi map to the invariant $\EHL$.

\begin{theorem}\label{thm:ehl_sv}
	Let $K \subset (Y,\xi)$ be a null-homologous Legendrian knot.  Under the Stipsicz-V\'ertesi map $\Phi_{SV} : \SFHL(-Y,K) \to \HFKH(-Y,K)$, the class $\EHL(K)$ is identified with the LOSS invariant $\SLH(K)$.
\end{theorem}

\begin{proof}\label{pf:ehl_sv}
	The main theorem of \cite{StV} states that under the map 
\[
	\phi_{SV} : \SFH(-Y(K),-\Gamma_K) \to \HFKH(-Y,K),
\]
the HKM invariant $\EH(K)$ is identified with $\SLH(K)$.  Combining this result with the definition of $\EHL(K)$ and the commutativity of the diagram shown in Figure~\ref{fig:big_SV} defining the map $\Phi_{SV}$, we have that $\Phi_{SV}(\EHL(K)) = \SLH(K)$.
\end{proof}

The theorem below illustrates how the LIMIT invariant of a Legendrian or transverse knot relates to the classical contact invariant of the ambient space.

\begin{theorem}\label{thm:leg_ehl_eh}
	Let $K \subset (Y,\xi)$ be a null-homologous Legendrian knot.  Under the map $\Phi_{2h} : \SFHL(-Y,K) \to \HFH(-Y)$ induced by 2-handle attachment, the class $\EHL(K)$ is identified with the contact invariant $\EH(Y,\xi)$ of the ambient space.
\end{theorem}

The proof of this theorem is similar to that of Theorem~\ref{thm:ehl_sv}, so we omit it. The key observation is that since the HKM gluing maps respect contact invariants, the constituent maps defining $\Phi_{2h}$ each identify the elements $\EH(K_i)$ with $\EH(Y,\xi)$.  Otherwise, the proof is identical.



\subsection{The sutured inverse limit homology of a knot} 
\label{sub:top_inverselimits}

As usual, given a knot $K$ in a closed 3--manifold $Y$, we let $Y(K)$ denote the complement of an open tubular neighborhood of $K$.  Choosing a framing on $K$ is equivalent to choosing a longitude $\lambda$ on $\partial Y(K)$. Let $\Gamma_\mu$ be the union of two disjoint copies of the meridian of $K$ on $\partial Y(K)$, and consider the sutured manifold $(Y(K),\Gamma_\mu)$.

Using notation from the end of Subsection~\ref{sub:tt} we define a {\em longitudinal completion} of $Y(K)$ to be
\[
	(Y(K),\Gamma_\lambda) = [(Y(K),\Gamma_\mu) \cup \widetilde{A}_0^-]/\sim,
\]
where $T^2\times \{0\}$ is identified to $\partial Y(K)$ so that $\{pt\}\times S^1$ is mapped to the chosen longitude $\lambda$ of $K$ and the dividing curves on $T^2\times\{0\}$ are mapped to the sutures $\Gamma_\mu$ on $\partial Y(K)$. The manifold $(Y(K),\Gamma_\lambda)$ is naturally a sutured manifold with sutures $\Gamma_\lambda$ coming from the dividing curves on $\partial Y(K)$.  That is, $\Gamma_\lambda$ consists of two longitudinal curves. 

For notational ease in the following discussion, we will henceforth denote the longitudinal suture set $\Gamma_\lambda$ by $\Gamma_0$.  

As noted in Subsection~\ref{sub:tt} there are convex tori $\widetilde{T}_i$ in $\widetilde{A}_0^-\subset (Y(K),\Gamma_0)$ whose dividing curves are parallel to $\lambda+i\mu$ and such that $\widetilde{T}_i$ is closer to the (convex) boundary of $(Y(K),\Gamma_0)$ than $\widetilde{T}_j$ if $j>i$. Thus we have a sequence of sutured manifolds $(Y(K), \Gamma^+_i)$ given as the closure of the component of the complement of $\widetilde{T}_i$ in $(Y(K),\Gamma_0)$ not containing the boundary of $(Y(K),\Gamma_0)$, with sutures coming from the dividing curves of $\widetilde{T}_i$. (As in Section~\ref{sub:top_limits}, strictly speaking, we have a sequence of distinct manifolds $\{Y_i(K)\}$, each contained in its successor.  However, as before, since each of the $Y_i(K)$ are pairwise diffeomorphic, we drop the subscript to avoid obscuring the discussion.)

Note that for any $j>i$ we have the inclusion $(Y(K),\Gamma^+_j)\subset (Y(K),\Gamma^+_i)$, and that $\overline{(Y(K),\Gamma^+_i) \backslash (Y(K),\Gamma^+_j)}$ has a contact structure on it. More specifically 
\[
	\overline{(Y(K),\Gamma^+_{i}) \backslash (Y(K),\Gamma^+_{j})}
\] 
is the contact manifold  $C_{-i,-j}^-$. Using the HKM gluing maps in sutured Floer homology discussed in Theorem~\ref{thm:hkm_gluing}, we have maps 
\[
	\phi'_{ji}:\SFH(-Y(K),-\Gamma^+_{j}) \to \SFH(-Y(K),-\Gamma^+_{i})
\]
if $i\leq j$.  Just as in Proposition~\ref{prop:limitgroup} we have the following result.
\begin{proposition}\label{prop:limitgroup2}
	Let $K$ be a knot in $Y$. With the notation above the collection 
\[
	(\{\SFH(-Y(K),-\Gamma^+_{i})\},\{\phi'_{ji}\})
\] 
of sutured Floer homology groups and maps together form a directed system. \hfill \qed
\end{proposition}

This leads us to the following definition.

\begin{definition}\label{def:limitgroup2}
	Let $K \subset (Y,\xi)$ be a null-homologous Legendrian knot and consider the associated directed system $(\{\SFH(-Y^e_i(K),-\Gamma^+_{i})\},\{\phi^e_{ji}\})$ given by Proposition~\ref{prop:limitgroup2}.  The \emph{sutured inverse limit homology} of $(-Y,K)$ is defined by taking the inverse limit
\[
	\SFHIL(-Y,K) = \varprojlim_{\phi'_{ji}} \SFH(-Y(K),-\Gamma^+_{i}).
\]
\end{definition}

One may easily show, as in the proof of Theorem~\ref{thm:limit_top}, that this invariant is independent of the choice of longitude. 

\begin{theorem}\label{thm:limit_top2}
The sutured limit invariant $\SFHIL(-Y,K)$ depends only on the knot type of $K$ in $Y$ and not the choice of framing used in the definition. 
\end{theorem}

Analogously to the sutured limit homology, we can define a $U$-action. To this end we set $\phi'_-=\phi'_{i+1,i}$ and obtain the cofinal sequence 
\begin{center}\label{fig:alglimit2}
\begin{tikzpicture}	[<-,>=stealth',auto,thick]
	\node (a) at (0,0){$\SFH(-Y(K),-\Gamma_0)$} ;
	\node (b) at (4.25,0) {$\SFH(-Y(K),-\Gamma^+_{1})$} ;
	\node (c) at (8.75,0) {$\SFH(-Y(K),-\Gamma^+_{2})$} ;
	\node (d) at (11.75,0) {$\dots$} ;
	
	\draw (a) edge node[above] {\small $\phi'_-$} (b);
	\draw (b) edge node[above] {\small $\phi'_-$} (c);
	\draw (c) edge node[above] {\small $\phi'_-$} (d);
\end{tikzpicture}
\end{center}
from which $\SFHIL(-Y,K)$ can be computed. 

Each $\phi'_-$ is defined using the contact structure on the basic slice $B^-_{-i}$. We can similarly define
\[
	\psi'_+:\SFH(-Y(K),-\Gamma^+_{i+1})\to \SFH(-Y(K),-\Gamma^+_{i}).
\]
using the basic slice $B^+_{-i}$.

The same arugments used in the proof of Proposition~\ref{pro:pos_commute} show that the maps $\phi'_-$ and $\psi'_+$ together fit into the commutative diagram shown below.
\begin{figure}[htbp]
	\centering
	\begin{tikzpicture}	[<-,>=stealth',auto,thick]
		\node (a) at (0,0){$\SFH(-Y(K),-\Gamma_0)$} ;
		\node (b) at (4.75,0) {$\SFH(-Y(K),-\Gamma^+_{1})$} ;
		\node (c) at (9.5,0) {$\SFH(-Y(K),-\Gamma^+_{2})$} ;
		\node (d) at (12.5,0) {$\dots$} ;
		\node (e) at (0,-2){$\SFH(-Y(K),-\Gamma^+_{1})$} ;
		\node (f) at (4.75,-2) {$\SFH(-Y(K),-\Gamma^+_{2})$} ;
		\node (g) at (9.5,-2) {$\SFH(-Y(K),-\Gamma^+_{3})$} ;
		\node (h) at (12.5,-2) {$\dots$} ;

		\draw (a) edge node[above] {\small $\phi'_-$} (b);
		\draw (b) edge node[above] {\small $\phi'_-$} (c);
		\draw (c) edge node[above] {\small $\phi'_-$} (d);
		\draw (e) edge node[above] {\small $\phi'_-$} (f);
		\draw (f) edge node[above] {\small $\phi'_-$} (g);
		\draw (g) edge node[above] {\small $\phi'_-$} (h);
		\draw (a) edge node[right] {\small $\psi'_+$} (e);
		\draw (b) edge node[right] {\small $\psi'_+$} (f);
		\draw (c) edge node[right] {\small $\psi'_+$} (g);
	\end{tikzpicture}
\end{figure}

Thus the collection of maps $\{\psi'_+\}$ together induce a well-defined map on sutured inverse limit homology
\[
	\Psi': \SFHIL(-Y,K) \to \SFHIL(-Y,K).
\]
As an immediate consequence, we obtain the following theorem.

\begin{theorem}\label{thm:U_action2}
	Let $K$ be a smoothly embedded null-homologous knot in a 3-manifold $Y$.  The sutured inverse limit homology $\SFHIL(-Y,K)$ of the pair $(Y,K)$ can be given the structure of an $\F [U]$-module, where $U$ acts on elements of $\SFHIL(-Y,K)$ via the map $\Psi'$:
\[
	U \cdot [\x] := \Psi'([\x]).
\]\qed
\end{theorem}

The sutured inverse limit homology groups can be endowed with a well-defined Alexander grading using the method discussed in Section~\ref{sub:alexander}.

\subsubsection{A Natural Map} 

Recall that $(Y(K),\Gamma_\mu)$ sits as a sutured submanifold of $(Y(K),\Gamma_\lambda)$, and that the basic slices $\widehat{A}^-_i$ gives a contact structure on $\overline{(Y(K),\Gamma^+_i) \backslash (Y(K),\Gamma_\mu)}$.  Thus, the HKM gluing map from Theorem~\ref{thm:hkm_gluing} gives a maps
\[
	\phi_{dSV}:\SFH(-Y(K),-\Gamma_\mu) \to \SFH(-Y(K),-\Gamma^+_i)
\] 
and since $\SFH(-Y(K),-\Gamma_\mu)$ is isomorphic to $\HFKH(-Y,L)$, we have the commutative diagram in Figure~\ref{fig:big_SV2}.
\begin{figure}[htbp]
	\centering
	\begin{tikzpicture}	[<-,>=stealth',auto,thick]
		\node (a) at (0,0){$\SFH(-Y(K),-\Gamma_0)$} ;
		\node (b) at (4.5,0) {$\SFH(-Y(K),-\Gamma^+_{1})$} ;
		\node (c) at (9.5,0) {$\SFH(-Y(K),-\Gamma^+_{2})$} ;
		\node (d) at (12.5,0) {$\dots$} ;
		\node (e) at (0,-2){$\HFKH(-Y,K)$};

		\draw (a) edge node[above] {\small $\phi'_-$} (b);
		\draw (b) edge node[above] {\small $\phi'_-$} (c);
		\draw (c) edge node[above] {\small $\phi'_-$} (d);
		\draw (a) edge node[right] {\small $\phi_{dSV}$} (e);
		\draw (b) edge node[right] {\small $\quad\phi_{dSV}$} (e);
		\draw (c) edge node[right] {\small $\qquad\phi_{dSV}$} (e);
	\end{tikzpicture}
	\caption{ }
	\label{fig:big_SV2}
\end{figure}

It follows that the maps $\{\phi_{dSV}: \HFKH(-Y,K) \to \SFH(-Y(K),-\Gamma^+_i)\}$ together induce a map to the sutured inverse limit homology.
\begin{proposition}\label{pro:lim_hat2}
	Let $K$ be a null-homologous knot in a 3-manifold $Y$.  There exists a well-defined, grading preserving map $\Phi_{\mr{dSV}}:\HFKH(-Y,L) \to \SFHIL(-Y,L) $ which is induced by the constituent maps depicted in Figure~\ref{fig:big_SV2}.  \qed
\end{proposition}


\subsubsection{A Legendrian/transverse invariant in sutured inverse limit homology} 
\label{sub:definition_inverse_leg_invt}

Let $K \subset (Y,\xi)$ be a Legendrian knot.  Let $\nu(K)$ be a standard neighborhood of $K$ and notice that the sutured manifold $(Y(K),\Gamma_\mu)$ used to define $\SFHIL(-Y,K)$ is obtained from $\overline{(Y \backslash \nu(K))}$ by attaching a negative bypass as in the Stipsicz-V\'ertesi's construction from Section~\ref{sub:relationships}. As noted there, we have a contact structure $\overline{\xi}_K$ on the sutured manifold $(Y(K),\Gamma_\mu)$ and hence we have a contact structure $\overline{\xi}_{K,i}$ on each $(Y(K),\Gamma^+_i)$ by extending $\overline{\xi}_K$ by the contact structure on $\widetilde{A}^-_i$. From this we obtain a contact invariant
\[
	\EH(\overline{\xi}_{K,i})\in \SFH(-Y(K),-\Gamma^+_i)
\]
for each $i$ and, as when defining the direct limit Legendrian invariant, each element in the collection $\{\EH(\overline{\xi}_{K,i})\}$ is taken to another element in the collection by the $\phi'_{ji}$ maps used in the definition of $\SFHIL(-Y,K)$. Thus, we can define a inverse limit invariant as well. 

\begin{definition}\label{def:leg_lim_invt2}
	Let $K \subset (Y,\xi)$ be a Legendrian knot.  We define the {\em inverse LIMIT invariant} of $K$ to be the element $\EHIL(K) \in \SFHIL(-Y,K)$ given as the residue class of the collection $\{\EH(\overline{\xi}_{K,i})\}$ of HKM invariants associated to the $K$ inside $\SFHIL(-Y,K)$.
\end{definition}

From Definition~\ref{def:leg_lim_invt2}, we see that the class $\EHIL$ defines a Legendrian invariant.  Notice that the bypass attached to $\overline{(Y \backslash \nu(K))}$ to obtain the complement of the negatively stabilized $K$ embeds in the Stipsicz-V\'ertesi bypass but the bypass attached to $\overline{(Y \backslash \nu(K))}$ to obtain the complement of the positively stabilized $K$ when glued to the Stipsicz-V\'ertesi bypass yields an overtwisted contact structure. From this one easily concludes the following result. 

\begin{theorem}\label{thm:ehl_neg_stab2}
	Let $K$ be a null-homologous Legendrian knot and let $K_-$ and $K_+$ denote its negative and positive Legendrian stabilization, respectively, then $\EHIL(K_-) = \EHIL(K)$ and $\EHIL(K_+) = 0$. 
\end{theorem}

It follows immediately from this theorem that $\EHIL$ defines a transverse invariant through Legendrian approximation.

\begin{definition}\label{def:trans_lim_invt2}
	Let $K \subset (Y,\xi)$ be a transverse knot and $L_K$ a Legendrian approximation of $K$.  We define $\EHIL(K) := \EHIL(L_K)$.
\end{definition}

Lastly we observe the following result. 

\begin{theorem}\label{thm:ehl_sv2}
	Let $K \subset (Y,\xi)$ be a null-homologous Legendrian knot.  Under the map $\Phi_{dSV} : \HFKH(-Y,K)\to \SFHIL(-Y,K)$, defined in Proposition~\ref{pro:lim_hat2}, the the LOSS invariant $\SLH(K)$ is mapped to the class $\EHIL(K)$.
\end{theorem}




\part{Identifying the sutured limit homology package with the knot Floer homology package}\label{part2}

This part of the paper is devoted to proving our main theorems connecting the sutured limit invariants defined in Part~\ref{part1} with the more standard knot Floer homology package.

\section{Bordered Sutured Floer Homology} 
 \label{backgroud2}

We begin by reviewing some of the basic constructions and definitions from bordered sutured Floer homology.  For a more thorough and elementary treatment, we refer to interested reader to the book \cite{LOT1} on bordered Floer homology by Lipshitz, Ozsv\'ath and Thurston, and to the third author's paper \cite{Za1} extending this theory to the sutured category.

\subsection{Sutured manifolds and surfaces} 
\label{sub:sutured_manifolds}

We recall the definition of a sutured 3-manifold, originally due to Gabai \cite{Ga}.

\begin{definition}\label{def:bal_sut}
	A {\it sutured manifold} is a pair $(Y,\Gamma)$, where $Y$ is an oriented 3-manifold with boundary and $\Gamma$ is a collection of oriented, disjoint, simple closed curves on $\partial Y$ called {\it sutures}.  We further require that $Y$ contains no closed components, all boundary components of $Y$ have sutures and that the suture set $\Gamma$ divides $\partial Y$ into two regions $R_+$ and $R_-$ satisfying $\chi(R_+) = \chi(R_-)$.
\end{definition}

\begin{remark}
	The definition presented above is actually that of a {\it balanced annular sutured manifold} \cite{Ju}.  Since the sutured 3-manifolds encountered in Heegaard Floer theory are annular and generally satisfy the balancing condition, it is customary to omit the words ``balanced'' and ``annular'' when referring to such a space.
\end{remark}

Paralleling the above, in \cite{Za1}, the third author introduced the following 2-dimensional analogue of a sutured manifold.

\begin{definition}\label{def:sut_surf}
	A {\it sutured surface} is a pair $\SF = (F,\Lambda)$, where $F$ is an oriented 2-manifold, and $\Lambda$ is a collection of oriented, disjoint points on $\partial F$ called {\it sutures}.  We further require that $F$ contains no closed components, and that the suture set $\Lambda$ intersects each component of $\partial F$ nontrivially, dividing it into two components $S_+$ and $S_-$ satisfying $\partial S_\pm = \pm \Lambda$.
\end{definition}

\begin{definition}\label{def:sut_div}
	Let $\SF = (F,\Lambda)$ be a sutured surface.  A {\it dividing set} for $\SF$ is a finite collection $\Gamma$ of disjoint, embedded, oriented arcs and simple closed curves in $F$ for which $\partial \Gamma = -\Lambda$, as oriented submanifolds.  We further require that the dividing set $\Gamma$ separates $F$ into two regions $R_+$ and $R_-$ with $\partial R_\pm = (\pm \Gamma) \cup S_\pm$.
\end{definition}

For examples of a sutured surface, and a dividing set on a sutured surface, see Figure~\ref{fig:sut_surf_div}.

\begin{figure}[htbp]
	\centering
	{\begin{picture}(135,82)
	\put(0,0){\includegraphics[scale=2.25]{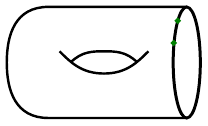}}
	\put(133,40){$S+$}
	\put(98,60){$S_-$}
	\put(118,65){$-$}
	\put(116,52){$+$}
	\end{picture}
	}
	\hspace{30px}
	{\begin{picture}(135,82)
	\put(0,0){\includegraphics[scale=2.25]{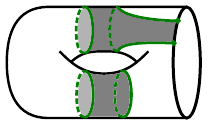}}
	\put(17,40){$R_+$}
	\put(93,23){$R_+$}
	\put(62,19){$R_-$}
	\put(75,60){$R_-$}
	\end{picture}
	}
	\caption{On the left is a sutured surface. On the right is a sutured surface equipped with a dividing set. }
	\label{fig:sut_surf_div}
\end{figure}
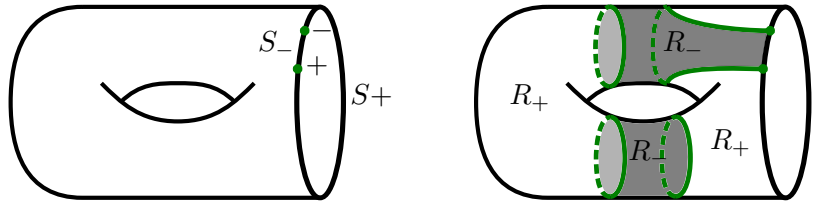


\subsection{Arc diagrams and bordered sutured manifolds} 
\label{sub:arc_diagrams_and_bordered_sutured_manifolds}

\begin{definition}\label{def:arc_diag}
	An {\it arc diagram} of rank $k$ is a triple $\SZ = (\BZ,\ba,M)$ consisting of a finite collection $\BZ$ of oriented arcs, a set of $2k$ disjoint points $\ba = \{a_1,\dots,a_{2k} \} \subset \BZ$, and a 2-to-1 matching $M : \ba \to \{1,\dots,k\}$ such that the 1-manifold obtained by performing 0-surgery along each 0-sphere $M^{-1}(i)$ in $\BZ$ has no closed components.
\end{definition}

Given an arc diagram $\SZ$, one can associate a a graph $G(\SZ)$ obtained from $\BZ$ by attaching 1--cells to points in $\ba$ according to the matching $M$. In addition, one can associate a sutured surface $\SF(\SZ) = (F(\SZ),\Lambda(\SZ))$ to it in the following way.  Starting with the product $\BZ \times [0,1]$, attach (oriented) 2-dimensional 1-handles along the 0-spheres in $M^{-1}(i) \times \{0\}$, for $i = 1,\dots,k$.  The suture set is given by $\Lambda(\SF) = -(\partial \BZ \times \{1/2\})$, and the positive and negative regions are the portions of the boundary $\partial F(\SZ)$ containing $\BZ \times \{1\}$ and $\BZ \times \{0\}$ respectively, see Figure~\ref{fig:par_surf}. We also notice that there is an obvious embedding of $G(\SZ)$ in $\SF(\SZ)$ such that $\BZ$ goes to $\BZ\times \{1/2\}$ and the 1-cells map to (extensions) of the cores of the 1-handles. When discussing the subset of arcs $\BZ$ inside $\SF(\SZ)$, we will always mean $\BZ\times \{1/2\}$.


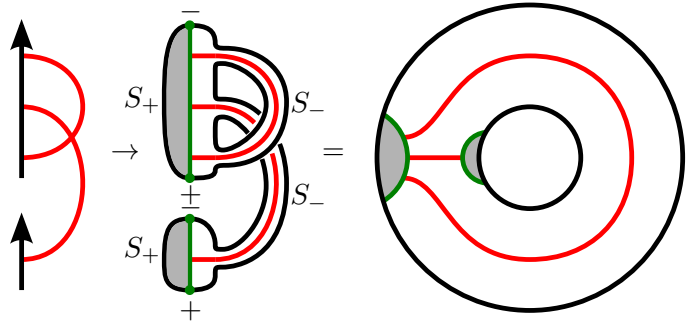
\begin{figure}[htbp]
\centering
	\begin{picture}(230,130)
		\put(0,5){\includegraphics[scale=1.2]{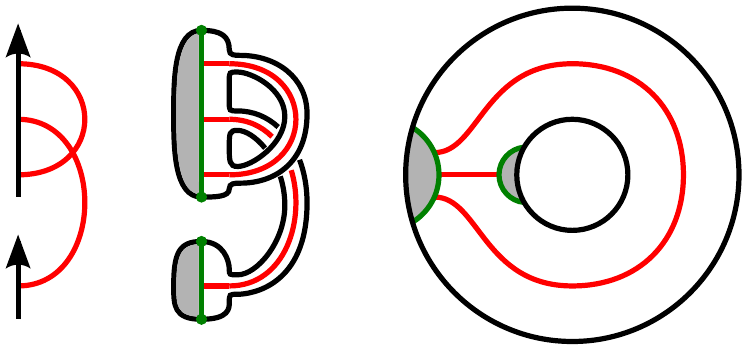}}
		\put(66,49){$+$}
		\put(66,118){$-$}
		\put(66,44){$-$}
		\put(66,5){$+$}
		\put(120,65){$=$}
		\put(40,65){$\to$}
		\put(45,85){$S_+$}
		\put(45,28){$S_+$}
		\put(108,83){$S_-$}
		\put(108,48){$S_-$}
	\end{picture}
	\caption{An arc diagram and its associated parametrized sutured surface.}
	\label{fig:par_surf}
\end{figure}


Let $\SF = (F,\Lambda)$ be a sutured surface and let $\SF(\SZ)$ be the sutured surface associated to an arc diagram $\SZ$.  If there exists a proper diffeomorphism $\iota : \SF(\SZ) \to \SF$, then we say that $\SZ$ {\it parametrizes} $\SF = (F,\Lambda)$.

\begin{definition}\label{def:bord_sut}
	A {\it bordered sutured manifold} $\Y = (Y,\Gamma,\SZ)$ is a (not necessarily balanced) sutured manifold $(Y,\Gamma)$, together with an embedding of the sutured surface $\SF(\SZ)$ into $\partial Y$ that sends $\BZ$, in an orientation preserving way, into $\Gamma$.
\end{definition}

An example of a bordered sutured manifold is depicted in Figure~\ref{fig:bord_sut_manifold}.

\begin{figure}[htbp]
	\centering
	\begin{picture}(240,100)
		\put(0,0){\includegraphics[scale=2.75]{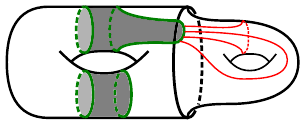}}
		\put(220,85){$\SF(\SZ)$}
		\put(115,35){$R_+$}
		\put(25,50){$R_+$}
		\put(95,73){$R_-$}
		\put(77,22){$R_-$}
	\end{picture}
		\caption{A bordered sutured manifold with sutured surface $\SF$ parametrized by the arc diagram $\SZ$.}
		\label{fig:bord_sut_manifold}
\end{figure}
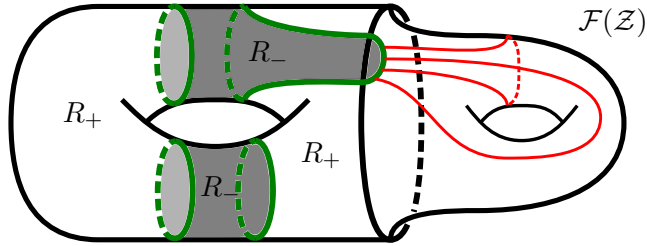

\begin{remark}
So far, the discussion has been focused exclusively on the bordered sutured category.  In ``classical'' bordered Floer homology, the concept of an arc diagram is replaced by that of a {\it pointed matched circle} --- roughly, an arc diagram with a single arc, whose tip and tail are identified via a marked point.  When working with bordered Heegaard diagrams, and when computing their corresponding invariants, this marked point plays the role of a usual basepoint in Heegaard Floer theory, as the sutures do in the discussion to follow. So, a ``classical" bordered manifold can be thought of as a sutured bordered manifold with suture a circle bounding a disk. 
\end{remark}


\subsection{The strands algebra} 
\label{sub:the_strands_algebra}

We now recall the definition of the strands algebra and bordered algebra from \cite{Za1}. These will both be differential graded algebras, but we omit the discussion of gradings in what follows and refer the interested reader to \cite{Za1}.

\begin{definition}\label{def:strand_alg}
	The {\it strands algebra} $\A(n,k)$ is a free $\F$-module with generators $\mu = (S,T,\phi)$, where $S$ and $T$ are both $k$-element subsets of $\{1,\dots,n\}$, and $\phi: S \to T$ is a non-decreasing bijection.  We denote by $\mr{Inv}(\mu)$ the set of inversions of the map $\phi$, that is, pairs $i<j$ in $S$ such that $\phi(i)>\phi(j)$. We also denote the cardinality of $\mr{Inv}(\mu)$ by $\mr{inv}(\mu) = \mr{inv}(\phi)$.  Multiplication in $\A(n,k)$ is then given by
\[
	(S,T,\phi) \cdot (U,V,\psi) = 
		\begin{cases}
			(S,V,\psi \circ \phi) \;\;\; & \text{if} \; T = U \; \text{and} \; \mr{inv}(\phi) + \mr{inv}(\psi) = \mr{inv}(\psi \circ \phi),\\
			0 & \text{else}.
		\end{cases}
\]
The differential is obtained by summing over all possible ways of ``resolving'' inversions (see below).
\end{definition}

The strands algebra is so-called because it has an obvious interpretation in terms of moving strands from the points of $S$ to those of $T$.  From this perspective, the differential corresponds precisely to resolving topological crossings between two strands.  There is an additional interpretation of the strands algebra in terms of Reeb chords along $\BZ$, see below and \cite{Za1}, and we use the terms ``strand'' and ``Reeb chord'' interchangeably according to context.

In the bordered sutured setting, we require a slight generalization of the strands algebra.  The {\it extended strands algebra} of the tuple $(n_1,\dots,n_\ell;k)$ is
\[
	\A(n_1,\dots,n_\ell;k) = \bigoplus_{k_1 + \dots + k_\ell = k} \A(n_1,k_1) \otimes \dots \otimes \A(n_\ell,k_\ell).
\]
We can view $\A(n_1,\dots,n_\ell;k)$ as a subalgebra of $\A(n_1 + \dots + n_\ell,k)$ by thinking of the components $\A(n_i,k_i)$ as acting on $\{(n_1 + \dots + n_{i-1})+1,\dots,(n_1 + \dots + n_{i-1})+n_i\}$ instead of $\{1,\dots,n_i\}$.

Let $(\BZ,\ba)$ be a finite collection of oriented arcs and a subset of $2k$ points as in Section~\ref{sub:arc_diagrams_and_bordered_sutured_manifolds} above.  Denote by $\BZ_i$ the $i^{\text{th}}$ oriented arc in $\BZ$ and let $\ba_i = \BZ_i \cap \ba$ be the subset of $\ba$ contained in $\BZ_i$.  The strands algebra associated to the pair $(\BZ,\ba)$ is
\[
	\A'(\BZ,\ba) = \bigoplus_{i = 1}^{2k}\A(|\ba_1|,\dots,|\ba_\ell|;i)
\]

Let $\SZ = (\BZ,\ba,M)$ be a rank $k$ arc diagram and $\A'(\BZ,\ba)$ the strands algebra associated to $(\BZ,\ba)$.  To each $i$-element subset $S \subset \{1,\dots,2k\}$, there exists and idempotent $I(S) = (S,S,\mr{id_S}) \in \A'(\BZ,\ba)$.  If $s \subset \{1,\dots,k\}$ is an $i$-element subset, then a {\it section} of $s$ is a subset $S \subset M^{-1}(s)$ such that $M|_S:S \to s$ is a bijection.  For each subset $s \subset \{1,\dots,k\}$ there is an idempotent
\[
	I_s = \sum_{S \; \text{a section of} \; s} I(S),
\]
obtained by summing over over all idempotents associated to sections over $s$.

\begin{definition}\label{def:strands_idemp}
	The ground ring $\SI(\SZ)$ associated to the arc diagram $\SZ = (\BZ,\ba,M)$ is the rank $2^k$ subalgebra of $\A'(\BZ,\ba)$ spanned by the collection of idempotents $\{I_s \; | \; s \subset \{1,\dots,k\}\}$.
\end{definition}

If we let $\SI(\SZ,i)$ denote the subalgebra of $\SI(\SZ)$ generated by the set $\{I_s \; | \; s \subset \{1,\dots,k\}, |s| = i\}$, then there exists a natural decomposition
\[
	\SI(\SZ) = \bigoplus_{i=0}^n \SI(\SZ,i).
\]

It is frequently convenient to focus on the subset of triples $(S,T,\phi)$, where $S,T \subset \{1,\dots,2k\}$, and $\phi:S \to T$ is a {\it strictly} increasing bijection.  In such a situation, we say that a subset $U \subset \{1,\dots,2k\}$ {\it completes} the pair $(S,T)$ if $U \cap (S \cup T) = \emptyset$.  Given, $(S,T,\phi)$ as above, we let
\[
	a_i(S,T,\phi) = \sum_{\substack{U \; \text{completes} \; (S,T)\\ |U \cup S| = i}} (S \cup U, T \cup U, \phi_U) \in \A'(\BZ,\ba),
\]
where $\phi_U|_S = \phi$, and $\phi_U|_U = \id_U$.

\begin{definition}\label{def:strands_assoc}
	The bordered algebra associated to the arc diagram $\SZ = (\BZ,\ba,M)$ is the algebra 
\[
	\A(\SZ) = \SI(\SZ) \cdot \A'(\BZ,\ba) \cdot \SI(\SZ) \subset \A'(\BZ,\ba).
\]
	It is generated over $\F$ by $\SI(\SZ)$ and elements of the form $I \cdot a_i(S,T,\phi) \cdot I$.
\end{definition}

The bordered algebra $\A(\SZ)$ is a module over the idempotent subalgebra $\SI(\SZ)$ and decomposes as a direct sum
\[
	\A(\SZ) = \bigoplus_{i = 0}^k \A(\SZ,i),
\]
where the constituents $\A(\SZ,i) = \SI(\SZ,i) \cdot \A(\SZ) \cdot \SI(\SZ,i)$ are modules over $\SI(\SZ,i)$.

One can alternatively describe the strands algebra $\A(\SZ)$ in terms of Reeb chords as follows.  If $\SZ = (\BZ,\ba,M)$ is an arc diagram, then, up to isotopy, there exists a unique (compatibly oriented) contact structure on the collection of arcs $\BZ$.  If we endow $\BZ$ with this contact structure, then the elements of $\ba \subset \BZ$ are Legendrian.  In this case, there exists a family of positively oriented Reeb chords whose beginning and end-points lie in $\ba$.  If $\brho = \{\rho_1,\dots,\rho_n\}$ is a collection of Reeb chords in $(\BZ,\ba)$, then we let $\brho^- = \{\rho_1^-,\dots,\rho_n^-\}$ and $\brho^+ = \{\rho_1^+,\dots,\rho_n^+\}$ denote the beginning and endpoints of the elements of $\brho$ respectively.

The idea is to use Reeb chords as geometric manifestations of the strictly increasing pairing functions discussed above.  For this to be possible, we must introduce an appropriate compatibility condition.

\begin{definition}\label{def:rho_comp}
	Let $\SZ=(\BZ, \ba, M)$ be an arc-diagram.  A collection of Reeb chords $\brho = \{\rho_1,\dots,\rho_n\}$ in $(\BZ,\ba)$, where $| \ba|=2k$, is said to be $i$-compatible if none of the $\rho_j$ are constant, the points $M(\rho_1^-),\dots,M(\rho_n^-)$ and, independently, $M(\rho_1^+),\dots,M(\rho_n^+)$ are all distinct, and $\#(M(\brho^-)\cup M(\brho^+)) \leq k - (i - n)$.
\end{definition}

Thinking of a collection of Reeb chords as a strictly increasing pairing function, the final condition above guarantees the existence of at lease one $(i-n)$-element subset $s \subset \{1,\dots,k\}$ which is disjoint from $M(\brho^-)\cup M(\brho^+)$, and which ``completes'' $\rho$.  That is, if $\rho$ is an $i$-compatible collection of Reeb chords and $s$ is an $i$-completion, then
\[
	a(\rho,s) = \sum_{S \; \text{is a section of} \; s} (\brho^- \cup S, \brho^+ \cup S, \phi_S),
\]
defines an element of $\A(\SZ,i)$, where $\phi_S(\rho_i^-) = \rho_i^+$ and $\phi_S$ is the identity on $S$.  Defining
\[
	a_i(\brho) = \sum_{s \; \text{an i-completion of} \; \brho} a(\brho,s),
\]
we see that $\A(\SZ,i)$ is generated over $\SI(\SZ)$ by the collection of elements $\{a_i(\brho)\}$, where $\brho$ is an $i$-compatible collection of Reeb chords.


\subsection{$A_\infty$-modules and Type-D structures} 
 \label{sub:_a_infty_modules_and_type_d_structures}

We now review basic definitions surrounding $A_\infty$-modules and Type-D structures.  Although everything that follows can be extended to $\Z$-coefficients, we work exclusively over $\F = \Z/2$ since this is all that is needed to define the bordered invariants, and to avoid sign complications. 

Recall an $A_\infty$-algebra over $\F$ is a pair $\A=(A,\{\mu_i\})$ where $A$ is graded $\F$-module and the $\mu_i$ are a sequence of multiplication maps 
\[
	\mu_i:A^{\otimes i}\to A[2-i]
\]
for $i=1,2 \ldots$, satisfying, for each $n$, the compatibility conditions
\[
	\sum_{i+j=n+1}\sum_{l=1}^{n-j+1} \mu_i(a_1\otimes \dots \otimes a_{l-1}\otimes\mu_j(a_l\otimes\dots\otimes a_{l+j-1})\otimes a_{l+j}\otimes \dots \otimes a_n)=0
\]
and the unital condition that there is an element $1\in A$ for which $\mu_2(a,1)=\mu_2(1,a)=a$ for all $a\in A$ and $\mu_i$ vanishes on $i$-tuples containing $1$ for $i\not=2$. 
Here and throughout this paper we denote by $A[n]$ is the module $A$ with grading shifted {\it down} by $n$ and by $A^{\otimes i}$ the tensor product over $\F$ of $i$ copies of $A$. If $\mu_i=0$ for $i>2$ then $A$ is simply a differential graded algebra (with differential $\mu_1$ and multiplication $\mu_2$).

\begin{definition}\label{def:a_infty}
	Let $\A$ be a unital, (graded) $A_\infty$-algebra over $\F$, with multiplication maps $\mu_i$, and $\SI$ the subalgebra of idempotents with orthogonal basis $\{I_i\}$ satisfying $\sum I_i = 1 \in \A$.  A (right unital) $A_\infty$-module over $\A$ is a graded module $M$ over the base ring $\SI$
\[
	M_\A = \bigoplus_i M \cdot I_i,
\]
together with a family of homogeneous maps
\begin{equation}\label{eqn:a_inf_mod}
	m_i : M \otimes A^{\otimes i-1} \to M[2-i], \;\; i \geq 1
\end{equation}
which together satisfy the $A_\infty$ structure conditions
\begin{align*}
	0 = \sum_{j = 1}^{n-1} \sum_{i = 1}^{n-j} m_n(x \otimes a_1 & \otimes \dots \otimes \mu_j(a_i \otimes \dots \otimes a_{i + j}) \otimes \dots \otimes a_{n-1})\\
	& + \sum_{i = 1}^n m_{n-i+1}(m_i(x \otimes a_1 \otimes \dots a_{i-1})  \otimes \dots \otimes a_{n-1}).
\end{align*}
and unital conditions
\begin{align*}
	m_2(x \otimes 1) &= x\\
	m_i(x \otimes \dots \otimes 1 \otimes \dots) &= 0, \;\; i \geq 2.
\end{align*}
We say that the $A_\infty$-module $M_\A$ is {\it bounded} if $m_i = 0$ for all sufficiently large $i$. 
\end{definition}

It is frequently convenient to represent a structure equation like \eqref{eqn:a_inf_mod} graphically.  For the case of an $A_\infty$-module, this is depicted in Figure~\ref{fig:typeA_structure}.

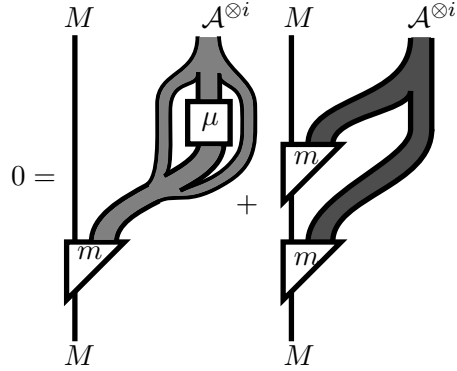
\begin{figure}[htbp]
\centering
	\begin{picture}(140,130)
		\put(0,5){\includegraphics[scale=1]{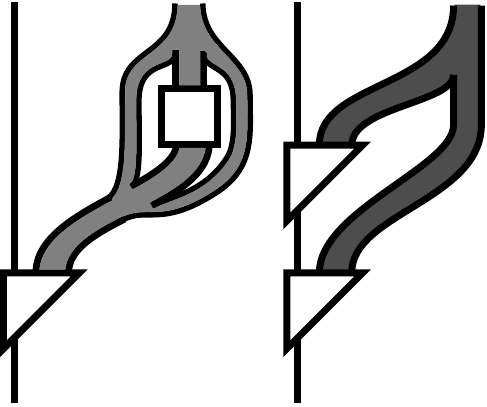}}
		\put(0,125){$M$}
		\put(83,125){$M$}
		\put(0,-3){$M$}
		\put(83,-3){$M$}
		\put(52,125){$\A^{\otimes i}$}
		\put(130,125){$\A^{\otimes i}$}
		\put(52,88){$\mu$}
		\put(5,37){$m$}
		\put(65,55){$+$}
		\put(-20,65){$0 =$}
		\put(87,35){$m$}
		\put(87,73){$m$}
	\end{picture}
	\caption{The structure equation for an $A_\infty$ module $M$.}
	\label{fig:typeA_structure}
\end{figure}

\begin{remark}
If $\SZ$ is an arc diagram, then the associated strands algebra $\A(\SZ)$ is actually a $DG$-algebra --- meaning that $\mu_i = 0$ for all $i > 2$.  Thus, the first summand in the structure equation for an $A_\infty$-module over such an $\A(\SZ)$ involves only terms containing $\mu_1$ and $\mu_2$.
\end{remark}

\begin{definition}\label{def:type_d}
	Let $\A$ be a unital $DG$-algebra over $\F$, with idempotent subalgebra $\SI$ as above.  A {\it (left) Type-$D$ structure} over $\A$ is a graded module $N$ over the base ring $\SI$
\[
	^\A N = \bigoplus_i I_i \cdot N
\]
together with a homogeneous map
\[
	\delta: N \to (A \otimes N)[1],
\]
satisfying the compatibility relation
\[
	(\mu_1 \otimes \id_N) \circ \delta + (\mu_2 \otimes \id_N) \circ (\id_A \otimes \delta) \circ \delta = 0
\]
\end{definition}

Iterating, we obtain a collection of maps indexed by $k \in \{0,1,\dots\}$
\[
	\delta_k : N \to (A^{\otimes k} \otimes N)[k],
\]	
where
\[
	\delta_k = \begin{cases} \id_N & \text{for} \; k = 0\\
	(\id_A \otimes \delta_{k-1}) \circ \delta \;\; & \text{for} \; k \geq 0.
	\end{cases}
\]
The structure equation for a Type-$D$ modules is shown in Figure~\ref{fig:typeD_structure}.  We say a Type-$D$ structure $^\A N$ is {\it bounded} if $\delta_k = 0$ for all sufficiently large $k$.

\begin{figure}[htbp]
\centering
	\begin{picture}(65,130)
		\put(0,5){\includegraphics[scale=1]{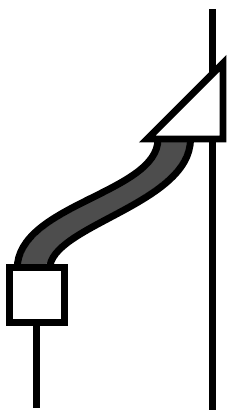}}
		\put(55,-2){$M$}
		\put(5,-2){$\A$}
		\put(55,125){$M$}
		\put(55,88){$\overline{\delta}$}
		\put(6,38){$\mu$}
		\put(-20,65){$0 =$}
	\end{picture}
	\caption{The structure equation for a Type-$D$ structure on $M$.}
	\label{fig:typeD_structure}
\end{figure}
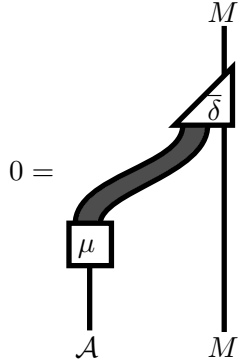

Given two Type-$D$ structures $(N,\delta)$ and $(N',\delta')$ over $\A$ an $\F$-module homomorphism $\psi: N\to A\otimes N'$ is a $D$-structure homomorphism if it satisfies
\[
	(\mu_2\otimes \id_{N'})\circ (\id_A \otimes \psi)\circ \delta + (\mu_2\otimes \id_{N'})\circ (\id_{A}\otimes \delta_{N'})\circ \psi + (\mu_1\otimes \id_{N'})\circ \psi=0.
\]
Given two $D$-structure homomorphisms $\phi: N\to A\otimes N'$  and $\psi:N'\to A\otimes N''$ their composition is defined to be the $D$-structure homomorphism $\psi\circ \phi$ form $N$ to $N''$ defined by 
\[
	(\mu_2\otimes \id_{N''})\circ(\id_A\otimes \psi)\circ \phi.
\]
We say that two Type-$D$ structure homomorphisms $\phi:N\to A\otimes N'$ and $\psi:N\to A\otimes N'$ are homotopic if there is a $D$-structure homotopy between them.  That is, a $\F$-module homomorphism $h: N \to A\otimes N'[-1]$ satisfying 
\[
	(\mu_2\otimes \id_{N'})\circ (\id_A \otimes h)\circ \delta + (\mu_2\otimes \id_{N'})\circ (\id_{A}\otimes \delta_{N'})\circ h + (\mu_1\otimes \id_{N'})\circ h=\psi-\phi.
\]

Given an $A_\infty$-module $M_\A$ and a Type-$D$ structure $^\A N$, at least one of which is bounded, we can form their {\it box tensor product} $M_\A \boxtimes {^\A N} = (M \otimes_\SI N, \partial^\boxtimes)$, with differential given by the formula
\[
	\partial^\boxtimes(x \otimes y) = \sum_{k = 0}^\infty (m_{k+1} \otimes \id_N)(x \otimes \delta_k(y)). 
\]
The boundedness assumption ensures that the above sum is finite.
\begin{figure}[htbp]
\centering
	\begin{picture}(63,130)
		\put(0,5){\includegraphics[scale=1]{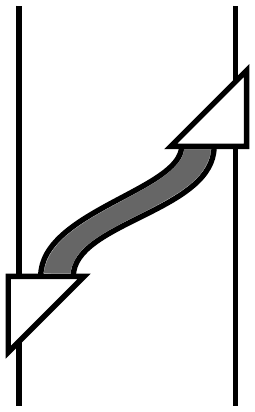}}
		\put(0,125){$M$}
		\put(65,125){$N$}
		\put(0,-3){$M$}
		\put(65,-3){$N$}
		\put(60,85){$\overline{\delta}$}
		\put(4,35){$m$}
		\put(-30,65){$\partial^\boxtimes = $}
	\end{picture}
	\caption{The structure equation for $\partial^\boxtimes$.}
	\label{fig:box_structure}
\end{figure}
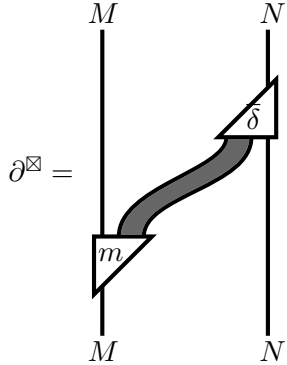

The following definition from \cite{LOT1} shows how to induce maps on box tensor products. 

\begin{definition}\label{def:boxmap}
	Let $\phi: N \to N'$ be a map of $D$-modules and $\id_M$ the identity. then {\it box tensor product} of $\id_M$ and $\phi$ is a map 
\[
	\id_M \boxtimes \phi: M \boxtimes N \to M \boxtimes N'
\]
given by
\[
	\id_M \boxtimes \phi (x \otimes y) = \sum_{k=0}^\infty (m_{k+1} \otimes \id_{N'})\circ(x \otimes \phi_k(y)),
\]
where the maps $\phi_k\to A^{\otimes k} \otimes N$ are defined inductively by
\[	
	\phi_k := \sum_{i+j = k-1} (\id_{A^{\otimes (i+1)}} \otimes \delta'_j) \circ (\id_{A^{\otimes i}} \otimes \phi) \circ \delta_i.
\]

\end{definition}

Graphically, the map $\id_M \boxtimes \phi$ can represented as in Figure~\ref{fig:boxmap}.

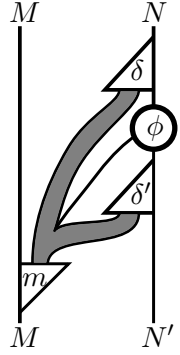
\begin{figure}[htbp]
\centering
	\begin{picture}(55,125)
		\put(0,5){\includegraphics[scale=1]{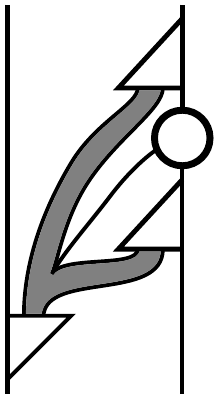}}
		\put(-2,120){$M$}
		\put(48,120){$N$}
		\put(-2,-3){$M$}
		\put(48,-3){$N'$}
		\put(44,98){$\delta$}
		\put(50,77){$\phi$}
		\put(44,51){$\delta'$}
		\put(3,20){$m$}
	\end{picture}
	\caption{The map $\mathrm{Id}_{M} \boxtimes \phi$.}
	\label{fig:boxmap}
\end{figure}

We will also make use of $A_\infty$-bimodules. As the $A_\infty$-algebras we will be concerned with are differential graded algebras we will define our bimodules over such algebras. See \cite{LOT2} for the more general definition. 
\begin{definition}
Let $\A$ and $\B$ be two differential graded algebras with underlying $\F$-modules $A$ and $B$, differential denoted $\partial_\A$ and $\partial_\B$ and multiplication denoted $\mu_\A$ and $\mu_\B$, respectively.  A Type-$DA$ structure over $\A$ and $\B$, denoted $^\A M_\B$, is a graded vector space $M$ over $\F$ together with a collection of graded maps
\[
	m_k:M\otimes B^{\otimes (k-1)}\to A\otimes M[2-k]
\]
satisfying 
\begin{align*}
	\sum_{p=1}^k (\mu_\A\otimes \id_M)&\circ(\id_\A\otimes m_{k-p+1})\circ(m_p\otimes \id_{B^{\otimes (k-p)}})+(\partial_\A \otimes \id_M)\circ m_k\\
	&+\sum_{p=0}^{k-2} m_k\circ (\id_M\otimes \id_{B^{\otimes p}}\otimes \partial_\B\otimes \id_{B^{\otimes(k-p-2)}})\\
	&+\sum_{p=0}^{k-3} m_k\circ(\id_M\otimes\id_{B^{\otimes p}}\otimes \mu_\B\otimes \id_{B^{\otimes(k-p-3)}})=0,
\end{align*}
for all $k\geq 0$, and the unital condition that $m_2(x\otimes 1)=1\otimes x$ and $m_k$ is zero when any entry is $1$ for $k\not=2$.
\end{definition}
Given such a Type-$DA$ structure, we can define maps $m_k^i:M\otimes B^{\otimes (k-1)} \to A^{\otimes i}\otimes M[1+i-k]$ by setting $m_1^0=\id_M, m_k^0=0$ for $k>1$, $m_k^1=m_k$ and then inductively define 
\[
	m_k^i=\sum_{j=0}^{k-1} (\id_{A^{\otimes(i-1)}} \otimes m_{j+1})\circ(m_{k-j}^{i-1}\otimes \id_{B^{\otimes j}}).
\]
Notice that, in the case that $A$ or $B$ is the trivial algebra, we get an $A_\infty$-module over $\B$ or a Type-$D$ structure over $\A$ (by ignoring the $m_k^i$ for $k>1$), respectively. There are notions of maps between Type-$DA$ modules and homotopies between such maps analogous to those for Type-$D$ modules discussed above. For details see \cite{LOT2}.

Now, given two Type-$DA$ structures $^\A M_\B$ and $^\B N_\C$ with maps $\{m_k^i\}$ and $\{n_l^j\}$, respectively, we define their {\em box tensor product} $^\A M_\B \boxtimes {^\B N}_\C$ to be the Type-$DA$ structure $^\A(M\otimes N)_\C$ with operations 
\[
(m\boxtimes n)^i_k= \sum_{j\geq 1} (m_j^i\otimes \id_N)\circ(\id_M\otimes n_k^{j-1}).
\]
If both $\A$ and $\C$ are trivial then one notes that this definition agrees with the box tensor product defined above. 


\subsection{The bordered invariants} 
\label{sub:the_bordered_invariants}
Let $\SZ = (\BZ,\ba,M)$ be an arc diagram.

\begin{definition}\label{def:bord_sut_heegaard}
A bordered sutured Heegaard diagram is a quadruple $\SH = (\Sigma, \balpha,\bbeta, \SZ)$ comprised of the following
\begin{itemize}
	\item $\Sigma$ a compact surface with no closed components
	\item $\balpha = \balpha^c \cup \balpha^a$ a collection of pairwise disjoint, properly embedded circles $\balpha^c$ and arcs $\balpha^a$ in $\Sigma$.
	\item $\bbeta$ a collection of pairwise disjoint, properly embedded circles.
	\item An embedding of the associated graph $G(\SZ) \to \Sigma$ such that $\BZ$ is sent to $\partial \Sigma$ in an orientation preserving way, and the 1-cells of $\mathcal{G}(\SZ)$ are identified with the arcs $\balpha^a$.
\end{itemize}
We further require that each component of $\Sigma - (\balpha^c \cup \balpha^a)$ and each component of $\Sigma - \bbeta$ intersects $\partial \Sigma - \BZ$.
\end{definition}

From a bordered sutured Heegaard diagram $\SH = (\Sigma, \balpha,\bbeta, \SZ)$ we can constructed a bordered sutured 3--manifolds $(Y,\Gamma,\SZ)$ as follows. The manifold $Y$ is simply $\Sigma\times [0,1]$ with 2--handles attached to $\Sigma\times\{1\}$ along the circles in $\bbeta$ and attached to $\Sigma\times\{0\}$ along the circles in $\balpha^c$. The sutures $\Gamma$ are $\partial\Sigma\times\{1/2\}$. Finally we construct the embedding of $\SF(\SZ)$ into $\partial Y$ by first embedding $G(\SZ)$. To this end, we embed $\BZ$ as $\BZ\times \{1/2\}\subset\partial \Sigma\times \{1/2\}$ and to each $\alpha \in \balpha^a$ we attach the 1-cell $(\partial \alpha\times[0,1/2])\cup \alpha\times\{0\}$. Now a small neighborhood of $G(\SZ)$ in $\partial Y$ gives an embedding of $\SF(\SZ)$ into $\partial Y$. 

Given a bordered sutured Heegaard diagram $\SH = (\Sigma, \balpha, \bbeta, \SZ)$, a {\it generator} for bordered sutured Floer homology is a collection of intersections $\x = (x_1,\dots,x_g)$ in $\balpha \cap \bbeta$ such that exactly one point comes from each $\alpha$-circle, exactly one point comes from each $\beta$ circle, and at most one point comes from each $\alpha$-arc.  We denote this generating set $\SG(\SH)$.

Let $\SH$ be a bordered sutured Heegaard diagram and $\x \in \SG(\SH)$ a generator.  The set of $\alpha$-arcs containing points of the generator $\x$ is denoted $o(\x)$ and is called the set of {\it occupied} arcs.  Similarly, we denote by $\overline{o}(\x) = \balpha^a - o(\x)$ the complement of the set of occupied arcs.

To a bordered sutured Heegaard diagram $\SH = (\Sigma,\balpha,\bbeta,\SZ)$ one associates two basic algebraic objects.  The first is a right $A_\infty$-module over the bordered strands algebra $\A(\SZ)$ and is denoted $\BSA(\SH)_{\A(\SZ)}$.  The second is a left Type-$D$ structure over $\A(-\SZ)$, and is denoted $^{\A(-\SZ)}\BSD(\SH)$.  More specifically the $A_\infty$-module $\BSA(\SH)$ is $X(\SH)$, where $X(\SH)$ is the $\F$-vector space generated by $\SG(\SH)$ and given a right $\SI(\SZ)$-module structure by the action
\[
	\x \cdot I(s) = \begin{cases}
						\x \;\;\; &\text{if} \; s = o(\x)\\
						0 & \text{else}.
					\end{cases}
\]
In a similar spirit, the Type-$D$ structure $\BSD(\SH)$ is defined to be $\A(-\SZ)\otimes_{\SI(-\SZ)} X(\SH)$ where the left action of $\SI(-\SZ)$ on $X(\SH)$ is given by
\[
	I(s) \cdot \x = \begin{cases}
						\x \;\;\; &\text{if} \; s = \overline{o}(\x)\\
						0 & \text{else}.
					\end{cases}
\]
The $A_\infty$-module and Type-$D$ structures on $\BSA(\SH)$ and $\BSD(\SH)$, respectively, are obtained by counting holomorphic curves in $\Sigma \times [0,1] \times \R$ with appropriate asymptotic --- we refer the reader to \cite{LOT1} and \cite{Za1} for details, but make a few remarks in the next subsection that are relevant to our computations below. 

The third author showed in \cite{Za1} that  $\BSA(\SH)$ and $\BSD(\SH)$ are both invariants of the underlying bordered sutured manifold specified by $\SH$ up to homotopy equivalence.  In addition, the third author established a pairing theorem which describes how these invariants behave if one glues together two bordered sutured 3-manifolds along a common parametrized sutured surface $\SZ$.

\begin{theorem}\label{thm:zarev_pairing}
	Let $(Y_1,\Gamma_1,\SZ)$ and $(Y_2,\Gamma_2,-\SZ)$ be two bordered sutured 3-manifolds and $(Y,\Gamma) = (Y_1 \cup_\SZ Y_2, \Gamma_1 \cup \Gamma_2)$ the sutured 3-manifold obtained by gluing them together along $\SZ$.  Then there exists a graded homotopy equivalence
\[
	\SFC(Y,\Gamma) \simeq \BSA(Y_1,\Gamma_1) \boxtimes \BSD(Y_2,\Gamma_2) \simeq \BSA(Y_1,\Gamma_1) \widetilde{\otimes}_{\A(\SZ)} \BSD(Y_2,\Gamma_2),
\]
where $\widetilde{\otimes}$ denotes the derived ($A_\infty$) tensor product.
\end{theorem}

Moving to bi-modules, let $\SH = (\Sigma,\balpha,\bbeta,-\SZ_1\cup \SZ_2)$ he a bordered sutured Heegaard diagram where the parameterized surface has two disjoint components one parameterized by $-\SF(\SZ_1)$ and the other by $\SF(\SZ_2)$. We define the Type-$DA$ structure $^{\A(\SZ_1)} \BSDA(\SH)_{\A(\SZ_2)}$ to be $\A(\SZ_1)\otimes_{\SI(\SZ_1)} X(\SH)$ where the left $\SI(-\SZ_1)$ and right $\SI(\SZ_2)$ module structures on $X(\SH)$ are defined by
\[
	I(s_1) \cdot \x \cdot I(s_2)=   \begin{cases}
						\x\;\:\; &\text{if } s_1=\overline{o}(\x) \text{ and } s_2 = o(\x)\\
						0 & \text{else}.
					\end{cases}
\]
The operators $\{m_k\}$ are defined by counting holomorphic curves in $\Sigma \times [0,1] \times \R$ with appropriate asymptotic. 

We also have a generalization of Theorem~\ref{thm:zarev_pairing}.
\begin{theorem}
Given two bordered sutured manifolds $(Y_1,\Gamma_1,-\SZ_1\cup \SZ_2)$ and $(Y_2,\Gamma_2, -\SZ_2\cup \SZ_3)$  we can glue them together along $\SF(\SZ_2)$ to get a sutured cobordism from $\SF(\SZ_1)$ to $\SF(\SZ_2)$. Then there is a graded homotopy equivalence of bimodules
\[
	\BSDA(Y_1\cup_{\SF(\SZ_2)} Y_2) \simeq \BSDA(Y_1)\boxtimes \BSDA(Y_2) \simeq \BSDA(Y_1) \widetilde{\otimes} \BSDA(Y_2).
\] 
\end{theorem}


\subsection{Nice Heegaard diagrams} 
 \label{sub:nice_heegaard_diagrams}
 In \cite{SW}, Sarkar and Wang showed how to compute the differential in Heegaard-Floer homology combinatorially if the Heegaard diagram was ``nice''. As discussed in \cite{Za1}, the same is true in the bordered sutured category. 
 
 A bordered sutured Heegaard diagram $\SH = (\Sigma, \balpha,\bbeta, \SZ)$ is {\em nice} if every region of $\Sigma\setminus (\balpha\cup \bbeta)$ either contains part of $\partial \Sigma\setminus \BZ$ or is a disk with at most four vertices. For such diagrams we can define the Type-$D$ structure $\delta$ on $\BSD(\SH)$ as follows. For each generator $\x$ of $\BSD(\SH)$ the differential $\delta(\x)$ is computed as follows:
 \begin{enumerate}
 \item\label{it1} Suppose $\y$ is another generator that differs from $\x$ in only one double point and $S$ is a convex bigon embedded in $\Sigma$ with boundary consisting of one arc from $\balpha$ and one arc from $\bbeta$, no points of $\x\cap \y$ in the interior of the bigon, traversing $\partial S$ near one of the double points of $\x$ in the direction induced from the orientation of $S$ one encounters the arc from $\bbeta$ and then the one from $\balpha$, and traversing the boundary near one of the double points of $\y$ one encounters the arcs in the opposite order. Then $S$ contributes $I(\overline{o}(\x)) \otimes \y$ to $\delta(\x)$. 
 \item\label{it2} Suppose $\y$ is another generator that differs from $\x$ at exactly two double points and $S$ is a convex rectangle embedded in $\Sigma$ with boundary consisting of two arcs from $\balpha$ and two arcs from $\bbeta$, no points of $\x\cap \y$ in the interior of the rectangle,  traversing $\partial S$ near one of the double points of $\x$ in the direction induced from the orientation of $S$ one encounters the arc from $\bbeta$ and then the one from $\balpha$, and traversing the boundary near one of the double points of $\y$ one encounters the arcs in the opposite order. Then $S$ contributes $I(\overline{o}(\x)) \otimes \y$ to $\delta(\x)$. 
 \item\label{it3} Suppose $\y$ is another generator that differs from $\x$ in only one double point and and $S$ is a convex rectangle embedded in $\Sigma$ with boundary consisting of two arcs from $\balpha$, one arc from $\bbeta$ and one Reeb chord $-\rho\in \BZ$,  no points of $\x\cap \y$ in the interior of the rectangle,  traversing $\partial S$ near one of the double points of $\x$ in the direction induced from the orientation of $S$ one encounters the arc from $\bbeta$ and then the one from $\balpha$, and traversing the boundary near one of the double points of $\y$ one encounters the arcs in the opposite order. Then $S$ contributes $I(\overline{o}(\x)) a(\rho)I(\overline{o}(\y))\otimes \y$ to $\delta(\x)$. 
 \end{enumerate}
 
Not all of our diagrams are nice but when computing Type-$D$ structures all the regions we compute will be (possibly immersed) bigons and rectangles, or embedded annuli. They will count towards $\delta$ in an entirely analogous way to the situation for nice diagrams (see for example, the computations in Appendix A of \cite{LOT1}).  When computing Type-$AD$ structures we make similar counts but in that case we will also need to consider (possibly immersed) annuli with varying numbers of  corners on each boundary components. (Again see the computations in Appendix A of \cite{LOT1} for the fact that these are counted analogously.)
 

\subsection{$\spinc$ structures in bordered sutured Floer homology} 
 \label{sub:_spinc_structures_in_bordered_sutured_floer_homology}

Let $Y$ be a 3-manifold (possibly with boundary), and $X \subset Y$ a subspace of $Y$.  Fix a non-zero vector field $v_0$ on the subspace $X$.  As discussed in Section~\ref{sub:knot_alex} for the case of knot complements (where $X = \partial Y = T^2$), the space of {\it relative $\spinc$-structures} $\spinc(Y,X,v_0)$, or simply $\spinc(Y,X)$ consists of non-vanishing vector fields $v$ on $Y$, such that $v|_{\partial Y} = v_0$, considered up to homology in $Y - X$.  The set $\spinc(Y,X,v_0)$ is an affine space over $\mr{H}_2(Y,X)$ whenever it is non-empty.

In the bordered setting, there are two flavors of (relative) $\spinc$-structure which depend on a chosen boundary normalization convention.  To understand the two conventions, let $\SH$ be a bordered sutured Heegaard diagram given by a boundary-compatible Morse function $f$, and let $\x \in \SG(\SH)$ be a generator of the associated bordered sutured Floer complex.  Consider the gradient vector field $\nabla f$, which vanishes only at the critical points of $f$.  Each intersection of $\x$ lies on a unique gradient trajectory connecting an index-1 and index-2 critical point of $f$.  As these intersections have opposite parity, we can alter the vector field $\nabla f$ to be non-vanishing in a neighborhood of each of these trajectories.  The few remaining critical points of $f$ are all contained in $F(\SZ) \subset \partial Y$.  It is straightforward to modify $\nabla f$ in a neighborhood of each such critical point to be nonzero.  The resulting nonzero vector field is denoted $v(\x)$.

Consider the vector fields $v_0 = v(\x)|_{\partial Y - F} = \nabla f|_{\partial Y - F}$ and $v_{o(\x)} = v(\x)|_{\partial Y}$.  The vector field $v_{o(\x)}$ depends on the collection of occupied $\alpha$-arcs $o(\x)$, while $v_0$ does not depend on any information coming from the generator $\x$.  For different choices of boundary-compatible Morse function and metric, the vector fields $v_0$ and $v_{o(\x)}$ may vary within a contractible set.  Thus, we consider the following two sets of relative $\spinc$-structures: $\spinc(Y,\partial Y - F)$ and $\spinc(Y,\partial Y, o)$, where $o \subset \{1,\dots,k\}$ lists the collection of occupied $\alpha$-arcs.

Let $\SH = \SH_1 \cup \SH_2$ be a decomposition of a (bordered) sutured Heegaard diagram into a pair of bordered sutured heegaard diagrams, glued along their common boundary.  If $\x \in \SG(\SH)$ is a generator of the (bordered) sutured Floer complex associated to $\SH$, then $\x = \x_1 \otimes \x_2$, where $\x_1 \in \SG(\SH_1)$ and $\x_2 \in \SG(\SH_2)$.  By fixing the natural vector field associated to $\s(\x) = \s(x_1 \otimes \x_2)$ along the gluing surface, we further obtain a decomposition of $\s(\x)$ into the pair of relative $\spinc$-structure: $\s(x_1) \in \spinc(Y_1,\partial Y_1, o(\x_1))$ and $\s(x_2) \in \spinc(Y_2,\partial Y_2, \overline{o}(\x_2))$.  Moreover, we have a splitting of Chern classes
\[
	c_1(\x_1 \otimes \x_2) = c_1(\x_1) \otimes c_1(\x_2).
\]


\subsection{Bordered invariants and knot Floer homology} 
\label{sub:bordered_invariants_and_knot_floer_homology}

Given a null-homologous knot $K \subset Y$, one can form the bordered manifold by removing a tubular neighborhood of $K$ and parametrizing the resulting torus boundary using the meridian and (Seifert-framed) longitude to $K$.  The resulting space is a bordered $3$-manifold which is canonically associated to the knot $K$.  From this space we compute $A_\infty$ and Type-$D$ modules denoted $\CFA(K)_{\A_T}$ and ${}^{\A_T} \CFD(K)$, each over the torus algebra (see Section~\ref{sub:param_torus}).

Lipshitz, Ozsv\'ath and Thurston observed in \cite{LOT1} that the knot Floer homology groups $\HFKM(Y,K)$ and $\HFKH(Y,K)$ can be recovered form either of the bordered invariants $\CFA(K)$ or $\CFD(K)$.  Concretely, this can be obtained by considering the doubly-pointed, bordered Heegaard diagram $\SH_c = (\mscr{H},z,w)$ for the solid torus shown in Figure~\ref{fig:minusdiag}.

\begin{figure}[htbp]
	\centering
	\begin{picture}(93,97)
		\put(0,0){\includegraphics[scale=1]{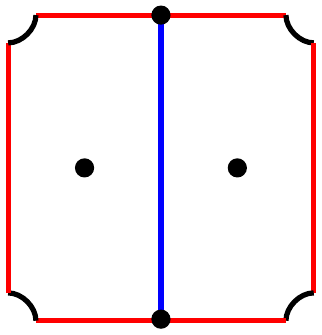}}
		\put(50,6){$x$}
		\put(27,6){\color{red}$2$}
		\put(-5,55){\color{red}$1$}
		\put(50,94){$x$}
		\put(27,94){\color{red}$2$}
		\put(93,55){\color{red}$1$}
		\put(27,52){$w$}
		\put(70,52){$z$}
		\put(-1,-1){$3$}
		\put(-1,89){$2$}
		\put(88,90){$1$}
		\put(88,-1){$0$}
	\end{picture}
		\caption{Diagram $(\mscr{H},z,w)$ yielding $\HFKM$.}
		\label{fig:minusdiag}
\end{figure}

The doubly-pointed diagram $\SH_c$ in Figure~\ref{fig:minusdiag} specifies a solid torus $S^1 \times D^2$ together with its core curve $\gamma = S^1 \times \mr{pt}$.  The arcs $\alpha_1$ and $\alpha_2$ specify the longitude and meridian of $\gamma$ respectively.  Identifying these curves of this solid torus with their pairs on the complement $Y \backslash \nu(K)$ amounts to performing infinity-surgery and further identifies the core curve $\gamma$ of $S^1 \times D^2$ with the knot $K$.

To $\SH_c$ one associates three distinct Type-$D$ modules, which compute the knot Floer homology groups $\HFKM(Y,K)$, $\HFKP(Y,K)$ and $\HFKH(Y,K)$, respectively, when paired with $\CFA(K)$.  We denote the first by $K^- := {}^{\A_T} \CFDM(\SH_c)$.  This module is generated over $\F[U]$ by the single intersection point $x$.  It is understood that holomorphic disks passing over the second basepoint $w$ in $\SH_c$, with multiplicity $n$, contribute a factor of $U^n$ to the differential.  Thus, the Type-$D$ module $K^-$ is specified by the relation
\[
	\delta(U^i \cdot x) = \rho_{23} \otimes (U^{i+1} \cdot x),
\]
or, graphically
\begin{center}
\begin{tikzpicture}	[->,>=stealth',auto,thick]
	\node (a) at (0,0){$x$} ;
	\node (b) at (2.5,0) {$U \cdot x$} ;
	\node (c) at (5,0) {$U^2 \cdot x$} ;
	\node (d) at (7.5,0) {$U^3 \cdot x$} ;
	\node (e) at (10,0) {$\dots$} ;
	
	\draw (a) edge node[above] {$\rho_{23}$} (b);
	\draw (b) edge node[above] {$\rho_{23}$} (c);
	\draw (c) edge node[above] {$\rho_{23}$} (d);
	\draw (d) edge (e);
\end{tikzpicture}
\end{center}

We denote the second Type-$D$ module associated to $\SH_c$ by $K^+:= {}^{\A_T} \CFDP(\SH_c)$.  As an $\F[U]$-module, it is equal to $\F[U^{-1}]$.  Again, we have that holomorphic disks passing over the second basepoint $w$ in $(\SH,z,w)$, with multiplicity $n$, contribute a factor of $U^n$ to the differential.  The Type-$D$ structure on $K^+$ is specified graphically as
\begin{center}
\begin{tikzpicture}	[->,>=stealth',auto,thick]
	\node (a) at (0,0){$x$} ;
	\node (b) at (2.25,0) {$U^{-1} \cdot x$} ;
	\node (c) at (5,0) {$U^{-2} \cdot x$} ;
	\node (d) at (7.75,0) {$U^{-3} \cdot x$} ;
	\node (e) at (10,0) {$\dots$.} ;
	
	\draw (b) edge node[above] {$\rho_{23}$} (a);
	\draw (c) edge node[above] {$\rho_{23}$} (b);
	\draw (d) edge node[above] {$\rho_{23}$} (c);
	\draw (e) edge node[above] {$\rho_{23}$} (d);
\end{tikzpicture}
\end{center}

Finally, there is a third Type-$D$ structure associated to $(\SH,z,w)$, which we denote $K_\infty := {}^{\A_T} \CFD(\SH_c)$.  The module $K_\infty$ is again generated by the single intersection $x$, but now over the field $\F$.  The differential on $K_\infty$ is trivial:
\[
	\delta(x) = 0.
\]


\subsection{Gluing maps} 
\label{sub:gluing_maps}

In \cite{Za2}, the third author defined a gluing map for sutured 3-manifolds.  Specifically, he showed the following.

\begin{theorem}[Zarev \cite{Za2}]\label{thm:rumen_gluing}
	Let $(Y_1,\Gamma_1)$ and $(Y_2,\Gamma_2)$ be balanced sutured 3-manifolds which can be glued along some surface $F$.  Then there exists a well-defined map
\[
	\Psi_F : \SFH(Y_1,\Gamma_1) \otimes \SFH(Y_2,\Gamma_2) \to \SFH(Y_1 \cup Y_2, \Gamma_1 \cup \Gamma_2),
\]
which is symmetric, associative and equals the identity for topologically trivial gluing.
\end{theorem}

These maps are well-defined, even when the two manifolds are bordered sutured, provided the surface along which the gluing is performed is part of the sutured boundary. It has further been shown by the third author \cite{Za3} that the gluing map given by Theorem~\ref{thm:rumen_gluing} are equivalent to the contact gluing map defined by Honda, Kazez and Mati\'c in \cite{HKM3}.  

The main advantage of the bordered sutured interpretation of the Honda-Kazez-Mati\'c (HKM) gluing map is that it is defined purely algebraically.  The utility of this algebraic perspective will become apparent quickly.  Indeed, the algebraic framework surrounding bordered sutured Floer homology, combined with some standard non-vanishing results for the HKM gluing maps will propel many of the computations that follow. 



\section{Parametrized Surfaces and Associated Algebras} 
\label{sec:algebras}

We now give explicit descriptions of the strands algebras to be encountered in subsequent sections.  There are three such surfaces and they are depicted in Figures~\ref{fig:sutured_surfaces} and~\ref{fig:sutured_surfaces2}.

\begin{figure}[htbp]
	\centering
	\begin{picture}(145,100)
		\put(0,0){\includegraphics[scale=1]{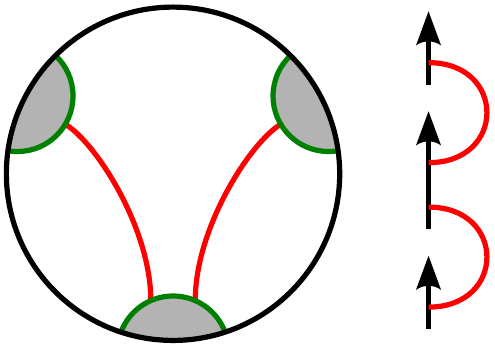}}
		\put(111,8){$a_1$}
		\put(111,37){$a_2$}
		\put(111,50){$a_3$}
		\put(111,80){$a_4$}
		\put(142,65){\color{red}$1$}
		\put(142,23){\color{red}$2$}
	\end{picture}
	\hspace{20pt}
	\begin{picture}(142,103)
		\put(2,0){\includegraphics[scale=1]{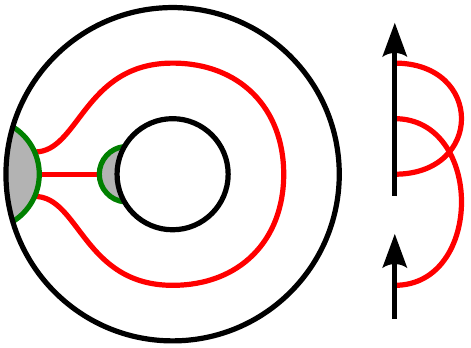}}
		\put(104,16){$a_1$}
		\put(104,47){$a_2$}
		\put(104,63){$a_3$}
		\put(104,80){$a_4$}
		\put(140,65){\color{red}$1$}
		\put(140,37){\color{red}$2$}
	\end{picture}
		\caption{On the left the sutured disk $\SF_D = (D^2,\Lambda_D)$ and corresponding arc diagram $\W_D$. On the right  sutured annulus $\SF_A = (A,\Lambda_A)$ and corresponding arc diagram $\W_A$.}
	\label{fig:sutured_surfaces}
\end{figure}
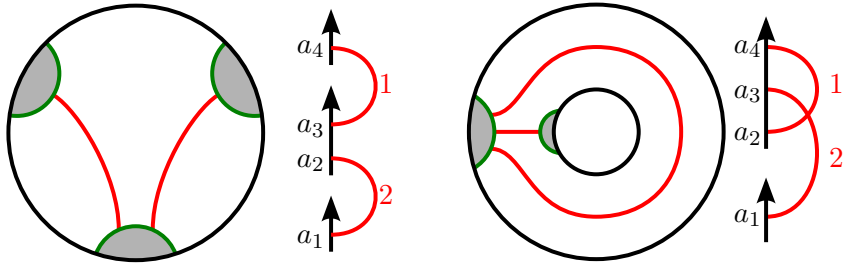

\subsubsection*{The Parametrized Sutured Disk} 
 \label{sub:param_disk}

We begin with the sutured surface $\SF_D = (D^2,\Lambda_D)$, depicted on the left in Figure~\ref{fig:sutured_surfaces}.  This surface is topologically a disk, with suture set $\Lambda_D$ consisting of six marked points along its boundary.  Figure~\ref{fig:sutured_surfaces} also depicts a parametrization of $\SF_D$ and the corresponding arc diagram $\W_D$.

From Figure~\ref{fig:sutured_surfaces}, we see that there is a single Reeb chord $\rho$ in $\W_D$ connecting $a_2$ to $a_3$.  In turn, the algebra $\A(\W_D)$ has trivial differential and decomposes as a direct sum of three subalgebras: $\A(\W_D,0)$, $\A(\W_D,1)$ and $\A(\W_D,2)$.  

The summands $\A(\W_D,0) = \langle I_\emptyset \rangle$ and $\A(\W_D,2) = \langle I_{12} \rangle$ are both trivial, while the $\A(\W_D,1)$-summand is given by
\begin{equation*}
	\A(\W_D,1) = \langle I_1,I_2,\rho' \rangle,
\end{equation*}
where $\rho' = a(\{\rho\},\emptyset)$ and the idempotent compatibilities are given by
\begin{equation*}
	I_2 \, \rho' \, I_1 = \rho'.
\end{equation*}


\subsubsection*{The Parametrized Sutured Annulus} 
 \label{sub:param_annulus}

Next, we turn  to the sutured surface $\SF_A = (A,\Lambda_A)$, depicted on the right in Figure~\ref{fig:sutured_surfaces}.  Topologically, this is an annulus with suture set $\Lambda_A$ consisting of a pair of marked points on each boundary component.  As before Figure~\ref{fig:sutured_surfaces} also depicts a parametrization of $\SF_A$ together with the corresponding arc diagram $\W_A$.

In $\W_A$, there are three Reeb chords --- $\rho_1$ connecting $a_2$ to $a_3$, $\rho_2$ connecting $a_3$ to $a_4$, and $\rho_{12}$ connecting $a_2$ to $a_4$.  It follows that the algebra $\A(\W_A)$ decomposes as a sum of three subalgebras: $\A(\W_A,0)$, $\A(\W_A,1)$ and $\A(\W_A,2)$.  

As before, the summand $\A(\W_A,0) = \langle I_\emptyset \rangle$ is trivial.  

The $\A(\W_A,1)$-summand, is described by
\begin{equation*}
	\A(\W_A,1) = \langle I_1,I_2, \rho_1', \rho_2', \rho_{12}' \rangle,
\end{equation*}
where $\rho_1' = a(\{\rho_1\},\emptyset)$, $\rho_2' = a(\{\rho_2\},\emptyset)$ and $\rho_{12}' = a(\{\rho_{12}\},\emptyset)$.  The idempotent compatibilities and nontrivial products in $\A(\W_A,1)$ are given by 
\begin{equation*}
I_1 \, \rho_1' \, I_2 = \rho_1', \;\;\; I_2 \, \rho_2' \, I_1 = \rho_2', \;\;\; I_2 \, \rho_{12}' \, I_2 = \rho_{12}', \;\;\; \rho_1' \rho_2' = \rho_{12}'.
\end{equation*}  

Finally, we have that
\[
	\A(\W_A,2) = \langle I_{12}, \rho_{12}'', \rho_2'' \cdot \rho_1'' \rangle,
\]
where $\rho_{12}'' = a(\{\rho_{12}\},\{2\})$ and $\rho_2'' \cdot \rho_1'' = a(\{\rho_1,\rho_2\},\emptyset)$.  The idempotent compatibilities are given by 
\[
	I_{12} \, \rho_{12}'' \, I_{12} = \rho_{12}'', \;\;\; I_{12} \, \rho_2'' \cdot \rho_1'' \, I_{12} = \rho_2'' \cdot \rho_1'',
\]  
and there are no nontrivial products.  There is a single nontrival differential in $\A(\W_A,2)$ given by $\partial \rho_{12}'' = \rho_2'' \cdot \rho_1''$.


\subsubsection*{The Parametrized Sutured Torus} 
 \label{sub:param_torus}
\begin{figure}[htbp]
	\centering
	\begin{picture}(145,90)
		\put(0,0){\includegraphics[scale=1]{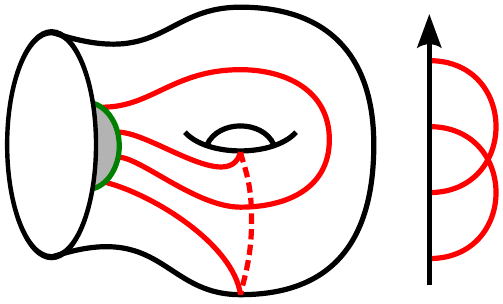}}
		\put(112,12){$a_1$}
		\put(112,30){$a_2$}
		\put(112,50){$a_3$}
		\put(112,67){$a_4$}
		\put(146,52){\color{red}$1$}
		\put(146,26){\color{red}$2$}
	\end{picture}
	\caption{The sutured torus $\SF_T = (T,\Lambda_T)$ and corresponding arc diagram $\W_T$. }
	\label{fig:sutured_surfaces2}
\end{figure}
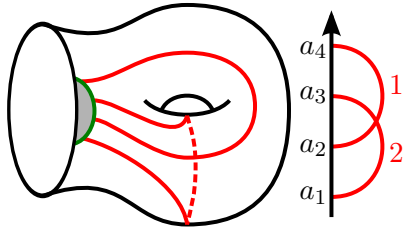

Finally, we consider surface $\SF_T = (T,\Lambda_T)$, depicted in Figure~\ref{fig:sutured_surfaces2}.  As a sutured surface, this is a punctured torus with a suture set $\Lambda_T$ consisting of a pair of marked point along its boundary.  An explicit parametrization of $\SF_T$ by $\W_T$ is also shown in Figure~\ref{fig:sutured_surfaces2}.

In $\W_T$, there are several Reeb chords --- $\rho_1$ from $a_1$ to $a_2$, $\rho_2$ from $a_2$ to $a_3$, $\rho_3$ from $a_3$ to $a_4$, $\rho_{12}$ from $a_1$ to $a_3$, $\rho_{23}$ from $a_2$ to $a_4$ and, finally $\rho_{123}$ from $a_1$ to $a_4$.  

The algebra $\A(\W_T)$ associated to this punctured torus is equivalent to the ``torus algebra'' from \cite{LOT1}.  As above, $\A(\W_T)$ decomposes as a sum of three subalgebras: $\A(\W_T,0)$, $\A(\W_T,1)$ and $\A(\W_T,2)$.

In this case, only the $\A(\W_T,0)$-summand is trivial: $\A(\W_T,0) = \langle I_\emptyset \rangle$.  

We have that $\A(\W_T,1)$ is given by 
\begin{equation*}
	\A(\W_T,1) = \langle I_1,I_2, \rho_1', \rho_2', \rho_3', \rho_{12}', \rho_{23}', \rho_{123}' \rangle,
\end{equation*}
where $\rho_1' = a(\{\rho_1\},\emptyset)$, $\rho_2' = a(\{\rho_2\},\emptyset)$, $\rho_3' = a(\{\rho_3\},\emptyset)$, $\rho_{12}' = a(\{\rho_{12}\},\emptyset)$, $\rho_{23}' = a(\{\rho_{23}\},\emptyset)$ and $\rho_{123}' = a(\{\rho_{123}\},\emptyset)$.  The idempotent compatibilities and nontrivial products in $\A(\W_T,1)$ are given by 
\begin{align*}
\nonumber I_2 \, \rho_1' \, I_1 = \rho_1', \;\;\; I_1 \, \rho_2' \, I_2 &= \rho_2', \;\;\; I_2 \, \rho_3' \, I_1 = \rho_3',\\
I_2 \, \rho_{12}' \, I_2 = \rho_{12}', \;\;\; I_1 \, \rho_{23}' \, I_1 &= \rho_{23}', \;\;\; I_2 \, \rho_{123}' \, I_1 = \rho_{123}',\\
\nonumber \rho_1' \rho_2' = \rho_{12}', \;\;\; \rho_2' \rho_3' = \rho_{23}', \;\;\;& \rho_1' \rho_{23}' = \rho_{123}', \;\;\; \rho_{12}' \rho_3' = \rho_{123}'
\end{align*}
Since it is not strictly needed in the discussion to follow, we leave it as an interesting exercise for the reader to compute the summand $\A(\W_T,2)$.  As a hint, this subalgebra contains both nontrivial products and nontrivial differentials.



\section{Bypass Attachment Maps} 
\label{sec:bypass_attachment_maps}

In this section, we begin discussing a general method for computing the HKM map induced on sutured Floer homology by bypass attachment.  The key result from \cite{Za3}, which propels this computation in this section, states that the HKM gluing maps extends to the bordered sutured category.  Specifically, the third author proves the following extension of Theorem~\ref{thm:hkm_gluing}.

\begin{theorem}\label{thm:hkm_za_gluing}
 	Let $\Y_1 = (Y_1,\Gamma_1,\SZ)$ and $\Y_2 = (Y_2, \Gamma_2,\SZ)$ be two bordered sutured 3--manifolds such that $Y_1 \subset Y_2$, $Y_2 \backslash \mr{int}(Y_1)$ is balanced sutured, and let $\xi$ be a contact structure on $Y_2 \backslash \mr{int}(Y_1)$ with convex boundary divided by $\Gamma_1\cup\Gamma_2$.  Then there exists a map of Type-$D$ structures induced by the contact structure $\xi$
\[
	\phi_{\xi} : \BSD(-\Y_1) \to \BSD(-\Y_2),
\]
which is natural with respect to gluings of bordered sutured 3--manifolds along subsets of the boundary which are parametrized sutured surfaces.
\end{theorem}

The discussion to follow focuses on a small tubular neighborhood of a bypass attachment arc.  Within this neighborhood, there exists a natural sequence of three suture sets, each obtained from the last by a bypass attachment.  Together, this sequence is known as a ``bypass exact triangle''.  It is so-called because it represents an exact triangle in Honda's contact category (see \cite{Ho}).  The sequence of attachments and resulting dividing sets are depicted in Figure~\ref{fig:byptriang}.
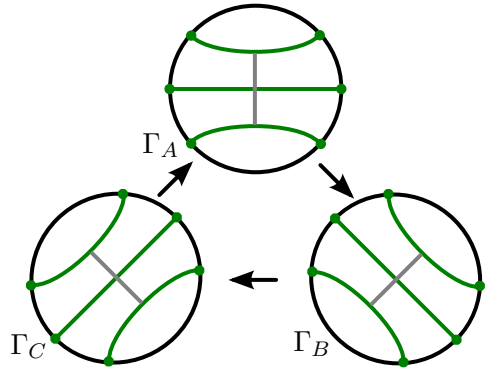
\begin{figure}[htbp]
	\centering
	\begin{picture}(175,140)
		\put(0,0){\includegraphics[scale=1]{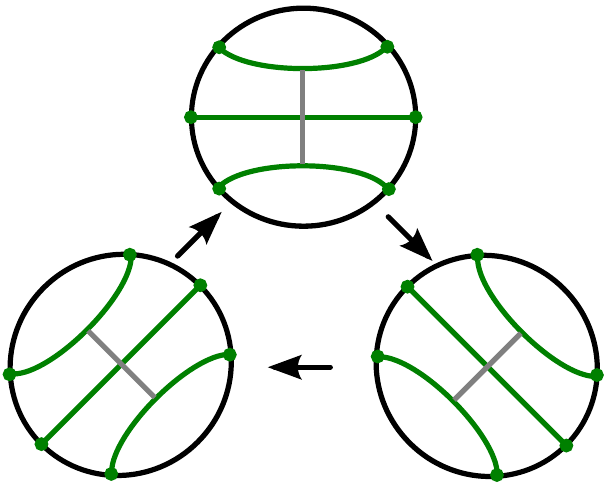}}
		\put(45, 82){$\Gamma_A$}
		\put(102, 7){$\Gamma_B$}
		\put(-5, 6){$\Gamma_C$}
	\end{picture}
	\caption{Honda's bypass exact triangle.}
	\label{fig:byptriang}
\end{figure}

\begin{theorem}[Honda \cite{Ho}]\label{thm:bypass_exact_triangle}
Denote the sutured manifolds in the above mentioned sequence by $(Y,\Gamma_A)$, $(Y,\Gamma_B)$ and $(Y,\Gamma_C)$ respectively.  The trio of HKM gluing maps induced by the collection of bypass attachments moving between these three manifolds together form an exact triangle in sutured Floer homology:
	\[\text{
	\begin{tikzpicture}	[->,>=stealth',auto,thick]
		\node (a) at (0,0){$\SFH(-Y,-\Gamma_C)$} ;
		\node (b) at (4.25,0) {$\SFH(-Y,-\Gamma_B)$.} ;
		\node (c) at (2.1,1.5){$\SFH(-Y,-\Gamma_A)$};
		\draw (b) edge node[above] {\small $\phi_C$} (a);
		\draw (a) edge node[left] {\small $\phi_B$} (c);
		\draw (c) edge node[right] {\small $\phi_A$} (b);
	\end{tikzpicture}}\]
\end{theorem}
In Subsection~\ref{sub:bord_analogue} we will reduce this theorem to a computation about simple bordered sutured manifolds and in the following section we will make the relevant computations giving a bordered sutured Floer proof of this theorem. 

\subsection{The Bordered Analogue} 
 \label{sub:bord_analogue}

We translate the above discussion into the language of bordered sutured Floer homology by decomposing the sutured manifold $(-Y,-\Gamma_A)$, as depicted in Figure~\ref{fig:byp_decomp}, into a pair of bordered sutured manifolds: $(-Y,\Gamma',\W_D)$ and $(D^2 \times I,\Gamma_A,-\W_D)$.
\begin{figure}[htbp]
	\centering
	\begin{picture}(163,112)
		\put(0,0){\includegraphics[scale=1]{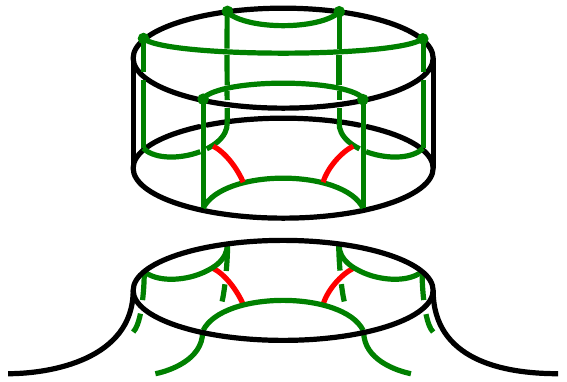}}
		\put(130,77){$(-D^2 \times I,\Gamma_A',-F_1)$}
		\put(140,15){$(-Y,\Gamma',F_1)$}
	\end{picture}
	\caption{Decomposing $(-Y,\Gamma_A)$.}
	\label{fig:byp_decomp}
\end{figure}

Mirroring the discussion above, we consider the trio of bordered sutured manifolds $\D_A = (D^2 \times I,\Gamma_A',-\W_D)$, $\D_B = (D^2 \times I,\Gamma_B',-\W_D)$ and $\D_C = (D^2 \times I,\Gamma_C',-\W_D)$ depicted in Figure~\ref{fig:bordbyptriang}.  The sutured manifolds $(-Y,-\Gamma_A)$, $(-Y,-\Gamma_B)$ and $(-Y,-\Gamma_C)$ are obtained from $\Y = (-Y,\Gamma',\W_D)$ by gluing on $\D_A$, $\D_B$ or $\D_C$ respectively.
\begin{figure}[htbp]
	\centering
	\begin{picture}(260,155)
		\put(0,0){\includegraphics[scale=1]{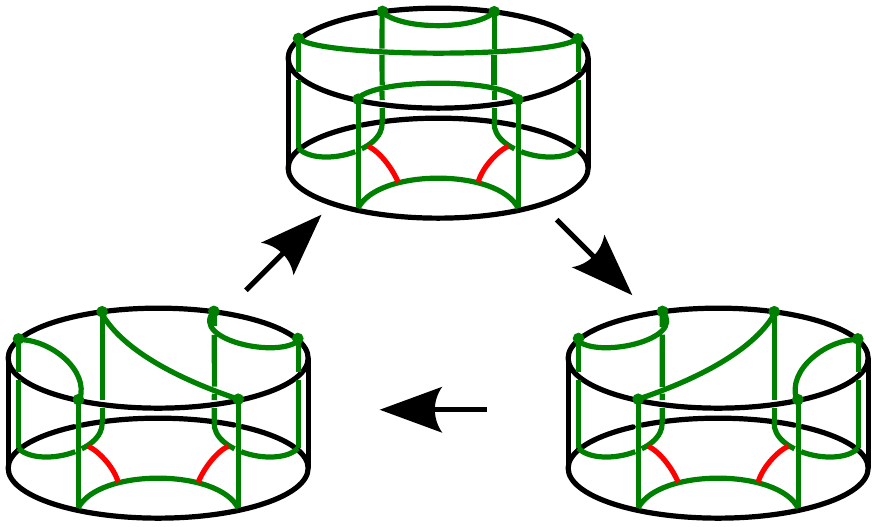}}
		\put(66,142){$\Gamma_A'$}
		\put(147,53){$\Gamma_B'$}
		\put(-14,53){$\Gamma_C'$}
	\end{picture}
	\caption{The bordered bypass exact triangle.}
	\label{fig:bordbyptriang}
\end{figure}
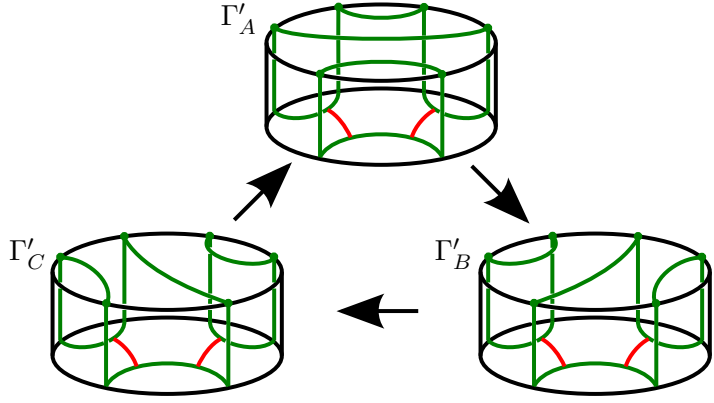

Below we compute the trio of HKM gluing maps $\phi_A'$, $\phi_B'$ and $\phi_C'$ induced by the bypass attachments in Figure~\ref{fig:byptriang} on the Type-$D$ structures ${}^{\A(\W_D)} \BSD(\D_A)$, ${}^{\A(\W_D)} \BSD(\D_B)$ and ${}^{\A(\W_D)} \BSD(\D_C)$.  We will then see that  $\BSD(\D_A)$ is the mapping cone of $\phi'_B$ with $\phi'_A$ and $\phi'_C$ being the projection and inclusion maps respectively.  It then follows from Theorem~\ref{thm:hkm_za_gluing} that the gluing maps $\phi_A$, $\phi_B$ and $\phi_C$ are equivalent to $H_*(\I \boxtimes \phi_A')$, $H_*(\I \boxtimes \phi_B')$ and $H_*(\I \boxtimes \phi_C')$, under the identifications  
\begin{align*}
	\SFH(-Y,-\Gamma_A) &\cong  H_*(\BSA(\Y) \boxtimes \BSD(\D_A))\\
	\SFH(-Y,-\Gamma_B) &\cong  H_*(\BSA(\Y) \boxtimes \BSD(\D_B))\\
	\SFH(-Y,-\Gamma_C) &\cong  H_*(\BSA(\Y) \boxtimes \BSD(\D_C)).
\end{align*} 
Properties of the derived tensor product then imply that $\SFC(-Y,\Gamma_A)\cong\BSA(\Y) \boxtimes \BSD(\D_A)$ is the mapping cone of $\I \boxtimes \phi_B$  with $\I \boxtimes \phi_A$ and $\I \boxtimes \phi_C$ being the projection and inclusion maps respectively. Theorem~\ref{thm:bypass_exact_triangle} then follows upon taking homology. 


\subsection{Bordered Bypass Attachment Maps} 
\label{sub:bordered_bypass_attachment_maps}

Our computation of the maps $\phi_A'$, $\phi_B'$ and $\phi_C'$ proceeds as follows.  

We begin by computing the Type-$D$ structures $\BSD(\D_A)$, $\BSD(\D_B)$ and $\BSD(\D_C)$.  From the form of these modules, it follows that there exist unique nontrivial maps which connect these Type-$D$ structures in sequence and these maps form a mapping cone as discussed above. The fact that the maps $\phi_A'$, $\phi_B'$ and $\phi_C'$ are nontrivial follows from Theorem~\ref{thm:hkm_za_gluing} and the fact that there exist contact manifolds with convex boundaries which are related by a single bypass attachment and whose sutured contact invariants are nonzero. 

\subsubsection*{Computation of Modules} 
\label{sub:bypass_modules}

We now compute the Type-$D$ structures $\BSD(\D_A)$, $\BSD(\D_B)$ and $\BSD(\D_C)$.  Figure~\ref{fig:disk_diagrams} depicts the parametrized bordered sutured manifolds $\D_A$, $\D_B$ and $\D_C$.
\begin{figure}[htbp]
	\centering
		\includegraphics[scale=0.9]{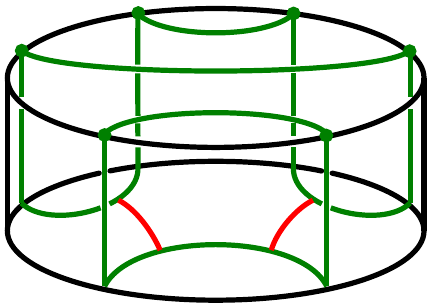}
	\hspace{5pt}
		\includegraphics[scale=0.9]{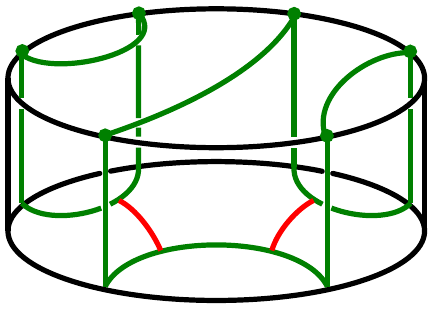}
	\hspace{5pt}
		\includegraphics[scale=0.9]{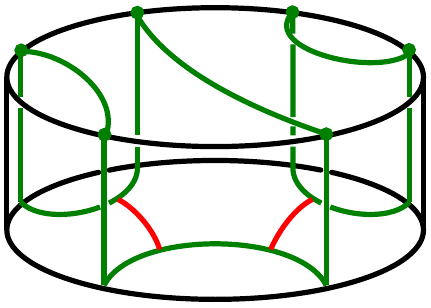}
	\caption{From left to right the parameterized bordered sutured manifolds $\D_A$, $\D_B$ and $\D_C$. }
	\label{fig:disk_diagrams}\label{fig:bord_disk_a}\label{fig:bord_disk_b}\label{fig:bord_disk_c}
\end{figure}

As we will now verify, the corresponding bordered sutured Heegaard diagrams $\SH_A$, $\SH_B$ and $\SH_C$ are shown in  Figure~\ref{fig:disk_diagrams2}. 
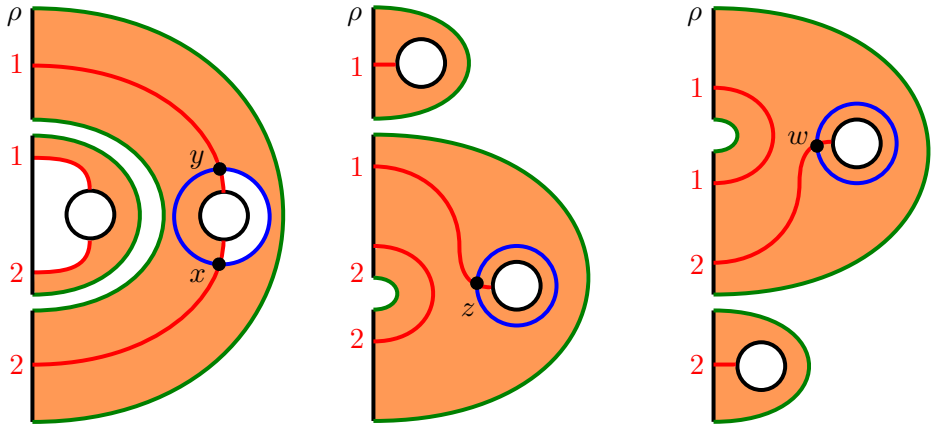
\begin{figure}[htbp]
	\centering
	\begin{picture}(125,163)
		\put(10,0){\includegraphics[scale=1]{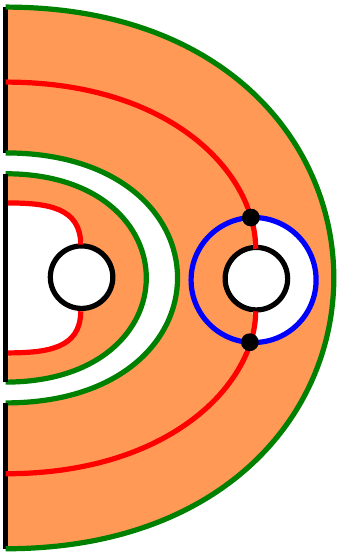}}
		\put(3,20){\color{red}$2$}
		\put(3,55){\color{red}$2$}
		\put(3,99){\color{red}$1$}
		\put(3,134){\color{red}$1$}
		\put(71,100){$y$}
		\put(71,55){$x$}
		\put(2,154){$\rho$}
	\end{picture}
	\begin{picture}(125,163)
		\put(10,0){\includegraphics[scale=1]{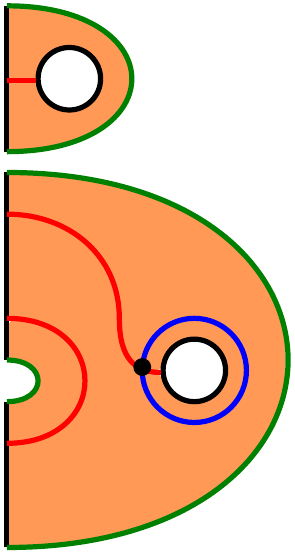}}
		\put(3,30){\color{red}$2$}
		\put(3,55){\color{red}$2$}
		\put(3,95){\color{red}$1$}
		\put(3,133){\color{red}$1$}
		\put(45,42){$z$}
		\put(2,154){$\rho$}
	\end{picture}
	\begin{picture}(125,163)
	\put(10,0){\includegraphics[scale=1]{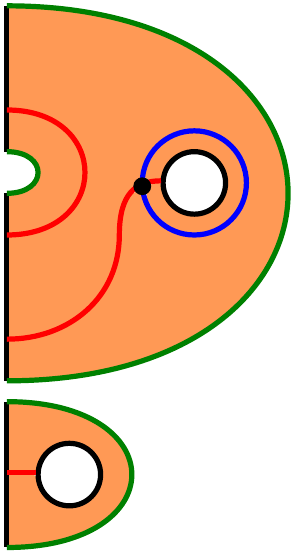}}
		\put(3,20){\color{red}$2$}
		\put(3,58){\color{red}$2$}
		\put(3,90){\color{red}$1$}
		\put(3,125){\color{red}$1$}
		\put(40,107){$w$}
		\put(2,154){$\rho$}
	\end{picture}
	\caption{From left to right he bordered sutured Heegaard diagrams  $\SH_A$, $\SH_B$ and $\SH_C$ associated to $\D_A$, $\D_B$ and $\D_C$, respectively. Each Heegaard surface is an annulus obtained by identifying the two round black circles in the diagrams. The blue circles give $\bbeta$ and the red arcs form $\balpha$. }
	\label{fig:disk_diagrams2}\label{fig:disk_a_diag}\label{fig:disk_b_diag}\label{fig:disk_c_diag}
\end{figure}
We will check that $\SH_A$ does indeed give a Heegaard diagram for $\D_A$ and leave it to the reader to make the analogous arguments for the other two diagrams. We first notice that $\SH_A=(\Sigma,\balpha,\bbeta, \SZ)$ where $\SZ$ is the arc diagram for $\SF_D$ from Section~\ref{sub:param_disk} and consists of the vertical black lines and the red arcs. Now recall from Section~\ref{sub:the_bordered_invariants} how to build a bordered sutured manifold form a bordered sutured Heegaard diagram. The manifold $Y$ is obtained from the Heegaard surface $\Sigma$ by attaching 2--handles to $\Sigma\times [0,1]$ along the $\balpha$ and $\bbeta$ circles. Thus $Y$ is a solid torus (that is an annulus times interval) with one 2--handle attached along a longitude. So $Y$ is clearly a 3-ball. 
\begin{figure}[htbp]
	\centering
		{\includegraphics{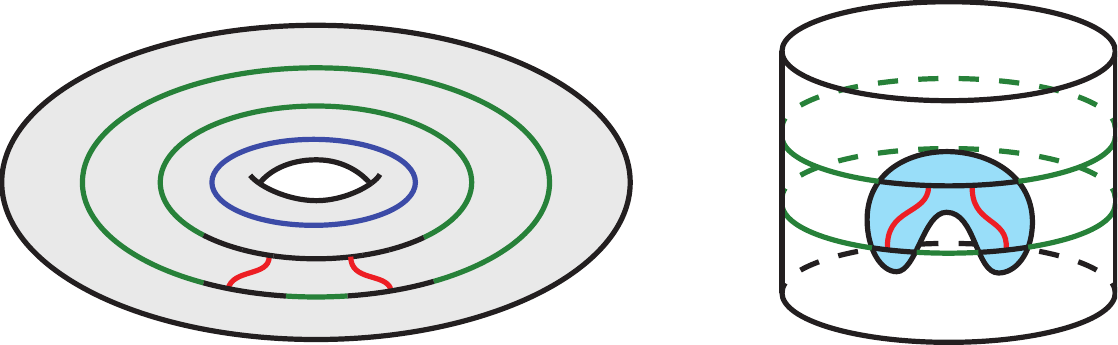}}
		\caption{On the left is $\partial (\Sigma\times[0,1])$ with $\bbeta$ shown in blue, the sutures $\Gamma$ the union of the green and black arcs and $\mathcal{G}(\SZ)$ the union of the black and red arcs. On the right is $\partial Y$ with $\SF(\SZ)$ shaded.}
	\label{fig:BuildBSM}
\end{figure}
Now on the left in Figure~\ref{fig:BuildBSM} we see the boundary of $\Sigma\times [0,1]$ with the embedding of the graph of the arc diagram $\W_D$ for $\SF_D$, the curve $\bbeta$ and the dividing curves $\Gamma$ shown. After surgery on $\beta$ we see $\partial Y$ on the left of Figure~\ref{fig:BuildBSM} together with the embedding of $\SF_D=\SF(\W_D)$. This is clearly isotopic to the bordered sutured manifold $\D_A$ in Figure~\ref{fig:disk_diagrams}.

To the diagram $\SH_A$, depicted in Figure~\ref{fig:disk_a_diag}, we associate the Type-$D$ structure $M_A := {}^{\A(\W_D)} \BSD(\D_A)$. As a module, $M_A$ is generated by two elements, $x$ and $y$, whose idempotent compatibilities are given by
\[
	I_1 \cdot x = x, \;\;\; I_2 \cdot y = y.
\]
The diagram $\SH_A$ is a nice in the sense of Section~\ref{sub:nice_heegaard_diagrams} so we can use the algorithm there to compute the boundary map $\delta$. 
From Figure~\ref{fig:disk_a_diag}\footnote{As an aid to the reader, regions adjacent to a sutured boundary components in bordered sutured Heegaard diagrams have been lightly shaded orange.  This signals that any domain which contributes nontrivially to the corresponding differential in Floer homology must have multiplicity zero in that region.}, we see that there is a single domain contributing to the boundary map $\delta$.  It corresponds to a source $S$ which is a rectangle from $y$ to $x$ with one edge mapping to $-\rho$.  Thus, the only nontrivial term in the structure map $\delta$ is given by 
\[
	\delta(y) = \rho' \otimes x.
\]

Next, to the diagram $\SH_B$, depicted in Figure~\ref{fig:disk_b_diag}, we associate the Type-$D$ structure $M_B := {}^{\A(\W_D)} \BSD(\D_B)$.  As a module, $M_B$ is generated by a single element $z$ whose idempotent compatibility is given by
\[
	I_2 \cdot z = z.
\]
The boundary map in this case is trivial since all regions in $\SH_B$ are adjacent to sutured portions of the boundary.

Finally, to the diagram $\SH_C$ in Figure~\ref{fig:disk_c_diag}, we associate the Type-$D$ structure $M_C := {}^{\A(\W_D)} \BSD(\D_C)$.  As above, $M_C$ is generated by a single element $w$, whose idempotent compatibility is given by
\[
	I_1 \cdot w = w,
\]
with trivial boundary map owing to the fact that all regions in $\SH_C$ are adjacent to portions of the boundary which are sutured.


\subsubsection*{Nontrivial Maps} 
 \label{sub:bypass_gluing_maps}

Having computed the Type-$D$ structures $M_A$, $M_B$ and $M_C$, the task of computing the HKM gluing maps $\phi_A'$, $\phi_B'$ and $\phi_C'$ is essentially a triviality.  The reason is that any map of Type-$D$ structures must respect idempotent compatibilities.  Combining this with the previously observed nontriviality requirement, one quickly checks that the desired maps are determined as follows.

The only nontrivial, idempotent compatible map from $M_A$ to $\A(\W_D) \otimes M_B$ is
\begin{align*}
\nonumber	\phi_A' : M_A &\to \A(\W_D) \otimes M_B\\
	\phi_A'(y) &= I_2 \otimes z\\
\nonumber	\phi_A'(x) &= 0
\end{align*}

Similarly, the only nontrivial, idempotent compatible map from $M_B$ to $\A(\W_D) \otimes M_C$ is
\begin{align}\label{eqn:bpmap}
	\phi_B' : M_B &\to \A(\W_D) \otimes M_C\\
\nonumber \phi_B'(z) &= \rho' \otimes w
\end{align}

Finally, the only nontrivial, idempotent compatible map from $M_C$ to $\A(\W_D) \otimes M_A$ is
\begin{align*}
	\phi_C' : M_C &\to \A(\W_D) \otimes M_A\\
\nonumber	\phi_C'(w) &= I_1 \otimes x
\end{align*}

Observe that the Type-$D$ structure $M_A$ is the mapping cone of the Type-$D$ morphism $\phi_B'$, and that the maps $\phi_A'$ and $\phi_C'$ are the projection and inclusion maps respectively.




\section{Limit Invariants and Knot Floer Homology} 
\label{sec:limit_is_knot}

In this section, we prove a version of Theorem~\ref{thm:lim_to_minus}, which states that there exists an isomorphism relating the sutured limit homology of a null-homologous knot with the minus version of knot Floer homology.  Specifically, we establish an isomorphism of $\F[U]$-modules, deferring the identification of absolute Alexander and ($\Z/2$) Maslov gradings until Section~\ref{sec:gradings}.

Proving Theorem~\ref{thm:lim_to_minus} requires a precise understanding of how the gluing maps induced by positive and negative Legendrian stabilization act on the sutured Floer homology of a Legendrian knot complement.  Computation of these gluing maps, at the level necessary to prove Theorem~\ref{thm:lim_to_minus}, has not historically been possible.  These gluing maps are defined in terms of inclusions of complexes, which, themselves are generally explicitly computable in only elementary situations.  Our strategy employs bordered Floer homology, which provides a sufficiently robust background structure that circumventing the general computability problem is possible in the present situation.

We have broken this section into three main parts.  The first discusses the overall geometric setup, remarking on an appropriate method of decomposing a knot complement which leads to simplified gluing map computations.  Next, we compute the various modules and bi-modules which will be encountered in later computations.  Finally, we conclude the section with a proof of the ungraded version of Theorem~\ref{thm:lim_to_minus}.

\subsection{The Geometric Setup} 
 \label{sub:the_geometric_setup}

Let $Y$ be a 3--manifold and $K \subset Y$ a null-homologous knot with chosen Seifert surface $F$.  Consider the sutured manifold $(-Y(K),-\Gamma_\lambda) = (-Y(K),-\Gamma_0)$, where $Y(K)$ denotes the complement of an open tubular neighborhood of $K$, and $\Gamma_\lambda$ consists of two oppositely-oriented Seifert-framed sutures on the boundary of $Y(K)$.  

We decompose $(-Y(K),-\Gamma_0)$ into a pair of bordered sutured manifolds
\[
	(-Y(K),-\Gamma_0) = (-Y(K),\Gamma',\SF_T) \cup \T_0,
\]
consisting of the knot complement $(-Y(K),\Gamma',\SF_T)$ together with a thickened, punctured torus $\T_0$, as depicted in Figure~\ref{fig:bordknotdecomp}.

\begin{figure}[htbp]
	\centering
	\begin{picture}(187,187)
	\put(0,0){\includegraphics[scale=1]{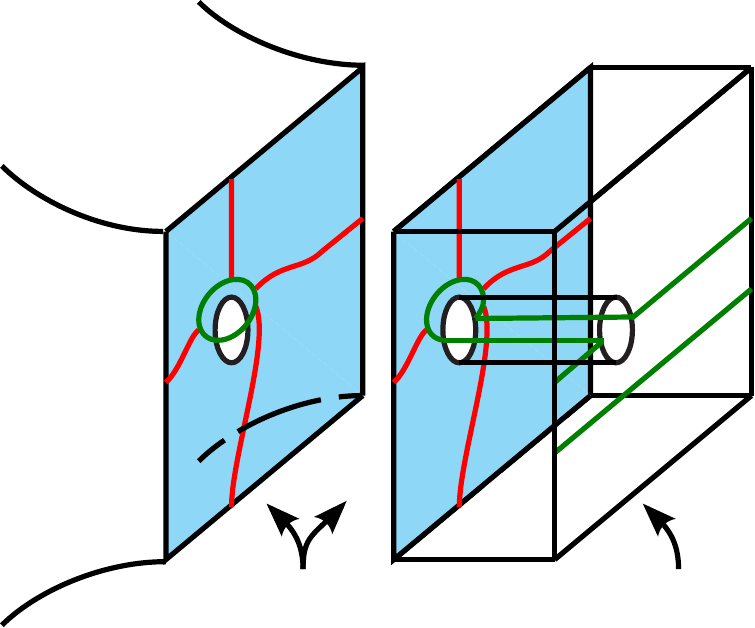}}
	\put(66,5){bordered}
	\put(185,5){sutured}
	\end{picture}
	\caption{Decomposing $(-Y(K),-\Gamma_K)$ into the union $(-Y(K),\Gamma',\SF_T) \cup (T \times I,\Gamma_K',-\SF_T)$. (The light blue regions are the parameterized surfaces $\SF_T$ as seen on the ``front" of $Y(K)$ and the ``back" of $T\times I$.)}
	\label{fig:bordknotdecomp}
\end{figure}

Decomposing $(-Y(K),-\Gamma_K)$ in this way, we can restrict our attention to computing the gluing maps induced on bordered sutured Floer homology by attaching positive or negative bypasses to the simpler space $\T_0 = (T \times I,\Gamma_K',-\SF_T)$, depicted in Figure~\ref{fig:torfill1}.  We denote the space resulting from $n$ such bypass attachments by $\T_n$ (see Figure~\ref{fig:torfill3}).

\begin{figure}[htbp]
	\centering
		\includegraphics[scale=1]{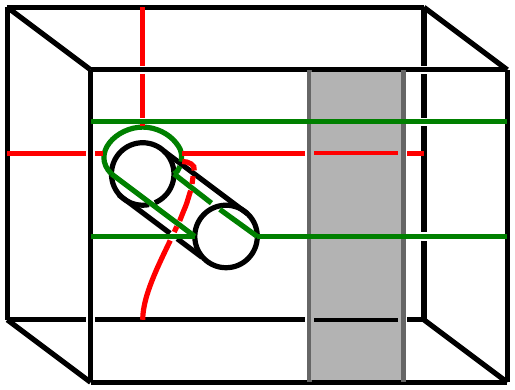}
	\hspace{20px}
		\includegraphics[scale=1]{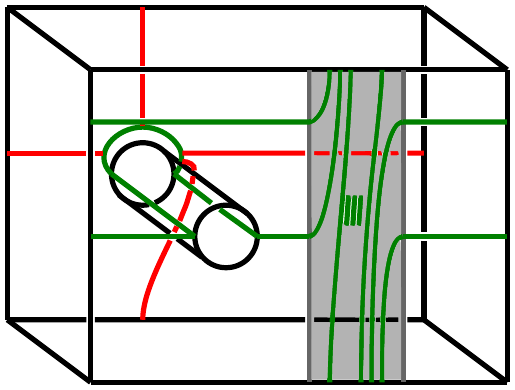}
	\caption{The  bordered sutured manifold $\T_0 = (T,-\Gamma_0',-\W_T)$ and $\T_n = (T,-\Gamma_n',-\W_T)$ on the left hand and, respectively, right hand sides. The annular strip $A$ is shaded in grey.}
	\label{fig:torfill}\label{fig:torfill3}\label{fig:torfill1}
\end{figure}

As described in Section~\ref{sub:tt}, the bypass attachments corresponding to positive or negative basic slice layer can be constrained to lie in the annular strip $A$, which is shaded grey in Figure~\ref{fig:torfill}.  The collection of sutures $\Gamma_{n+1}$ on the resulting space $\T_{n+1}$ are obtained from those on $\T_n$ by applying a single negative Dehn twist along a core curve of the annulus $A$ (note the orientation reversal).

This observation suggests that we consider the further decomposition of $\T_n$ into a union
\[
	\T_n = \T \cup \C_n \cup \A_0
\]
of three bordered sutured manifolds as depicted in Figure~\ref{fig:decomptor}.  Like the $\T_n$, the space $\T$ is diffeomorphic to a thickened punctured torus.  The spaces $\C_n$ and $\A_0$ are each diffeomorphic to thickened annuli $A \times I$ (see Sections \ref{sub:simp_ann} and \ref{sub:twisted_annulus} for descriptions of spaces $\A_n$ and $\C_m$).

\begin{figure}[htbp]
	\centering
		\includegraphics[scale=1]{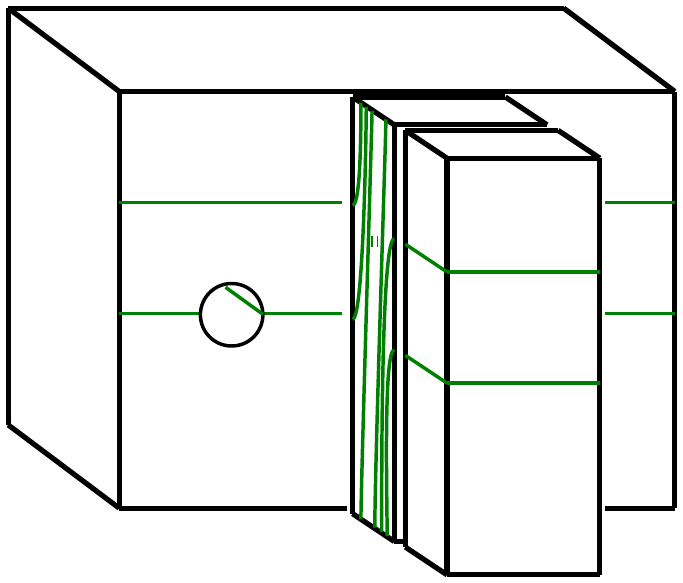}
	\caption{A decomposition of $\T_n$ into the union $\T \cup \C_n \cup \A_0$.}
	\label{fig:decomptor}
\end{figure}

As before, we can use this further decomposition of the $\T_n$ to systematically compute the relevant gluing maps. This is the content of Lemmas \ref{lem:basic_bpann}, \ref{lem:twist_bpann} and \ref{lem:twist_bptor}, which we wield to ultimately establish Theorem~\ref{thm:lim_to_minus}.


\subsection{Computations of Bordered Sutured Modules and Bi-Modules} 
\label{sub:comp_mod}

In this section, we compute the various modules and bi-modules which will be used below in the proof of Theorem~\ref{thm:lim_to_minus}.  

\subsubsection{Torus Modules} 
 \label{sub:torus_modules}

We begin by computing the Type-$D$ modules $\BSD(\T_0)$ and $\BSD(\T_n)$ associated to the spaces depicted in Figures \ref{fig:torfill1}.  Admissible bordered sutured Heegaard diagrams for these spaces are presented in Figure~\ref{fig:torfilldiag}.

\begin{figure}[htbp]
	\centering
	\begin{picture}(110,118)
		\put(0,0){\includegraphics[scale=0.65]{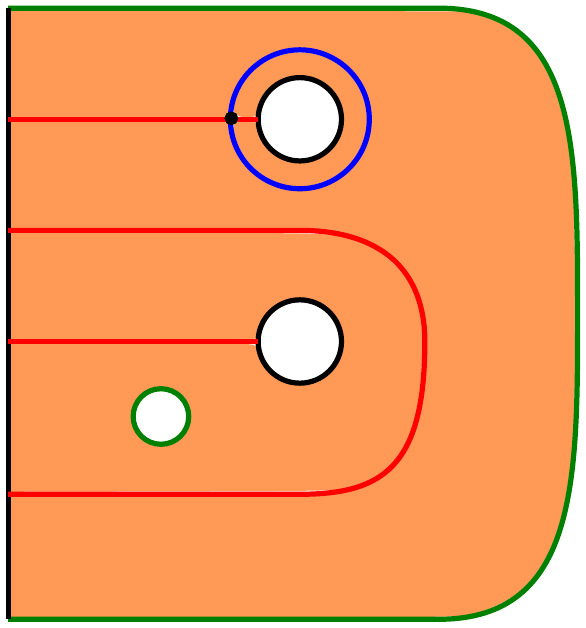}}
		\put(-7,22){\color{red}$2$}
		\put(-7,50){\color{red}$1$}
		\put(-7,70){\color{red}$2$}
		\put(-7,92){\color{red}$1$}
		\put(35,98){$a$}
		\put(-7,116){$\rho$}
	\end{picture}
	\hspace{20px}
	\begin{picture}(110,118)
		\put(0,0){\includegraphics[scale=0.65]{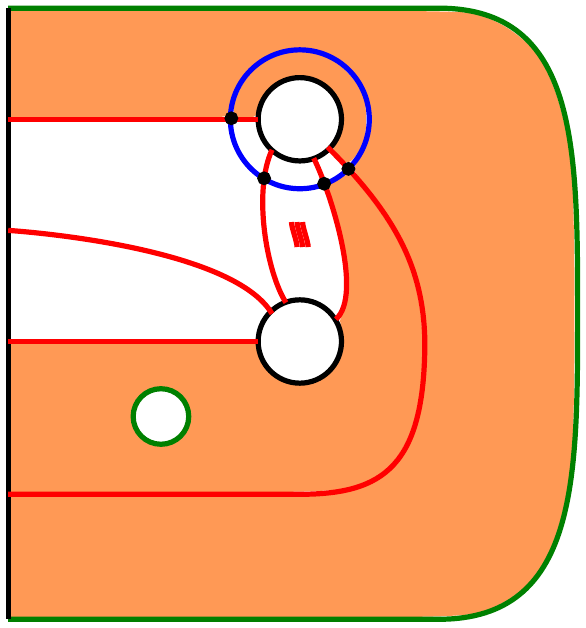}}
		\put(35,98){$a$}
		\put(39,82){$b_1$}
		\put(70,83){$b_n$}
		\put(-7,22){\color{red}$2$}
		\put(-7,50){\color{red}$1$}
		\put(-7,70){\color{red}$2$}
		\put(-7,92){\color{red}$1$}
		\put(-7,116){$\rho$}
	\end{picture}
	\caption{Bordered sutured Heegaard diagrams for the spaces $\T_0$ and $\T_n$, $n \geq 1$. (As usual the black circles are identified.)}
	\label{fig:torfilldiag}\label{fig:torfilldiag3}\label{fig:torfilldiag1}
\end{figure}
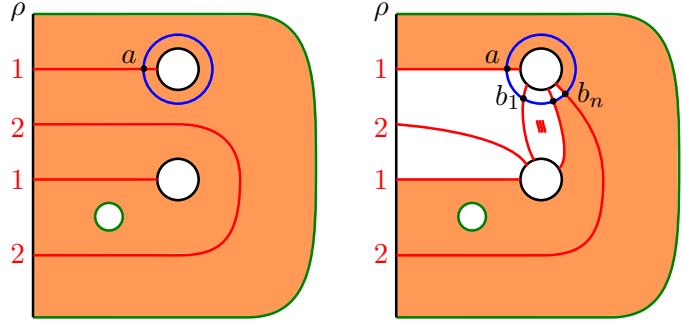

To the diagram on the left in Figure~\ref{fig:torfilldiag1}, we associate the Type-$D$ module $K_0 := {}^{\A(\W_T)} \BSD(\T_1)$, defined over the strand algebra $\A(\W_T)$ from Section~\ref{sec:algebras}. The module $K_0$ is generated by a single element $a$, whose idempotent compatibility is given by
\[
	I_2 \cdot a = a.
\]

The boundary operator on $K_0$ is trivial since all of the regions in the bordered sutured Heegaard diagram are adjacent to portions of the boundary which are sutured.

Similarly, to the of diagram on the right in Figure~\ref{fig:torfilldiag3}, representing the bordered sutured manifold $\T_n$, we associate the Type-$D$ module $K_n := {}^{\A(\W_T)} \BSD(\T_n)$. For an integer $n \geq 1$, the module $K_n$ is generated by the collection of intersections $\{a,b_1,\dots,b_n\}$.  The idempotent compatibilities of these generators are given by
\[
	I_2 \cdot a = a, \;\;\; I_1 \cdot b_i = b_i.
\]

From the diagram shown in Figure~\ref{fig:torfilldiag3}, we compute the following nontrivial terms in the boundary operator on $K_n$:
\begin{align*}
	\delta(b_1) &= \rho_2' \otimes a\\
	\nonumber \delta(b_i) &= \rho_{23}' \otimes b_{i-1}, \;\;\; i = 2,\dots n.
\end{align*}

To see this, observe that, for each $i \geq 1$, there is a single domain contributing to $\delta(b_i)$.  In the case $i = 1$, it corresponds to a domain that is described in the algorithm for computing $\delta$ for nice diagrams in Section~\ref{sub:nice_heegaard_diagrams}. For $i \geq 2$, there are many non-trivial (index 1) domains emanating from $b_i$, but only one which (potentially) contributes non-trivially to the boundary operator --- the domain connecting $b_i$ to $b_{i-1}$.  The domain is a 6-gon with with 2 edges going to Reeb chords $-\rho_2$ and $-\rho_3$, in that order. The contribution to $\delta$ is analogous to the algorithm described in Section~\ref{sub:nice_heegaard_diagrams} except that the two Reeb cords are multiplied together to yield $\rho_{23}$ when determining their contribution to $\delta$. 

\begin{remark}
In the discussion to follow, we will not generally provide justification for our boundary operator computations.  Instead, we leave them as straight-forward exercises for the reader, each of which follows from a line of reasoning similar to that above.
\end{remark}


\subsubsection{Torus Bi-module} 
 \label{sub:torus_bmodules}

We now compute the Type-$DA$ bi-module associated to the space $\T := (T \times [0,1],\Gamma,-\W_T \cup \W_A)$, depicted in Figure~\ref{fig:torusdiag}.  An admissible bordered sutured Heegaard diagram for $\T$ is also shown on the right of Figure~\ref{fig:torusdiag}.

\begin{figure}[htbp]
	\centering
			\includegraphics[scale=1]{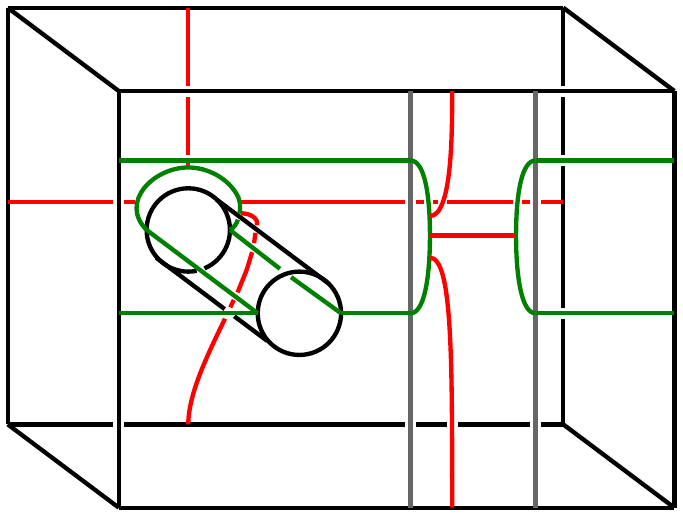}
	\hspace{10px}
	\begin{picture}(146,149)
		\put(0,0){\includegraphics[scale=1]{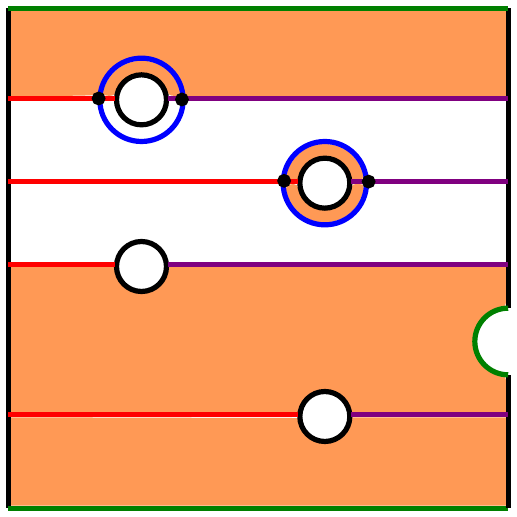}}
		\put(-7,25){\color{red}$2$}
		\put(-7,68){\color{red}$1$}
		\put(-7,93){\color{red}$2$}
		\put(-7,118){\color{red}$1$}
		\put(148,25){\color{red}$2$}
		\put(148,68){\color{red}$1$}
		\put(148,93){\color{red}$2$}
		\put(148,118){\color{red}$1$}
		\put(21,123){$x$}
		\put(53,123){$y$}
		\put(75,100){$y$}
		\put(105,100){$x$}
		\put(-7,147){$\rho$}
		\put(148,147){$\pi$}
	\end{picture}
	\caption{On the left hand side is the decorated sutured cobordism $\T$ between the sutured surfaces $F_3$ and $F_1$. On the right hand side is an admissible bordered sutured Heegaard diagram for $\T$. (As usual the black circles are identified by reflections along horizontal lines.)}
	\label{fig:borderedtorus}\label{fig:torusgeom}\label{fig:torusdiag}
\end{figure}
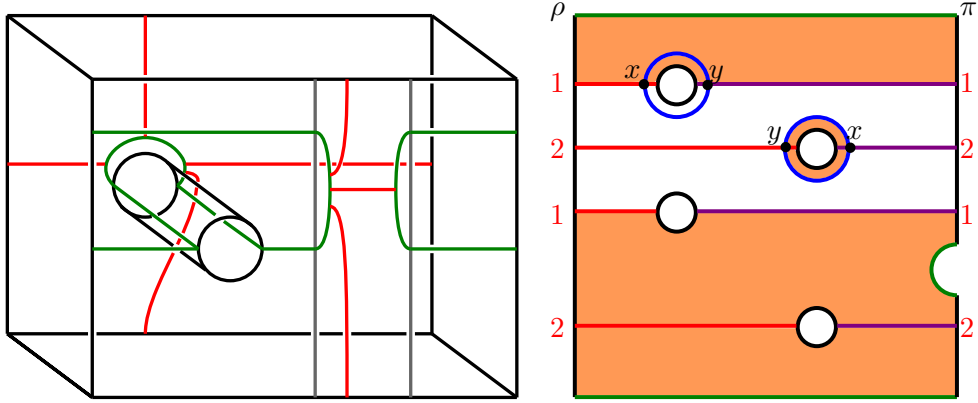

To this bordered sutured Heegaard diagram, we associate the Type-$DA$ bi-module $N := {}^{\A(\W_T)} \BSDA(\T)_{\A(W_A)}$.  

We content ourselves with computing the middle summand $\BSDA(\T,1)$, since that is all we need in our proof of Theorem~\ref{thm:lim_to_minus}.  Indeed, the modules ($A_\infty$ or Type-$D$) to be paired with $N$ have all other summands vanishing, owing to the fact that they arise from bordered sutured manifolds with a single bordered boundary component.  

From Figure~\ref{fig:torusdiag}, we see that $\BSDA(\T,1)$ is generated by elements $x$ and $y$, whose idempotent compatibilities are given by
\[
	I_2 \cdot x \cdot I_2 = x, \;\;\; I_1 \cdot y \cdot I_1 = y.
\]
(Other collections of intersection points will not lie in $\BSDA(\T,1)$.)

There are three domains which each contribute a term to $m_2$ --- two correspond to 8-gons with 2 edges mapping to Reeb chords and one which is an annulus with one boundary component having 4 edges --- one mapping to a Reeb chord --- and the other having 6 edges with 2 mapping to Reeb chords.  The nontrivial operations in $\BSDA(\T,1)$ are given by:
\begin{align*}
\nonumber	m_2(y,\pi_1') &= \rho_2' \otimes x\\
 	m_2(x,\pi_2') &= \rho_3' \otimes y\\
\nonumber 	m_2(y,\pi_{12}') &= \rho_{23}' \otimes y.
\end{align*}


\subsubsection{Annular Modules} 
 \label{sub:simp_ann}

We now focus on the collection of spaces $\A_0 := (A \times [0,1],\Gamma_0,-\W_A)$ and $\A_n := (A \times [0,1],\Gamma_n,-\W_A)$ depicted in Figure~\ref{fig:annA1}.  Admissible bordered sutured Heegaard diagrams for the $\A_n$ are given in Figure~\ref{fig:annC1}.

\begin{figure}[htbp]
	\centering
	\includegraphics[scale=1]{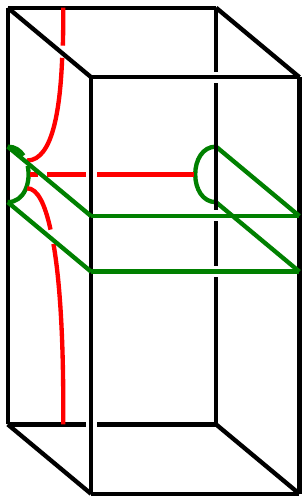}
	\hspace{60px}
	\includegraphics[scale=1]{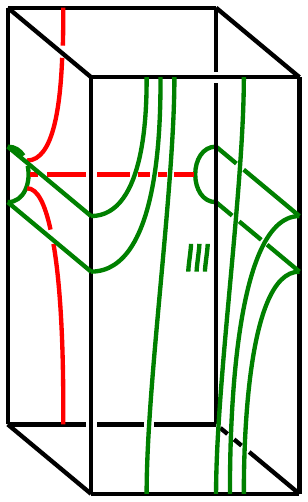}
		\caption{Thickened bordered sutured annuli representing the ``twist region''.}
	\label{fig:annA}\label{fig:annA1}\label{fig:annA3}
\end{figure}

\begin{figure}[htbp]
	\centering
	\begin{picture}(110,150)
		\put(0,0){\includegraphics[scale=1]{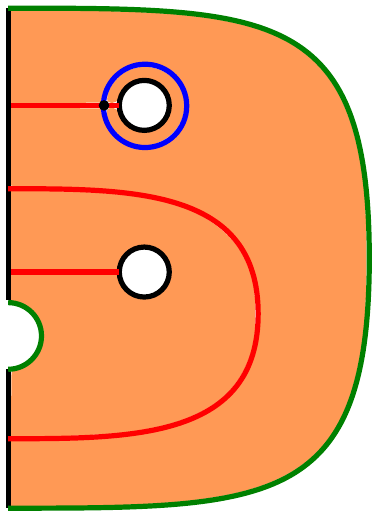}}
		\put(-7,18){\color{red}$2$}
		\put(-7,66){\color{red}$1$}
		\put(-7,90){\color{red}$2$}
		\put(-7,115){\color{red}$1$}
		\put(20,120){$a$}
		\put(-7,148){$\rho$}
	\end{picture}
	\hspace{40px}
	\begin{picture}(110,150)
		\put(0,0){\includegraphics[scale=1]{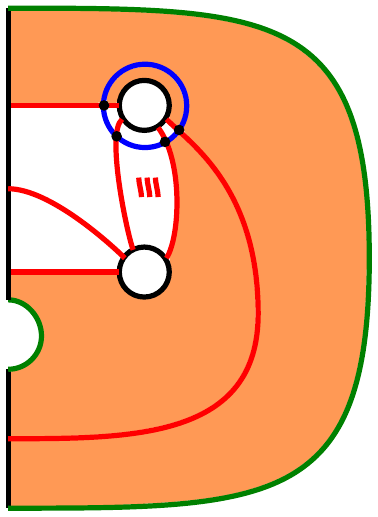}}
		\put(22,120){$a$}
		\put(22,104){$b_1$}
		\put(55,110){$b_n$}
		\put(-7,18){\color{red}$2$}
		\put(-7,66){\color{red}$1$}
		\put(-7,90){\color{red}$2$}
		\put(-7,115){\color{red}$1$}
		\put(-7,148){$\rho$}
	\end{picture}
	\caption{Bordered sutured Heegaard diagrams for the thickened bordered sutured annuli $\A_0$ and $\A_n$. (As usual the black circles are identified.)}
	\label{fig:annC}\label{fig:annC1}\label{fig:annC3}
\end{figure}
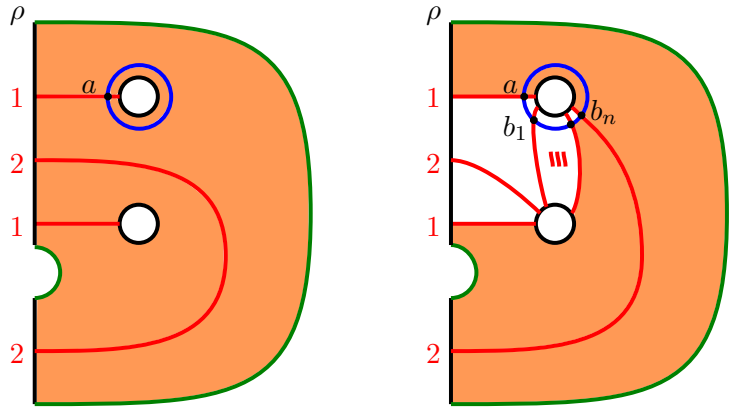

To the diagram on the left hand side of Figure~\ref{fig:annC1}, we associate the Type-$D$ module $M_0 := {}^{\A(\W_A)} \BSD(\A_0)$. The module $M_0$ is generated by a single element $a$, whose idempotent compatibility is given by
\[
	I_2 \cdot a = a.
\]

The corresponding boundary operator $\delta$ is trivial since all of the regions in the corresponding bordered sutured Heegaard diagram are adjacent to portions of the boundary which are sutured.

Similarly, to the collection of diagrams depicted on the right hand side of Figure~\ref{fig:annC3}, we associate the Type-$D$ modules $M_n := {}^{\A(\W_A)} \BSD(\A_n)$. For an integer $n \geq 1$, the module $M_n$ is generated by the collection of intersections $\{a,b_1,\dots,b_n\}$.  The idempotent compatibilities of these generators are given by
\[
	I_2 \cdot a = a, \;\;\; I_1 \cdot b_i = b_i.
\]

From the diagram shown in Figure~\ref{fig:annC3}, we see that the nontrivial terms in the boundary map on $\delta$ are given by:
\begin{align}\label{eqn:AAn_boundary}
	\delta(b_1) &= \rho_1' \otimes a\\
	\nonumber \delta(b_i) &= \rho_{12}' \otimes b_{i-1}, \;\;\; i = 2,\dots n.
\end{align}

The justification for this calculation is identical to that for the torus modules $K_n$ in Section~\ref{sub:torus_modules}.


\subsubsection{Annular Bi-modules} 
 \label{sub:twisted_annulus}

Here, we compute the Type-$DA$ bi-modules associated to the collection of spaces $\C_n := (A \times [0,1],\Gamma_n,-\W_A \cup \W_A)$ depicted in Figure~\ref{fig:annbimod1}.  Admissible bordered sutured Heegaard diagrams for the $\C_n$ are given in Figure~\ref{fig:annbimod2}, to which we associate the Type-$DA$ bi-module $C_n := {}^{\A(\W_A)} \BSDA(\C_n)_{\A(W_A)}$.

\begin{figure}[htbp]
	\centering
		\includegraphics[scale=1]{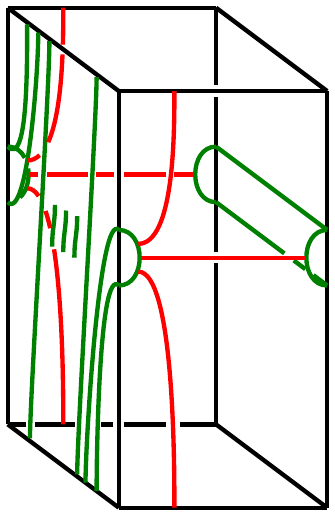}
	\hspace{30px}
	\begin{picture}(148,148)
		\put(0,0){\includegraphics[scale=1]{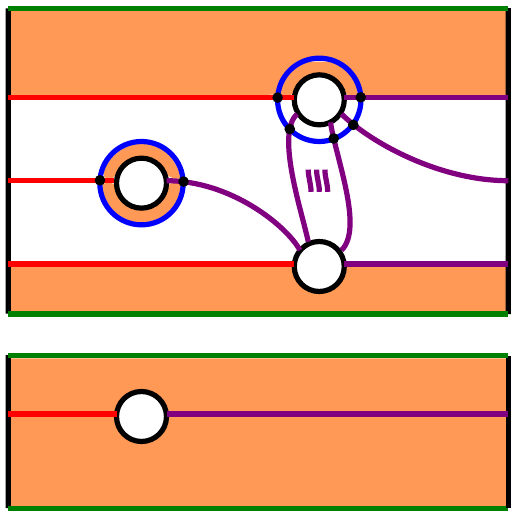}}
		\put(-5,25){\color{red}$2$}
		\put(-5,69){\color{red}$1$}
		\put(-5,93){\color{red}$2$}
		\put(-5,117){\color{red}$1$}
		\put(149,25){\color{red}$2$}
		\put(149,69){\color{red}$1$}
		\put(149,93){\color{red}$2$}
		\put(149,117){\color{red}$1$}
		\put(73,123){$a$}
		\put(73,107){$b_1$}
		\put(106,111){$b_n$}
		\put(107,123){$c$}
		\put(-5,146){$\rho$}
		\put(149,146){$\pi$}
	\end{picture}
	\caption{Doubly bordered annulus and its bordered diagram. (As usual the black circles are identified by reflections along horizontal lines.)}
	\label{fig:annbimodA}\label{fig:annbimod1}\label{fig:annbimod2}
\end{figure}
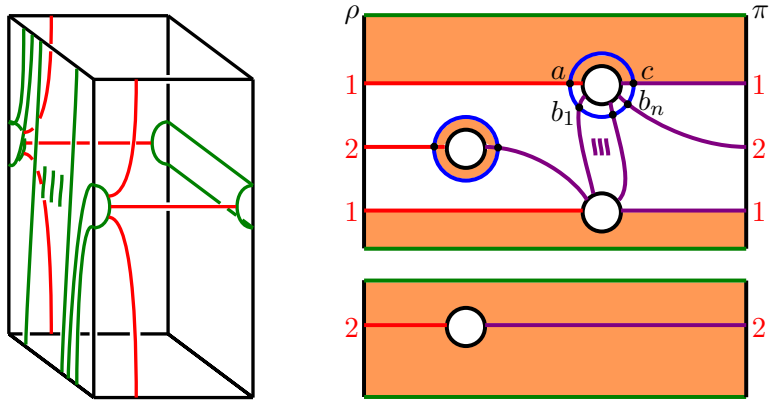

As was the case above for the bi-module $N$, we content ourselves with computing the middle summand $\BSDA(\C_n,1)$ since, that is all we use in our proof of Theorem~\ref{thm:lim_to_minus}.  The bi-module $\BSDA(\C_n,1)$ is generated by the intersections, $\{a,b_1,\dots,b_n,c\}$, whose idempotent compatibilities are given by
\[
	I_2 \cdot a \cdot I_2 = a, \;\;\; I_1 \cdot b_i \cdot I_2 = b_i, \;\;\; I_1 \cdot c \cdot I_1 = c.
\]
(Notice that for each intersection point in $\{a,b_1,\dots,b_n,c\}$ there is a unique second intersection between the $\balpha$ and $\bbeta$ curves that will produce an element of $\BSDA(\C_n,1)$, so we denote the generators of $\BSDA(\C_n,1)$ by the points $\{a,b_1,\dots,b_n,c\}$.)

From Figure~\ref{fig:annbimod2}, we see that the nontrivial operations in $\BSDA(\A_n,1)$ are given as follows:
\begin{align*}
\nonumber	m_1(b_1) &= \rho_1 \otimes a\\
\nonumber	m_1(b_i) &= \rho_{12} \otimes b_{i-1}\\
\nonumber	m_2(c,\pi_1) &= \rho_{12} \otimes b_n\\
	m_2(c,\pi_{12}) &= \rho_{12} \otimes c\\
\nonumber	m_{k+2}(a,\pi_2,\pi_{12},\dots,\pi_{12},\pi_1) &= \rho_2 \otimes b_k\\
\nonumber	m_{k+2}(b_i,\pi_2,\pi_{12},\dots,\pi_{12},\pi_1) &= I_1 \otimes b_{i+k}\\
\nonumber	m_{k+2}(b_{n-k},\pi_2,\pi_{12},\dots,\pi_{12}) &= I_1 \otimes c\\
\nonumber	m_{n+2}(a,\pi_2,\pi_{12},\dots,\pi_{12}) &= \rho_2 \otimes c,
\end{align*}
where $k\geq 1$. 

The above computation is straightforward but somewhat involved, and entirely analogous to the computations made in Appendix A of \cite{LOT1}. 


\subsubsection{Bypass Attachment Annuli} 
 \label{sub:bypass_annuli}

We conclude by computing the Type-$DA$ bi-modules associated to the two spaces $\B_1 := (A \times [0,1],\Gamma',-\W_A \cup \W_D)$ and $\B_2 := (A \times [0,1],\Gamma'',-\W_A \cup \W_D)$ depicted in Figure~\ref{fig:annbypA1}. These spaces will be used to study the attachment of bypass  by gluing them to the bordered sutured manifolds $\D_A$, $\D_B$ or $\D_C$ from Subsection~\ref{sub:bord_analogue}.

\begin{figure}[htbp]
	\centering
	\includegraphics[scale=1]{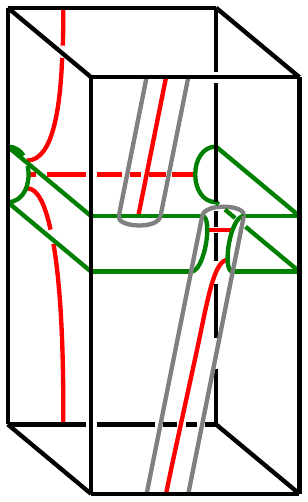}
	\hspace{85px}
	\includegraphics[scale=1]{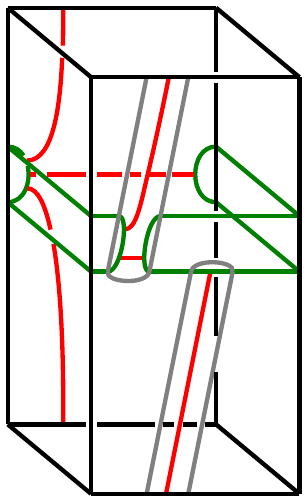}
	\caption{Bordered sutured manifolds for the two possible bypass attachments to an annulus: $\B_1 := (A \times [0,1],\Gamma',-\W_A \cup \W_D)$ and $\B_2 := (A \times [0,1],\Gamma'',-\W_A \cup \W_D)$ }
	\label{fig:annbypA1}\label{fig:annbypB1}
\end{figure}	
	
\begin{figure}[htbp]
	\centering
	\begin{picture}(147,147)
	\put(0,0){\includegraphics[scale=1]{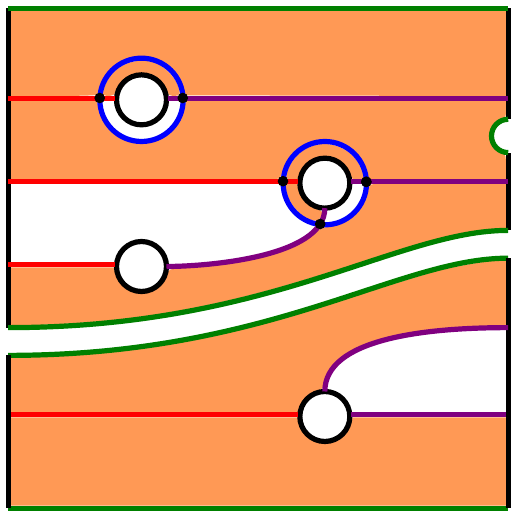}}
	\put(-7,27){\color{red}$2$}
	\put(-7,69){\color{red}$1$}
	\put(-7,93){\color{red}$2$}
	\put(-7,117){\color{red}$1$}
	\put(149,118){\color{red}$1$}
	\put(149,93){\color{red}$2$}
	\put(149,52){\color{red}$1$}
	\put(149,27){\color{red}$2$}
	\put(73,99){$d$}
	\put(92,77){$e$}
	\put(107,99){$f$}
	\put(-7,145){$\rho$}
	\put(149,145){$\pi$}
	\end{picture}
	\hspace{30px}
	\begin{picture}(147,147)
	\put(0,0){\includegraphics[scale=1]{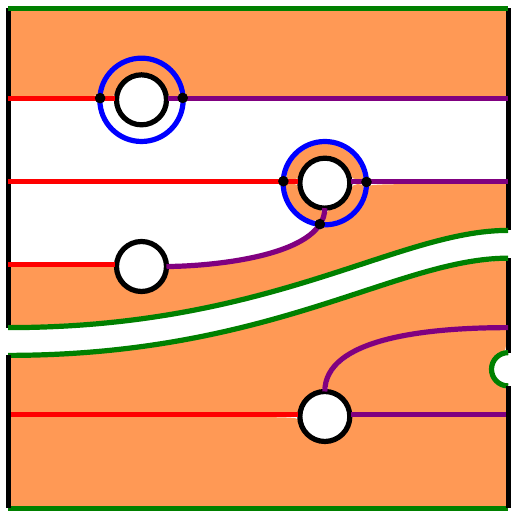}}
	\put(-7,27){\color{red}$2$}
	\put(-7,69){\color{red}$1$}
	\put(-7,93){\color{red}$2$}
	\put(-7,117){\color{red}$1$}
	\put(149,118){\color{red}$1$}
	\put(149,93){\color{red}$2$}
	\put(149,52){\color{red}$1$}
	\put(149,27){\color{red}$2$}
	\put(73,99){$d$}
	\put(92,77){$e$}
	\put(107,99){$f$}
	\put(-7,145){$\rho$}
	\put(149,145){$\pi$}
	\end{picture}
	\caption{Heegaard diagrams for the possible bypass attachments to an annulus $\B_1$ and $\B_2$. (As usual the black circles are identified by reflections along horizontal lines.)}
	\label{fig:annbypA2}\label{fig:annbypB2}
\end{figure}
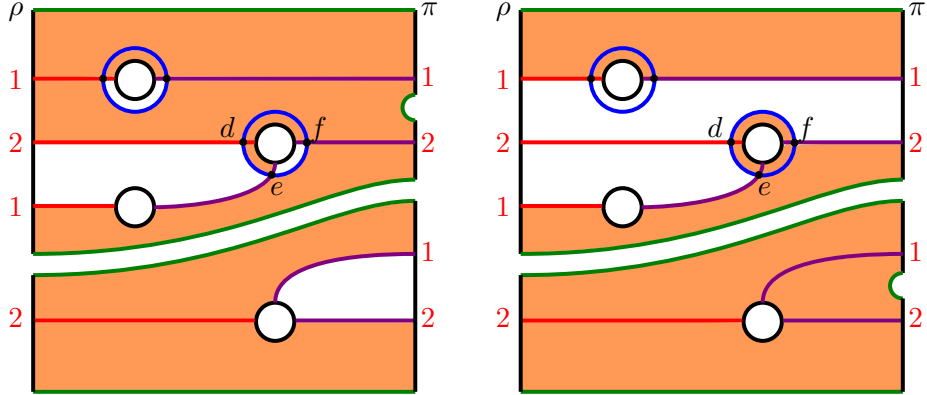

To the bordered sutured Heegaard diagrams shown on the left and right hand side of Figure~\ref{fig:annbypA2}, we associate the Type-$DA$ bi-modules $B_1 := {}^{\A(\W_A)} \BSDA(\B_1)_{\A(W_D)}$ and $B_2 := {}^{\A(\W_A)} \BSDA(\B_2)_{\A(W_D)}$, respectively.  As before, we compute only the summand $\BSDA(\B_i,1)$, since that is all we need in our proof of Theorem~\ref{thm:lim_to_minus}.  

The bi-module $\BSDA(\B_1,1)$ is generated by the intersections, $\{d,e,f\}$, whose idempotent compatibilities are given by
\[
	I_1 \cdot d \cdot I_1 = d, \;\;\; I_2 \cdot e \cdot I_1 = e, \;\;\; I_2 \cdot f \cdot I_2 = f.
\]

There are two nontrivial domains, one contributes a term to $m_1$ and another to $m_2$.
The nontrivial operations in $\BSDA(\B_1,1)$ are given by:
\begin{align*}
	m_1(d) &= \rho_1' \otimes e\\
\nonumber 	m_2(f,\pi') &= I_2 \otimes e.
\end{align*}

The bi-module $\BSDA(\B_2,1)$ is again generated by the intersections, $\{d,e,f\}$, with idempotent compatibilities
\[
	I_1 \cdot d \cdot I_1 = d, \;\;\; I_2 \cdot e \cdot I_1 = e, \;\;\; I_2 \cdot f \cdot I_2 = f.
\]

There are again two domains which contribute terms to either $m_1$ or $m_2$.
The nontrivial operations in $\BSDA(\B_2,1)$ are given by:
\begin{align*}
	m_1(d) &= \rho_1' \otimes e\\
\nonumber 	m_2(f,\pi') &= \rho_2' \otimes d.
\end{align*}
\



\subsection{Computation of Gluing Maps} 
 \label{sub:computation_of_gluing_maps}

Having finished computing the various modules and bi-modules which are needed to prove Theorem~\ref{thm:lim_to_minus}, we now determine the various maps on Floer homology induced by either positive or negative bypass attachment.  

After reversing orientation, the lemma below describes the effect on bordered Floer homology of attaching either a positive or negative bypass to the bordered sutured annulus $\A_0$.

\begin{lemma}\label{lem:basic_bpann}
	The map on bordered sutured Floer homology induced by attaching a positive bypass to the thickened annulus $\A_0$ of slope zero is given by the following equation:
\begin{align*}
	\psi_p: M_0 &\to M_1\\
	\nonumber a &\mapsto I_2 \otimes a.
\end{align*} 
	
	Similarly, the map on bordered sutured Floer homology induced by attaching a negative bypass to the thickened annulus $\A_0$ of slope zero is given by the following equation:
\begin{align*}
	\psi_n: M_0 &\to M_1\\
\nonumber a &\mapsto \rho_2' \otimes b_1.
\end{align*}
\end{lemma}

\begin{proof}

The proof of Lemma~\ref{lem:basic_bpann} centers around the following key diagram.

\begin{center}
\begin{tikzpicture}[->,>=stealth',shorten >=1pt,auto,node distance=2.8cm,semithick]
	\node  (A) at (0,0) {$B_2 \boxtimes M_B$};
	\node  (B) at (0,2) {$B_1 \boxtimes M_B$};
	\node  (C) at (3.5,0) {$B_2 \boxtimes M_C$};
	\node  (D) at (3.5,2) {$B_1 \boxtimes M_C$};
	\node  (E) at (0,1) {$M_0$};
	\node  (F) at (3.5,1) {$M_1$};

	\path (A) edge [bend right] node[below] {\small $\id_{B_2} \boxtimes \phi_B$} (C)
		  (A) edge [-,double] (E)
		  (C) edge [-,double] (F)
		  (B) edge [-,double] (E)
		  (D) edge [-,double] (F)
		  (E) edge [bend right] node[below] {$\psi_n$} (F)
		  (E) edge [bend left] node[above] {$\psi_p$}(F)
	      (B) edge [bend left] node[above] {\small $\id_{B_1} \boxtimes \phi_B$} (D);
\end{tikzpicture}
\end{center}

The Type-$D$ module in the upper left-hand corner of this diagram correspond to gluing the bordered sutured $3$-manifolds $\B_1$ and $\D_B$ depicted in Figures \ref{fig:annbypA1} and \ref{fig:bord_disk_b} together along their bordered boundaries.  The resulting bordered sutured $3$-manifold is the thickened annulus $\A_0$ of slope zero depicted in Figure~\ref{fig:annA1}.  

Similarly, the Type-$D$ module in the lower left-hand corner of this diagram correspond to gluing the bordered sutured $3$-manifolds $\B_2$ and $\D_B$ depicted in Figures \ref{fig:annbypB1} and \ref{fig:bord_disk_b} together along their bordered boundaries.  The resulting bordered sutured $3$-manifold is, again, the thickened annulus $\A_0$ of slope zero depicted in Figure~\ref{fig:annA1}. 

The same is true for the modules on the right-hand side of the diagram. The Type-$D$ module in the upper right-hand corner of this diagram correspond to gluing the bordered sutured $3$-manifolds $\B_1$ and $\D_C$ depicted in Figures \ref{fig:annbypA1} and \ref{fig:bord_disk_c} together along their bordered boundaries, while the module in the lower right corresponds to gluing the spaces $\B_2$ and $\D_C$ in Figures \ref{fig:annbypB1} and \ref{fig:bord_disk_c}.  The resulting bordered sutured $3$-manifolds are both equal to the thickened annulus $\A_1$ of slope one depicted in Figure~\ref{fig:annC3}. 

Observe that there exist canonical identifications between the Type-$D$ modules $B_1 \boxtimes M_B$ and $B_2 \boxtimes M_B$ with $M_0$ given by:
\begin{equation}\label{eqn:annzero}
	f \otimes z = a = f' \otimes z.
\end{equation}

Similarly, by idempotent considerations, and by Equation~\eqref{eqn:AAn_boundary}, we see that there are canonical identifications of the Type-$D$ modules $B_1 \boxtimes M_C$ and $B_2 \boxtimes M_C$ with $M_1$ given by
\begin{align}\label{eqn:annone}
	d \otimes w &= b_1 = d' \otimes w \text{ and }\\
\nonumber	e \otimes w &= a = e' \otimes w.
\end{align}

We now turn to computing the associated gluing maps $\psi_p$ and $\psi_n$ induced by positive and negative bypass attachment respectively.  These maps are each defined by  
\[
	\id_{B_1} \boxtimes \phi_B: B_1 \boxtimes M_B \to B_1 \boxtimes M_C
\]
and
\[
	\id_{B_2} \boxtimes \phi_B: B_2 \boxtimes M_B \to B_2 \boxtimes M_C
\]
under the identifications given by Equations \ref{eqn:annzero} and \ref{eqn:annone}, respectively.  To see that the first map corresponds to positive stabilization and the second to negative stabilization, simply compare the associated bordered sutured manifolds shown in Figure~\ref{fig:annbypA1} with Figure~\ref{fig:slice_bp}, and recall that in the former setting, orientations are reversed.

 Applying Definition~\ref{def:boxmap} to compute the map $\id_{B_1} \boxtimes \phi_B$:
\begin{align*}
	\id_{B_1} \boxtimes \phi_B(f \otimes z) &= m_2(f,\pi') \otimes w\\
\nonumber	&= I_2 \otimes (e \otimes w)
\end{align*}

Similarly, applying Definition~\ref{def:boxmap}, to the map $\id_{B_2} \boxtimes \phi_B$, we obtain:
\begin{align*}
	\id_{B_2} \boxtimes \phi_B(f' \otimes z) &= m_2(f',\pi')\\
\nonumber	&= \rho_2' \otimes (d' \otimes w)
\end{align*} 

Therefore, under the identifications given in Equations \eqref{eqn:annzero} and \eqref{eqn:annone}, we have that
\[
	\psi_p(a) = I_2 \otimes a
\]
and
\[
	\psi_n(a) = \rho_2' \otimes b_1,
\]
completing the proof of Lemma~\ref{lem:basic_bpann}

\end{proof}


\begin{lemma}\label{lem:twist_bpann}
	The map on bordered sutured Floer homology induced by attaching a positive bypass to the thickened annulus $\A_m$ of slope $m$ is given by the following equation
\begin{align*}
\nonumber \psi_{p,m}: M_m &\to M_{m+1}\\
		b_i &\mapsto I_1 \otimes b_i\\
\nonumber a &\mapsto I_2 \otimes a.
\end{align*}
	
	Similarly, the map on bordered sutured Floer homology induced by attaching a negative bypass to the thickened annulus $\A_m$ of slope $m$ is given by the following equation
\begin{align*}
	\psi_{n,m}: M_m &\to M_{m+1}\\
	\nonumber b_i &\mapsto I_1 \otimes b_{i+1}\\
	\nonumber a &\mapsto \rho_2' \otimes b_1.
\end{align*}
\end{lemma}

\begin{proof}  The proof of Lemma~\ref{lem:twist_bpann} is very similar to that of Lemma~\ref{lem:basic_bpann}.  In this case, the proof centers around the following key diagram.

\begin{center}
\begin{tikzpicture}[->,>=stealth',shorten >=1pt,auto,node distance=2.8cm,
	                    semithick]

	\node  (A) at (-1.35,0) {$M_m = C_m \boxtimes M_0$};
	\node  (B) at (4.1,0) {$C_m \boxtimes M_1 = M_{m+1}$};

	\path (A) edge [bend left] node {\small $\psi_{p,m} :=\id_{C_m} \boxtimes \psi_p$} (B)
		  (A) edge [bend right] node[below] {\small $\psi_{n,m} := \id_{C_m} \boxtimes \psi_n$}
 (B);
\end{tikzpicture}
\end{center}

There exists a canonical identification between $C_m \boxtimes M_0$ and $M_m$ given by
\begin{align*}
	a \otimes a &= a \text{ and }\\
\nonumber	b_i \otimes a &= b_i.
\end{align*}

Similarly, there exists a canonical identification between $C_m \boxtimes M_1$ and $M_{n+1}$ is given by
\begin{align*}
\nonumber	a \otimes a &= a,\\
	b_i \otimes a &= b_i, \;\;\; i = 1,\dots n, \text{ and }\\
\nonumber	c \otimes b_1 &= b_{n+1}.
\end{align*}

As in the proof of Lemma~\ref{lem:basic_bpann}, Lemma~\ref{lem:twist_bpann} now follows by applying Definition~\ref{def:boxmap} to compute the maps $\id_{C_m} \boxtimes \psi_p$ and $\id_{C_m} \boxtimes \psi_n$.

\end{proof}

\begin{lemma}\label{lem:twist_bptor}
	The map on bordered sutured Floer homology induced by attaching a positive bypass to the thickened punctured torus $\T_m$ of slope $m$ is given by the following equation
	\begin{align*}
	\nonumber	\eta_{p,m}: K_n &\to K_{m+1}\\
		b_i &\mapsto I_1 \otimes b_i\\
	\nonumber a &\mapsto I_2 \otimes a.
	\end{align*}
	
	Similarly, the map on bordered sutured Floer homology induced by attaching a negative bypass to the thickened punctured torus $\T_m$ of slope $m$ is given by the following equation
	\begin{align*}
	\nonumber	\eta_{n,m}: K_m &\to K_{m+1}\\
		b_i &\mapsto I_1 \otimes b_{i+1}\\
	\nonumber a &\mapsto \rho_3' \otimes b_1.
	\end{align*}
\end{lemma}

\begin{proof}
As before, the proof of Lemma~\ref{lem:twist_bptor} centers around the following key diagram.

\begin{center}
\begin{tikzpicture}[->,>=stealth',shorten >=1pt,auto,node distance=2.8cm,
	                    semithick]

	\node  (A) at (-1.25,0) {$K_m = N \boxtimes M_m$};
	\node  (B) at (4,0) {$N \boxtimes M_{n+1} = K_{m+1}$};

	\path (A) edge [bend left] node {\small $\eta_{p,m} := \id_N \boxtimes \psi_{p,m}$} (B)
	  	(A) edge [bend right] node[below] {\small $\eta_{n,m+1} := \id_N \boxtimes \psi_{n,m}$}
	 (B);
\end{tikzpicture}
\end{center}

We leave the remainder of the proof as an exercise to the reader, noting that the argument is similar to those establishing Lemmas \ref{lem:basic_bpann} and \ref{lem:twist_bpann} above.

\end{proof}


\subsubsection{Proof of Theorem~\ref{thm:lim_to_minus}} 
 \label{sub:proof_of_theorem_thm:lim_to_minus}

Having set up the necessary algebraic machinery, we now complete the proof of Theorem~\ref{thm:lim_to_minus}.

As shown in Section~\ref{sub:torus_modules}, the Type-$D$ module $K_n$ associated to the thickened, punctured, bordered sutured torus $\T_n$ of slope $n$ torus is given by
\begin{center}\label{eqn:ntormod}
\begin{tikzpicture}	[->,>=stealth',auto,thick]
	\node (a) at (0,0){$b_{n}$};
	\node (b) at (2,0) {$b_{n-1}$};
	\node (c) at (4,0) {$\dots$};
	\node (d) at (6,0) {$b_1$};
	\node (e) at (8,0) {$a.$};
	
	\draw (a) edge node[above] {\small $\rho_{23}$} (b);
	\draw (b) edge node[above] {\small $\rho_{23}$} (c);
	\draw (c) edge node[above] {\small $\rho_{23}$} (d);
	\draw (d) edge node[above] {\small $\rho_{2}$} (e);
\end{tikzpicture}
\end{center}
(Recall this diagram shows the generators of the Type-$D$ structures as vertices and the edges denote the map $\delta$.)

According to Lemma~\ref{lem:twist_bptor}, positive and negative bypass attachments induce maps
\begin{align*}
	\eta_{p,m}: K_m &\to K_{m+1}\\
	b_i &\mapsto I_1 \otimes b_i\\
	a &\mapsto I_2 \otimes a
\end{align*}
and
\begin{align*}
	\eta_{n,m}: K_m &\to K_{m+1}\\
	b_i &\mapsto I_1 \otimes b_{i+1}\\
	a &\mapsto \rho_3' \otimes b_1
\end{align*}
respectively.

The groups $K_i$ and maps $\eta_{p,j}$ and $\eta_{n,j}$ can, therefore, be organized neatly into the diagram given in Figure~\ref{fig:big_diag}. The columns of this diagram depict the Type-$D$ modules $K_n$.  The south eastern pointing (blue) arrows depict the maps $\eta_{p,m}:K_m \to K_{m+1}$, while the eastern pointing (red) arrows depict the maps $\eta_{n,m}: K_m \to K_{m+1}$.

\begin{figure}[htbp]
\begin{tikzpicture}	[->,>=stealth',auto,thick]
	\node (a) at (0,0){$a$};
	\node (b) at (2,0) {$b_1$};
	\node (c) at (4,0) {$b_2$};
	\node (d) at (6,0) {$\dots$};
	\node (e) at (8,0) {$b_{n-1}$};
	\node (f) at (10,0) {$\dots$};
	\node (g) at (2,-1.5){$a$};
	\node (h) at (4,-1.5) {$b_1$};
	\node (i) at (6,-1.5) {$\dots$};
	\node (j) at (8,-1.5) {$b_{n-2}$};
	\node (k) at (10,-1.5) {$\dots$};
	\node (l) at (4,-3) {$a$};
	\node (m) at (6,-3) {$\dots$};
	\node (n) at (8,-3) {$b_{n-3}$};
	\node (o) at (10,-3) {$\dots$};
	\node (p) at (6,-4.5) {$\dots$};
	\node (q) at (8,-4.5) {$\dots$};
	\node (r) at (10,-4.5) {$\dots$};
	\node (s) at (8,-6) {$a$};
	\node (t) at (10,-6) {$\dots$};	
	\node (u) at (10,-7.5) {$\dots$};
		
	\draw (a) edge[red] node[black,above] {$\rho_3$} (b);
	\draw (b) edge[red] (c);
	\draw (c) edge[red] (d);
	\draw (d) edge[red] (e);
	\draw (e) edge[red] (f);
	\draw (g) edge[red] node[black,above] {$\rho_3$} (h);
	\draw (h) edge[red] (i);
	\draw (i) edge[red] (j);
	\draw (j) edge[red] (k);
	\draw (l) edge[red] node[black,above] {$\rho_3$} (m);
	\draw (m) edge[red] (n);
	\draw (n) edge[red] (o);
	\draw (s) edge[red] node[black,above] {$\rho_3$} (t);
	
	\draw (a) edge[blue] (g);
	\draw (b) edge[blue] (h);
	\draw (c) edge[blue] (i);
	\draw (d) edge[blue] (j);
	\draw (e) edge[blue] (k);
	\draw (g) edge[blue] (l);
	\draw (h) edge[blue] (m);
	\draw (i) edge[blue] (n);
	\draw (j) edge[blue] (o);
	\draw (l) edge[blue] (p);
	\draw (p) edge[blue] (s);
	\draw (n) edge[blue] (r);
	\draw (s) edge[blue] (u);

	\draw (b) edge node[right] {$\rho_2$} (g);
	\draw (c) edge node[right] {$\rho_{23}$} (h);
	\draw (e) edge node[right] {$\rho_{23}$} (j);
	\draw (h) edge node[right] {$\rho_2$} (l);
	\draw (j) edge node[right] {$\rho_{23}$} (n);
	\draw (n) edge node[right] {$\rho_{23}$} (q);
	\draw (q) edge node[right] {$\rho_2$} (s);
\end{tikzpicture}
\caption{Direct limit diagram.}
\label{fig:big_diag}
\end{figure}
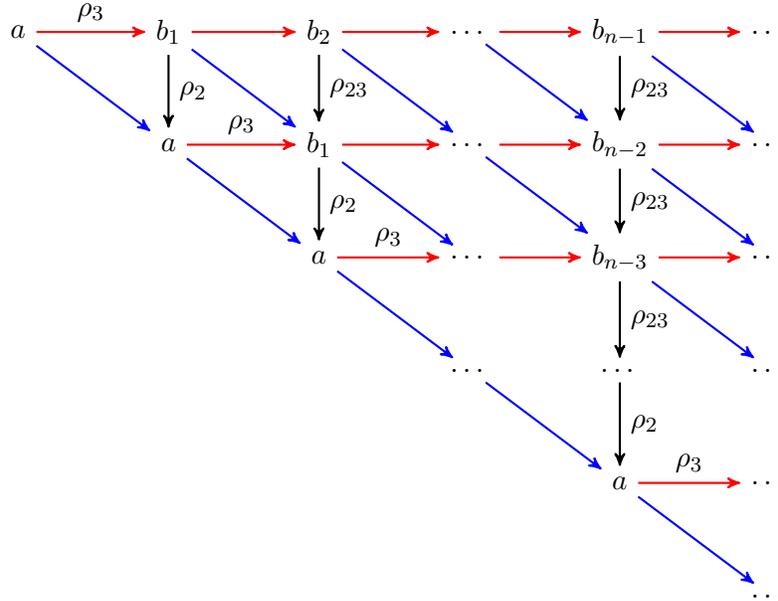

As discussed in Section~\ref{sec:limits}, the invariant $\SFHL(-Y,K)$ is obtained by taking a directed limit of sutured Floer homology groups over HKM gluing maps arising form iterated negative basic slice attachments.  In the bordered setting, this means that we are taking the directed limit of the collection of modules $\{K_m\}$ and maps $\{\eta_{n,m}\}$ connecting them.  Doing so, the resulting module is given as follows.
\begin{center}
\begin{tikzpicture}	[->,>=stealth',auto,thick]
	\node (a) at (0,0){$\delta_0$} ;
	\node (b) at (2,0) {$\delta_1$} ;
	\node (c) at (4,0) {$\delta_2$} ;
	\node (d) at (6,0) {$\delta_3$} ;
	\node (e) at (8,0) {$\dots$} ;
		
	\draw (a) edge node[above] {$\rho_{23}$} (b);
	\draw (b) edge node[above] {$\rho_{23}$} (c);
	\draw (c) edge node[above] {$\rho_{23}$} (d);
	\draw (d) edge (e);
\end{tikzpicture}
\end{center}

From Figure~\ref{fig:big_diag}, we see that the two flavors of $\eta$-maps commute in the sense that
\[
	\eta_{n,i+1} \circ \eta_{p,i} = \eta_{p,i+1} \circ \eta_{n,i} : K_i \to K_{i+2},
\]
corresponding geometrically to the fact that the associated bypasses attachments can be made along disjoint annuli. As discussed in Section~\ref{sub:leg_limits} for the limit invariant $\SFHL(-Y,K)$, this commutativity property implies that the collection of maps $\{\eta_{p,m}\}$ together yield a well-defined $U$-action on the module $\underrightarrow{K} := \underrightarrow{\mr{Lim}}_n K_n$ which sends each $\delta_i$ to $\delta_{i+1}$.

It follows that the Type-$D$ module $\underrightarrow{K}$ is isomorphic under the identification given below to the Type-$D$ module $K^-$ yielding the minus version of knot Floer homology that was discussed in Section~\ref{sub:bordered_invariants_and_knot_floer_homology}.
\begin{center}
		\begin{tikzpicture}	[->,>=stealth',auto,thick]
			\node (aa) at (-.75,1.45) {$\underrightarrow{K}:=$};
			\node (bb) at (-.75,0.05) {$K^-=$};
			\node (a) at (0,1.5){$\delta_0$} ;
			\node (b) at (2.5,1.5) {$\delta_1$} ;
			\node (c) at (5,1.5) {$\delta_2$} ;
			\node (d) at (7.5,1.5) {$\delta_3$} ;
			\node (e) at (10,1.5) {$\dots$} ;
			\node (f) at (0,0){$x$} ;
			\node (g) at (2.5,0) {$U \cdot x$} ;
			\node (h) at (5,0) {$U^2 \cdot x$} ;
			\node (i) at (7.5,0) {$U^3 \cdot x$} ;
			\node (j) at (10,0) {$\dots$} ;

			\draw (a) edge node[above] {$\rho_{23}$} (b);
			\draw (b) edge node[above] {$\rho_{23}$} (c);
			\draw (c) edge node[above] {$\rho_{23}$} (d);
			\draw (d) edge (e);
			\draw (f) edge node[above] {$\rho_{23}$} (g);
			\draw (g) edge node[above] {$\rho_{23}$} (h);
			\draw (h) edge node[above] {$\rho_{23}$} (i);
			\draw (i) edge (j);
			\draw (a) edge[bend left] node[above] {$U \cdot$} (b);
			\draw (b) edge[bend left] node[above] {$U \cdot$} (c);
			\draw (c) edge[bend left] node[above] {$U \cdot$} (d);
			\draw (d) edge[bend left] node[above] {$U \cdot$} (e);
			\draw (a) edge[<->] (f);
			\draw (b) edge[<->] (g);
			\draw (c) edge[<->] (h);
			\draw (d) edge[<->] (i);
		\end{tikzpicture}
\end{center}

Finally, on the level of sutured Floer homology, we have
\begin{align*}
	\SFHL(-Y,K) &:= \varinjlim \SFH(-Y(K),-\Gamma_i)\\
		&\cong \varinjlim H_*(\BSA(-Y(K),\Gamma',\SF_T) \boxtimes \BSD(\ST_n))\\
		&\cong H_*(\varinjlim(\BSA(-Y(K),\Gamma',\SF_T) \boxtimes \BSD(\ST_n)))\\
		&\cong H_*(\BSA(-Y(K),\Gamma',\SF_T) \boxtimes \varinjlim \BSD(\ST_n))\\
		&\cong H_*(\BSA(-Y(K),\Gamma',\SF_T) \boxtimes \underrightarrow{K})\\
		&\cong H_*(\BSA(-Y(K),\Gamma',\SF_T) \boxtimes K^-)\\
		&\cong \HFKM(-Y,K).
\end{align*}
In the above, the third equality follows form work of B\"okstedt and Neeman \cite{BN}, who showed that the homology functor and direct limits commute in the homotopy category of complexes.  The forth equality follows from a standard fact asserting the commutativity of direct limits and ($A_\infty$) tensor products.

This completes the proof of Theorem~\ref{thm:lim_to_minus}. \qed



\section{Equivalence of Legendrian invariants} 
\label{sec:equiv_leg}

In this section, we prove Theorem~\ref{thm:lim_to_minus_leg}, which states that the LIMIT invariant $\EHL$ defined in Section~\ref{sub:leg_limits} agrees with the LOSS invariant $\SL$ under the identification given by Theorem~\ref{thm:lim_to_minus}.  This result, together with the main theorems of \cite{StV} and \cite{BVV}, completes a body of work which clarifies relationships between the various Legendrian and transverse invariants defined in the context of Heegaard Floer theory.

Let $K \subset (Y,\xi)$ be a given null-homologous Legendrian knot inside the contact 3-manifold $(Y,\xi)$.  Recall that the LIMIT invariant $\EHL(K)$ is defined to be the residue class of the collection of HKM invariants $\{\EH(S^i_-(K))\}$ inside the sutured limit homology group $\SFHL(-Y,K)$.  To relate the LIMIT and and LOSS invariants, we begin by constructing a bordered sutured diagram which can be simultaneously completed to compute either the HKM invariant $\EH(K)$ or the LOSS invariant $\SL(K)$.

With this in mind, recall that the LOSS invariant is defined using an open book and collection of basis arcs as depicted on the left hand side in Figure~\ref{fig:hkm_lossA1} (see Section~\ref{sub:loss_invt}).  In this figure, we see the Legendrian knot $K$ sitting on the page $S_{1/2}$ of the open book $(S,\phi)$ for the ambient contact 3-manifold $(Y,\xi)$.  Observe that $K$ is pierced by the single basis element $a_0$ transversally in a single point.  Following the process discussed in Section~\ref{sub:loss_invt}, we obtains the doubly-pointed Heegaard diagram denoted $\HD$ for the pair $(-Y,K)$.  The collection of intersections $\x := \{x_0,\dots,x_n\}$ on the page $S_{1/2}$ defines a generator of $\CFKM(-Y,K)$, and the LOSS invariant is defined as
\[
	\SL(K) := [\x] \in \HFKM(-Y,K).
\]

\begin{figure}[htbp]
	\centering
	\begin{picture}(143,75)
		\put(0,0){\includegraphics[scale=1]{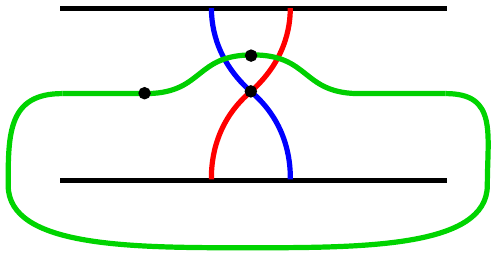}}
		\put(110,52){$L$}
	\end{picture}
	\hspace{20px}
		\includegraphics[scale=1]{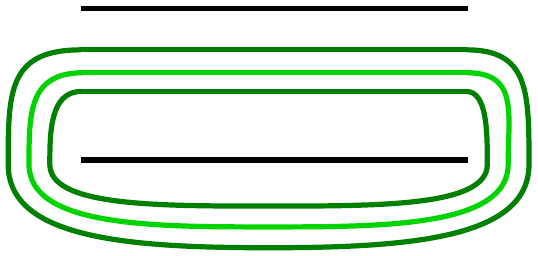}
	\caption{On the left is the doubly pointed Heegaard diagram defining the LOSS invariant. On the right is the sutured Heegaard diagram constructed by Stipsicz and V\'ertesi computing the HKM invariant.}
	\label{fig:hkm_lossA}\label{fig:hkm_lossA1}\label{fig:hkm_lossA2}
\end{figure}

In \cite{StV}, Stipsicz and V\'ertesi showed how to slightly modify the open book decomposition $(S,\phi)$ for $(Y,\xi)$ to produce a partial open book decomposition $(S,P,\phi_P)$ for the space $(Y(K),\xi_K)$ obtained by removing an open standard neighborhood of $K$.  The result of their procedure is depicted on the right hand side of Figure~\ref{fig:hkm_lossA2}.  If $P$ denotes the result of removing an open tubular neighborhood of $K$ from $S$ and we set $\phi_P=\phi|_P$, then the modified partial open book is equal to $(S,P,\phi_P)$.  Given a basis $\{a_0,\dots,a_n\}$ for $S$ adapted to $K$, we obtain a new basis $\{a_1,\dots,a_n\}$ for $P$ by dropping the arc which previously intersected $K$.

We denote by $\HD'$ the sutured Heegaard diagram obtained from the partial open book $(S,P,\phi_P)$ and basis arcs $\{a_1,\dots,a_n\}$.  Again, the collection of intersections $\x' := \{x_1,\dots,x_n\}$ on the subsurface $P$ defines a cycle in $\SFC(\HD')$, and the HKM invariant is equal to
\[
	\EH(K) := [\x'] \in \SFH(-Y(K),-\Gamma_K).
\]

To deduce a relationship between the LIMIT and LOSS invariants, we begin by modifying the Heegaard diagrams $\HD'$ and $\HD$ as shown in Figure~\ref{fig:hkm_lossB1} to obtain new diagrams denoted by by $\THD'$ and $\THD$ respectively.  

\begin{remark}
Figure~\ref{fig:hkm_lossB1} depicts a pair of Heegaard diagrams one can use to compute the LOSS and HKM Legendrian knot invariants. These invariants each live in the homology of an appropriate manifold with reversed orientation.  This ambient orientation reversal is typically effectuated on the level of Heegaard diagrams by reversing the roles of the $\alpha$ and $\beta$-curves.  That is, exchanging the Heegaard diagram $(\Sigma,\balpha,\bbeta)$ for $(\Sigma,\bbeta,\balpha)$.  This same orientation reversal can be accomplished by exchanging all $\alpha$-curves for $\beta$-curves and vice-versa, while retaining the usual diagrammatic ordering $(\Sigma,\balpha,\bbeta)$.  This second convention conforms more naturally with preexisting conventions in bordered Floer theory --- that all decompositions occur along $\alpha$-curves --- and is what we adopt in the discussion to follow.
\end{remark}

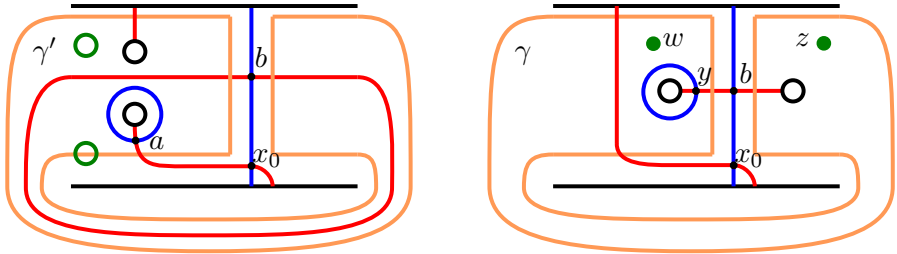
\begin{figure}[htbp]
	\centering
	\begin{picture}(155,100)
		\put(0,0){\includegraphics[scale=1]{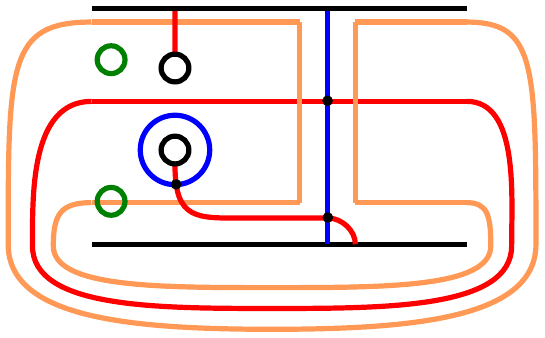}}
		\put(56,41){$a$}
		\put(96,72){$b$}
		\put(95,36){$x_0$}
		\put(12,75){$\gamma'$}
	\end{picture}
	\hspace{20px}
	\begin{picture}(155,100)
		\put(0,0){\includegraphics[scale=1]{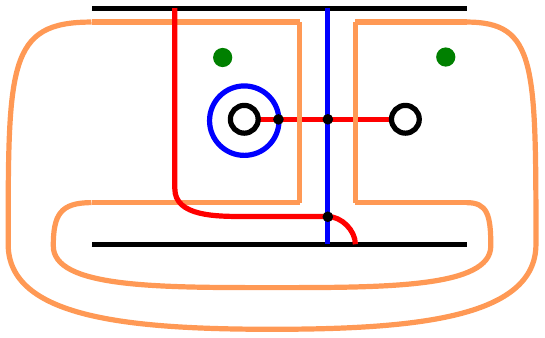}}
		\put(118,80){$z$}
		\put(68,80){$w$}
		\put(81,67){$y$}
		\put(97,66){$b$}
		\put(95,36){$x_0$}
		\put(12,75){$\gamma$}
	\end{picture}
	\caption{Modified Heegaard diagrams for the HKM diagram, left, and LOSS digram, right. (As usual the black circles are identified.)}
	\label{fig:hkm_lossB1}\label{fig:hkm_lossC1}
\end{figure}

\begin{figure}[htbp]
	\centering
	\begin{picture}(108,152)
		\put(0,0){\includegraphics[scale=1]{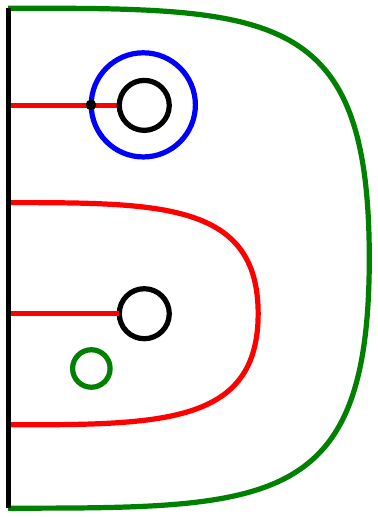}}
		\put(-7,23){\color{red}$2$}
		\put(-7,55){\color{red}$1$}
		\put(-7,85){\color{red}$2$}
		\put(-7,115){\color{red}$1$}
		\put(19,122){$a$}
		\put(-7,148){$\rho$}
	\end{picture}
	\hspace{67px}
	\begin{picture}(108,152)
		\put(0,0){\includegraphics[scale=1]{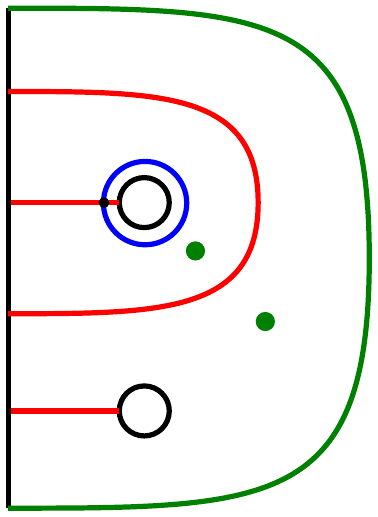}}
		\put(-7,23){\color{red}$2$}
		\put(-7,55){\color{red}$1$}
		\put(-7,87){\color{red}$2$}
		\put(-7,119){\color{red}$1$}
		\put(22,95){$y$}
		\put(57,80){$w$}
		\put(80,59){$z$}
		\put(-7,148){$\rho$}
	\end{picture}
	\caption{Bordered Heegaard diagrams for the HKM and LOSS pieces of $\HD$. (As usual the black circles are identified.)}
	\label{fig:hkm_lossB}\label{fig:hkm_lossB2}\label{fig:hkm_lossC2}
\end{figure}
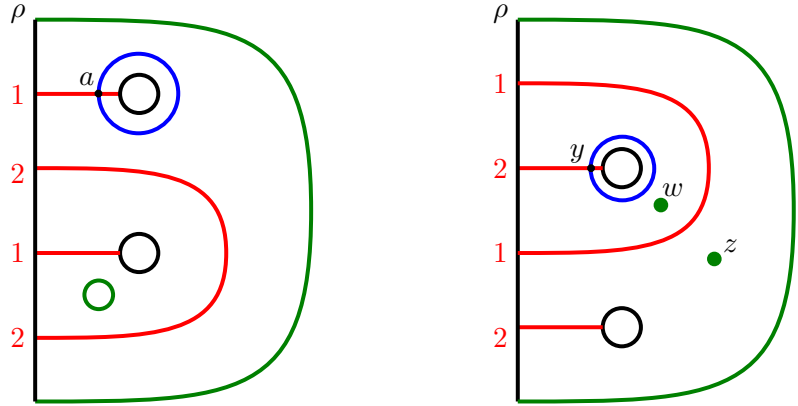

On the HKM side, the Heegaard diagram $\THD'$ is obtained from $\HD'$ by performing a pair of stabilizations.  To see this we show how to destabilize $\THD'$ to get $\HD'$. First destabilize the Heegaard diagram by erasing the two back circles, the blue circle and the red circle that runs over the removed handle corresponding to the black circles. Now for the second destabilization remove the remaining blue circle and then surger along the red circle. One may easily check the resulting Heegaard diagram is isotopic to the one on the right of Figure~\ref{fig:hkm_lossA}.

There is a canonical isomorphism on homology induced by the chain map $\CF(\HD') \to \SFC(\THD')$ which acts on generators by sending $\y \in \SFC(\HD')$ to $(a,b,\y) \in \SFC(\THD')$. Correspondingly, in this new diagram, the generator representing the HKM invariant is given by the collection of intersections $(a,b,x_1,\dots,x_n)$.

On the LOSS side, the Heegaard diagram $\THD$ is obtained form $\HD$ by performing a single  stabilization.  Again, there is a canonical isomorphism on homology induced by the chain map $\CFKM(\HD) \to \CFKM(\THD)$ which acts on generators by sending $\y \in \CFKM(\HD)$ to $(y,\y)$.  In this case, the new new generator representing the LOSS invariant is given by the collection of intersections $(y,x_0,\dots,x_n)$.

Observe that the Heegaard diagrams $\THD'$ and $\THD$ differ only in the the $C$-shaped region bounded by the orange curves $\gamma'$ and $\gamma$ shown in Figure~\ref{fig:hkm_lossC1}.  We denote the common bordered Heegaard diagram lying outside the curves $\gamma'$ and $\gamma$ by $\HD_K$.

The decomposition of $\THD'$ along $\gamma'$ is the diagrammatic equivalent to the bordered sutured decomposition 
\[
	(-Y(K),-\Gamma_K) = (-Y(K),\Gamma',\SF_T) \cup \T_0
\]
used to establish Theorem~\ref{thm:lim_to_minus} and discussed in detail in Section~\ref{sub:the_geometric_setup}.  It corresponds to the removal of a $T \times I$ neighborhood of $\partial (-Y(K))$ from the sutured manifold $(-Y(K),-\Gamma_K)$. 

The bordered sutured Heegaard diagram for the portion of $\THD'$ contained within $\gamma'$ is depicted on the left hand side of Figure~\ref{fig:hkm_lossB2}.  It is identical to the bordered sutured diagram shown in Figure~\ref{fig:torfilldiag} for the space $\T_0$.

In a similar spirit, decomposing the diagram $\THD$ along the orange curve $\gamma$ corresponds to excising a tubular neighborhood $\nu(K)$ of the knot $K$ from the 3-manifold $Y$.  In this case, however, the portion of $\THD$ contained within the curve $\gamma$ forms a doubly pointed bordered Heegaard diagram for the core curve ``$K$'' of the solid torus neighborhood $\nu(K)$ (see Figure~\ref{fig:hkm_lossC2}).

As noted above, the Heegaard diagrams $\THD'$ and $\THD$ which compute the invariants HKM and LOSS invariants have been specially constructed to agree outside the curves $\gamma'$ and $\gamma$ respectively.  This construction allows us to track the image of the HKM invariant under the gluing maps induced by negative stabilization, and, ultimately, the image of $\EHL(K)$ under the isomorphism given by Theorem~\ref{thm:lim_to_minus} identifying $\SFHL(-Y,K)$ with $\HFKM(-Y,K)$.  

In Section~\ref{sub:computation_of_gluing_maps}, we performed a detailed computation of the HKM gluing maps induced by stabilization on the type $D$ modules $K_n$.  Proving Theorem~\ref{thm:lim_to_minus_leg} requires that we also understand at least part of the parallel story on the type $A$ side.  Lemma~\ref{lem:leg_neg_stab1} below computes the portion of the type $A$ module $\BSA(\HD_K)_{\A(\W_T)}$ needed to to establish Theorem~\ref{thm:lim_to_minus_leg}.

\begin{lemma}\label{lem:leg_neg_stab1}
	In the type $A$ module $\BSA(\HD_K)_{\A(\W_T)}$ we have the following operations
\begin{align*}
	m_2((b,x_1,\dots,x_n),\pi_3) &= (x_0,x_1,\dots,x_n)\\
	m_3((b,x_1,\dots,x_n),\pi_3,\pi_2) &= 0
\end{align*}
\end{lemma}

\begin{proof}[Proof of Lemma~\ref{lem:leg_neg_stab1}]
	In order to contribute a term to $m_2((b,x_1,\dots,x_n),\pi_3)$ a necessary condition is that the corresponding domain must have multiplicity zero in the regions bordering the Reeb chords $\pi_1$, and $\pi_2$ and multiplicity one in the region bordering the Reeb chord $\pi_3$.

This fact has two important consequences.  First, it forces the region $D$ southwest of the intersection $b$ and northwest of the intersection $x_0$ in Figure~\ref{fig:hkm_lossB1} to have multiplicity one, and that all other regions of bordering $\gamma'$ have multiplicity zero.  This, in turn implies that the intersection points $x_1,\dots,x_n$ are all fixed under the operation $m_2((b,x_1,\dots,x_n),\pi_3)$. 

It follows that
\[
	m_2((b,x_1,\dots,x_n),\pi_3) = (x_0,x_1,\dots,x_n),
\]
with the sole nontrivial term corresponding to a source $S$ which is topologically a rectangle with one edge mapping to $\pi_3$.

In a similar spirit, for a domain $D$ to contribute non-trivially to $m_3((b,x_1,\dots,x_n),\pi_3,\pi_2)$, it must be the case that $D$ have multiplicity one in the regions bordering the Reeb cords $\pi_3$ and $\pi_2$ and multiplicity zero in the region bordering the Reeb cord $\pi_1$ and in the region bordering the suture.

From this observation, we can read off the multiplicities of the regions surrounding the intersection point $b$.  Beginning with the northeast region and moving clockwise, these multiplicities are $0$, $0$, $1$, and $1$, respectively.  It follows from this multiplicity calculation that the intersection $b$ is fixed under the operation $m_3((b,x_1,\dots,x_n),\pi_3,\pi_2)$.

Denote by $A$, $B$, $C$ and $D$ the multiplicities of the regions surrounding the intersection point $x_0$, beginning with the northeast region and listed in clockwise order.  Since the intersection point $x_0$ and $b$ lie on a common $\beta$-curve, and since $b$ is must be fixed by $m_3((b,x_1,\dots,x_n),\pi_3,\pi_2)$, the multiplicities of the regions surround $x_0$ satisfy the relation
\[
	A+C = B+D,
\]
with each of $A$, $B$, $C$ and $D$ non-negative by positivity of intersection.  From the paragraph above, however, we know that $A = 0$, $C = 0$ and $D = 1$.  Thus, $B = -1$, contradicting positivity of intersection.

From the above, we conclude that no such (positive) domain exists, and that
\[
	m_3((b,x_1,\dots,x_n),\pi_3,\pi_2) = 0
\]
\end{proof}

We are now ready to proceed with the proof of Theorem~\ref{thm:lim_to_minus_leg}.  

\begin{proof}[Proof of Theorem~\ref{thm:lim_to_minus_leg}]\label{pf:lim_min_leg}

As discussed above, the HKM Legendrian invariant $\EH(L)$ is represented by the generator
\[
	(b,x_1,\dots,x_n) \otimes a \in \CFA(\mscr{H_1},{\bf \alpha}, {\bf \beta}) \boxtimes K_0.
\]

Under the map $\id \boxtimes \eta_{n,0}$ induced by the first negative stabilization, we have by Lemmas~\ref{lem:twist_bptor} and~\ref{lem:leg_neg_stab1} that
\begin{align*}
	\id \boxtimes \eta_{n,0}((b,x_1,\dots,x_n) \otimes a) &= m_2((b,x_1,\dots,x_n),\pi_3) \otimes b_1\\
	& \;\;\;\;\;\;\; + m_{3}(b,x_1,\dots,x_n),\pi_3,\pi_2) \otimes a\\
	&= (x_0,x_1,\dots,x_n) \otimes b_1.
\end{align*}

Continuing, we have that for each integer $i \geq 1$,
\begin{align*}
	\id \boxtimes \eta_{n,i}((x_0,x_1,\dots,x_n) \otimes b_i) &= \sum_{k=0}^i m_{2+k}((x_0,x_1,\dots,x_n),I_1,\pi_{23},\dots,\pi_{23}) \otimes b_{i-k+1}\\
	& \;\;\;\;\;\;\; + m_{2+i+1}((x_0,x_1,\dots,x_n),I_1,\pi_{23},\dots,\pi_{23},\pi_2) \otimes a\\
	&= m_2((x_0,x_1,\dots,x_n),I_1) \otimes b_{i+1}\\
	&= (x_0,x_1,\dots,x_n) \otimes b_{i+1}.
\end{align*}

Thus, under the isomorphism given by Theorem~\ref{thm:lim_to_minus}, the LIMIT invariant $\EHL(K)$ is identified with the class $[(x_0,\dots,x_n) \otimes y]$.  Since this is, by construction, the LOSS invariant, Theorem~\ref{thm:lim_to_minus_leg} follows. 

\end{proof}


\section{The Stipsicz-V\'ertesi Attachment Map and Sutured Limit Homology} 
\label{sec:SV_limit}

In this section, we prove Theorem~\ref{thm:SV_map} --- that the map induced on sutured limit homology by the SV attachment is equivalent under the identification given by Theorem~\ref{thm:lim_to_minus} to the map 
\[
	\HFKM(-Y,K) \to \HFKH(-Y,K)
\]
induced on knot Floer homology by setting the formal variable $U$ equal to zero at the chain level. Our proof of Theorem~\ref{thm:SV_map} is similar to that of Theorem~\ref{thm:lim_to_minus} in the sense that local computations of the HKM gluing maps can be utilized to deduce the desired global result.

Let $K \subset (Y,\xi)$ be a null-homologous Legendrian knot.  Recall that the SV attachment is given by gluing an appropriately signed basic slice to the complement $(Y(K),\xi_K)$.  As in the case of positive or negative Legendrian stabilization, the SV attachment can be effectuated through bypass attachment, as shown in Figure~\ref{fig:sv_attach}.  

\begin{figure}[htbp]
	\centering
		\includegraphics[scale=1]{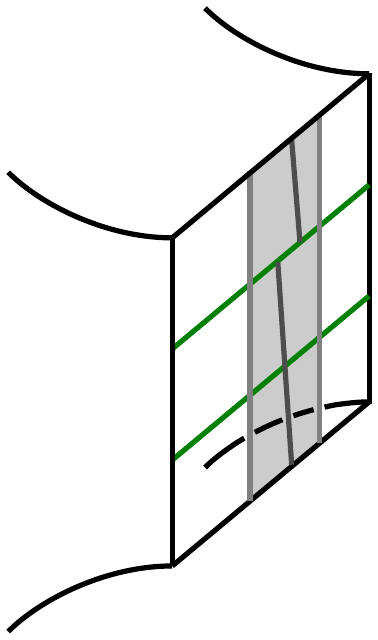}
	\hspace{20px}
		\includegraphics[scale=1]{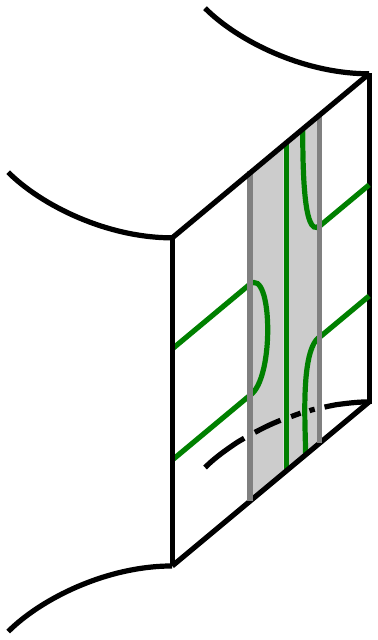}
	\caption{The bypass attachment arc that realizes the Stipsicz-V\'ertesi basic slice attachment. The arc and dividing curves before the attachment are shown on the left hand side and the result on the level of dividing curves is shown on the right hand side. }
	\label{fig:sv_attach}\label{fig:sv_attach_before}\label{fig:sv_attach_after}
\end{figure}

Reversing orientation, the left hand side of Figure~\ref{fig:sv_attach_before} depicts the sutured boundary of the complement $(-Y(K),-\Gamma_K)$.  The attaching curve for the SV bypass attachment is shown in dark grey and is lies within a vertical annulus (meridional when measured with respect to the knot $K$).  To ensure compatibility with negative Legendrian stabilization, the endpoints of the SV bypass attachment are chosen to lie on the dividing curve shown.

The contact 3-manifold which results from the SV bypass attachment is shown on the right hand side of Figure~\ref{fig:sv_attach_after}.  As a sutured manifold, this space is equal to $(-Y(K),-\Gamma_\mu)$.  It is obtained from $Y$ by removing an open tubular neighborhood of the knot $K$ and placing two parallel meridional sutures along its torus boundary.

Since the SV bypass attachment can be performed within a vertical annulus, we may apply the techniques used in Section~\ref{sec:limit_is_knot} to establish Theorem~\ref{thm:SV_map}.  First, we decompose the sutured 3-manifolds $(-Y(K),-\Gamma_n)$ as
\[
	(-Y(K),-\Gamma_n) = (-Y(K),\Gamma',\SF_T) \cup \T \cup \C_n \cup \A_0,
\]
where $\T$, $\C_n$ and $\A_0$ are discussed in Sections~\ref{sub:torus_bmodules}, \ref{sub:twisted_annulus} and~\ref{sub:simp_ann}, respectively.
Next, Lemmas \ref{lem:SV_ann_map}, \ref{lem:SV_twistann_map} and \ref{lem:SV_twisttor_map}, compute the HKM gluing maps on bordered sutured Floer homology induced by the SV attachment performed on the spaces $\A_0$, $\A_n = \C_n \cup \A_0$ and $\T_n = \T \cup \A_n$ respectively. Finally, in Lemma~\ref{lem:SV_twisttor_map} below we deduce the HKM gluing map induced on the limit module $\underrightarrow{K}$, which is then seen to agree with the map given by setting the formal variable $U$ equal to zero under the identification between $\underrightarrow{K}$ and the module $K$ which gives rise to $\HFKM(-Y,K)$.

We now recall from Section~\ref{sub:bordered_invariants_and_knot_floer_homology} the definitions of the Type-$D$ modules $K^-$ and $K_\infty$, associated to the doubly-pointed solid torus, which computes the minus and hat variants of knot Floer homology.  The module $K^-$ is given by
\begin{center}
\begin{tikzpicture}	[->,>=stealth',auto,thick]
	\node (a) at (0,0){$x$} ;
	\node (b) at (2.5,0) {$U \cdot x$} ;
	\node (c) at (5,0) {$U^2 \cdot x$} ;
	\node (d) at (7.5,0) {$U^3 \cdot x$} ;
	\node (e) at (10,0) {$\dots$,} ;
	
	\draw (a) edge node[above] {$\rho_{23}$} (b);
	\draw (b) edge node[above] {$\rho_{23}$} (c);
	\draw (c) edge node[above] {$\rho_{23}$} (d);
	\draw (d) edge node[above] {$\rho_{23}$} (e);
\end{tikzpicture}
\end{center}
where each of the $U^i \cdot x$ live in the idempotent $I_1$. The module $K_\infty$ is generated by the single element $x$, which lives in idempotent $I_1$ and satisfies $\delta(x) = 0$.  

At the level of Type-$D$ modules, the natural map $\HFKM(Y,K) \to \HFKH(Y,K)$ given by setting $U$ equal to zero at the chain level is given by
\begin{align*}
	K^- &\to K_\infty\\
	x &\mapsto x\\
	U^i \cdot x &\mapsto 0, \;\;\; i \geq 1.
\end{align*}
Equivalently, in the language of sutured limit homology, under the isomorphism identifying the Type-$D$ modules $\underrightarrow{K}$ and $K$, this map is given by
\begin{align*}
	\underrightarrow{K} &\to K_\infty\\
	\delta_0 &\mapsto x\\
	\delta_i &\mapsto 0, \;\;\; i \geq 1.
\end{align*}

Following the strategy discussed several paragraphs above, we now study the map on bordered sutured Floer homology induced by attaching a SV bypass to the space $\A_0$.  Restricted to this space, the SV bypass attachment is depicted on the left hand side of Figure~\ref{fig:sv_attach_ann}.  The arc of attachment is shown in grey.  The bordered sutured manifold which results from this attachment is denoted by $\A_\infty$ and depicted in the middle of Figure~\ref{fig:sv_attach_ann_res}.  A corresponding bordered sutured Heegaard diagram for the space $\A_\infty$ is show on the right hand side of Figure~\ref{fig:sv_attach_ann_diag}.

\begin{figure}[htbp]
	\centering
		\includegraphics[scale=1]{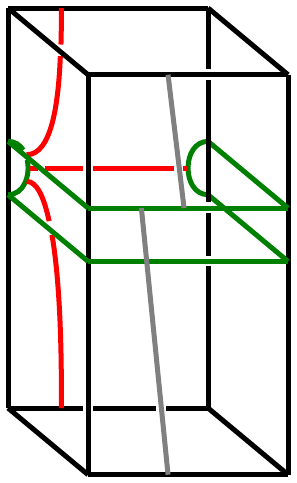}
	\hspace{20px}
		\includegraphics[scale=1]{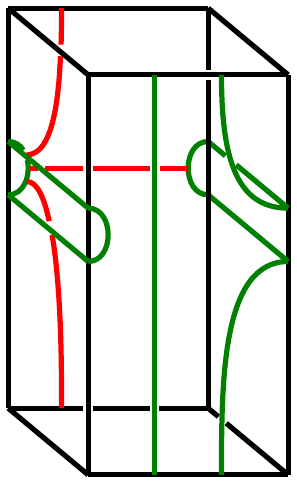}
	\hspace{20px}
	\begin{picture}(72,137)
		\put(5,0){\includegraphics[scale=1]{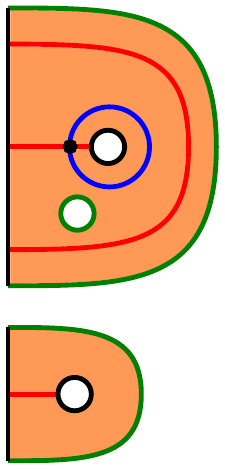}}
		\put(0,121){\color{red}$1$}
		\put(0,90){\color{red}$2$}
		\put(0,60){\color{red}$1$}
		\put(0,18){\color{red}$2$}
		\put(0,135){$\rho$}
		\put(15,96){$w$}
	\end{picture}
	\caption{On the left hand side is the SV bypass attachment viewed on $\A_0$. The middle figure is the space $\A_\infty$ which results from attaching a SV bypass to $\A_n$. On the right hand side is a bordered sutured Heegaard diagram for the space $\A_\infty$. (As usual the black circles are identified.)}
	\label{fig:sv_attach_mfd_diag}\label{fig:sv_attach_ann}\label{fig:sv_attach_ann_diag}\label{fig:sv_attach_ann_res}
\end{figure}
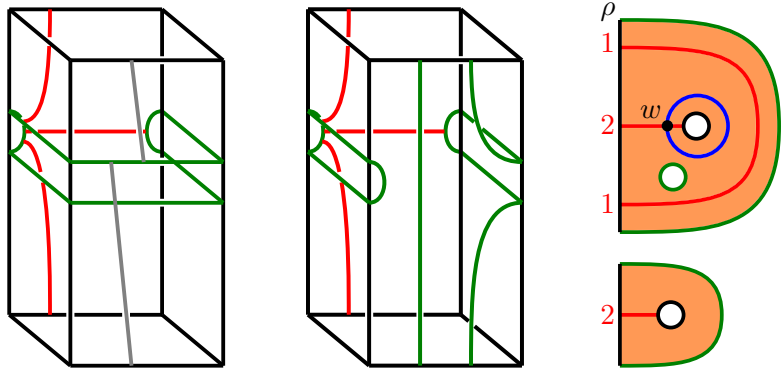

Using the conventions already set forth in Section~\ref{sub:comp_mod}, we associate to the diagram shown on the right hand side of Figure~\ref{fig:sv_attach_ann_diag} the Type-$D$ module $M_\infty := {}^{\A(\W_A)}\BSD(\A_\infty)$, which is defined over the strand algebra $\A(\W_A)$.  The module $M_\infty$ is generated by the single element $w$, whose idempotent compatibility is given by
\[
	I_1 \cdot w = w.
\]
The corresponding boundary map $\delta$ is trivial since all of the regions in the associated bordered sutured Heegaard diagram are adjacent to portions of the boundary which are sutured.

As discussed in Section~\ref{sub:comp_mod}, the Type-$D$ module $M_0$ associated to the space $\A_0$ is generated by a single element $a$, whose idempotent compatibility is given by $I_2 \cdot a = a$.

\begin{lemma}\label{lem:SV_ann_map}
	The map on bordered sutured Floer homology induced by the Stipsicz-V\'ertesi attachment to the thickened annulus $\A_0$ of slope zero is given by the following equation:
\begin{align*}
	\phi_{SV}: M_0 &\to M_\infty\\
	\nonumber a &\mapsto \rho_2' \otimes w
\end{align*}
\end{lemma}

\begin{proof}\label{pf:ann_sv}
	Observe that the HKM map induced on Type-$D$ modules by the SV bypass attachment must be non-trivial.  This follows, for instance, from the fact that there exist Legendrian knots whose LOSS hat invariants are nontrivial --- the Legendrian unknot with maximal Thurston-Bennequin invariant is such a knot.

It is now elementary to check that the unique non-trivial map $M_0 \to M_\infty$ of Type-$D$ modules is precisely the map $\phi_{SV}$ given in the statement of Lemma~\ref{lem:SV_ann_map}.
\end{proof}

Next, we compute the HKM gluing map induced by attaching a SV bypass to the spaces $\A_n = \C_m \cup \A_0$.

\begin{lemma}\label{lem:SV_twistann_map}
	The map on bordered sutured Floer homology induced by the Stipsicz-V\'ertesi attachment to the thickened annulus $\A_m$ of slope $m$ is given by the following equation:
\begin{align*}
\nonumber	\phi_{SV}: M_m &\to M_\infty\\
	b_n &\mapsto I_1 \otimes w\\
\nonumber b_i &\mapsto 0\qquad\qquad \text{for $i<n$}\\
\nonumber a &\mapsto 0
\end{align*}
\end{lemma}

\begin{proof}\label{pf:SV_twistann_map}
	This follows by a computation which similar to those given in the proofs of Lemmas \ref{lem:twist_bpann} and \ref{lem:twist_bptor}.  In this case, the proof centers around the following key diagram.
\begin{center}
\begin{tikzpicture}[->,>=stealth',shorten >=1pt,auto,node distance=2.8cm, semithick]
	\node  (A) at (0,0) {$M_m = C_m \boxtimes M_0$};
	\node  (B) at (5.5,0) {$C_m \boxtimes M_\infty = M_\infty$};

	\path (A) edge node {\small $\id_{C_m} \boxtimes \phi_{SV}$} (B);
\end{tikzpicture}
\end{center}

As before, the canonical identification between $C_m \boxtimes M_0$ and $M_m$ is given by 
\begin{align*}
	a \otimes a &= a\\
	b_i \otimes a &= b_i
\end{align*}
The identification between $C_m \boxtimes M_\infty$ and $M_\infty$ is given by
\[
	c \otimes w = w.
\]

Lemma~\ref{lem:SV_twistann_map} now follows immediately from these identifications and Definition~\ref{def:boxmap} applied to the map $\I_{C_m} \boxtimes \phi_{SV}$.
\end{proof}

We now compute the HKM gluing map induced by attaching a SV bypass to the spaces $\T_n = \T \cup \A_m$.

\begin{lemma}\label{lem:SV_twisttor_map}
	The map on bordered sutured Floer homology induced by the Stipsicz-V\'ertesi attachment to the thickened punctured torus $\T_m$ of slope $m$ is given by the following equation:
\begin{align*}
\nonumber	\phi_{SV}: K_n &\to K_\infty\\
	b_n &\mapsto I_1 \otimes x\\
\nonumber b_i &\mapsto 0\\
\nonumber a &\mapsto 0
\end{align*}
\end{lemma}

\begin{proof}\label{pf:SV_twisttor_map}
	This follows from a computation similar to that given in the proof of Lemma~\ref{lem:SV_twistann_map}.  In this case, the computation centers around the following key diagram.
\begin{center}
\begin{tikzpicture}[->,>=stealth',shorten >=1pt,auto,node distance=2.8cm, semithick]
	\node  (A) at (0,0) {$K_m = N \boxtimes M_m$};
	\node  (B) at (5.5,0) {$N \boxtimes M_\infty = K_\infty$};

	\path (A) edge node {\small $\id_N \boxtimes \phi_{SV}$} (B);
\end{tikzpicture}
\end{center}

We leave the remaining details as an elementary exercise to the reader.
\end{proof}

We are now in position to prove Theorem~\ref{thm:SV_map}.

\begin{proof}[Proof of Theorem~\ref{thm:SV_map}]\label{pf:SV_map}
	From Lemma~\ref{lem:SV_twisttor_map}, we see that the HKM gluing map induced on the limit module $\underrightarrow{K}$ by attaching a SV bypass is given by: 
\begin{align*}
	\Phi_{SV}: \underrightarrow{K} &\to K_\infty\\
	\delta_0 &\mapsto I_1 \otimes x\\
	\delta_i &\mapsto 0, \;\;\; i \geq 1.
\end{align*}

As discussed earlier in this section, under the identification between $\SFHL(-Y,K)$ and $\HFKM(-Y,K)$ given by Theorem~\ref{thm:lim_to_minus}, this is precisely the analogue at the level of Type-$D$ modules of the natural map $\HFKM(-Y,K) \to \HFKH(-Y,K)$ induced by setting the formal variable $U$ equal to zero at the chain level.
\end{proof}


\section{Two handle attachments and the proof of Theorem~\ref{thm:2handle}} 
\label{sec:proof_lim_hfh}

We now prove Theorem~\ref{thm:2handle}.  Recall that this theorem states that the HKM gluing map
\[
	\Phi_{2h}: \SFHL(-Y,K) \to \HFH(-Y),
\]
which is induced by meridional 2-handle attachment is equivalent to the map
\[
	\HFKM(-Y,K) \to \HFH(-Y),
\]
which is given by setting the formal variable $U$ equal to the identity at the chain level.  

Our proof of Theorem~\ref{thm:2handle} is substantially similar to that of Theorem~\ref{thm:SV_map} given in Section~\ref{sec:SV_limit}.  As before, we begin by decomposing the sutured 3-manifolds $(-Y(K),-\Gamma_n)$ as
\[
	(-Y(K),-\Gamma_n) = (-Y(K),\Gamma',\SF_T) \cup \T \cup \C_n \cup \A_0,
\]
where $\T$, $\C_n$ and $\A_0$ are discussed in Sections~\ref{sub:torus_bmodules}, \ref{sub:twisted_annulus} and~\ref{sub:simp_ann}, respectively.
Lemmas~\ref{lem:2han_ann0}, \ref{lem:2han_ann} and~\ref{lem:2han_tor} compute the HKM gluing maps induced by performing meridional contact 2-handle attachments on the spaces $\A_0$, $\A_n = \C_n \cup \A_0$ and $\T_n = \T \cup \A_n$ respectively. From Lemma~\ref{lem:2han_tor}, we are then able to compute the gluing map induced on the limit module $\underrightarrow{K}$, which we show agrees with the map given by setting the formal variable $U$ equal to the identity under the identification between $\underrightarrow{K}$ and the module $K$ which gives rise to $\HFKM(-Y,K)$.

The result of attaching a meridional contact 2-handle to the space $(Y(K),\Gamma_K)$ and rounding edges is depicted in Figure~\ref{fig:2handle2}.  The new space has a single convex boundary component which is a 2-sphere containing a dividing curve.  This space is equal to $Y(1)$ as a sutured manifold (see \cite{Ju}).  

\begin{figure}[htbp]
	\centering
		\includegraphics[scale=1]{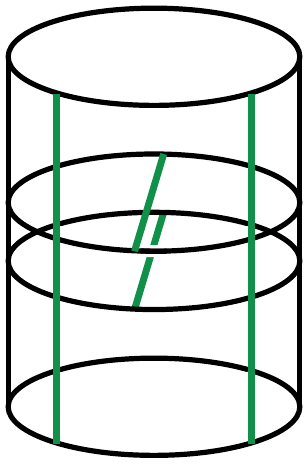}
	\hspace{20px}
		\includegraphics[scale=1]{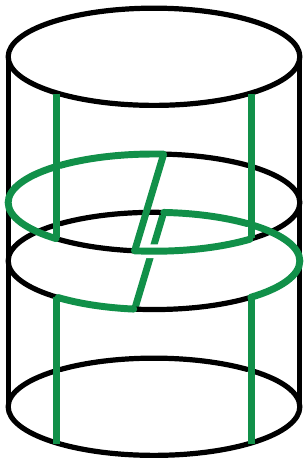}
	\caption{The contact $2$-handle attachment. The manifold $Y(K)$ is outside the torus in the figure. The 2-handle attached along a meridian is the $D^2\times [0,1]$ shown inside the torus. On the right hand side the dividing curves are shown after corners have been rounded.}
	\label{fig:2handle_attach}\label{fig:2handle1}\label{fig:2handle2}
\end{figure}

Let $\TF$ be $\T_0$ with a 2-handle attached to along a meridian (that is along the grey annulus in Figure~\ref{fig:torfill}). This is a bordered sutured solid torus and is shown on the left hand side of Figure~\ref{fig:torus_fill}. We obtain the space $Y(1)$ form the bordered sutured $(Y(K),\Gamma',\SF_T)$ by attaching $\TF$.  The Heegaard diagram for $\TF$ is shown in Figure~\ref{fig:torus_fill2} and it is equivalent to that depicted in Figure~\ref{fig:minusdiag}, but with second basepoint $w$ removed and the first basepoint $z$ incorporated into the boundary.

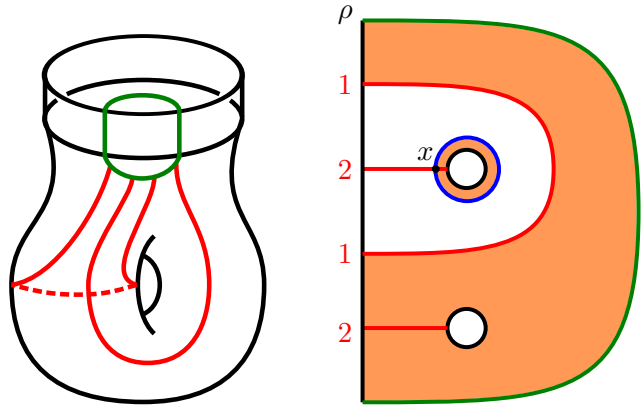
\begin{figure}[htbp]
	\centering
		\includegraphics[scale=1]{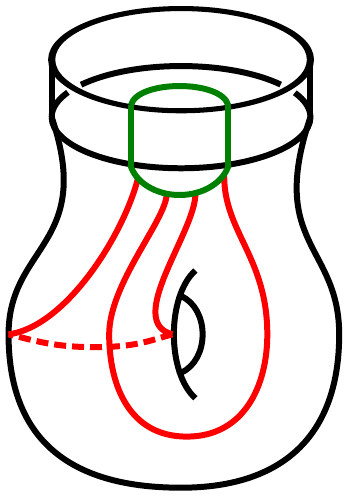}
	\hspace{25px}
	\begin{picture}(113,154)
		\put(0,0){\includegraphics[scale=1]{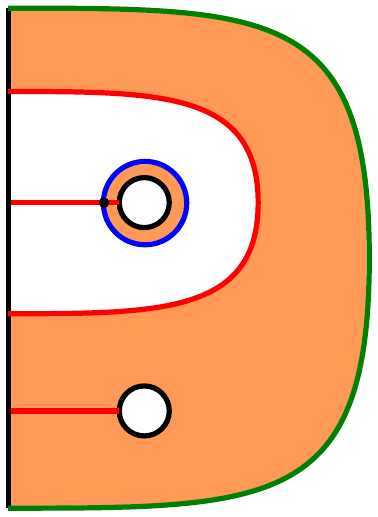}}
		\put(-7,25){\color{red}$2$}
		\put(-7,55){\color{red}$1$}
		\put(-7,87){\color{red}$2$}
		\put(-7,119){\color{red}$1$}
		\put(-7,148){$\rho$}
		\put(23,94){$x$}
	\end{picture}
	\caption{The bordered sutured manifold $\TF$, left, and its Heegaard diagram, right. (As usual the black circles are identified.)}
	\label{fig:torus_fill}\label{fig:torus_fill1}\label{fig:torus_fill2}
\end{figure}

To the diagram in Figure~\ref{fig:torus_fill}, we associate the type $D$ module $K_{\mathrm{fill}} := {}^{\A(\W_T)}\BSD(\T_{\mathrm{fill}})$, defined over the strand algebra $\A(\W_T)$.  The module $K_{\mathrm{fill}}$ is generated by the single element $x$, whose idempotent compatibility is given by
\[
	I_1 \cdot x = x.
\]
The corresponding boundary map is given by
\begin{equation*}
	\delta(x) = \rho_{23}' \otimes x.
\end{equation*}

To prove Theorem~\ref{thm:2handle}, we must understand the bordered sutured analogues of both the contact $2$-handle, and the corresponding gluing map induced on (bordered) sutured Floer homology.  The left hand side of Figure~\ref{fig:2handle4} depicts the bordered analogue $\D_{2h} = (D^2,\Gamma_{2h},-\W_A)$ of a contact 2-handle.  In the bordered world, the act of ``attaching'' a 2-handle is, locally, given by exchanging the annular bordered sutured manifold $\A_0$ for $\D_{2h}$.

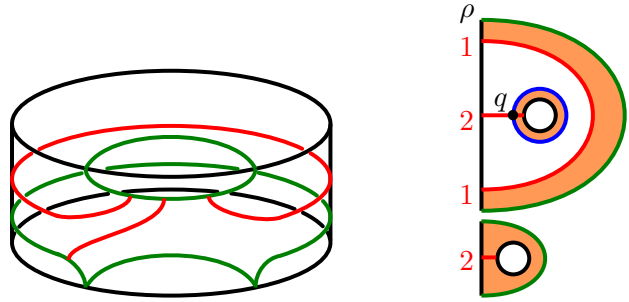
\begin{figure}[htbp]
	\centering
		\includegraphics[scale=1]{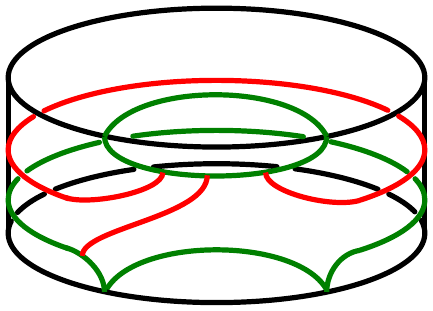}
	\hspace{30px}
	\begin{picture}(88,110)
		\put(15,0){\includegraphics[scale=1]{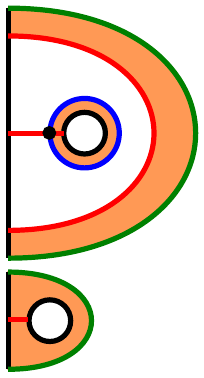}}
		\put(9,12){\color{red}$2$}
		\put(9,37){\color{red}$1$}
		\put(9,65){\color{red}$2$}
		\put(9,93){\color{red}$1$}
		\put(22,74){$q$}
		\put(9,108){$\rho$}
	\end{picture}
	\caption{A 2-handle written as a bordered sutured manifold, left, and its Heegaard diagram, right. (As usual the black circles are identified.)}
	\label{fig:2handle}\label{fig:2handle4}\label{fig:2handle5}
\end{figure}

To the diagram in Figure~\ref{fig:2handle5}, we associate the type $D$ module $K_{2h}:= {}^{\A(\W_A)} \BSD(\D_{2h})$. The module $K_{2h}$ is generated by the single intersection $q$, with idempotent compatibility given by
\[
	I_1 \cdot q = q.
\]
The corresponding boundary map is given by the equation
\begin{equation*}
	\delta(q) = \rho_{12}' \otimes q
\end{equation*}

As in the case of bypass attachment above, we can apply the third author's equivalence of gluing maps result from \cite{Za3} to compute the map induced by 2-handle attachment on bordered sutured Floer homology.  In this instance, the bordered analogue of attaching such a handle is locally given by exchanging the space $\A_0$ for $\D_{2h}$.  The third author's result, coupled with known non-vanishing results for  invariants of contact structures,  implies that the HKM gluing map must be given by \emph{some} nonzero map $\phi_{2h}$, connecting the Type-$D$ modules associated to $\A_0$ and $\D_{2h}$ respectively.  As there is a unique such non-trivial map, we have established the following lemma.

\begin{lemma}\label{lem:2han_ann0}
	The map on bordered sutured Floer homology induced by meridianal  
	 2-handle attachment to the thickened annulus $\A_0$ is given by the following equation:
\begin{align*}
	\phi_{2h}: M_0 &\to K_{2h} \\
	\nonumber a &\mapsto \rho_2' \otimes q
\end{align*}\qed
\end{lemma}

Recall that the map $\Phi_{2h}$ is induced by the collection of gluing maps $\{\phi_{2h}\}$ which come from attaching a contact $2$-handle to a meridional curve on the convex boundary of $(Y(K),\xi_{K_n})$, as depicted in Figure~\ref{fig:2handle}. The collection of constituent gluing maps defining $\Phi_{2h}$ are computed in precisely the same manor as those defining $\Phi_{SV}$ in Section~\ref{sec:SV_limit}.  

Observe that, topologically, attaching a contact 2-handle to any of the spaces $\A_n$ results in a copy of $\D_{2h}$.  Similarly, as described above, attaching a meridional contact 2-handle to any of the spaces $\T_n$ results in the space $\TF$.  Lemmas~\ref{lem:2han_ann} and~\ref{lem:2han_tor} below compute the corresponding gluing maps and are the analogues of Lemmas~\ref{lem:SV_twistann_map} and~\ref{lem:SV_twisttor_map} from Section~\ref{sec:SV_limit} respectively.  Because the computations are so similar, we leave them as exercises to the interested reader.

\begin{lemma}\label{lem:2han_ann}
	The map on bordered sutured Floer homology induced by attaching a contact 2-handle to the thickened annulus $\A_n$ of slope $n$ is given by the following equation:	
\begin{align*}
	\phi_{2h}: M_n &\to K_{2h}\\
	\nonumber b_k &\mapsto I_1 \otimes q\\
	\nonumber a &\mapsto \rho_2' \otimes q
\end{align*}
\qed
\end{lemma}

\begin{lemma}\label{lem:2han_tor}
	The map on bordered sutured Floer homology induced by attaching a contact 2-handle to the thickened punctured torus $\T_n$ of slope $n$ is given by the following equation:	
\begin{align*}
	\phi_{2h}: K_n &\to \KF\\
	\nonumber b_k &\mapsto I_1 \otimes x\\
	\nonumber a &\mapsto \rho_3' \otimes x
\end{align*}
\qed
\end{lemma}

\begin{proof}[Proof of Theorem~\ref{thm:SV_map}]\label{pf:lim_hat_closed}
	From Lemma~\ref{lem:2han_tor}, it follows that the map induced by contact 2-handle attachment on the limit type $D$ module $\underrightarrow{K}$ is given by
\begin{align*}
	\Phi_{2h}: \underrightarrow{K} &\to \KF\\
	\nonumber \delta_i &\mapsto I_1 \otimes x
\end{align*}

It follows immediately that, under the identification 
\begin{align}
	\nonumber \underrightarrow{K} &\to K^-\\
	\nonumber \delta_i &\mapsto U^i \cdot x
\end{align}
between $\underrightarrow{K}$ and $K$, the Type-$D$ module computing $\HFKM$, the map $\Phi_{2h}$ agrees with that induced by setting the formal variable $U$ equal to the identity at the chain level.
\end{proof}


\section{Inverse Limit Invariants and Knot Floer Homology} 
\label{sec:plus}

In this section, we briefly discuss methods for establishing the results detailed in Section~\ref{sub:top_inverselimits} concerning the sutured inverse limit invariants.  Generally speaking, proofs of these theorems are simple translates of the corresponding results in the direct limit setting, and we therefore leave them as straightforward exercise for the interested reader.

\subsection{Identifying Invariants} 
\label{sub:identifying_invts}

Here we present an outline of the proof of Theorem~\ref{thm:lim_to_plus}.  Recall that this theorem states that there exists an isomorphism of $\F[U]$-modules.
\[
	I_+: \SFHIL(-Y,K) \to \HFKP(-Y,K).
\]

To obtain this result, we adopt the same general strategy used to prove Theorem~\ref{thm:lim_to_minus} in Section~\ref{sec:limit_is_knot}.

Let $\{(Y(K),\Gamma^+_i)\}$ be the collection of sutured manifolds which are obtained subsets of the longitudinal completion $(Y(K),\Gamma_\lambda) = (Y(K),\Gamma_0)$, as described in Section~\ref{sub:top_inverselimits}.  Recall that for $j>i$, we have an inclusion $(Y(K),\Gamma^+_j) \subset (Y(K),\Gamma^+_i)$, and that each of the differences $\overline{(Y(K),\Gamma^+_i) \backslash (Y(K),\Gamma^+_{i+1})}$ can naturally be given the structure of a basic slice.

Consistently choosing the ``negative'' sign for each of the above basic slices gives rise to the cofinal sequence
\begin{center}
\begin{tikzpicture}	[<-,>=stealth',auto,thick]
	\node (a) at (0,0){$\SFH(-Y(K),-\Gamma^+_0)$} ;
	\node (b) at (4.25,0) {$\SFH(-Y(K),-\Gamma^+_1)$} ;
	\node (c) at (8.75,0) {$\SFH(-Y(K),-\Gamma^+_2)$} ;
	\node (d) at (11.75,0) {$\dots$,} ;
	
	\draw (a) edge node[above] {\small $\phi'_-$} (b);
	\draw (b) edge node[above] {\small $\phi'_-$} (c);
	\draw (c) edge node[above] {\small $\phi'_-$} (d);
\end{tikzpicture}
\end{center}
the inverse limit of which is define to the the sutured inverse limit homology $\SFHIL(-Y,K)$.

As was the case for sutured limit homology, one obtains a natural $U$-action on $\SFHIL(-Y,K)$ via positive basic slice attachment.

We begin by decomposing the spaces $(-Y(K),-\Gamma^+_n)$ as in Section~\ref{sub:the_geometric_setup}, to obtain
\[
	(-Y(K),-\Gamma^+_n) = (-Y(K),\Gamma',\SF_T) \cup \T^+_n,
\]
where $(-Y(K),\Gamma',\SF_T)$ is the knot complement and $\T^+_i$ is the bordered sutured manifold obtained from $\T_0$ (see Section~\ref{sub:the_geometric_setup}) by applying $i$ {\it positive} Dehn twists along the core curve of a meridional annulus $A$.

Using the same techniques employed in Section~\ref{sub:comp_mod}, we compute Type-$D$ modules $K_n^+ := {}^{\A(\W_T)} \BSD(\T^+_n)$.  The $K^+_n$ are generated by $\{a,b_1,\dots,b_n\}$, with idempotent comatibilities
\[
	I_2 \cdot a = a, \;\;\; I_1 \cdot b_i = b_i,
\]
and differential described graphically as
\begin{center}
\begin{tikzpicture}	[->,>=stealth',auto,thick]
	\node (a) at (0,0){$b_{n}$};
	\node (b) at (2,0) {$b_{n-1}$};
	\node (c) at (4,0) {$\dots$};
	\node (d) at (6,0) {$b_1$};
	\node (e) at (8,0) {$a.$};
	
	\draw (b) edge node[above] {\small $\rho_{23}$} (a);
	\draw (c) edge node[above] {\small $\rho_{23}$} (b);
	\draw (d) edge node[above] {\small $\rho_{23}$} (c);
	\draw (e) edge node[above] {\small $\rho_{3}$} (d);
\end{tikzpicture}
\end{center}

Analogues of Lemmas \ref{lem:basic_bpann}, \ref{lem:twist_bpann} and \ref{lem:twist_bptor}, adapted to this context, again give rise to the key diagram depicted in Figure~\ref{fig:inverse_limit}.

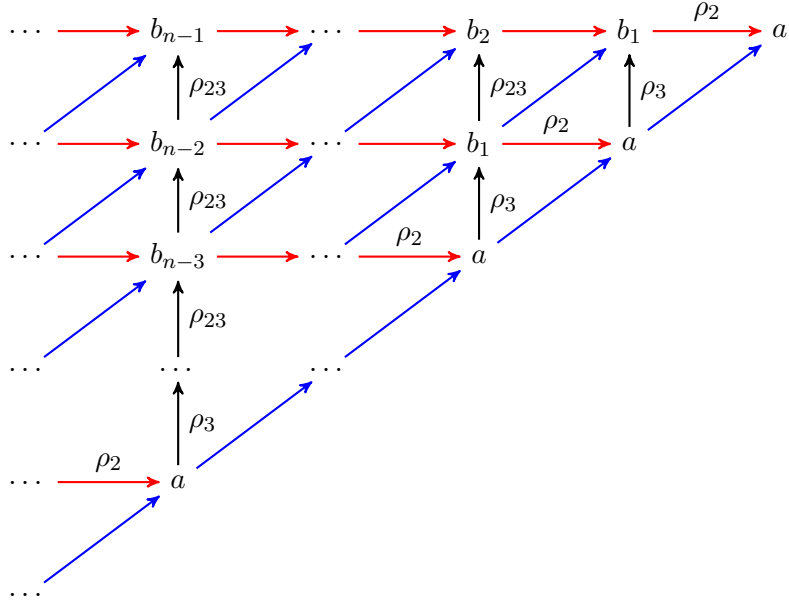
\begin{figure}[htbp]
\begin{tikzpicture}	[->,>=stealth',auto,thick]
	
	\node (a) at (10,0){$a$};
	\node (b) at (8,0) {$b_1$};
	\node (c) at (6,0) {$b_2$};
	\node (d) at (4,0) {$\dots$};
	\node (e) at (2,0) {$b_{n-1}$};
	\node (f) at (0,0) {$\dots$};
	\node (g) at (8,-1.5){$a$};
	\node (h) at (6,-1.5) {$b_1$};
	\node (i) at (4,-1.5) {$\dots$};
	\node (j) at (2,-1.5) {$b_{n-2}$};
	\node (k) at (0,-1.5) {$\dots$};
	\node (l) at (6,-3) {$a$};
	\node (m) at (4,-3) {$\dots$};
	\node (n) at (2,-3) {$b_{n-3}$};
	\node (o) at (0,-3) {$\dots$};
	\node (p) at (4,-4.5) {$\dots$};
	\node (q) at (2,-4.5) {$\dots$};
	\node (r) at (0,-4.5) {$\dots$};
	\node (s) at (2,-6) {$a$};
	\node (t) at (0,-6) {$\dots$};	
	\node (u) at (0,-7.5) {$\dots$};
		
	\draw (b) edge[red] node[black,above] {$\rho_2$} (a);
	\draw (c) edge[red] (b);
	\draw (d) edge[red] (c);
	\draw (e) edge[red] (d);
	\draw (f) edge[red] (e);
	\draw (h) edge[red] node[black,above] {$\rho_2$} (g);
	\draw (i) edge[red] (h);
	\draw (j) edge[red] (i);
	\draw (k) edge[red] (j);
	\draw (m) edge[red] node[black,above] {$\rho_2$} (l);
	\draw (n) edge[red] (m);
	\draw (o) edge[red] (n);
	\draw (t) edge[red] node[black,above] {$\rho_2$} (s);
	
	\draw (g) edge[blue] (a);
	\draw (h) edge[blue] (b);
	\draw (i) edge[blue] (c);
	\draw (j) edge[blue] (d);
	\draw (k) edge[blue] (e);
	\draw (l) edge[blue] (g);
	\draw (m) edge[blue] (h);
	\draw (n) edge[blue] (i);
	\draw (o) edge[blue] (j);
	\draw (p) edge[blue] (l);
	\draw (s) edge[blue] (p);
	\draw (r) edge[blue] (n);
	\draw (u) edge[blue] (s);

	\draw (g) edge node[right] {$\rho_3$} (b);
	\draw (h) edge node[right] {$\rho_{23}$} (c);
	\draw (j) edge node[right] {$\rho_{23}$} (e);
	\draw (l) edge node[right] {$\rho_3$} (h);
	\draw (n) edge node[right] {$\rho_{23}$} (j);
	\draw (q) edge node[right] {$\rho_{23}$} (n);
	\draw (s) edge node[right] {$\rho_3$} (q);
\end{tikzpicture}
\caption{Indirect limit diagram.}
\label{fig:inverse_limit}
\end{figure}

Figure~\ref{fig:inverse_limit} depicts the bordered sutured analogues of the HKM gluing maps induced by positive and negative basic slice attachment.  Specifically, the eastern pointing (red) arrows depict the Type-$D$ maps induced by negative basic slice attachment, while the northeastern pointing (blue) arrows depict the Type-$D$ maps induced by positive basic slice attachment.

Taking the inverse limit over the horizontal maps which correspond to negative basic attachments, we obtain the Type-$D$ module $\underleftarrow{K} := \varprojlim_n{K^+_n}$, described graphically as
\begin{center}
\begin{tikzpicture}	[->,>=stealth',auto,thick]
	\node (a) at (0,0){$\delta_0$} ;
	\node (b) at (2,0) {$\delta_1$} ;
	\node (c) at (4,0) {$\delta_2$} ;
	\node (d) at (6,0) {$\delta_3$} ;
	\node (e) at (8,0) {$\dots$.} ;
		
	\draw (b) edge node[above] {$\rho_{23}$} (a);
	\draw (c) edge node[above] {$\rho_{23}$} (b);
	\draw (d) edge node[above] {$\rho_{23}$} (c);
	\draw (e) edge node[above] {$\rho_{23}$} (d);
\end{tikzpicture}
\end{center}
We further see that the $U$-action on the Type-$D$ module $\underleftarrow{K}$ is given by left-translation, sending $\delta_i$ to $\delta_{i-1}$ for $i > 0$ and $\delta_0$ to $0$.

It immediately follows that the Type-$D$ module $\underleftarrow{K}$ is isomorphic under the below identification to the Type-$D$ module $K^+$ from Section~\ref{sub:bordered_invariants_and_knot_floer_homology} which gives rise to the plus variant of knot Floer homology.
\begin{center}
		\begin{tikzpicture}	[->,>=stealth',auto,thick]
			\node (aa) at (-.75,1.45) {$\underleftarrow{K}:=$};
			\node (bb) at (-.75,0.05) {$K^+=$};
			\node (a) at (0,1.5){$\delta_0$} ;
			\node (b) at (2.5,1.5) {$\delta_1$} ;
			\node (c) at (5,1.5) {$\delta_2$} ;
			\node (d) at (7.5,1.5) {$\delta_3$} ;
			\node (e) at (10,1.5) {$\dots$} ;
			\node (f) at (0,0){$x$} ;
			\node (g) at (2.5,0) {$U^{-1} \cdot x$} ;
			\node (h) at (5,0) {$U^{-2} \cdot x$} ;
			\node (i) at (7.5,0) {$U^{-3} \cdot x$} ;
			\node (j) at (10,0) {$\dots$} ;

			\draw (b) edge node[above] {$\rho_{23}$} (a);
			\draw (c) edge node[above] {$\rho_{23}$} (b);
			\draw (d) edge node[above] {$\rho_{23}$} (c);
			\draw (e) edge (d);
			\draw (g) edge node[above] {$\rho_{23}$} (f);
			\draw (h) edge node[above] {$\rho_{23}$} (g);
			\draw (i) edge node[above] {$\rho_{23}$} (h);
			\draw (j) edge (i);
			\draw (b) edge[bend right] node[above] {$U \cdot$} (a);
			\draw (c) edge[bend right] node[above] {$U \cdot$} (b);
			\draw (d) edge[bend right] node[above] {$U \cdot$} (c);
			\draw (e) edge[bend right] node[above] {$U \cdot$} (d);
			\draw (a) edge[<->] (f);
			\draw (b) edge[<->] (g);
			\draw (c) edge[<->] (h);
			\draw (d) edge[<->] (i);
		\end{tikzpicture}
\end{center}

Finally, on the level of sutured Floer homology, we have
\begin{align*}
	\SFHIL(-Y,K) &:= \varprojlim \SFH(-Y(K),-\Gamma^+_i)\\
		&\cong \varprojlim H_*(\BSA(-Y(K),\Gamma',\SF_T) \boxtimes \BSD(\ST^+_n))\\
		&\cong H_*(\BSA(-Y(K),\Gamma',\SF_T) \boxtimes \varprojlim \BSD(\ST^+_n))\\
		&\cong H_*(\BSA(-Y(K),\Gamma',\SF_T) \boxtimes K^+)\\
		&\cong \HFKP(-Y,K),
\end{align*}
as in the proof of Theorem~\ref{thm:lim_to_minus} in Section~\ref{sub:computation_of_gluing_maps}.

This finishes our sketch of the proof of Theorem~\ref{thm:lim_to_plus}.


\subsection{Identifying the Natural Map} 
 \label{sub:identifying_nat_map}

Here we provide a sketch of the proof of Theorem~\ref{thm:SV_map2}, which, under the isomorphism $I_+$ from Theorem~\ref{thm:lim_to_plus}, identifies the induced map
\[
	\Phi_{dSV}: \HFKH(-Y,K) \to \SFHIL(-Y,K)
\]
with the natural map on knot Floer homology
\[
	\iota_*: \HFKH(-Y,K) \to \HFKP(-Y,K)
\]
induced by the inclusion of complexes.

Recall that the map $\Phi_{dSV}$ is induced by the HKM gluing maps associated to the collection basic slice attachments $\{\widehat{A}^-_i\}$, performed along the boundary of the meridional completion $(Y(K),\Gamma_\mu)$.  

The proof of Theorem~\ref{thm:SV_map2} is similar to that of to Theorem~\ref{thm:SV_map}.  The basic idea is to decompose the relevant spaces in the usual way as unions of bordered sutured manifolds, thus localizing the associated computation to a neighborhood of the original boundary.  Specifically, we have
\[
	(-Y(K),-\Gamma_\mu) = (-Y(K),\Gamma',\SF_T) \cup \ST_\mu,
\]
and
\[
	(-Y(K),-\Gamma^+_i) = (-Y(K)\Gamma',\SF_T) \cup \ST^+_i
\]

Having decomposed the sutured manifolds as above, the goal shifts to computing the maps of Type-$D$ structures
\[
	\phi_{dSV}: K_\infty \to K^+_i
\]
induced by negative basic slice or, equivalently, negative bypass attachment.  The key lemma in proving Theorem~\ref{thm:SV_map2} is the following.

\begin{lemma}\label{lem:dSV_map}
	The map on bordered sutured Floer homology induced by the above described negative basic slice attachment $\widehat{A}^-_i$, performed along the boundary of the meridional completion, is given by the following equation:
	\begin{align*}
		\nonumber	\phi_{dSV}: K_\infty  &\to K^+_n\\
		x &\mapsto I_1 \otimes b_n\\
	\end{align*}
\end{lemma}

We leave it as an exercise for the interested reader to verify Lemma~\ref{lem:dSV_map}, with the following hint.  First, focus on the (non-trivial) map $\phi_{dSV}: K_\infty \to K^+_0$, whose target is the Type-$D$ module associated to the longitudinal completion.  This map is unique, and given by $\phi_{dSV}(x) = \rho_2 \otimes a$.  From here, the Lemma follows by iteratively tensoring with the bi-module from \cite{LOT1} corresponding to a negative Dehn twist along the meridian to obtain the maps $\phi_{dSV}:K_\infty \to K^+_n$ above.

It follows from Lemma~\ref{lem:dSV_map} that, upon taking the inverse limit, the induced map $\Phi_{dSV}$ is given by.
\begin{align*}
	\Phi_{dSV}: K_\infty &\to \underleftarrow{K}\\
	x &\mapsto I_1 \otimes \delta_0.
\end{align*}
Under the identification of $\SFHIL(-Y,K)$ and $\HFKP(-Y,K)$ given by Theorem~\ref{thm:lim_to_plus}, we see that $\Phi_{dSV}$ is precisely the analogue at the level of Type-$D$ modules of the natural map $\iota_*: \HFKH(-Y,K) \to \HFKP(-Y,K)$ induced by inclusion of complexes.

This finishes our sketch of the proof of Theorem~\ref{thm:SV_map2}.



\section{Gradings} 
\label{sec:gradings}

In this section, we show how to extend the proof of Theorem~\ref{thm:lim_to_minus}, presented in Section~\ref{sec:limit_is_knot} to include an identification of gradings.  To avoid unnecessary complications, we will assume in what follows that our ambient 3-manifold $Y$ is an integral homology sphere, though the results generalize to any three manifold and null-homologous knot.

\subsection{Alexander Grading} 
\label{sub:Alex_ident}

Let $K$ be a null-homologous knot in the 3-manifold $Y$, and let $\SH = (\Sigma,\balpha,\bbeta,z,w)$ be a doubly-pointed Heegaard diagram for the pair $(Y,K)$.  Recall from the discussion in Section~\ref{sub:knot_alex} that if $[F,\partial F]$ is a homology class of Seifert surface for the knot $K$, then one defines the Alexander grading of a generator $\x \in \SG(\SH)$ of $\CFKM(\SH)$ via the formula
\[
	A_{[F,\partial F]}(\x) = \frac{1}{2}\langle c_1(\s(\x),t_\mu), [F,\partial F] \rangle,
\]
which is then extended to all of $\CFKM(Y,K)$ via linearity and the relation
\[
	A_{[F,\partial F]}(U \cdot \x) = A_{[F,\partial F]}(\x) - 1.
\]

In Section~\ref{sub:alexander}, we extend the above Alexander grading to the sutured setting.  Namely, if $\SH_i$ is a sutured Heegaard diagram for the sutured manifold $(-Y(K),-\Gamma_i)$, obtained from the complement $Y(K)$ by and placing a pair of oppositely oriented sutures which run once longitudinally and $i$-times meridionally on the resulting boundary.  Then, to a generator $\x \in \SG(\SH_i)$ one assigns the Alexander grading
\[
	A_{[F,\partial F]}(\x) = \frac{1}{2}\langle c_1(\s(\x),t_\mu), [F,\partial F] \rangle,
\]
and extends linearly to all of $\SFC(\SH_i)$.

It was further observed in Section~\ref{sub:alexander} that, based on the relative first Chern class computations in Section~\ref{ssec:convexandspin}, the gluing maps $\phi_-$ and $\psi_+$ are homogeneous of Alexander degree plus and minus $1/2$ respectively.  In particular, the collections of maps
\[
	\phi_-: \SFH(-Y(K),-\Gamma_i)[(i - 1)/2] \to \SFH(-Y(K),-\Gamma_{i+1})[((i+1) - 1)/2],
\]
and
\[
	\psi_+: \SFH(-Y(K),-\Gamma_i)[(i - 1)/2] \to \SFH(-Y(K),-\Gamma_{i+1})[((i+1) - 1)/2],
\]
were all seen to be Alexander-homogeneous of degree $0$ and $-1$ respectively.  Taking the direct limit over the collection $\{\phi_-\}$, the sutured limit homology inherited a natural Alexander grading from the above formulae.

Our identification of these two Alexander gradings proceeds in two steps.  In the first, we show that the two gradings must agree up to an overall shift independent of $Y$ and $K$.  This is the content of Proposition~\ref{prop:Alex_shift}.  From here, it suffices to identify a single class $[\x] \in \SFHL(-Y,K)$ whose Alexander grading is preserved under the isomorphism given in Theorem~\ref{thm:lim_to_minus}.  This is accomplished in Proposition~\ref{prop:trans_alex} using Legendrian/transverse invariants.

\begin{proposition}\label{prop:Alex_shift}
Let $K$ be a null-homologous knot in a rational homology 3-sphere $Y$. Under the isomorphism given in Theorem~\ref{thm:lim_to_minus}, the Alexander gradings defined on $\SFHL(-Y,K)$ and $\HFKM(-Y,K)$ agree up to an overall shift.
\end{proposition}

\begin{proof}

Let $(-Y(K),-\Gamma_i)$ be as above, and consider the bordered sutured decomposition
\[
	(-Y(K),-\Gamma_i) = (-Y(K),\Gamma',\SF_T) \cup \T_i
\]
from Section~\ref{sub:the_geometric_setup}, where the parametrization on the common boundary $\SF_T$ is given by the pair of $\alpha$-arcs consisting of a meridian and a $0$-framed longitude.

To the bordered sutured manifold $(-Y(K),\Gamma',\SF_T)$, we associate the $A_\infty$-module $\CFA(-Y,K)$.  The complexes $\CFKM(-Y,K)$ and $\SFC(-Y(K),-\Gamma_i)$ can are obtained (up to homotopy) from $\CFA(-Y,K)$ by forming the box tensor products with the modules $\CFD^-(\SH_c)$ (from Section~\ref{sub:bordered_invariants_and_knot_floer_homology}) and $K_i$ (from Section~\ref{sub:torus_modules}), respectively.  In this setting, generators of $\CFKM(-Y,K)$ are all of the form $\y \otimes x$, while generators of $\SFC(-Y(K),-\Gamma_i)$ are either of the form $\y \otimes b_j$ or $\y' \otimes a$, where $\y$ and $\y'$ live in idempotents $I_1$ and $I_2$ respectively.

The Alexander grading computation can similarly be decomposed in either setting.  Namely,
\begin{align*}
	A_{[F,\partial F]}(m \otimes n) &= \frac{1}{2} \langle c_1(\s(m \otimes n)),[F, \partial F] \rangle\\
				&= \frac{1}{2} \langle c_1(\s(m)) \oplus c_1(\s(n)),[F, \partial F] \rangle\\
				&= \frac{1}{2} \langle c_1(\s(m)),[F, \partial F] \rangle + \frac{1}{2} \langle c_1(\s(n)),[A, \partial A] \rangle,
\end{align*}
where $A = \partial F \times I$ is an annular extension of the Seifert surface $F$ through $T^2 \times I$.

We recall that the Alexander grading on $\CFKM(-Y,K)$ is characterized by the above formula, together with the fact that multiplication by $U$ drops the grading by $1$.  Thus, the task of identifying the Alexander gradings, up to an overall shift is equivalent to showing that for generators $b_j$ and $b_\ell$ of $K_i$,
\[
	\langle c_1(\s(b_j)),[A, \partial A] \rangle - \langle c_1(\s(b_\ell)),[A, \partial A] \rangle = 2\cdot (j - \ell),
\]
since these elements represent the elements $U^j \cdot x$ and $U^\ell \cdot x$ in the limit.

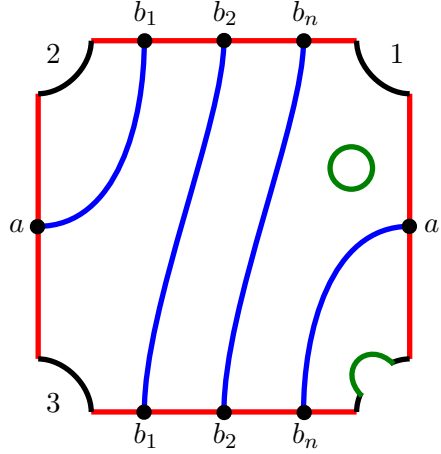
\begin{figure}[htbp]
	\centering
	\begin{picture}(150,160)
		\put(0,8){\includegraphics[scale=1.0]{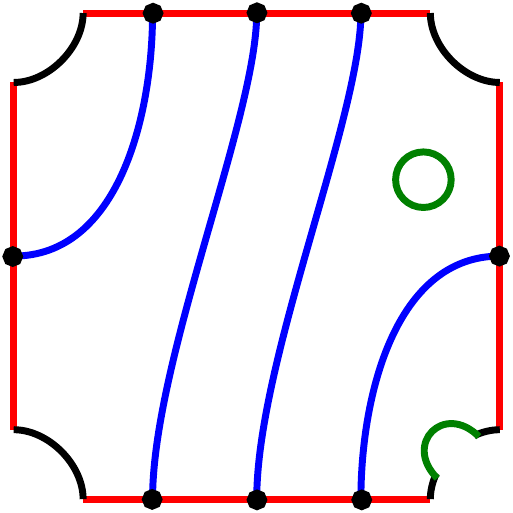}}
		\put(40,160){$b_1$}
		\put(70,160){$b_2$}
		\put(100,160){$b_n$}
		\put(40,0){$b_1$}
		\put(70,0){$b_2$}
		\put(100,0){$b_n$}
		\put(-7,80){$a$}
		\put(150,80){$a$}
		\put(7,12){$3$}
		\put(7,144){$2$}
		\put(137,144){$1$}
	\end{picture}
	\caption{An alternative view of the bordered sutured Heegaard diagram depicted in Figure~\ref{fig:torfilldiag3}.}
	\label{fig:TorusAltView}
\end{figure}

To see this, consider Figure~\ref{fig:TorusAltView}.  Here, we see an alternate view of the bordered Sutured Heegaard diagram originally depicted in Figure~\ref{fig:torfilldiag3}.  From this diagram, it immediately follows that for generators $b_j,b_\ell \in K_i$ intersecting the longitudinal $\alpha$ arc, $\epsilon(b_j,b_\ell) = (j-\ell)[\mu]$ and, in turn, that their associated $\spinc$-structures differ by $(j-\ell)\mr{PD}([\mu])$.  Therefore,
\begin{align*}
	\langle c_1(\s(b_j)),[A, \partial A] \rangle - \langle c_1(\s(b_\ell)),[A, \partial A] \rangle =& \langle c_1(\s(b_j) - \s(b_\ell))  ,[A, \partial A] \rangle\\
	=& \langle 2 \cdot \mr{PD}([\mu]), [A, \partial A] \rangle\\
	=& 2 \cdot (j - \ell).
\end{align*}
\end{proof}

We now turn to the task of showing that the Alexander gradings defined on $\SFHL(-Y,K)$ and on $\HFKM(-Y,K)$ agree on-the-nose, not just up to an overall shift.  As observed above, and in light of Proposition~\ref{prop:Alex_shift}, it suffices to demonstrate this equality for a single nontrivial element.  In fact, it suffices to prove this equality for {\it some} knot $K$ in {\it some} 3-manifold $Y$.  This is because,  as shown in the proof of Proposition~\ref{prop:Alex_shift}, any shift in Alexander grading can be computed strictly within the context of the Type-$D$ modules $K_i$ and $\CFD^-(\SH_c)$.

\begin{proposition}\label{prop:trans_alex}
There exists a knot $K$ contained in the 3-sphere $S^3$ and non-trivial elements $g \in \SFHL(-Y,K)$ and $h \in \HFKM(-Y,K)$ which are identified under the isomorphism $\Phi: \SFHL(-Y,K) \to \HFKM(-Y,K)$ such that $A(g) = A(h)$.
\end{proposition}

\begin{proof}
Let $K$ be a given knot type in $S^3$.  Consider the standard tight contact structure $\xi_{std}$ on $S^3$, and let $L$ be a Legendrian representative of the knot type $K$.  It was shown in Section~\ref{sub:relationships} that the Alexander grading of $\EH(L) \in \SFH(-Y(K),-\Gamma_{tb(L)})$ is precisely equal to the negative rotation number
\[
	A(\EH(L)) = -\frac{r(L)}{2}.
\]
Recall that the Thurston-Bennequin number is, by definition, the difference between the contact framing on $L$ and the Seifert framing.  This, in turn, is precisely equal to the slope of the sutures on the Legendrian knot complement.  Thus, in the sequence
\begin{center}
	\begin{tikzpicture}	[->,>=stealth',auto,thick]
		\node (a) at (0,0){$\SFH(-Y(K),-\Gamma_{-tb})[(-tb-1)/2]$} ;
		\node (b) at (6.95,0) {$\SFH(-Y(K),-\Gamma_{-tb - 1})[((-tb
		-1)-1)/2]$} ;
		\node (c) at (11.5,0) {$\dots$,} ;
		\draw (a) edge (b);
		\draw (b) edge (c);
	\end{tikzpicture}
\end{center}
we see immediately that $A(\EHL(L)) = (tb(L) - r(L) + 1)/2$, agreeing the corresponding value for the Alexander grading of the LOSS invariant $\SL(L)$. 

To finish the proof, it suffices to check that either $\EHL(L)$ or $\SL(L)$ is non-zero.  This follows immediately from the fact that the invariants $\EHL(L)$ is identified with contact invariant $c(S^3,\xi_{std})$ under the map $\Phi_{2h}: \SFHL(-Y,K) \to \HFH(-Y)$, and the latter contact invariant is nontrivial. 
\end{proof}


\subsection{Maslov Grading} 
\label{sub:Maslov_ident}

We now turn our attention to understanding the homological or ``Maslov'' grading in sutured limit homology.  Specifically, we aim to show that the $\F[U]$-module $\SFHL(-Y,K)$ inherits an absolute $\Z/2$-grading which can be canonically identified with the usual absolute $\Z/2$-grading in knot Floer homology.

With this goal in mind, we begin by recalling the following useful fact, which characterizes the behavior of the absolute $\Z/2$-grading on sutured Floer homology under the gluing maps defined by Honda, Kazez and Mati\'c.

\begin{theorem}[Honda-Kazez-Mati\'c \cite{HKM3}, Gripp-Huang \cite{GH}]\label{thm:HKM_grad}
Let $(Y_1,\Gamma_1)$ and $(Y_2,\Gamma_2)$ be balanced sutured 3-manifolds such that $Y_1 \subset Y_2$, and let $\xi$ be a contact structure on $Y_2 \backslash \mr{int}(Y_1)$ with sutured contact invariant $\EH(\xi)$.  Then the Honda-Kazez-Mati\'c gluing map, on the chain level, is homogeneous of degree $\mr{gr}(\EH(\xi))$
\[
\phi_\xi: \SFC(-Y_1,-\Gamma_1) \to \SFC(-Y_2,-\Gamma_2)[\mr{gr}(\EH(\xi))],
\]
and descends to a degree $\mr{gr}(\EH(\xi))$ map on the homology level. 
\end{theorem}

Although the above result is not explicitly stated in \cite{HKM3}, the result is implicit in the proof of the main result of that paper (stated here as Theorem~\ref{thm:hkm_gluing}) --- that well-defined, natural contact gluing maps exist in the sutured category.  It's truth can be derived from results of Gripp and Huang \cite{GH} characterizing the absolute gradings in Heegaard Floer theory in terms of homotopy classes of vector fields.  

To better understand the context of the above result, consider the following variant of Honda, Kazez and Mati\'c's construction.  Start with two balanced sutured manifolds $(Y_1,\Gamma_1)$ and $(Y_2,\Gamma_2)$, each of which can be equipped with contact structures inducing the specified suture sets on their boundary.  Now suppose that $\Sigma_i$ are sutured subsurfaces (with dividing sets) of the boundaries $\partial Y_i$, which are compatible in the sense that $\Sigma \cong \Sigma_1 \cong -\Sigma_2$. Then, it is possible to glue together the two sutured manifolds along the $\Sigma_i$ to form a new balanced suture manifold
\[
	(Y,\Gamma) = (Y_1,\Gamma_1) \cup_{\Sigma} (Y_2,\Gamma_2).
\]
In this setting, Honda, Kazez, and Mati\'c's gluing theorem states that there exists a map
\[
	\phi_{\Sigma}: \SFC(-Y_1,-\Gamma_1) \otimes \SFC(-Y_2,-\Gamma_2) \to \SFC(-Y,-\Gamma)
\]
which is obtained as an inclusion of complexes, and which is homogeneous of degree zero.  Furthermore, if the sutured manifold $(Y_2,-\Gamma_2)$ is equipped with the compatible contact structure $\xi_2$, then the usual gluing map is given by
\[
	\phi_{\xi_2}(\, \cdot\, ) = \phi_{\Sigma}(\, \cdot \otimes \EH(\xi_2)),
\]
viewed as a map from $\SFC(-Y_1,-\Gamma_1)$ to $\SFC(-Y,-\Gamma)$.

With the above result in mind, we turn to the problem at hand --- determining an absolute $\Z/2$-grading on the sutured limit homology groups.  We begin with the following useful observation.

\begin{proposition}\label{prop:maslov_basic_slice}
The absolute $\Z/2$-grading of the contact invariant associated to either a positive or negative basic slice is zero.
\end{proposition}

\begin{proof}
Recall from Section~\ref{sub:tt} that if $B_\pm = (T^2 \times I,\xi_\pm)$ is either a positive or negative basic slice, then it contains a convex torus $T$ which decomposes $B_\pm$ into two basic slices of the same (original) sign.

This decomposition allows us to compute the absolute grading of $\EH(T^2 \times I,\xi)$ by applying Theorem~\ref{thm:HKM_grad}.  Specifically, we have
\begin{align*}
	\mr{gr}(\EH(B_\pm)) &= \mr{gr(\EH(B_\pm \cup_T B_\pm))}\\
	&= \mr{gr}(\EH(B_\pm)) + \mr{gr}(\EH(B_\pm))\\
	&= 0 \; (\mr{mod}\, 2).
\end{align*}
\end{proof}

It follows from Proposition~\ref{prop:maslov_basic_slice}, together with the above result of Honda, Kazez and Mati\'c, that the maps $\phi_i$ which give rise to sutured limit homology are necessarily all Maslov-homogeneous of degree 0.  In turn, we see that the sutured limit homology group $\SFHL(-Y,K)$ inherits an absolute $\Z/2$-grading, which we refer to as the {\it Maslov grading}.

\begin{theorem}\label{thm:maslov_ident}
Under the isomorphism $\Phi: \SFHL(-Y,K) \to \HFKM(-Y,K)$, given by Theorem~\ref{thm:lim_to_minus}, the absolute $\Z/2$-gradings on $\SFHL(-Y,K)$ is identified with the absolute $\Z/2$-grading on $\HFKM(-Y,K)$.
\end{theorem}

\begin{proof}
Recall that Theorem~\ref{thm:SV_map} states that, under the isomorphism $I_-: \SFHL(-Y,K) \to \HFKM(-Y,K)$, the Stipsicz-V\'ertesi map $\Phi_{SV}$ is identified with the canonical map on knot Floer homology which is induced by setting the formal variable $U$ equal to zero at the chain level.  That is to say, the following diagram commutes:
\begin{center}
\begin{tikzpicture}	[->,>=stealth',auto,thick]
	\node (a) at (0,0){$\SFHL(-Y,K)$} ;
	\node (b) at (4,0) {$\HFKM(-Y,K)$} ;
	\node (e) at (2,-2) {$\HFKH(-Y,K)$.} ;

	\draw (a) edge node[above] {\small $I_-$} (b);
	\draw (a) edge node[left] {\small $\Phi_{\mathrm{SV}}$} (e);
	\draw (b) edge node[right] {($p_*$)} (e);
\end{tikzpicture}
\end{center}
Moreover, the maps $I_-$, $\Phi_{\mathrm{SV}}$ and $p_*$ are all defined on the chain level and on the chain level fit into the analogous commutative diagram. Recall that $\CFKM(-Y,K)$ is generated by elements that map non-trivially to $\CFKH(-Y,K)$ and their images under $U$. Since we know the effect of $U$ on grading is well understood we need to see that for elements that map non-trivially to  $\CFKH(-Y,K)$ their grading in $\CFKM(-Y,K)$ and in the chain group computing $\SFHL(-Y,K)$ are the same. 

Focusing first on the knot Floer homology side of this story, recall that the ``$U=0$ map'' $\CFKM(-Y,K) \to \CFKH(-Y,K)$ preserves the absolute $\Z/2$-grading.  

In a similar spirit, recall that the Stipsicz-V\'ertesi map is induced on $\SFHL(-Y,K)$ by via a collection of basic slice attachments to the set of sutured knot complements $\{(-Y(K),-\Gamma_i)\}$.  Topologically, each Stipsicz-V\'ertesi attachment yields the sutured manifold $(-Y(K),-\Gamma_\mu)$.  The sutured Floer homology of this space is isomorphic to $\HFKH(-Y,K)$, via an isomorphism that preserves absolute grading.  These facts, in conjunction with Proposition~\ref{prop:maslov_basic_slice}, shows that the map $\Phi_{SV}$ is necessarily homogeneous of degree zero, finishing the proof of Theorem~\ref{thm:maslov_ident}.
\end{proof}



\section{Examples} 
\label{sec:examples}

In this section, we present some examples which highlight distinctions between the various Legendrian and transverse invariants defined and discussed in this paper.

Recall that there are essentially three flavors of Legendrian or transverse invariants under consideration in this paper: HKM, LOSS, and LIMIT.  Both the HKM and LIMIT invariants are defined geometrically and correspond to contact invariants naturally associated to a given Legendrian or transverse knot.  The LOSS invariants, on the other hand, are defined via compatible open book decompositions, thus obscuring obvious connections between the invariants and the ambient geometry.

Before delving into the examples, recall that we have correspondences identifying some of the Legendrian and transverse invariants discussed above.  Specifically, Theorem~\ref{thm:lim_to_minus_leg} asserts that the LOSS minus invariant is identified with the ``direct'' LIMIT invariant under the isomorphism given by Theorem~\ref{thm:lim_to_minus}.  Stipsicz and V\'ertesi also provided a geometric interpretation of the LOSS hat invariant as the contact invariant of a contact manifold (with convex boundary) canonically associated to a Legendrian or transverse knot (see Section~\ref{sub:relationships}).

There is a fourth set of invariants of Legendrian and transverse knots in the standard contact 3--sphere $(S^3,\xi_{std})$, which is defined combinatorially by Ozsv\'ath, Szab\'o and Thurston via grid diagrams, and which are referred to as the GRID invariants \cite{OSzT}.  It was shown by the second author, in joint work with Baldwin and V\'ertesi \cite{BVV}, that these invariants agree with the LOSS invariants where the two are simultaneously defined.

\subsection{Comparing the HKM and LOSS invariants} 
\label{sub:comp_hkm_loss}

Here, we compare the HKM and LOSS invariants by providing an example of a Legendrian knot whose HKM invariant is non-zero, but whose LOSS invariants vanish.

The example is a non-loose Legendrian unknot $U_{OT}$ in an overtwisted contact structure $\xi$ on $S^3$.  The Legendrian knot $U_{OT}$ is shown in Figure~\ref{fig:HKMLOSSminEx} as an essential embedded curve on the open book decomposition $(A,D_\gamma^-)$ for $S^3$.  The pages of the open book decomposition $(A,D_\gamma^-)$ are annuli, and the monodromy map is given by a single negative Dehn twist along a core curve $\gamma$.  Figure~\ref{fig:HKMLOSSminEx} depicts the associated multi-pointed Heegaard diagram associated to the Legendrian knot $U$, used to compute the LOSS invariants.

\begin{figure}[htbp]
	\centering
	\begin{picture}(205,135)
		\put(0,0){\includegraphics[scale=1.25]{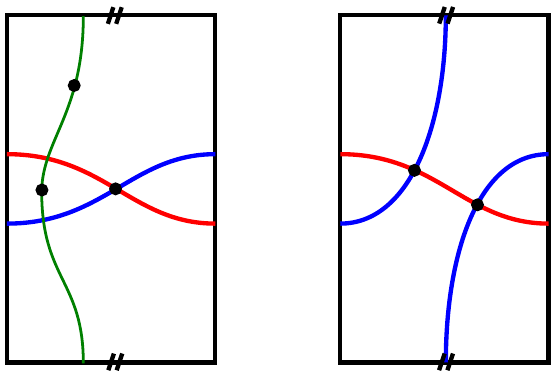}}
		\put(-20, 10){$S_{1/2}$}
		\put(100, 10){$-S_0$}
		\put(42, 74){$x$}
		\put(144, 80){$a$}
		\put(168, 67){$b$}
		\put(30, 25){$U$}
		\put(20, 65){$w$}
		\put(32, 103){$z$}
	\end{picture}
	\caption{A non-loose Legendrian unknot $U_{OT}$ with non-vanishing HKM invariant and vanishing LOSS invariant in an overtwisted contact 3--sphere.}
	\label{fig:HKMLOSSminEx}
\end{figure}
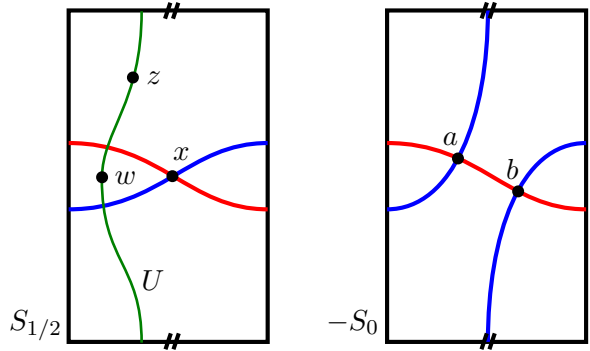

We see immediately that the Legendrian $U_{OT}$ is indeed an unknot.  To see that its corresponding LOSS invariants vanish, observe that they are represented in $\HFKM(-S^3,U)$ and $\HFKH(-S^3,U)$ by the intersection point $x$ in Figure~\ref{fig:HKMLOSSminEx}.  We immediately see that $\partial b = x$ in either case.  Thus, the class $[x]$ vanishes in homology.

On the other hand, we can see that the HKM invariant of $U_{OT}$ is non-vanishing as follows.  Stipsicz and V\'ertesi showed that the HKM invariant can be computed form the open book decomposition $(A,D_\gamma^-)$ by removing an open tubular neighborhood of the curve $U_{OT}$, and deleting the $\alpha$ and $\beta$-curves it intersects.  The result is a sutured Heegaard diagram consisting of a topological disk and empty sets of $\alpha$ and $\beta$-curves.  The contact invariant is represented by the empty set of intersections, which is non-zero in homology.  If the reader prefers, this same non-vanishing result can be shown by positively stabilizing the open book decomposition $(A,D_\gamma^-)$ along a boundary-parallel arc to ensure the existence of a non-trivial set of intersections after removing a neighborhood of $U_{OT}$. (Yet another way to see this invariant is non-zero is to note the knot is a non-loose knot and so its complement is tight. Any tight contact structure on a solid torus has non-zero contact invariant.)

To show Golla's characterization of $EH(K)$ in the tight contact structure on $S^3$ discussed in Section~\ref{sub:aux_quest} does not hold in general we note that there is a loose unknot $U'$ in the same contact structure as $U_{UT}$ that has the same classical invariants as $U_{OT}$. We just argued that $\EH(U_{OT})$ is non-zero but it is clear that $\EH(U')$ is zero. Moreover from the above arguments one may easily conclude that $\SL(U_{OT})=\SL(-U_{OT})=0=\SL(U')=\SL(-U')$.


\subsection{Comparing the LOSS invariants} 
\label{sub:comp_loss_invts}

Relationships and differences between the LOSS minus and hat invariants are well-documented, and we refer the interested reader to the original paper \cite{LOSS} by Lisca, Ozsv\'ath, Stipsicz and Szabo for more information.

Here, we simply remark that examples of Legendrian or transverse knots for which the minus version of the LOSS invariant is non-vanishing, while the hat invariant vanishes are easy to generate.  Indeed, recall that the minus version of the LOSS invariant is identified with the contact invariant of the ambient manifold under the natural map
\[
	\pi_*: \HFKM(-Y,K) \to \HFH(-Y),
\]
induced by setting the variable $U$ equal to the identity at the chain level.  Thus, if $(Y,\xi)$ is any contact 3--manifold with non-vanishing contact invariant $c(Y,\xi) \neq 0$, then for any null-homologous Legendrian (resp. transverse) knot $K \subset (Y,\xi)$, we have that $\SL(K) \neq 0$ (resp. $\ST(K) \neq 0$).  

On the other hand, if $K' \subset (Y,\xi)$ is any positively stabilized Legendrian (resp. transverse) knot, then $\SLH(K') = 0$ (resp. $\STH(K') = 0$).

Since the standard contact 3--sphere $(S^3,\xi_{std})$ satisfies the condition that $c(S^3,\xi_{std})$ does not vanish, we see that the desired examples exist in abundance.  In fact, using the isomorphism given in \cite{BVV} relating the LOSS and GRID invariants, one can produce a multitude of non-destabilizable examples satisfying the same vanishing properties --- the Etnyre-Honda $(2,3)$-cable of the $(2,3)$-torus knot $T_{EH}$ \cite{EH3} is such an example by a result of Ng, Ozsv\'ath and Thurston \cite{NOT}.


\subsection{Comparing the LOSS and LIMIT invariants} 
\label{sub:comp_loss_limit}

We now turn to the task of comparing the LOSS and LIMIT invariants.  In light of Theorem~\ref{thm:lim_to_minus_leg}, and the discussion in Section~\ref{sub:comp_loss_invts} above, we focus on understanding differences in information content between the hat version of the LOSS invariant and the ``inverse'' version of the LIMIT invariant. 

Recall from Section~\ref{sub:definition_inverse_leg_invt} that if $K$ is a null-homologous Legendrian or transverse knot in a contact 3--manifold $(Y,\xi)$, then there exists a Legendrian (resp. transverse) invariant $\EHIL(K)$ taking values in the sutured inverse limit homology group $\SFHIL(-Y,K)$.  According to Theorem~\ref{thm:lim_to_plus}, the latter group is isomorphic to the plus version of knot Floer homology
\[
	I_+: \SFHIL(-Y,K) \to \HFKP(-Y,K).
\]
By Theorem~\ref{thm:lim_to_plus_leg}, the inverse limit invariant $\EHIL(K)$ can be identified with the image of the hat version of the LOSS invariant under the natural map
\[
	\iota_*: \HFKH(-Y,K) \to \HFKP(-Y,K),
\]
induced by inclusion.

With the above in mind, consider the following example from \cite{LOSS} of a non-loose Legendrian $(2,3)$-torus knot $T_{(2,3)}$ in the overtwisted contact structure $\xi$ on $S^3$ with Hopf invariant $d_3(\xi) = -1$.  The knot is depicted in Figure~\ref{fig:LOSSLimitEx}.

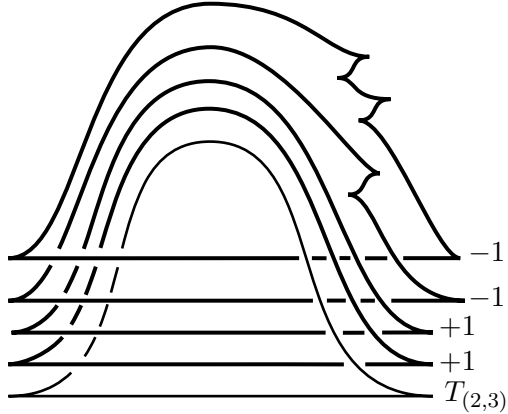
\begin{figure}[htbp]
	\centering
	\begin{picture}(180,157)
		\put(0,0){\includegraphics[scale=1.0]{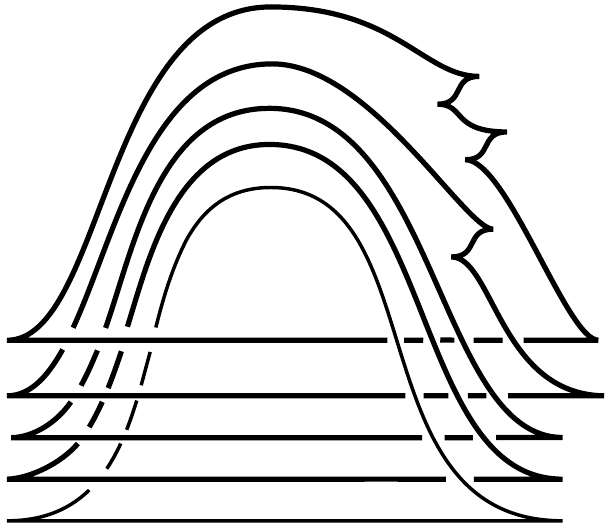}}
		\put(167, 0){$T_{(2,3)}$}
		\put(165, 12){$+1$}
		\put(165, 25){$+1$}
		\put(176, 36){$-1$}
		\put(176, 53){$-1$}
	\end{picture}
	\caption{A non-loose Legendrian right-handed trefoil $T_{(2,3)}$ in the overtwisted contact structure $\xi$ on $S^3$ with Hopf-invariant $d_3(\xi) = -1$.}
	\label{fig:LOSSLimitEx}
\end{figure}

In \cite{LOSS}, Lisca, Ozsv\'ath, Stipsicz and Szab\'o show that $\SLH(T_{(2,3)}) \neq 0$, and identify the specific class in $\HFKH(-S^3,T_{(2,3)}) \cong \HFKH(S^3,T_{(2,-3)})$ representing $\SLH(T_{(2,3)})$ --- the unique non-zero class in Alexander grading zero.

Figure~\ref{fig:LHTcomplex} depicts the knot Floer chain complex $(\CFKI(T_{(2,-3)}),\partial^\infty)$ associated to the left-handed trefoil knot $T_{(2,-3)}$.  In the drawing, the vertical $j$-axis records Alexander grading, while the horizontal $i$-axis keeps track of the (negative) $U$-power.  In this context, we view $\CFKM(T_{(2,-3)})$, $\CFKP(T_{(2,-3)})$, and $\CFKH(T_{(2,-3)})$ as the sub, quotient and sub-quotient complexes $C(i \leq 0)$, $C(i \geq 0)$, and $C(i=0)$ respectively.

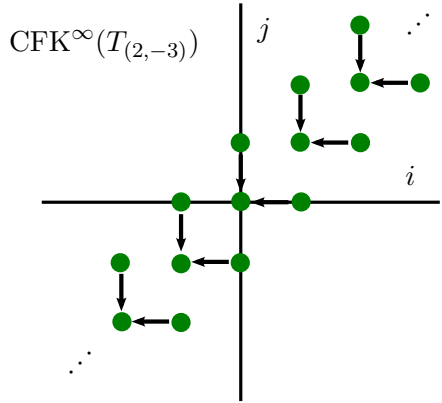
\begin{figure}[htbp]
	\centering
	\begin{picture}(160,160)
		\put(0,0){\includegraphics[scale=1.0]{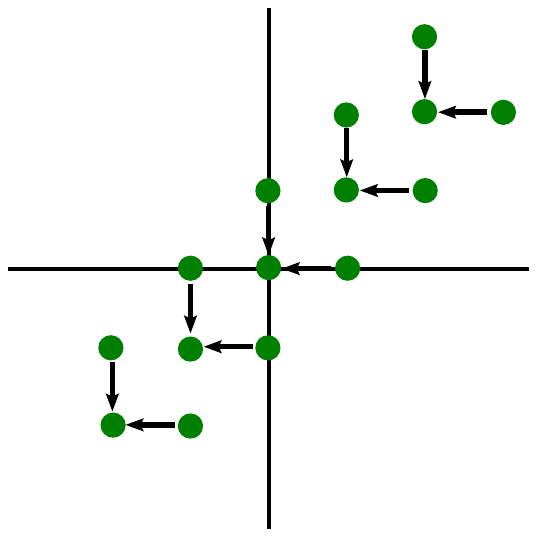}}
		\put(84, 140){$j$}
		\put(140, 84){$i$}
		\put(140,140){\rotatebox{45}{$\dots$}}
		\put(12,12){\rotatebox{45}{$\dots$}}
		\put(-10, 135){$\CFKI(T_{(2,-3)})$}
	\end{picture}
	\caption{The knot Floer complex associated to the left-handed trefoil $T_{(2,-3)}$.}
	\label{fig:LHTcomplex}
\end{figure}

After forming the associated graded objects with respect to the Alexander filtration, we see immediately from Figure~\ref{fig:LHTcomplex}, that the image of $\SLH(T_{(2,-3)})$ under the the map $\iota_*: \HFKH(T_{(2,-3)}) \to \HFKP(T_{(2,-3)})$, induced by inclusion, vanishes.  It follows, therefore, form Theorem~\ref{thm:lim_to_plus_leg} that $\EHIL(T_{(2,3)}) = 0$.

To see that the invariant $\EHIL$ is not always zero, consider the maximum Thurston-Bennequin invariant unknot in the standard contact 3--sphere $U \subset (S^3,\xi_{std})$.  To compute the inverse limit invariant, we begin by performing a Stipsicz-V\'ertesi basic slice attachment to the boundary the complement of an standard (open) tubular neighborhood of $U$.  The result is a tight contact structure on a solid torus $(S^1 \times D^2,\xi)$, inducing two longitudinal dividing curves along the boundary.  It follows from the classification of tight contact structures on solid tori \cite{Honda00a} that any negative basic slice attachment that does not achieve the meridional slope on $S^1 \times D^2$ induces a tight contact structure on the solid torus which embeds into a Stein-fillable contact structure on a lens-space.  In particular, the contact invariants of these spaces are all non-zero, which implies that the inverse limit invariant $\EHIL(U)$ is non-vanishing.

\subsection{Vanishing slope computations} 

We note that one may easily compute the vanishing slopes, discussed in Section~\ref{sub:vanishing_slopes}, for the knots considered in this section to be:
\begin{align*}
	\mr{Van}^-(U_{OT}) &= (0,0), \;\;\;\;\;\; \mr{Van}^-(T_{EH}) = (0,-\infty),\\
	\mr{Van}^-(T_{(2,3)}) &= (0,-\infty), \;\;\;\;\;\, \mr{Van}^-(U) = (-1,0).
\end{align*}



\bibliographystyle{plain}
\nocite{*}
\bibliography{LimitInvts}

\end{document}